\documentclass[aap]{imsart}
\pdfoutput=1
\RequirePackage{amsthm,amsmath,amsfonts,amssymb}
\RequirePackage[numbers]{natbib}
 \usepackage{qtree}
 \usepackage{hyperref}
\usepackage{enumerate}
 \usepackage{paralist}
\usepackage{mathtools}

\startlocaldefs

\theoremstyle{Summary}
\newtheorem{summary}{Summary}

\newtheorem{lemma}{Lemma}[section]
\newtheorem{proposition}[lemma]{Proposition}

\newtheorem{theorem}[lemma]{Theorem}

\newtheorem{condition}{Condition}

\newtheorem{remark}{Remark}
\newtheorem{example}{Example}[section]

\newtheorem{defn}{Definition}[section]
\newcommand{\R}[0]{\mathbb{R}}

\newcommand{\N}[0]{\mathbb{N}}

\newcommand{\sF}[0]{\mathcal{F}}

\newcommand{\sA}[0]{\mathcal{A}}

\newcommand{\sC}[0]{\mathcal{C}}
\newcommand{\sD}[0]{\mathcal{D}}

\newcommand{\NN}{\mathbb{N}}
\newcommand{\AAA}{\mathbb{A}}
\newcommand{\II}{\mathbb{I}}
\newcommand{\RR}{\mathbb{R}}

\newcommand{\cle}{\mathcal{E}}
\newcommand{\clh}{\mathcal{H}}
\newcommand{\clf}{\mathcal{F}}
\newcommand{\clg}{\mathcal{G}}

\newcommand{\clc}{\mathcal{C}}

\newcommand{\clz}{\mathcal{Z}}

\newcommand{\cly}{\mathcal{Y}}

\newcommand{\clu}{\mathcal{U}}

\newcommand{\clm}{\mathcal{M}}
\newcommand{\cla}{\mathcal{A}}
\newcommand{\ch}{\check}
\newcommand{\Om}{\Omega}

\newcommand{\Ups}{\Upsilon}
\newcommand{\chij}{\chi_J}
\newcommand{\vsg}{\varsigma}
\newcommand{\yr}{\mathfrak{b}^r}
\newcommand{\y}{\mathfrak{b}}
\newcommand{\id}{\mbox{id}}

\endlocaldefs

\begin{document}

\begin{frontmatter}
\title{Simple Form Control Policies for Resource Sharing Networks with HGI Performance}
%\title{A sample article title with some additional note\thanksref{T1}}
\runtitle{Control Policies for HGI Performance}
%\thankstext{T1}{A sample of additional note to the title.}

\begin{aug}
%%%%%%%%%%%%%%%%%%%%%%%%%%%%%%%%%%%%%%%%%%%%%%%
%% Only one address is permitted per author. %%
%% Only division, organization and e-mail is %%
%% included in the address.                  %%
%% Additional information can be included in %%
\author[A]{\fnms{Amarjit}~\snm{Budhiraja}\ead[label=e1]{budhiraj@email.unc.edu}}
\and
\author[B]{\fnms{Dane}~\snm{Johnson}\ead[label=e2]{djohnson52@elon.edu}}
% \and
% \author[B]{\fnms{???}~\snm{???}\ead[label=e3]{???@???}}
%%%%%%%%%%%%%%%%%%%%%%%%%%%%%%%%%%%%%%%%%%%%%%
%% Addresses                                %%
%%%%%%%%%%%%%%%%%%%%%%%%%%%%%%%%%%%%%%%%%%%%%%
\address[A]{Department of Statistics and Operations Research,
University of North Carolina,
Chapel Hill, NC 27599, USA.\printead[presep={,\ }]{e1}}

\address[B]{Department of Mathematics and Statistics,
Elon University,
 Elon, NC 27244, USA.\printead[presep={,\ }]{e2}}
\end{aug}

\begin{abstract}
We consider a family of resource sharing networks, known as bandwidth sharing models, in heavy traffic with general service and interarrival times. These networks, introduced in Massouli\'{e} and Roberts (2000) as models for internet flows, have the feature that a typical job may require simultaneous processing by multiple resources in the network. We construct simple form threshold policies that asymptotically achieve the Hierarchical Greedy Ideal (HGI) performance.
This performance benchmark, which was introduced in Harrison et al. (2014), is characterized by the following two features: every resource works at full capacity whenever there is work for that resource in the system;  total holding cost of  jobs of each type at any instant is the minimum cost possible for the associated vector of workloads. The control policy we provide is explicit in terms of a finite collection of vectors which can be computed offline by solving a system of linear inequalities. Proof of convergence is based on path large deviation estimates for renewal processes, Lyapunov function constructions, and analyses of suitable sample path excursions.
\end{abstract}


\end{frontmatter}
%%%%%%%%%%%%%%%%%%%%%%%%%%%%%%%%%%%%%%%%%%%%%%
%% Please use \tableofcontents for articles %%
%% with 50 pages and more                   %%
%%%%%%%%%%%%%%%%%%%%%%%%%%%%%%%%%%%%%%%%%%%%%%
%\tableofcontents

%%%%%%%%%%%%%%%%%%%%%%%%%%%%%%%%%%%%%%%%%%%%%%
%%%% Main text entry area:
\section{Introduction}
In Harrison et al. (2014) \cite{harmandhayan} the authors formulated a challenging open problem of constructing simple form control policies, for Massouli\'{e} and Roberts \cite{masrob} Resource Sharing Networks (RSN) in heavy traffic, that asymptotically achieve the so called Hierarchical Greedy Ideal performance. In a recent work \cite{budjoh} we had made partial progress towards this goal. In the  current work we give a much more general treatment that removes some of the main restrictive assumptions of \cite{budjoh} and also provides a substantially more intuitive and easier to implement control policy.

We begin by describing the network model and the basic problem from \cite{harmandhayan}.
% ust like our previous paper Dupuis and Johnson (2020) \cite{budjoh}, this is an approach to addressing the open problem formulated in Harrison et al. (2014) \cite{harmandhayan} for Resource Sharing Networks (RSN) that were introduced in the work of Massouli\'{e} and Roberts \cite{masrob} as models for internet flows. 
Resource sharing networks of the form considered in this work, which were originally proposed as models for internet flows \cite{masrob} (the terminology is due to \cite{verloop2005stability}), are  quite general processing systems that have the distinguishing feature that a typical job may require simultaneous processing  by multiple resources.  
We remark that they are a subset of the collection of stochastic networks introduced in \cite{har2}
and have the key feature that they are `one pass' systems with no feedback.
A  RSN considered in this work consists of $I$ resources (labeled $1, \ldots , I$) where the $i$-th resource has processing rate capacity $C_i$, $i = 1, \ldots , I$.
Jobs of type $1, \ldots , J$ arrive according to independent renewal processes with distributions depending on the job-type.  The job-sizes of different job-types are iid with distribution (depending on the type)  supported on the positive real line. We make usual assumptions on mutual independence. In general a job  will require simultaneous processing by several network resources. This  is made precise through a $I\times J$ incidence matrix $K$ for which $K_{i,j}= 1$ if $j$-th job-type requires processing by resource $i$ and $K_{i,j}=0$  otherwise. The processing of a job is accomplished by allocating to it the same instantaneous {\em flow rate}
 by all the resources  responsible for its simultaneous processing, and a job departs from the system when the integrated flow rate equals the size of the job. Throughout this work the term Resource Sharing Network (RSN) will refer to a network of the above form.
If at any time instant $\y = (\y_1, \ldots , \y_J)'$  is the nonnegative vector of flow rates allocated to various job-types, then $\y$ must satisfy the capacity constraint $K\y \le C$, where $C= (C_1, \ldots , C_I)'$.  There is a holding cost per unit time which is a linear function of the queue length.  Specifically, we let $h_j>0$ denote the cost per unit time per job for the $j$-th job type.

Optimal resource allocation problems for such networks are in general hard and intractable. A common approach when the system is critically loaded is through certain diffusion approximations that replace the original allocation problem by a certain Brownian control problem (BCP) (cf. \cite{har1}). Although in special cases such BCP can be explicitly solved, in general closed form solutions are not available and using this approach for constructing asymptotically optimal simple form allocation policies for general RSN becomes challenging. To address this, in \cite{harmandhayan}, authors propose to focus on a less demanding goal than constructing asymptotically optimal policies, which is to construct control policies that achieve Hierarchical Greedy Ideal(HGI) performance in the heavy traffic limit. Roughly speaking, HGI performance is the (in general sub-optimal) cost for a control in the BCP which has the following two features: (a) {\em No idleness:} Every resource works at full capacity whenever there is work for that resource in the system; (b) {\em Instantaneous holding cost minimization:} Total holding cost of  jobs of each type at any instant is the minimum cost possible for the associated vector of workloads.
% {\color{red} HGI perfomance can be thought of as achieving optimal asymptotic cost among schemes which are restricted to what Harrison referred to as "reversible displacement" queue length movements (see \cite{har1}).  These movements are reversible in that they maintain the same resource capacity usage (or equivalenty maintain the same workload which is discussed in Section \ref{sec:propcon}), as opposed to queue length movements which underuse capacity whose reversal would require exceeding the capacity constraints.}\footnote{I hope these comments provide some understanding of the HGI goal.  When we're pursuing asymptotic optimality and the cost is not a monotonic function of the workload the concern is that we'll regret a short term cost reduction that cannot be reversed and will ultimately create a higher cost.  This may be relevant to R1R2 item 5, but I'm not sure if that item is about explaining HGI or discussing how the paper proves HGI performance for this algorithm. Feel free to remove this if you don't like it - DJ} 
Good performance of HGI control policies and comparison with other types of allocation policies, e.g. proportional fairness policies\cite{kelly1997charging, kelly1998rate, mo2000fair}
has been discussed in detail in \cite{harmandhayan} through simulations and numerical examples. This paper also put forward the challenging open problem of constructing simple form control policies for broad families of RSN that achieve HGI performance in the heavy traffic limit. 

In our recent work\cite{budjoh}  we made partial progress towards this open problem. Under a set of conditions on the topology of the RSN and system parameters (arrival and service rates, holding costs) the paper \cite{budjoh} constructed a sequence of threshold-type control policies, which, when the interarrival and service times are {\em exponentially distributed}, was proved to  asymptotically achieve HGI. One of the key assumptions imposed in \cite{budjoh} is the existence of a certain {\em ranking map} which identified a suitable form of priority among the various job-types.  This paper also provided some tractable sufficient conditions under which such a ranking map exists for a given network, nevertheless this is a restrictive condition and it is easy to produce examples of networks where this condition fails  (see for instance \cite[Example 6.6]{budjoh}).  Furthermore, in general constructing an explicit ranking map can be challenging.

In the current paper we revisit \cite{harmandhayan} and give a fairly complete solution to the open problem formulated there.  In particular we remove the key restrictive condition on the existence of a ranking map that was imposed in \cite{budjoh} and also allow for non-exponential interarrival and job-sizes. We impose a standard heavy traffic and  a stability condition (Condition \ref{cond:htc}), and a local traffic condition (Condition \ref{cond:locTraffic}). Both of these conditions were also assumed in \cite{harmandhayan} and \cite{budjoh}. The latter condition, first introduced in \cite{kankelleewil} is needed in order to ensure that the state space of the {\em workload process} is all of the positive orthant. Finally, although interarrival and job-sizes are not required to be exponential, we require them to satisfy a suitable exponential integrability condition (Condition \ref{eqn:mgfBnd}). Under these three conditions we introduce in Definition \ref{def:scheme} a sequence of control policies and prove that these policies achieve the HGI in the heavy traffic limit.  Our two main results are Theorem \ref{thm:disccost} and 
Theorem \ref{thm:ergcost}. The first proves the convergence to the HGI when the underlying cost is an infinite horizon discounted cost while the second shows the same result for the long-time cost per unit time (ergodic cost).
Specifically, HGI performance in the discounted cost setting is defined as the expectation of a functional of a reflected Brownian motion (with drift) in $\RR_+^I$ with normal reflections (namely the first expression in \eqref{eq:hgicosts}) while in the ergodic cost setting it is given as the expectation under the unique stationary distribution of the same reflected Brownian motion (namely the second expression in \eqref{eq:hgicosts}). The fact that these functionals involve the minimizer $\hat h$ defined in \eqref{eq:hhatdef}  captures the feature of instantaneous holding cost minimization, and the fact that the limit process is a reflected Brownian motion in $\RR_+^I$, which has the feature that the  reflection (which roughly corresponds to the  asymptotic idleness) occurs only when the process hits the faces of the orthant, captures the no-idleness property of HGI.

In addition to removing undesirable and restrictive assumptions, the other main contribution of this work is to the form of the control policy which we believe is significantly more intuitive and easy to implement than the policy presented in \cite{budjoh}. The control policy that we introduce in this work is given in terms of explicit thresholds determined from system parameters, and in addition requires the evaluation, for each $z\in\{0,1\}^J$, of $J$-dimensional vectors $v^b(z)$ and $v^c(z)$ (see below Proposition \ref{thm:fullCapVect}).  The evaluation of these  vectors can be done offline. Specifically, determining $v^b(z)$ for each fixed $z$ requires solving the inequalities $Av = 0$ and $Bv>0$ where $A$ and $B$ are given matrices (depending on $z$) with dimensions $r\times J$ and $s\times J$ where $r\le I$, $s\le J$. Similarly determining $v^c(z)$ for each fixed $z$ requires solving the inequalities $Av=0$, $Bv\ge 0$, $v\cdot f>0$ (or determining that no such $v$ exists) where $A$ and $B$ are given matrices of similar form and $f$ is a given vector in $\RR^J$. Once the vectors $v^c(z), v^b(z)$ and the thresholds are determined, the policy takes a simple explicit form which can be described by a single line (see \eqref{eq:singleline}) as opposed to the half-page description of the policy constructed in \cite{budjoh}.
Roughly speaking the policy works as follows. Consider a vector $z \in \{0,1\}^J$ representing the state of the system at some given time instant. An entry of $0$ in the vector $z$ means that the corresponding  job-type's queue is `far from empty' and an entry of $1$ means that the queue is `close to empty'. The vectors $v^c(z)$ and $v^b(z)$ will be used to make adjustments, depending on the state $z$, to the nominal instantaneous flow rate at the given instant to give the overall flow rate. The role of $v^c(z)$ will be to reduce the holding cost while keeping the net workload to be the same and ensuring that the queues close to empty do not get more than the nominal allocation. This vector helps with achieving the { second} feature of the HGI performance. The vector $v^b(z)$ is instrumental in achieving the { first} feature of the HGI performance by pushing the queues that are close to empty, away from $0$, so that idleness due to `blocking' is prevented. A detailed discussion of the policy can be found in Remark \ref{rem:policy}.

We now make some comments on the main ingredients in the proofs (additional discussion can be found in Remark \ref{rem:policy}). As noted previously, the two key characteristic properties of HGI performance are: {\em no idleness}, and {\em instantaneous holding cost minimization}.
 The main effort in the proof is to show that the sequence of control policies constructed here asymptotically have these two features. The main results that enable the verification of the first property are Propositions \ref{thm:skorokApprox} and \ref{thm:idleTimeBndMark} whereas the second feature emerges as a consequence of Proposition \ref{thm:costDiffResult}. Furthermore, as we need to consider an ergodic cost criterion, one needs to obtain suitable stability estimates that are uniform in the traffic parameter. This is done in Proposition \ref{thm:workloadExpBnd}. Proofs of these three results are the technical heart of this work. Some recurring tools in the various proofs are path large deviation estimate for renewal processes and excursion analyses of strong Markov processes. We were not able to find suitable references for the path large deviation estimates we need and so we provide self contained proofs of these results for reader's convenience (see Theorem \ref{thm:expTailBnd}). The excursion analyses that are used in the various proofs are guided by the form of our control policy which is described in terms of thresholds determined by a sequence of stopping times. The general approach of considering suitable excursions together with appropriate large deviation estimates in the analysis of control policies for networks in heavy traffic originates from the work of Bell and Williams\cite{belwil1, belwil2} (see also \cite{budgho1}, \cite{atakum} and \cite{budliusah} ). Our recent paper \cite{budjoh} also used an analogous approach, however in general the precise forms of excursions to consider,  and study of their properties, is problem specific and indeed such analyses constitute the major effort in the proofs. 
 
 The control policy we construct is in general not asymptotically optimal. However, it can be argued that when the function $\hat h$ introduced in \eqref{eq:hhatdef} is a nondecreasing function then the minimality properties of the Skorohod map imply that our policy is indeed asymptotically optimal. This is a consequence of a more general result on asymptotic lower bounds on costs of arbitrary control policies for resource sharing networks which will be reported elsewhere.
 
 The remaining paper is organized as follows. We close this section with the notations and conventions used in this work. Section \ref{sec:mainresults} presents the model description, our main assumptions, the sequence of control policies that we study, and our main results. Rest of the paper is devoted to the proofs of the main results, namely Theorems \ref{thm:disccost} and \ref{thm:ergcost}. Details of proof organization can be found at the end of Section \ref{sec:mainres}.

\subsection{Notation and Conventions.} 
For $j \in \N$, let $\sD^j = \sD([0,\infty) : \R^j)$ (resp. $\sD_+^j = \sD([0,\infty) : \R_+^j)$) denote the space of functions that are right continuous with left limits (RCLL) from $[0, \infty)$ to $\R^j$ (resp. $\R_+^j$) equipped with the usual Skorohod topology. 
Also, let  $\sC^j= \sC([0,\infty) : \R^j)$ (resp. $\sC_+^j = \sC_+([0,\infty) : \R^j)$) denote the space of continuous functions from
$[0, \infty)$ to $\R^j$ (resp. $\R_+^j$) equipped with the local uniform topology. 
For $T<\infty$, the
 spaces $\sD([0,T] : \R^j)$, $\sC([0,T] : \R^j)$, $\sD([0,T] : \R_+^j)$,  $\sC([0,T] : \R_+^j)$ are defined similarly.
 All stochastic processes in this work will have sample paths that are RCLL unless noted explicitly.  For $m \in \N$, we denote by $\AAA_m$ the set $\{1, 2, \ldots, m\}$ and $\chi_m$ the finite set of all vectors in $\RR^m$ with entries $0$ or $1$.
For $r \in \NN$, $\NN_{1/r} \doteq \frac{1}{r} \NN_0$  is the scaled (nonnegative) integer lattice.
 We will frequently do component-wise operations on vectors.  For instance, given two vectors $v^{1},v^{2}\in\mathbb{R}^{d}$ we will use $v^{1}v^{2}$ to denote component-wise multiplication. The vector $v^1/v^2$ is defined similarly.
%  , so
% \begin{equation*}
% v^{1}v^{2}=\left(v^{1}_{1}v^{2}_{1},v^{1}_{2}v^{2}_{2},...,v^{1}_{J}v^{2}_{J} \right).
% \end{equation*}
For a vector $v\in\mathbb{R}^{d}$ and a constant $c\in\mathbb{R}$ we will interpret $v\vee c$, $v \wedge c$, and $v-c$ component-wise, e.g.
$
v\vee c=\left( v_{1}\vee c,v_{2}\vee c,...,v_{J}\vee c\right).$
% \begin{equation*}
% v\wedge c=\left( v_{1}\wedge c,v_{2}\wedge c,...,v_{J}\wedge c\right),
% \end{equation*}
% and
% \begin{equation*}
% v- c=\left( v_{1}- c,v_{2}- c,...,v_{J}- c\right).
% \end{equation*}
We will also treat inputs to vectors of functions component-wise so for a vector $v\in\mathbb{R}^{d}$, a constant $c\in\mathbb{R}$, and a vector of functions $f=\left(f_{1},f_{2},...,f_{d}\right)$, from $\R$ to itself, we write
\begin{equation*}
f(v)=\left(f_{1}(v_{1}), f_{2}(v_{2}),...,f_{d}(v_{d})\right),\;\;
\mathbb{I}_{\{v\geq c\}}=\left(\mathbb{I}_{\{v_{1}\geq c\}},\mathbb{I}_{\{v_{2}\geq c\}},...,\mathbb{I}_{\{v_{d}\geq c\}}\right).
\end{equation*}
We will use coordinate-wise inequalities on vectors, e.g.  for  $v^1, v^2\in\mathbb{R}^{d}$ and $c\in\mathbb{R}$ the statements $v^1\ge v^2$
and $v^1\geq c$ mean  $v^1_{j}\geq v^2_{j}$  and $v^1_{j}\geq c$ for all $j\in\AAA_{d}$, resp. Inequalities for vector valued functions are interpreted pointwise and coordinate-wise.  For  $v\in\mathbb{R}^{d}$,
$|v |_{1}=\sum_{j=1}^{d}|v_{j}|$ and 
$|v |_{2}=\left(\sum_{j=1}^{d}|v_{j}|^{2}\right)^{\frac{1}{2}}$. 
As a convention, for a real sequence $\{a(l)\}_{l\in \NN}$, $\sum_{l=1}^{n} a(l)$ is taken to be $0$ if $n=0$.

\section{Main Results}
\label{sec:mainresults}
 A RSN in heavy traffic is described through a sequence of models, indexed by a traffic parameter $r \in \NN$, each of which has the same underlying network topology. 
Each network in the sequence consists of  $J$ types of jobs and $I$ resources and the network topology is described through an $I\times J$ matrix $K$ for which  $K_{i,j}=1$ if the $j$th job requires service from  the $i$th resource and $K_{i,j}=0$ otherwise. A job of type $j$ will be processed simultaneously by all resources in the set $\{i: K_{i,j}=1\}$ and so each resource in this set will allocate the same amount of processing capacity to the job at any instant.
As $r$ goes to $\infty$ the networks approach criticality in that the traffic intensity converges to $1$.
Specifically, as made precise in Condition \ref{cond:htc},
with the asymptotic load vector $\rho$ as defined in that condition, $C_i = \sum_{j=1}^J K_{i,j}\rho_j$ for every resource $i = 1, \ldots, I$ which says that the capacity of each resource equals the asymptotic load on the resource.
  In the $r$-th system $\{u^{r}_{j}(l)\}_{l=1}^{\infty}$ and $\{v^{r}_{j}(l)\}_{l=1}^{\infty}$ are the i.i.d, mutually independent, interarrival times and job-sizes for job type $j$, given on some probability space $(\Om, \clf, P)$,  with means $E[u^{r}_{j}(l)]=\frac{1}{\alpha^{r}_{j}}$ and $E[v^{r}_{j}(l)]=\frac{1}{\beta^{r}_{j}}$ and finite standard deviations $\sigma^{u,r}_{j}$ and $\sigma^{v,r}_{j}$.  We assume a first in, first out (FIFO) policy meaning that for each job type the oldest job in the queue is processed before another one is started.  In the case where the job sizes are exponentially distributed the ``memoryless" property implies that the precise manner of the allocation of the flow rate among jobs in the queue of a particular job type does not impact the distribution of the queue-length process.  Consequently if the job sizes are exponentially distributed we can drop the FIFO assumption in favor of something else, such as  a processor sharing policy that is common in the literature, without altering the results.
  
  We now introduce our main assumptions.
\subsection{Assumptions}

We will assume that
\begin{equation*}
P(u^{r}_{j}(1)>0)=P(v^{r}_{j}(1)>0)=1 \text{ for all } r \text{ and } j.
\end{equation*}
In fact, for notational convenience  we will  simply  assume (without loss of generality) that $v^{r}_{j}(l)>0$ and $u^{r}_{j}(l)>0$ for all $j\in\AAA_J$ and $l\in\mathbb{N}$.\\
 The following assumption on  finite moment-generating functions  in the neighborhood of the origin will be used to obtain certain large deviation estimates. Recall  that for $m \in \NN$, $\AAA_m=\{1,...,m\}$.
\begin{condition}
\label{eqn:mgfBnd}
There exists $\delta>0$ such that for all $y<\delta$ and $j\in\AAA_J$.
\begin{equation*}
\sup_{r>0}E\left[ e^{yv^{r}_{j}(1)}\right]<\infty,\;\;
\sup_{r>0}E\left[ e^{yu^{r}_{j}(1)}\right]<\infty .
\end{equation*}
\end{condition}
The following is our main {\em heavy traffic condition}.  
\begin{condition}
	\label{cond:htc}
For each $j\in \AAA_J$ there exist $\bar{\alpha}_{j},\bar{\beta}_{j}, \alpha_{j}, \beta_{j}, \sigma^{u}_{j},\sigma^{v}_{j}\in (0,\infty)$ such that $$\lim_{r\rightarrow\infty}r(\alpha^{r}_{j}-\alpha_{j})=\bar{\alpha}_{j},\;\; \lim_{r\rightarrow\infty}r(\beta^{r}_{j}-\beta_{j})=\bar{\beta}_{j},\;\; \lim_{r\rightarrow\infty}\sigma^{u,r}_{j}=\sigma^{u}_{j},\;\;  \lim_{r\rightarrow\infty}\sigma^{v,r}_{j}=\sigma^{v}_{j}.$$
Furthermore, with $\rho=\alpha/\beta$ we have $C=K \rho$ and with $\eta =\beta^{-2} (\bar{\alpha} \beta -\alpha\bar{\beta})$ and $\theta \doteq K\eta$ we have $\theta<0$.  Note that this implies $\lim_{r\rightarrow\infty}r(\rho^{r}-\rho)=\eta$ where $\rho^r = \alpha^r/\beta^r$.
\end{condition}
The $\theta<0$ part of the condition is a key stability assumption which will be crucially used in obtaining various types of uniform in time moment estimates (see e.g. Section \ref{sec:stabproofs}). 

Finally, we make the following local traffic assumption. It says that, every resource has at least one associated job-type that requires service from only that particular resource. This assumption was also made in  \cite{kankelleewil},  \cite{harmandhayan}, and \cite{budjoh}.
\begin{condition}
\label{cond:locTraffic}
For every $i\in\AAA_I$ there exists $j\in\AAA_J$ such that $K_{i,j}=1$ and $K_{l,j}=0$ for $l\neq i$.
\end{condition}

When the above condition does not hold the situation is quite different in that the HGI performance is not given by a reflected Brownian motion(RBM) in an orthant since the workload process may not achieve all vectors in $\mathbb{R}_+^I$. Consider for example the case where $I=2, J=2$ and the job-resource matrix $K$ is 
  $\begin{bmatrix}
1 & 1 \\
0 & 1\\
\end{bmatrix}$. In this case, if $\beta = [1,1]'$, the limiting Brownian motion will live in the wedge $\{(x_1, x_2): 0\le x_2 \le x_1<\infty\}$. Treating such settings will require additional new ideas and is beyond the scope of the current work.

\subsection{State Processes}
The arrival and service processes for job type $j \in \AAA_J$ are, resp., 
\begin{equation*}  
A^{r}_{j}(s)=\max\left\{n\geq 0: \sum_{l=1}^{n}u^{r}_{j}(l)\leq s\right\}, \;
S^{r}_{j}(s)=\max\left\{n\geq 0: \sum_{l=1}^{n}v^{r}_{j}(l)\leq s\right\}.
\end{equation*} 
The initial state of the system is described by the $J$-dimensional queue length vector $q^r \in \NN^{J}$
and residual arrival and service time vectors in $\RR_+^J$, defined as
%
%
%
% I added some notation so that the state of the system is a Markov process.  At time $0$ the (known) time until the next arrival of job type $j$ is $\Upsilon^{A,r}_{j}$ and the (known) time until the next service of job type $j$ is $\Upsilon^{S,r}_{j}$.  Normally these are set to $0$ but when we think about the process at a future time $t>0$ we have already begun processing the next job and waiting on the next arrival and the remaining times for these events are known based on the filtration defined below.  There are time dependent versions of these introduced in the next section and used in the proofs.  In order to use the Markov property I allowed time until the next service/arrival to be part of the initial condition.  Throughout I will refer to the vectors
\begin{equation*}
\Upsilon^{A,r}=\left(\Upsilon^{A,r}_{1},\Upsilon^{A,r}_{2},...,\Upsilon^{A,r}_{J}\right),\;\;
\Upsilon^{S,r}=\left(\Upsilon^{S,r}_{1},\Upsilon^{S,r}_{2},...,\Upsilon^{S,r}_{J}\right).
\end{equation*}
Here $\Upsilon^{A,r}_{j}$ [ resp. $\Upsilon^{S,r}_{j}$] represent the deterministic times after which the arrivals [resp. processing times] are governed by the renewal processes $\{A^{r}_{j}(s)\}_{s\ge 0}$ [resp.  $\{S^{r}_{j}(s)\}_{s\ge 0}$  ]. 
 These quantities capture initial state configurations when the  evolution is viewed onwards from a time instant at which the system has been in operation for some time.
 
A key ingredient in the state evolution of the queue-length process is a capacity allocation control policy which is described as a continuous $\RR_+^J$ valued stochastic process $B^r(\cdot)$ with appropriate measurability and feasibility properties that will be specified later in the section. Roughly speaking, for $j\in \AAA_J$, $B^r_j(t)$ specifies the cumulative amount of capacity used by type-$j$ jobs, in the $r$-th system, by time instant $t$.
Given such initial conditions and a control policy, the $J$-dimensional queue-length process is given by the equation
\begin{equation}\label{eq:qrt}
Q^{r}(t)=q^r+A^{r}\left(\left(t-\Upsilon^{A,r}\right)^{+}\right)+\mathcal{I}_{\{t\geq\Upsilon^{A,r}>0\}}-S^{r}\left(\left(B^{r}(t)-\Upsilon^{S,r}\right)^{+}\right)-\mathcal{I}_{\{B^{r}(t)\geq\Upsilon^{S,r}>0\}}.
\end{equation}
This evolution captures the fact that the queue length, of say the $j$-th job type, corresponds to the initial queue length $q^r_j$, plus all the arrivals that have occurred according to the renewal process $A^r_j$, with an additional arrival at time instant $\Upsilon^{A,r}$ if this quantity is positive, minus all the jobs that have been been completed according to the renewal process $S^r_j$, with an additional departure at instant $t$ where $B^r(t)= \Upsilon^{S,r}$ if the latter quantity is nonzero.

Let $M^{r}$ be the $J\times J$ diagonal matrix with entries $\{\frac{1}{\beta^{r}_{j}}\}_{j=1}^{J}$ and let $M$ be the $J\times J$ diagonal matrix with entries $\{\frac{1}{\beta_{j}}\}_{j=1}^{J}$. 
The  $I$-dimensional workload process $W^r(t)$ associated with the queue-length process $Q^r(t)$ is then given by the equation
\begin{eqnarray*}
W^r(t) \doteq KM^{r}Q^{r}(t)&=&KM^{r}q^r+KM^{r}\left(A^{r}\left(\left(t-\Upsilon^{A,r}\right)^{+}\right)-S^{r}\left(\left(B^{r}(t)-\Upsilon^{S,r}\right)^{+}\right)\right)\\
&&+KM^{r}\mathcal{I}_{\{t\geq\Upsilon^{A,r}>0\}}-KM^{r}\mathcal{I}_{\{B^{r}(t)\geq\Upsilon^{S,r}>0\}}.\\
\end{eqnarray*}
Note that this is a $I$-dimensional process which captures  the amount of work in the system at any instant for each of the $I$ resources in the network.\\

\noindent {\bf Scaled Processes.}  In order to study the behavior as the systems approach criticality, we will consider two types of scaling: diffusion scaling and fluid scaling. In both of these scalings, time is accelerated by a 
factor of $r^2$, but in  the first type of scaling the magnitude is scaled down by a factor of $r$ while in the second scaling  the magnitude is scaled 
down by factor of $r^2$. Processes obtained using diffusion scaling will typically be denoted using a `hat' symbol while the processes with fluid scaling 
will be denoted using a `bar' symbol. In particular, we define
 $$\hat{\Upsilon}^{A,r}=\frac{1}{r}\Upsilon^{A,r}, \;\; \hat{\Upsilon}^{S,r}=\frac{1}{r}\Upsilon^{S,r},\;\; \bar{\Upsilon}^{A,r}=\frac{1}{r^{2}}\Upsilon^{A,r}, \;\; \bar{\Upsilon}^{S,r}=\frac{1}{r^{2}}\Upsilon^{S,r}.$$
Similarly, we define $\hat Q^r(t) \doteq Q^r(r^2t)/r$, $\hat W^r(t) \doteq W^r(r^2t)/r$, and $\bar B^r(t) \doteq B^r(r^2t)/r^2$.
Letting,
\begin{align*} 
\hat{A}^{r}_{j}(s)&=\frac{1}{r}\max\left\{n\geq 0: \sum_{l=1}^{n}u^{r}_{j}(l)\leq r^{2}s\right\}-rs\alpha^{r}_{j},\\
\hat{S}^{r}_{j}(s)&=\frac{1}{r}\max\left\{n\geq 0: \sum_{l=1}^{n}v^{r}_{j}(l)\leq r^{2}s\right\}-rs\beta^{r}_{j}
\end{align*} 
we see that, with $\hat q^r = q^r/r$
\begin{eqnarray*}
\hat{Q}^{r}(t)&=& \hat q^r+\hat{A}^{r}\left(\left(t-\bar{\Upsilon}^{A,r}\right)^{+}\right)-\hat{S}^{r}\left(\left(\bar{B}^{r}(t)-\bar{\Upsilon}^{S,r}\right)^{+}\right)+\frac{1}{r}\mathcal{I}_{\{t\geq\bar{\Upsilon}^{A,r}>0\}}\nonumber\\
&&-\frac{1}{r}\mathcal{I}_{\{\bar{B}^{r}(t)\geq\bar{\Upsilon}^{S,r}>0\}}
+r\left(\alpha^{r}t-\beta^{r}\bar{B}^{r}(t)\right)-r\alpha^{r}\left(t\wedge \bar{\Upsilon}^{A,r}\right)+r\beta^{r}\left(\bar{B}^{r}(t)\wedge \bar{\Upsilon}^{S,r}\right)\label{eq:303n}
\end{eqnarray*}
and the corresponding diffusion-scaled workload process, using the identities $M^r \alpha^r = \rho^r$ and $K\rho=C$, is
\begin{eqnarray*}
\hat{W}^{r}(t) &=&
%&=&KM^{r}\hat{Q}^{r}(t) \\
KM^{r}\hat{q}^{r}+KM^{r} \hat{A}^{r}\left(\left(t-\bar{\Upsilon}^{A,r}\right)^{+}\right)-KM^{r}\hat{S}^{r}\left(\left(\bar{B}^{r}(t)-\bar{\Upsilon}^{S,r}\right)^{+}\right)\\
&&+\frac{1}{r}KM^{r}\mathcal{I}_{\{t\geq\bar{\Upsilon}^{A,r}>0\}}-\frac{1}{r}KM^{r}\mathcal{I}_{\{\bar{B}^{r}(t)\geq\bar{\Upsilon}^{S,r}>0\}}+rKM^{r}(\alpha^{r}t-\beta^{r}\bar{B}^{r}(t))\\
&& -rKM^{r}\alpha^{r}\left(t\wedge \bar{\Upsilon}^{A,r}\right)+rKM^{r}\beta^{r}\left(\bar{B}^{r}(t)\wedge \bar{\Upsilon}^{S,r}\right)
\\
% &=&KM^{r}\hat{q}^{r}+KM^{r} \hat{A}^{r}\left(\left(t-\bar{\Upsilon}^{A,r}\right)^{+}\right)+\frac{1}{r}KM^{r}\mathcal{I}_{\{t\geq\bar{\Upsilon}^{A,r}>0\}}-KM^{r}\hat{S}^{r}\left(\left(\bar{B}^{r}(t)-\bar{\Upsilon}^{S,r}\right)^{+}\right)\\
% &&-\frac{1}{r}KM^{r}\mathcal{I}_{\{\bar{B}^{r}(t)\geq\bar{\Upsilon}^{S,r}>0\}} +Krt\rho^{r}-rK\bar{B}^{r}(t) -rK\rho^{r}\left(t\wedge \bar{\Upsilon}^{A,r}\right) +rK\left(\bar{B}^{r}(t)\wedge \bar{\Upsilon}^{S,r}\right)
% \\
&=&KM^{r}\hat{q}^{r}+KM^{r} \hat{A}^{r}\left(\left(t-\bar{\Upsilon}^{A,r}\right)^{+}\right)-KM^{r}\hat{S}^{r}\left(\left(\bar{B}^{r}(t)-\bar{\Upsilon}^{S,r}\right)^{+}\right)\\
&&+\frac{1}{r}KM^{r}\mathcal{I}_{\{t\geq\bar{\Upsilon}^{A,r}>0\}}-\frac{1}{r}KM^{r}\mathcal{I}_{\{\bar{B}^{r}(t)\geq\bar{\Upsilon}^{S,r}>0\}}+rtK(\rho^{r}-\rho) \\
&&+r(Ct-K\bar{B}^{r}(t)) -rK\rho^{r}\left(t\wedge \bar{\Upsilon}^{A,r}\right) +rK\left(\bar{B}^{r}(t)\wedge \bar{\Upsilon}^{S,r}\right).
\end{eqnarray*}
It will be convenient to introduce the processes
\begin{align}
\hat{X}^{r}(t)\doteq& KM^{r}\hat{A}^{r}\left(\left(t-\bar{\Upsilon}^{A,r}\right)^{+}\right)+\frac{1}{r}KM^{r}\mathcal{I}_{\{t\geq\bar{\Upsilon}^{A,r}>0\}}-KM^{r}\hat{S}^{r}\left(\left(\bar{B}^{r}(t)-\bar{\Upsilon}^{S,r}\right)^{+}\right)\nonumber\\
&-\frac{1}{r}KM^{r}\mathcal{I}_{\{\bar{B}^{r}(t)\geq\bar{\Upsilon}^{S,r}>0\}} 
+rtK(\rho^{r}-\rho) -rK\rho^{r}\left(t\wedge \bar{\Upsilon}^{A,r}\right) +rK\left(\bar{B}^{r}(t)\wedge \bar{\Upsilon}^{S,r}\right) \label{eq:Xhat}
\end{align} 
and
\begin{equation*}
\hat{U}^{r}(t)=r(Ct-K\bar{B}^{r}(t))
\end{equation*}
so that, with $\hat w^r \doteq  KM^{r}\hat{q}^{r}$, we have
\begin{equation}\label{eq:eq424}
\hat{W}^{r}(t)=\hat w^r +\hat{X}^{r}(t)+\hat{U}^{r}(t).
\end{equation}
\subsection{Admissible Control Policies}
We now  specify what types of allocation policies are admissible. Roughly speaking an admissible control policy should not look into the future. This is made precise by introducing certain multiparameter filtrations.
Recall the probability space $(\Om, \clf, P)$ on which the sequences $\{u^{r}_{j}(l)\}_{l=1}^{\infty}$ and $\{v^{r}_{j}(l)\}_{l=1}^{\infty}$ are defined.
\begin{defn}
For $n=(n_{1},...,n_{J})\in\mathbb{N}_0^{J}$ and $m=(m_{1},...,m_{J})\in\mathbb{N}_0^{J}$ let 
\begin{equation*}
\mathcal{F}^{r}(n,m)=\sigma\{u^{r}_{j}(\tilde{n}_{j}), v^{r}_{j}(\tilde{m}_{j}):0\le \tilde{n}_{j}\leq n_{j},0\le \tilde{m}_{j}\leq m_{j}, j\in\AAA_J\},
\end{equation*}
where by convention $u^{r}_{j}(0)= v^{r}_{j}(0)=0$.
% We assume that for all $j\in\AAA_J$ and $r>0$ both $\{u^{r}_{j}(l)\}_{l=1}^{\infty}$ and $\{v^{r}_{j}(l)\}_{l=1}^{\infty}$ are independent of $Q^{r}(0),\Upsilon^{A,r},$ and $\Upsilon^{S,r}$.  Consequently for any $j\in\AAA_J$, $r>0$, and $n,m\in\AAA_J$ both $\{u^{r}_{j}(l)\}_{l=n_{j}+1}^{\infty}$ and $\{v^{r}_{j}(l)\}_{l=m_{j}+1}^{\infty}$ are independent of $\mathcal{F}^{r}(n,m)$.
Let 
\begin{equation*}
\mathcal{F}^{r}=\sigma\left\{\bigcup_{(n,m)\in\mathbb{N}^{2J}}\mathcal{F}^{r}(n,m)\right\}
\end{equation*}
\end{defn}
Note that $\{\mathcal{F}^{r}(n,m), n,m\in\mathbb{N}^{J}\}$ is a multiparameter filtration generated by the arrival and service times with the following partial ordering:
\begin{equation*}
(n,m)\leq (\tilde n,\tilde m) \text{ if and only if } n_{j}\leq \tilde n_{j} \text{ and } m_{j}\leq \tilde m_{j} \text{ for all } j\in\AAA_J.
\end{equation*}
For some basic definitions of multiparameter filtrations and multiparameter stopping times, see \cite[Chapter 2, Section 8]{ethier2009markov}.
An admissible control policy is defined as follows.
\begin{defn}
\label{def:AdmissableCntrl}
For $r \in \N$,
$B^{r}(\cdot)$ is an admissible policy  (for the $r$-th system) for the initial condition $(q^r, \Ups^r=(\Upsilon^{A,r}, \Upsilon^{S,r})) \in \N_0^J\times \R_+^{2J}$ if the following hold:
\begin{enumerate}[(a)]
\item The stochastic process $B^{r}(\cdot)$ has sample paths that  are absolutely continuous, nonnegative, nondecreasing functions from $[0,\infty)\rightarrow\mathbb{R}^{J}$ with $B^{r}(0)=0$.
\item $C \geq K\frac{d}{dt}B^{r}(t)$ for a.e. $t\geq 0$, a.s.
\item The process $Q^r(\cdot)$ defined by the right side of \eqref{eq:qrt} satisfies
$Q^{r}(t)\geq0$ for all $t\geq 0$.
\item For each $r\in \N$ and $t\ge 0$ consider the $\mathbb{N}^{2J}$ valued random variable
	\begin{equation*}
	\tau^{r}(t)\doteq (\tau^{r,A}_{1}(t),...,\tau^{r,A}_{J}(t),\tau^{r,S}_{1}(t),...,\tau^{r,S}_{J}(t))
	\end{equation*}
	where for  $j\in\mathbb{N}^{J}$
	\begin{equation*}
	 \tau^{r,A}_{j}(t)=\min\left\{n\geq 0: \sum_{l=1}^{n}u^{r}_{j}(l)\geq r^{2}\left(t-\bar{\Upsilon}^{A}_{j}\right)^{+}\right\}
	\end{equation*}
	and
	\begin{equation*}
	 \tau^{r,S}_{j}(t)=\min\left\{n\geq 0: \sum_{l=1}^{n}v^{r}_{j}(l)\geq r^{2}\left(\bar{B}^{r}_{j}(t)-\bar{\Upsilon}^{S}_{j}\right)^{+}\right\},
	\end{equation*}
	where by convention, $\tau^{r,A}_{j}(t)$ [resp. $\tau^{r,S}_{j}(t)$] is defined to be $0$ if $t \le \bar{\Upsilon}^{A}_{j}$ [resp. $\bar{B}^{r}_{j}(t)\le \bar{\Upsilon}^{S}_{j}$].
	 Then $\tau^{r}(t)$ is an $\{\mathcal{F}^{r}(n,m)\}$-stopping time for all $t\geq 0$.
\item Consider the filtration 
	\begin{equation*}
	\clg^r(t)\doteq \mathcal{F}^{r}(\tau^{r}(t))=\sigma\{A\in\mathcal{F}^{r}:A\cap\{\tau^{r}(t)\leq (n,m)\}\in \mathcal{F}^{r}(n,m) \text{ for all } (n,m)\in\mathbb{N}^{2J}\}
	\end{equation*}
	Then $B^{r}(r^{2}t)$ is $\{\clg^{r}(t)\}$-adapted.
\end{enumerate}
Define $\sA$ to be the set of admissable controls.
\end{defn}
Parts (b) and (c) give the feasibility requirements on the control policy whereas parts (d) and  (e) make precise the non-anticipativity property of an admissible   policy.

\subsection{Proposed Control Policy}
\label{sec:propcon}
In this section we introduce our proposed control policy which will be shown to achieve the HGI performance asymptotically.
Recall that $K$ is an $I\times J$ matrix and $M$ is a $J\times J$ diagonal matrix with diagonal entries $\{\frac{1}{\beta_{j}}\}_{j=1}^{J}$. 
Let $h\in \RR^J$ such that $h>0$ be a given holding cost vector. This vector will determine the discounted and ergodic cost functions that will be introduced later.

 For  $w\in \mathbb{R}^{I}_{+}$ define
\begin{equation}\label{eq:hhatdef}
\Lambda(w)=\{q\in \mathbb{R}^{J}_{+}:KMq=w\}, \;\;
\hat{h}(w)=\inf_{q\in\Lambda(w)}(h\cdot q).
\end{equation}
Let $v_{j}$ be the $j$th column of $K$ and note that due to the local traffic condition (Condition \ref{cond:locTraffic}) the span of $\{v_{j}\}_{j=1}^{J}$ is $\mathbb{R}^{I}$.  As a consequence, for any $w\in\mathbb{R}^{I}$,  $\Lambda(w)$ is a nonempty compact subset of $\mathbb{R}^{J}$.  In particular  there exists $q^{*} = q^*(w)\in\Lambda(w)$ such that $h\cdot q^{*}=\hat{h}(w)$.

Define 
\begin{equation*}
\mathcal{C}_{K}^{h}=\{u\in \mathbb{R}^{J}: \textbf{0}=Ku \text{ and } (h\beta)\cdot u = 0\}
\end{equation*}
and note that $\mathcal{C}_{K}^{h}$ is a linear subspace of $\text{ker}(K)$.  Note that either $\mathcal{C}_{K}^{h}=\text{ker}(K)$ or $\text{dim}(\mathcal{C}_{K}^{h})=\text{dim}(\text{ker}(K))-1=J-I-1$.  If $\mathcal{C}_{K}^{h}=\text{ker}(K)$ then for any $w\in\mathbb{R}_+^{I}$ all $q\in\mathbb{R}_+^{J}$ which satisfy $w=KMq$ have the same cost, namely if $q, \tilde q \in \R^J_+$ satisfy $KMq=KM\tilde q=w$ then $h\cdot q = h \cdot \tilde q$.  Note that this situation, where all queue length vectors which give the same workload have the same holding cost, is trivial in the sense that instantaneous holding cost minimization is no longer a concern and the focus is entirely on preventing idleness.  The policy that we will present below (see Definition \ref{def:scheme}) takes a simpler form in this setting in that the vector $v^c(z)$ used for rate allocation adjustment in \eqref{eq:singleline} is simply $0$. However, as was pointed out to us by Mike Harrison in a private communication, one can treat this trivial case in a more elementary manner by simply giving lowest priority to the local traffic jobs 
so that the non-local traffic class form a subcritical RSN.  As a result in the diffusion-scaled heavy traffic limit the non-local traffic queue-lengths go to $0$ and the local traffic queue-lengths give the desired reflected Brownian motion workload.  We have chosen to present the alternative policy for this case as in Definition \ref{def:scheme} since it allows us to cover the two cases (namely $\text{dim}(\mathcal{C}_{K}^{h})=\text{dim}(\text{ker}(K))-1$ and $\text{dim}(\mathcal{C}_{K}^{h})=\text{dim}(\text{ker}(K))$) in a unified approach.  

% Since any queue length for a given workload gives the same cost the problem becomes much easier because our only concern is ensuring the every server is almost always working at full capacity when it's workload is nonempty.  Specifically, if $\mathcal{C}_{K}^{h}=\text{ker}(K)$ the scheme described below is substantially simplificed by removing the $v^{c}(\cdot)$ which are used to move the system to a more cost efficient queue-length for a given workload.  Consequently we consider the case where $\mathcal{C}_{K}^{h}\neq\text{ker}(K)$ so $\text{dim}(\mathcal{C}_{K}^{h})=J-I-1$ because the alternative is an easier version of this.
We select an orthonormal basis of $\mathbb{R}^{J}$,  $(u_{1},...,u_{J})$, such that  $(u_{1},...,u_{J-I})$ is an orthonormal basis of $\text{ker}(K)$, 
and, in the case $\text{dim}(\mathcal{C}_{K}^{h})=\text{dim}(\text{ker}(K))-1=J-I-1$,
$\text{span}(u_{1},...,u_{J-I-1})=\mathcal{C}_{K}^{h}$, and $u_{J-I}$ is orthogonal to $\mathcal{C}_{K}^{h}$ and satisfies $h\beta\cdot u_{J-I}<0$.  \\ The latter  quantity will play an important role and we define
\begin{equation*}
\lambda\doteq h\beta\cdot u_{J-I}.
\end{equation*}
Note that since $\lambda<0$, adjusting the queue length by adding $\beta u_{J-I}$ reduces the cost while maintaining the same workload.
For  $q\in\mathbb{R}^{J}$ define
\begin{equation*}
\Xi (q)\doteq\{v\in\text{ker}(K):q+v\beta\geq 0 \}.
\end{equation*}
Note that $\Xi (q)$ is a compact set.
Let
\begin{equation*}
\tilde{d}(q)\doteq 
\begin{cases}
\sup_{v\in\Xi(q)}\{v\cdot u_{J-I}\} &\mbox{ if } \text{dim}(\mathcal{C}_{K}^{h})=\text{dim}(\text{ker}(K))-1=J-I-1\\
0&\mbox{ otherwise.}
\end{cases}
\end{equation*}
The following proposition gives a precise measure for how far away $q$ is  from the cost minimizing queue length for the workload $KMq$.  
%Proof is given in Section \ref{sec:gennet}.
\begin{proposition}[Proof in Section \ref{sec:gennet}]
\label{thm:lambdaEqu}
For all $q\in\mathbb{R}_{+}^{J}$ we have
$h\cdot q-\hat{h}\left(KMq\right)=|\lambda| \tilde{d}(q).$
\end{proposition}
\begin{figure}\centering
\includegraphics[height=1in, width=3in]{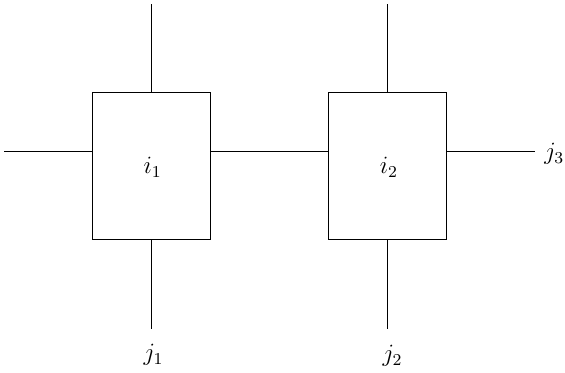}
\caption{2LLN Network}
\end{figure}
To illustrate the distinction between the trivial case where $\mathcal{C}_{K}^{h}=\text{ker}(K)$ and the nontrivial case where $\text{dim}(\mathcal{C}_{K}^{h})=\text{dim}(\text{ker}(K))-1$ we consider two examples of networks with $I=2$ resources and $J=3$ job types (referred to as 2LLN in \cite{harmandhayan}) pictured in Figure 1.
\begin{example}
\label{ex:trivial}
Consider a 2LLN network (see Figure 1) where $\alpha=\beta=(1,1,1)'$ so $\rho=(1,1,1)'$ and $C=(2,2)'$.  Let the holding cost vector be $h=(1,1,2)'$.
\end{example}

\begin{example}
\label{ex:nontrivial}
Just like Example \ref{ex:trivial} we consider a 2LLN network (see Figure 1) where $\alpha=\beta=(1,1,1)'$ so $\rho=(1,1,1)'$ and $C=(2,2)'$.  However, the holding cost vector is now $h=(1,1,1)'$.
\end{example}
For these 2LLN networks the incidence matrix is
$$
K=
\begin{bmatrix}
1 & 0 & 1 \\
0 & 1 & 1 
\end{bmatrix}
$$
so $\text{ker}(K)$ is one dimensional and an orthonormal basis of $\text{ker}(K)$ is the vector $u_1=\frac{1}{\sqrt{3}}\left(-1,-1,1\right)'$.  Consequently, for both Example \ref{ex:trivial} and Example \ref{ex:nontrivial}, all changes to the queue length that maintain the same workload involve adding a constant multiple of $\beta u_1$.
% and recall that part of HGI peformance is achieving the minimal cost for the given workload.
%Adding a multiple of the $u_1$ vector to our flow rate $\y$ does not change the capacity usage, $K\y$, and because in order to achieve HGI performance we want to work at full capacity, we can think of $u_1$ as characterizing the permissable adjustments to the flow rate we can make in order to reduce the cost.  For any constant $c\in\mathbb{R}$ subtracting $c u_1$ from the flow rate will, on average, adjust the queue length at a rate of $c \beta  u_1$ (and is workload neutral because $KM c\beta u_1= cK u_1=0$) which changes the cost at a rate of $h\cdot \beta u_1=h\beta\cdot u_1$.  
In Example \ref{ex:trivial} we have $h\beta \cdot  u_1=0$ so $\mathcal{C}_{K}^{h}=\text{ker}(K)$ and we say this network is trivial (in terms of instantaneous holding cost minimization) because all queue lengths which give the same workload have identical cost.  In Example \ref{ex:nontrivial} we have $\lambda= h\beta\cdot u_1=h\cdot u_1=-\frac{1}{\sqrt{3}}$ so $\text{dim}(\mathcal{C}_{K}^{h})=\text{dim}(\text{ker}(K))-1$ and in this case adding a positive multiple of $\beta u_1$ will maintain the same workload but reduce the cost. This is because the queue lengths $(1,1,0)'$ and $(0,0,1)'$ result in the same workload but $(0,0,1)'$ has a lower holding cost.  In contrast to Example \ref{ex:trivial} where all queue lengths are cost minimizing for their corresponding workloads, the only cost minimizing queue lengths in Example \ref{ex:nontrivial} are on the boundary where either $q_1=0$ or $q_2=0$ meaning we want as much workload as possible to come from type $3$ jobs.

Recall that $\chij$ denotes the finite set of all vectors in $\RR^J$ with entries $0$ or $1$.
% \begin{equation*}
% z^{q}_{j}\doteq \left\{
% \begin{array}{cc}
% 1, & \text{ if } q_{j}=0%
% \\
% \  &  \\
% 0, & \text{ otherwise. }%
% \end{array}
% \right.
% \end{equation*}
For  $z\in\chij$ define
\begin{equation*}
\mathcal{A}_{z}=\{ i\in\AAA_I:\text{ there exists }j\in\AAA_J\text{ such that }K_{i,j}=1\text{ and } z_{j}=0 \}
\end{equation*}
Also, for  $q\in\mathbb{R}^{J}_{+}$ define $z^{q}\in\chij$ by $z^q\doteq \II_{\{q=0\}}$.
Here, for a given queue-length state $q \in \R_+^J$,  $z^q\in\chij$ is a  vector of indicators which will tell us which queues are `empty' (corresponding to coordinates that equal $1$) and the set $\mathcal{A}_{z^q}$ will tell us  which resources are used by at least one job-type whose queue length is not empty.  In fact, in describing the scheme we will use approximate versions of such indicator vectors  given as $\chij$ valued processes $\mathcal{Z}^r(\cdot)$ (see Definition \ref{def:scheme}) which will tell us which queues are {\em near empty}  so we can push them away from the boundary to avoid `blocking' and ensure the servers can work at full capacity.  
In the policy we describe the nominal  allocation for each job type $j$ is $\rho_j$ (essentially the amount needed to keep up with the arrival rate). This nominal allocation is modified through  two types of vectors, $v^b$ and $v^c$ in $\mathbb{R}^{J}$, that represent the amount we  subtract (amounts can be positive or negative) from the nominal capacity allocation while maintaining capacity constraints.  A positive entry in these vectors indicates an under-allocation that should result in queue length growth and a negative entry results in queue length decline.  

We now introduce a subset $\clm$ of $\chij$ that will play an important role. In the case when $\text{dim}(\mathcal{C}_{K}^{h})=\text{dim}(\text{ker}(K))-1=J-I-1$, we define
\begin{equation*}
\clm\doteq \{ z\in\chij: \text{ there exists } v\in\text{ker}(K)\text{ such that } v_{j}\geq 0 \text{ if } z_{j}=1 \text{ and } v\cdot u_{J-I}>0\}.
\end{equation*}
We define $\mathcal{M}$ to be the empty set when $\text{dim}(\mathcal{C}_{K}^{h})=\text{dim}(\text{ker}(K))=J-I$.
The set $\mathcal{M}$ lists the configurations of empty and nonempty queues, as indicated by $z$, which allow us to reduce cost while maintaining the same workload.  In particular, the following proposition demonstrates that if $z^{q}\not\in\mathcal{M}$ then $q$ is a cost minimizing queue length for the workload $KMq$ and so for such configurations $q$ we cannot reduce cost while maintaining the same workload. 
\begin{proposition}
\label{thm:notMprovesOptimal}
Let $q\in\mathbb{R}^{J}_{+}$ be such that $z^{q}\not\in\mathcal{M}$. Then $\tilde{d}(q)=0$.
\end{proposition}
\begin{proof}
	The result is immediate when $\text{dim}(\mathcal{C}_{K}^{h})=\text{dim}(\text{ker}(K))=J-I$. Consider now the case $\text{dim}(\mathcal{C}_{K}^{h})=\text{dim}(\text{ker}(K))-1=J-I-1$.
Suppose $\tilde{d}(q)>0$.  Then there exists $v\in\Xi(q)$ such that $v\cdot u_{J-I}>0$.  However, $v\in\Xi(q)$ implies that $v\in\text{ker}(K)$ and $q+\beta v\geq 0$.  For any $j\in\AAA_J$ such that $z^q_{j}=1$, we have,
$
0\leq q_{j}+\beta_{j}v_{j}=\beta_{j}v_{j}
$
and since $\beta_{j}>0$ we must have $v_{j}\geq 0$.  Consequently there exists $v\in\text{ker}(K)$ such that $v_{j}\geq 0$ for all $j$ with $z^q_{j}=1 $. Also, $v\cdot u_{J-I}>0$, which says that $z^{q}\in\mathcal{M}$. Thus we have a contradiction and the result follows.
\end{proof}
The following proposition will allow us to construct an allocation policy which under-allocates resource capacity to queues that are close to empty to increase their queue lengths while simultaneously ensuring each resource utilizes its full capacity unless all of the job-types that use  this resource have queue-lengths  near $0$ (in other words, all resources in $\mathcal{A}_{\clz^r(t)}$, where $\clz^r(t)$ is as in Definition \ref{def:scheme}, allocate their full capacity).
%Proof  is given in Section \ref{sec:fullcapvec}.
\begin{proposition}[Proof in Section \ref{sec:fullcapvec}]
\label{thm:fullCapVect}
For any  $z\in\chij$ there exists $v\in\mathbb{R}^{J}$ such that for all $i\in\mathcal{A}_{z}$ we have $(Kv)_i=0$
% \begin{equation*}
% \sum_{j\in\AAA_J}K_{i,j}v_{j}=0
% \end{equation*}
 and for all $j\in\AAA_J$ such that $z_{j}=1$ we have $v_{j}>0$.
\end{proposition}
{\bf Vectors $v^c(z)$ and $v^b(z)$.} We now introduce our main control policy. It will be defined in terms of two vector functions $v^c$ on the finite set $\clm$ and  $v^b$ on the finite set $\chij$.
Define $\rho^* \doteq \min_{j\in \AAA_J} \{\rho_j\}$.
Recall that for any $z\in \mathcal{M}$ we can find $v^c(z) \in \text{ker}(K)$ such that
$v^{c}_{j}(z)\geq 0$ 
if $z_{j}=1$, and $v^{c}(z)\cdot u_{J-I}>0$. Without loss of generality we can assume that
$\rho-v^{c}(z)>\frac{\rho^*}{2}$. Also note that
\begin{align*}
	(h\beta)\cdot v^{c}(z) &= \sum_{m=1}^{J-I}  (v^{c}(z)\cdot u_m )((h\beta)\cdot u_m )=  (v^{c}(z)\cdot u_{J-I}) ((h\beta)\cdot u_{J-I} )\\
	&=
\lambda (v^{c}(z)\cdot u_{J-I}) \doteq \lambda_c(z)<0.\end{align*}
The $c$ superscript refers to the fact that these vectors will be used to reduce the cost while being workoad neutral.  
In particular the vectors $v^{c}(z)$ determine changes to bandwidth allocation (from nominal allocation) that reduces cost in a manner that the rate allocation to queues that are close to empty is not increased. 
Next we introduce the vector function $v^b$.  Let
\begin{equation}\label{def:tillam}
\tilde{\lambda}\doteq \max\{\lambda_{c}(z):z\in \mathcal{M}\}.
\end{equation}
From Proposition \ref{thm:fullCapVect},  for any $z\in \chij$ we can find $v^{b}(z)\in\mathbb{R}^{J}$ such that 
\begin{align}
	|v^{b}(z)|&\leq \min\left\{\frac{\rho^*}{4}, \frac{|\tilde{\lambda}|}{4|\beta||h|}\right\},\;\;  {(Kv^b(z))_i=0}
\mbox{ for all } i\in\mathcal{A}_{z},\nonumber\\
& \mbox{ and } v^{b}_{j}(z)>0 \mbox{ for all } j\in\AAA_J \mbox{ such that } z_{j}=1. \label{eq:919}
\end{align}
 The $b$ superscript refers to the fact that these vectors are used to keep the queue lengths away from the boundary. 
%while not dominating the $v^{c}(z)$ vectors if $z\in\mathcal{M}$.
\begin{defn}[{\bf Control Policy}]
\label{def:scheme}
 % For each $z\in \mathcal{M}$ we associate a pair $(v^{c}(z) , \lambda^{c}(z))\in \text{ker}(K)\times (-\infty,0)$ which satisifes $v^{c}_{j}(z)\geq 0$ if $z_{j}=1$, $\rho-v^{c}(z)>\frac{\rho^*}{2}$ component-wise, and $h\beta\cdot v^{c}(z)=\lambda^{c}(z)<0$.   Let
% \begin{equation*}
% \tilde{\lambda}\doteq \max\{\lambda^{c}(z):z\in \mathcal{M}\}.
% \end{equation*}
% Due to Theorem \ref{thm:fullCapVect} for all $z\in \chij$ we can define $v^{b}(z)\in\mathbb{R}^{J}$ such that $|v^{b}(z)|\leq \min\{\frac{\rho^*}{4},\frac{|\tilde{\lambda}|}{4|\beta||h|}\}$,
% \begin{equation*}
% \sum_{j\in\AAA_J}K_{i,j}v^{b}_{j}(z)=0
% \end{equation*}
% for all $i\in\mathcal{A}_{z}$, and $v^{b}_{j}(z)>0$ for all $j\in\AAA_J$ such that $z_{j}=1$.  The $b$ superscript refers to the fact that these vectors are used to keep the queue lengths away from the boundary while not overpowering the $v^{c}(z)$ vectors if $z\in\mathcal{M}$.
Let $c_{1}< c_{2}$ and $0<\kappa<\frac{1}{4}$ be arbitrary and let $\tilde{c}_{1}=\min_{j}\{\beta_{j}\}c_{1}$ and $\tilde{c}_{2}=\min_{j}\{\beta_{j}\}c_{2}$.
%(we're thinking of $\tilde{c}_{1},\tilde{c}_{2}$ as the queue length versions of $c_{1},c_{2}$).  
Define the $\chij$ valued process
$
\mathcal{Z}^{r}(t)\doteq \II_{\{Q^{r}(t)<\tilde{c}_{2}r^{\kappa}\}}$
% \left\{
% \begin{array}{cc}
% 1, & \text{ if } Q^{r}_{j}(t)<\tilde{c}_{2}r^{\kappa}\\
% \  &  \\
% 0, & \text{ otherwise }%
% \end{array}
% \right.
and, for $j \in \AAA_J$, consider the stopping times:
$
\tilde{\tau}^{r,j}_{1}=\inf\{ t\geq 0: Q^{r}_{j}(t)<\tilde{c}_{1}r^{\kappa}\},
$
and, for $l\ge 1$,
\begin{equation*}
\tilde{\tau}^{r,j}_{2l}=\inf\{ t\geq \tilde{\tau}^{r,j}_{2l-1}: Q^{r}_{j}(t)\geq \tilde{c}_{2}r^{\kappa}\},\;\;
\tilde{\tau}^{r,j}_{2l+1}=\inf\{ t\geq \tilde{\tau}^{r,j}_{2l}: Q^{r}_{j}(t)< \tilde{c}_{1}r^{\kappa}\}.
\end{equation*}
Define
$\mathcal{E}_{j}^{r}(t)\doteq \sum_{l=1}^{\infty} \II_{[\tilde{\tau} _{2l-1}^{r,j},\tilde{\tau} _{2l}^{r,j})}(t)$ for $j\in \N_j$ and let
\begin{equation}
x^{r}(t) \doteq \rho-v^{c}\left(\mathcal{Z}^{r}(t)\right)\II_{\{\mathcal{Z}^{r}(t)\in \mathcal{M} \}}-v^{b}\left(\mathcal{Z}^{r}(t)\right),\;\;\;
\yr(t) \doteq x^{r}(t)\II_{ \{\mathcal{E}^{r}(t)=0\}}. \label{eq:singleline}
\end{equation}
The control policy is then given as
$B^{r}(t)=\int_{0}^{t}\yr(s)ds$, $t\ge 0$.
\end{defn}
Going forward we will assume that $r$ is sufficiently large that $|\rho^{r}-\rho |<\frac{\rho^*}{4}$ component-wise and $(\tilde{c}_{2}-\tilde{c}_{1})r^{\kappa}>1$. \\

To  illustrate  the scheme  we now apply it to the 2LLN network in Example \ref{ex:nontrivial}.  Recall that in this simple example $\text{ker}(K)$ is the one-dimensional space spanned by $u_1=\frac{1}{\sqrt{3}}\left(-1,-1,1\right)'$ and $u_1$ satisfies $\lambda =h\beta\cdot u_1=-\frac{1}{\sqrt{3}}$.  Note that since the first two components of $u_1$ are negative we have $\clm=\{ (0,0,0)',(0,0,1)'\}$ and we can define 
\begin{equation*}
v^c((0,0,0)')=v^c((0,0,1)')=\frac{1}{\sqrt{3}}u_1.
\end{equation*}
This says that if both $q_1$ and $q_2$ are positive we can reduce the cost while maintaining the same workload by moving the queue length in the $\frac{1}{\sqrt{3}}\beta u_1=\frac{1}{3}\left(-1,-1,1\right)'$
direction, and, if not, the queue length provides the minimum cost for its corresponding workload.  For  $z\in\chij$ define 
\begin{align*}
 &v^b_1(z)=\frac{1}{36}\mathcal{I}_{\{z_1=1\}\cup\left(\{z_2=1\}\cap\{z_3=0\}\right)}-\frac{1}{36}\mathcal{I}_{\{z_1=0\}\cap\{z_3=1\}},\\
& v^b_2(z)=\frac{1}{36}\mathcal{I}_{\{z_2=1\}\cup\left(\{z_1=1\}\cap\{z_3=0\}\right)}-\frac{1}{36}\mathcal{I}_{\{z_2=0\}\cap\{z_3=1\}},\\
& v^b_3(z)=\frac{1}{36}\mathcal{I}_{\{z_3=1\}}-\frac{1}{36}\mathcal{I}_{\{z_3=0\}\cap\left(\{z_1=1\}\cup\{z_2=1\}\right)}.
\end{align*}
and note that this definition of $v^b(z)$ satisfies (\ref{eq:919}).  Observe that the magnitudes of $v^b(z)$ and $v^c(z)$ have been chosen so that while $v^b(z)$ is attempting to keep the queue lengths away from the boundary and to avoid any unnecessary idle time it does not overwhelm the cost reduction efforts of $v^c(z)$ that is trying to shift the workload to the less expensive type $3$ jobs. 

\begin{remark}
	\label{rem:policy}
The basic idea underlying the proposed scheme is as follows.
  We want $\tilde{d}\left(\hat{Q}^{r}(t)\right)$ close to zero so that (from Proposition \ref{thm:lambdaEqu}) our queue lengths are near cost-minimizing for the given workload.  When $\mathcal{Z}^{r}(t)\in \mathcal{M}$ the policy uses the vector $v^{c}(\mathcal{Z}^{r}(t))$ to reduce $\tilde{d}\left( \hat{Q}^{r}(t) \right)$ and when $\mathcal{Z}^{r}(t)\not\in \mathcal{M}$, due to Lemma \ref{thm:nearBndZ} and Proposition \ref{thm:costDiffBndQ}, $\tilde{d}\left( \hat{Q}^{r}(t) \right)$ is already close to $0$ so the 
  cost associated with that configuration of queue lengths is close to optimal for the corresponding workload.  We also want  the resources to be utilized at full capacity when their workloads are not near the origin so that asymptotically these workloads behave like reflected Brownian motions.  In order to ensure this we want to prevent all queues from being completely empty so that idleness, when there is work present, is avoided.  To achieve this behavior, when $\hat{Q}^{r}_{j}(t)$ falls below $\tilde{c}_{2}r^{\kappa-1}$ the allocation scheme attempts to increase the corresponding queue.  This is because, due to the property $v^c_j(z)\ge 0$ when $z_j=1$, the vectors $v^{c}(\clz^r(t))$ will not attempt to decrease queue lengths of job types that fall below this threshold 
 % while trying to reduce $\tilde{d}\left(\hat{Q}^{r}(t) \right)$ 
 while the vectors $v^{b}(z)$ will attempt to increase queue lengths of job types below this threshold.  If $\hat{Q}^{r}_{j}(t)$ continues to decline past $\tilde{c}_{1}r^{\kappa-1}$, so that $\mathcal{E}_{j}^{r}(t)=1$,  then we stop processing this job type altogether until the corresponding scaled queue length exceeds $\tilde{c}_{2}r^{\kappa-1}$ again.   The magnitudes of the vectors $v^{c}(z)$ and $v^{b}(z)$ are chosen so that while the vector $v^{b}(z)$ is keeping queue lengths nonempty and ensuring resources can operate at full capacity 
 % while their workloads are not near the boundary 
  it does not overwhelm  $v^{c}(z)$ and prevent it from reducing $\tilde{d}\left(\hat{Q}^{r}(t) \right)$ to make  the queue length configuration to be near cost-minimizing for the associated workload.  Note that as long as at least one job type that uses resource $i$ has a queue length greater than $\tilde{c}_{2}r^{\kappa-1}$ 
  and all of the job types that use resource $i$ are still being processed (meaning $\mathcal{E}_{j}^{r}(t)=0$ for all $j\in\AAA_J$ such that $K_{i,j}=1$)
   then  from the properties $(Kv^b(z))_i=0$ for $i \in \cla_z$, $v^c(z) \in \text{ker}(K)$, and $\yr_j(t)= x_j^r(t)$ for all $j$ with $\mathcal{E}_{j}^{r}(t)=0$, we see that
   under the policy resource $i$ is working at full capacity.  In particular, recalling the relation between $c_2$ and $\tilde c_2$, it
follows that
    for any $i\in\AAA_I$, and large $r$, if
\begin{equation}\label{eq:240n}
\hat{W}^{r}_{i}(t)\geq 2Jc_{2}r^{\kappa-1}
\mbox{ and }
\sum_{j\in\AAA_J} K_{i,j} \mathcal{I}_{\{\mathcal{E}_{j}^{r}(t)=1\}}=0
\mbox{ 
then }
\sum_{j\in\AAA_J}K_{i,j}\yr_{j}(t)=C_{i}.
\end{equation}
To argue that the workload process asymptotically behaves like a  reflected Brownian motion  we want a resource to (almost always) work at full capacity when its workload exceeds (the asymptotically $0$) level $2Jc_{2}r^{\kappa-1}$, namely, in view of \eqref{eq:240n}, we want to show that
$
\sum_{j\in\AAA_J} K_{i,j}\mathcal{I}_{\{\mathcal{E}_{j}^{r}(t)=1\}}>0$
does not happen too frequently.  This key property is established in Proposition \ref{thm:idleTimeBndMark} which provides an exponential decay bound on the probability of this happening frequently.
In addition, Proposition \ref{thm:costDiffResult} shows that under this scheme the difference between the holding cost of the queue-length process and the optimal cost for the corresponding workload is arbitrarily small for large values of $r$.  These two propositions, which capture the two main ingredients of the HGI performance, are crucial to 
the proofs of our main results, namely Theorems \ref{thm:disccost} and \ref{thm:ergcost}, which say that the cost associated with this scheme achieves the hierarchical greedy ideal performance.
\end{remark}

\subsection{Main Results}
\label{sec:mainres}

We now introduce some additional notation used in the rest of the paper.  Some of this notation will be used to expand the state space of the process under our scheme defined above so that it becomes a Markov process.\\
% Let $X\geq0$ be a $\mathcal{F}^{r}(0,0)$-measurable random variable, let $\tau^{r}(\cdot)$ be as in Definition \ref{def:AdmissableCntrl}, and note that due to Theorem \ref{thm:stopTime} $\tau^{r}(X)$ is a $\mathcal{F}^{r}(n,m)$ stopping time.  We use the following to refer to the future arrival and service behavior given the information we have at time $X$.
\begin{defn}
\label{def:newStartArrAndServ}
For $x \in \R_+$ and $j \in \AAA_J$, define
%Assume $X\geq 0$ is $\mathcal{F}^{r}(0,0)$-measurable and $P(X<\infty)=1$.  For all $j\in\{1,...,J\}$ define   
\begin{equation*}  
A^{r,x}_{j}(s)=\max\Big\{n\geq 0:\sum_{l=\tau^{r,A}_{j}(x)+1}^{\tau^{r,A}_{j}(x)+n}u^{r}_{j}(l)\leq s\Big\},\;
S^{r,x}_{j}(s)=\max\Big\{n\geq 0:\sum_{l=\tau^{r,S}_{j}(x)+1}^{\tau^{r,S}_{j}(x)+n}v^{r}_{j}(l)\leq s\Big\}
\end{equation*}
along with their diffusion-scaled versions
\begin{equation*} 
\hat{A}^{r,x}_{j}(s)=\frac{1}{r}A^{r,x}_{j}(r^2s)
-rs\alpha^{r}_{j},\;\;
\hat{S}^{r,x}_{j}(s)=\frac{1}{r}S^{r,x}_{j}(r^2s)-rs\beta^{r}_{j}.
\end{equation*} 
\end{defn}
 % Since $(\hat{A}^{r,t}(\cdot),\hat{S}^{r,t}(\cdot))$ do not start impacting the state of the system until  $\tau^{r,A}(t)$ jobs have arrived and (resp.) $\tau^{r,S}(t)$ jobs have been processed.  This is why we introduce the following.
We now introduce the following processes.
\begin{defn}
\label{def:xi}
For  $t\in [0,\infty)$ and $j\in \AAA_J$ define
\begin{equation*}
\bar{\xi}^{A,r}_{j}(t)=\frac{1}{r^{2}}\sum_{l=1}^{\tau^{r,A}_{j}(t)}u^{r}_{j}(l),\;\;\;
\bar{\xi}^{S,r}_{j}(t)=\frac{1}{r^{2}}\sum_{l=1}^{\tau^{r,S}_{j}(t)}v^{r}_{j}(l).
\end{equation*}
\end{defn}
We will also need the following  fluid-scaled residual arrival and service times at an arbitrary  instant $t$.
\begin{defn}
\label{def:Upsilon}
Let, for $t \ge 0$, $r\in \NN$ and $j \in \AAA_J$,
\begin{equation*}
\bar{\Upsilon}^{A,r}_{j}(t)\doteq 
\bar{\xi}^{A,r}_{j}(t)-\left(t-\bar{\Upsilon}^{A,r}_{j}\right), \;\;\;
\bar{\Upsilon}^{S,r}_{j}(t)\doteq 
\bar{\xi}^{S,r}_{j}(t)-\left(\bar{B}_{j}^{r}(t)-\bar{\Upsilon}^{S,r}_{j}\right).
\end{equation*}
In addition, define
\begin{equation*}
\hat{\Upsilon}^{A,r}_{j}(t)=r\bar{\Upsilon}^{A,r}_{j}(t),\;\;
\hat{\Upsilon}^{S,r}_{j}(t)=r\bar{\Upsilon}^{S,r}_{j}(t), \mbox{ and }
\hat{\Upsilon}^{r}(t)=\left( \hat{\Upsilon}^{A,r}(t),\hat{\Upsilon}^{S,r}(t)\right).
\end{equation*}
\end{defn}
In addition, we will use the following fluid-scaled version of the indicator vector $\cle^r$ which tells us which jobs are currently not being processed due to their queue lengths being close to $0$. 
\begin{defn}
\label{def:markProc}
 For $j \in \N_j$, $t\ge 0$ and $r \in \N$, define
$
\tilde{\mathcal{E}}^{r}_{j}(t) \doteq \mathcal{E}^{r}_{j}(r^{2}t)$.
\end{defn}
From \eqref{eq:303n}, for all $j\in\AAA_J$ and $0\leq s<t$, we have

\begin{align}
\hat{Q}^{r}_{j}(t)&=\hat{Q}^{r}_{j}(s)+r^{-1}\mathcal{I}_{\{t-s\geq \bar{\Upsilon}^{A,r}_{j}(s)>0\}}+r^{-1}A_{j}^{r,s}\left(r^2\left(t-s-\bar{\Upsilon}^{A,r}_{j}(s)\right)^{+}\right)\nonumber\\
&\quad -r^{-1}\mathcal{I}_{\{\bar{B}^{r}_{j}(t)- \bar{B}^{r}_{j}(s)\geq \bar{\Upsilon}^{S,r}_{j}(s)>0\}}-r^{-1}S_{j}^{r,s}\left(r^2\left(\bar{B}^{r}_{j}(t)-\bar{B}^{r}_{j}(s)- \bar{\Upsilon}^{S,r}_{j}(s)\right)^{+}\right).\label{eq:328}
\end{align}

It then follows that
% Consequently, for any $t,s>0$, $y=(q,\hat{\Upsilon},\tilde{\mathcal{E}})\in \mathbb{R}^{3J}_{+}\times \chij$, and $E\in\mathcal{B}(\mathbb{R}^{J}_{+})$ we have
% \begin{equation*}
% P\left(\hat{Q}^{r}(t+s)\in E| \left(\hat{Q}^{r}(t),\hat{\Upsilon}^{r}(t),\tilde{\mathcal{E}}^{r}(t) \right) =y \right)=P\left(\hat{Q}^{r}(s)\in E| \left(\hat{Q}^{r}(0),\hat{\Upsilon}^{r}(0),\tilde{\mathcal{E}}^{r}(0) \right) =y \right).
% \end{equation*}
% In particiular, 
$\hat{Y}^{r}(t) \doteq \left(\hat{Q}^{r}(t),\hat{\Upsilon}^{r}(t),\tilde{\mathcal{E}}^{r}(t) \right)$ is a strong Markov process with values in $\cly^r \doteq \NN_{1/r}^J\times \mathbb{R}^{2J}_{+}\times \chij$, with respect to the filtration $\{\clg^r(t)\}_{t\ge 0}$.
For $y=\left(\hat{q},\hat{\Upsilon},\tilde{\mathcal{E}}\right)\in \cly^r$ and bounded measurable $f:\cly^r\rightarrow\mathbb{R}$ we use the notation
\begin{align*}
&E_y\left[f\left(\left(\hat{Q}^{r}(t),\hat{\Upsilon}^{r}(t),\tilde{\mathcal{E}}^{r}(t) \right) \right)\right] \\
&= E\left[f\left(\left(\hat{Q}^{r}(t),\hat{\Upsilon}^{r}(t),\tilde{\mathcal{E}}^{r}(t) \right) \right) |\left(\hat{Q}^{r}(0),\hat{\Upsilon}^{r}(0),\tilde{\mathcal{E}}^{r}(0) \right)=y\right].
\end{align*}
% Note that with this notation, from the Markov property, for $0\leq s< t$ we have
% \begin{equation*}
% E\left[ f\left(\hat{Y}^{r}(t) \right)|\clg^r(s)) \right]= \clt_{t-s} f(\hat{Y}^{r}(s)).
% %
% % E_{\hat{Y}^{r}(s)}\left[ f\left(\hat{Y}^{r}(t-s) \right) \right].
% \end{equation*}
When $f = 1_A$, we will write $E_{y}\left[ f\left(\hat{Y}^{r}(t) \right)\right]$ as $P_y(\hat{Y}^{r}(t) \in A)$.\\

% \begin{equation*}
% y^{r}_{0}\doteq \left(q^{r},\textbf{0},\mathcal{I}_{\{q^{r}<r^{\kappa-1}\tilde{c}_{1}\}}\right)
% \end{equation*}
% where for all $r$ we have $y^{r}_{0}\in\mathcal{Y}^{r}_{0}$ and we have $\lim_{r\rightarrow\infty}q^{r}=\tilde{q}$ for some $\tilde{q}\in\mathbb{R}^{J}_{+}$.
Recall the holding cost vector $h$ from Section \ref{sec:propcon}. We consider two types of costs.  The first is the infinite horizon discounted cost.  Fix a discount factor $\vsg\in (0,\infty)$.  For  $r \in \NN$ and $y^{r}\in\cly^r$ the infinite horizon discounted cost associated with the control policy $B^r$ in Definition \ref{def:scheme} is
\begin{equation*}
J^{r}_{D}\left(B^{r},y^{r}\right)\doteq \int_{0}^{\infty}e^{-\vsg t}E_{y^{r}}\left[h\cdot \hat{Q}^{r}(t) \right]dt.
\end{equation*}
The second cost we consider is the long-term cost per unit time (also referred to as the ergodic cost).  Define, for $r\in \N$, $y^{r}\in\cly^r$, and $T>0$,
\begin{equation*}
J^{r,T}_{E}(B^{r},y^r)\doteq E_{y^{r}} \left[\frac{1}{T}\int_{0}^{T}h\cdot \hat{Q}^{r}(t)dt \right].
\end{equation*}
Then the long-term cost per unit time  for $B^r$ and $y^r\in\cly^r$ is 
\begin{equation*}
J^{r}_{E}(B^{r},y^r)\doteq \limsup_{T\rightarrow\infty} J^{r,T}_{E}(B^{r},y^r).
\end{equation*}

In order to describe the limit model under the heavy traffic scaling we now recall  the  definition of a Skorokhod map on the positive orthant $\RR_+^d$ associated with  normal reflections at the boundary. The reason such a Skorohod map emerges in our analysis is that, 
(i) due to the local traffic condition and the fact that (under HGI performance) every resource utilizes its full capacity whenever there is work for that resource in the system, the state space of the workload process is all of the positive orthant; and (ii) since this is a `one pass' system, idleness of one resource does not (directly) impact the flow of work to any other resource and as a consequence one does not get oblique reflections from idleness.
\begin{defn}
\label{def:skorohod}
Let $T\in (0,\infty)$ and $f\in\mathcal{D}\left([0,T]:\mathbb{R}^{d}\right)$.  We say that $(\phi,h)\in \mathcal{D}\left([0,T]:\mathbb{R}^{d}\right) \times \mathcal{D}\left([0,T]:\mathbb{R}^{d}\right)$ solves the Skorohod Problem for $f$ if 
\begin{inparaenum}[(a)]
%\item $\phi(0)=f(0)$,
\item $\phi(t)=f(t)+h(t)$ for all $t\in [0,T]$,
\item $h(\cdot)$ is nondecreasing and $h(0)=-f(0)\vee 0$,
\item $\phi(\cdot)\geq 0$,
\item $\int_{[0,\infty)}\mathcal{I}_{\{\phi_{i}(t)>0\}}dh_{i}(t)=0$ for all $i\in\mathbb{N}_{d}$.
\end{inparaenum}
\end{defn}
It is known that there is a unique solution to the above Skorohod problem with normal reflections (which can be essentially regarded as $d$-one dimensional Skorohod problems)  for every $f \in D([0,T]: \mathbb{R}^d)$ and denoting the unique $\phi$ associated with $f$ as $\Gamma_d(f)$, the Skorohod map $\Gamma_d: \mathcal{D}\left([0,T]:\mathbb{R}^{d}\right) \to \mathcal{D}\left([0,T]:\mathbb{R}^{d}\right)$  has the following Lipschitz property: There exists $K_{\Gamma_d} \in (0,\infty)$ such that for all $T>0$ and $f_1, f_2 \in D([0,T]: \mathbb{R}^d)$ 
$$\sup_{0\le t \le T}|\Gamma_d(f_1)(t) - \Gamma_d(f_2)(t)| \le K_{\Gamma_d} \sup_{0\le t \le T}|f_1(t)- f_2(t)|.$$
Note that for $f \in D([0,T]: \mathbb{R}^d)$, $\Gamma_d(f)_i = \Gamma_1(f_i)$ for all $i = 1, \ldots d$.
When $d=I$ we will write $\Gamma_d=\Gamma_I$ as simply $\Gamma$. Also it is easily verified that $K_{\Gamma_1}$ can be taken to be $2$. We refer the reader to \cite[Section 3.6.C]{karshr} for a discussion of the one dimensional Skorohod problem.

Consider diagonal matrices $\Sigma^u$  and $\Sigma^v$ given  with diagonal entries $\{\sigma^{u}_{j}\}_{j=1}^{J}$ and $\{\sigma^{v}_{j}\}_{j=1}^{J}$ resp., and define $\Sigma \doteq KM\left(\Sigma^u + \Sigma^v R\right)M^\mathsf{T}K^\mathsf{T}$.
Let $(\ch \Om, \ch \clf, \{\ch \clf_t\}, \ch P)$ be a filtered probability space on which is given an $I$-dimensional $\{\ch \clf_t\}$-Brownian motion $\{\hat{X}(t)\}$
with drift $\theta$ and covariance matrix $\Sigma$.  For $w_0 \in \RR_+^I$, let $\ch W^{w_0}$ be a $\mathbb{R}_+^I$ valued continuous stochastic process defined as 
\begin{equation}
	\label{eq:eqrbm}
\ch W^{w_0}(t) = \Gamma (w_0 +\hat{X}(\cdot))(t), \; t \ge 0.
\end{equation}
The process $\ch W^{w_0}$ is referred to as a $I$-dimensional reflected Brownian motion with initial value $w_0$, drift $\theta$ and covariance matrix $\Sigma$.
It is well known\cite{harwil1} that, under our conditions (specifically the property $\theta<0$), $\{\ch W^{w_0}\}_{w_0 \in \RR_+^I}$ defines a Markov process that has a unique invariant probability distribution which we denote as $\pi$. 

The {\bf hierarchical greedy ideal} (HGI) associated with the discounted and the ergodic cost, for $w_0 \in \R_+^I$,
 % associated with the costs $J_D^r(B^r, y^{r}_{0}) $ and $J_E^r(B^r) $
are given respectively as
\begin{align}
	\mbox{HGI}_D(w_0) \doteq \int_0^{\infty} e^{-\vsg t} E  \left[\hat{h}(\ch W^{w_0}(t))\right] dt, \;\;\; 
		\mbox{HGI}_E \doteq \int_{\RR^+_I} \hat{h}(w) \pi(dw). \label{eq:hgicosts}
\end{align}
% Note that due to our proof of the other bound if $\hat{h}$ is a nondecreasing function of the workload the costs $\mbox{HGI}_D(w_0)$ and $\mbox{HGI}_E$ are asymptotically optimal.\\

The following theorem says that our scheme achieves the hierarchical greedy ideal infinite horizon discounted cost.
%Proof is given in Section \ref{sec:thmdisccost}.
\begin{theorem}[Proof in Section \ref{sec:thmdisccost}]
	\label{thm:disccost}
	Let $y^{r} = (\hat{q}^r,\hat{\Upsilon}^r,\tilde{\mathcal{E}^r})\in \cly^r$ be an arbitrary sequence satisfying \\
$\lim_{r\rightarrow\infty}\hat q^{r}=\tilde{q}$ for some $\tilde{q}\in\mathbb{R}^{J}_{+}$ and $\sup_{r}r\hat{\Upsilon}^{r}<\infty$.
	 Define $w_0\doteq KM\tilde{q}$. Then\\
$\lim_{r\rightarrow\infty}J^{r}_{D}(B^{r},y^{r})=\mbox{HGI}_D(w_{0})$.
\end{theorem}
The next theorem gives  a similar result for the  long-term cost per unit time. 
%Proof is given in Section \ref{subsec:erg}.
\begin{theorem}[Proof in Section \ref{subsec:erg}]
\label{thm:ergcost}
Let $y^{r}=(\hat{q}^r,\hat{\Upsilon}^r,\tilde{\mathcal{E}^r})\in \cly^r$ be an arbitrary sequence satisfying $\sup_{r} \hat q^{r}<\infty$ and $\sup_{r}r\hat{\Upsilon}^{r}<\infty$.
Then
$
\lim_{r\rightarrow\infty}J^{r}_{E}(B^{r},y^r)=\mbox{HGI}_E$.
\end{theorem}
For some background and rationale for the terminology of HGI for the costs in \eqref{eq:hgicosts}, we refer the reader to \cite{harmandhayan}. 
Roughly speaking the fact that the asymptotic cost for our sequence of policies is given in terms of the function $\hat h$ corresponds to the feature of instantaneous holding cost minimization for the given workload (recall the definition of $\hat h$), and the fact that the limit is determined by a reflected Brownian motion, which has the feature that the  reflection (which roughly corresponds to the  asymptotic idleness) occurs only when the process hits the faces of the orthant (namely one of the coordinates is zero), captures the no-idleness property of HGI.

 In {the} rest of the paper we simplify the notation by leaving out the subscript $y$ on the expected value that specifies the initial condition unless it is particularly relevant in that situation.  
% In addition, I use the subscript $0$ on the initial condition (meaning $y_{0}$) to indicate that this is a ``standard'' initial condition form the set $\mathcal{Y}^{r}_{0}$.  If this $0$ subscript is missing do not assume that the initial condition is in the $\mathcal{Y}^{r}_{0}$.

The following result is immediate from the definition of the control policy $B^r(\cdot)$. For part (d) see Remark \ref{rem:policy} and for part (e) recall that $v_j^b(z)>0$ if $z_j=1$.
\begin{proposition}
\label{thm:schemeSum}
The scheme given in Definition \ref{def:scheme} has following properties:
\begin{enumerate}[(a)]
\item For any $j\in\AAA_J$ and $t\geq 0$ if $\tilde{\mathcal{E}}_{j}^{r}(t)=1$ then $\frac{d}{dt}\bar{B}^{r}_{j}(t)=0$.
\item For any $j\in\AAA_J$ and $t\geq 0$ if $\tilde{\mathcal{E}}_{j}^{r}(t)=0$ then $\tilde{\mathcal{E}}_{j}^{r}(s)=0$ for all $s\geq t$ such that
 $\bar{B}^{r}_{j}(s)-\bar{B}^{r}_{j}(t)<\bar{\Upsilon}^{S,r}_{j}(t))$.
%\tau^{r,S}_{j}(t))$
\item For $r$ sufficiently large for all $j\in\AAA_J$ and $t\geq 0$ we have $x^{r}_{j}(t)\geq \frac{\rho^*}{4}$ and consequently if $\mathcal{E}^{r}_{j}(t)=0$ then  $\yr_{j}(t)\geq \frac{\rho^*}{4}$.
\item For all $i\in\AAA_I$ and $t\geq 0$  if $\hat{W}^{r}_{i}(t)\geq 2Jc_{2}r^{\kappa-1}$ and $\sum_{j=1}^{J} K_{i,j}\mathcal{I}_{\{ \tilde{\mathcal{E}}_{j}^{r}(t)=1\}}=0$ then $\frac{d}{dt} (K\bar{B}^{r}(t))_i=( K\yr)_i(r^{2}t)=C_{i}$.
\item For $r$ sufficiently large there exists $\Delta>0$ such that for all $j\in\AAA_J$ and $t\geq 0$ if
 $Q^{r}_{j}(t)<\tilde{c}_{2}r^{\kappa}$ then $\beta^{r}_{j}\yr_{j}(t)\leq \alpha^{r}_{j}-\Delta$.
\end{enumerate}
\end{proposition}
% \begin{proof}
% This all comes directly from Definition \ref{def:scheme}.
% \end{proof}

The remainder of the paper is organized as follows.  Section 3 proves the main results of this work, namely Theorems \ref{thm:disccost} and \ref{thm:ergcost}, by introducing a set of seven propositions. 
{The rest} of the paper is devoted to the proof of these propositions.
Propositions \ref{thm:tightMeas} and 
\ref{thm:convMeas} give the tightness of certain path occupation measures and characterize the weak limit points. These results are needed only in the treatment of the ergodic cost. Analogous results when the interarrival times and service times are exponential were established in \cite{budjoh} and so we only provide a sketch of the proofs. These are in Section \ref{sec:tightness}. Proposition \ref{thm:XhatLimit} gives a functional central limit theorem which proceeds by standard methods using the heavy traffic condition and central limit theorem for renewal processes. Proof is sketched in Section \ref{sec:pfxhatlim}. Propositions \ref{thm:skorokApprox} and \ref{thm:idleTimeBndMark} are key ingredients in establishing the no-idleness feature of the HGI performance.  The first is proved in Section \ref{sec:thmskorok} while the second is proved in Section \ref{sec:thmidletimebnd} using certain large deviation estimates for renewal processes (Proposition \ref{thm:expTailBnd}) the proofs of which is deferred to Section \ref{sec:thmexpTailBnd}. Proposition \ref{thm:workloadExpBnd}  gives a key uniform in time exponential moment estimate. The proof of this result is in Section   \ref{sec:secworkloadexp}  based on several Lyapunov function lemmas, the proofs of which are relegated to Section \ref{sec:stabproofs}.
Finally Proposition \ref{thm:costDiffResult} is the key ingredient in showing that our policy achieves the instantaneous cost minimization feature of the HGI performance. Proof of this proposition is in Section \ref{sec:costdiffres}, based on some auxiliary lemmas that are proved in Sections \ref{sec:expoestim} and \ref{sec:gennetm}.

\begin{figure}[h!]\centering
\includegraphics[height=2in, width=3in]{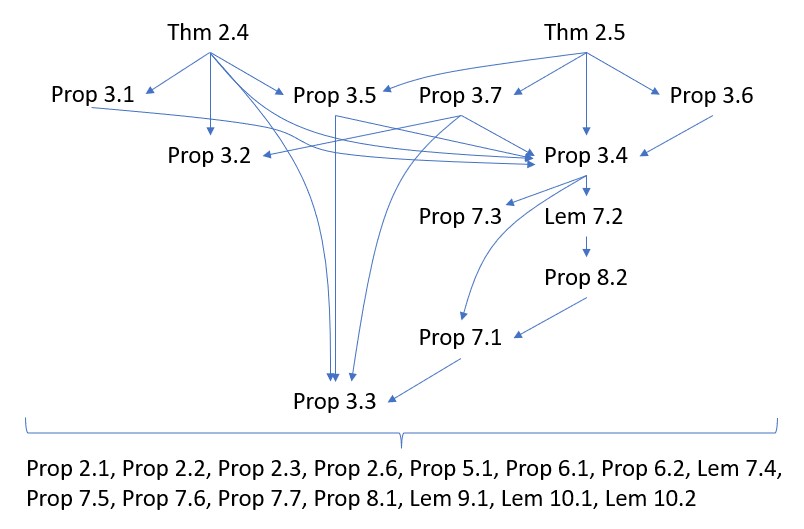}
\caption{The interdependence of results}\label{fig:ResultDependence}
\end{figure}

To describe the interdependence of results of the paper we provide Figure \ref{fig:ResultDependence} with the interpretation that the results can only depend on other results located lower in the figure.  To simplify the figure and to make it easier to read we have placed results which are not directly used in the proofs of HGI performance (Theorems \ref{thm:disccost} and \ref{thm:ergcost}) and are not heavily dependent on other results into a "foundation" group at the bottom.  The proofs of results in this foundation group depend on at most one other result in the paper which is also in the foundation group, but these foundation results may be used in the proofs of results higher up the figure. The interdependence of results outside of the foundation group is indicated using arrows, where, for instance, an arrow going from Thm 2.4 to Prop 3.1 indicates that Proposition \ref{thm:XhatLimit} is used in the proof of Theorem \ref{thm:disccost}.

\section{Proof of Main Theorems}
\subsection{Some Auxiliary Results}
Proof of the following proposition follows by standard methods. A sketch is given in Section \ref{sec:pfxhatlim}.
\begin{proposition}
\label{thm:XhatLimit}
 Let $y^{r} = (\hat{q}^r,\hat{\Upsilon}^r,\tilde{\mathcal{E}^r})\in \cly^r$ be an arbitrary sequence satisfying $\sup_{r}\hat{q}^{r}<\infty$ and $\lim_{r\rightarrow\infty}\hat{\Upsilon}^{r}=0$.  Then  $\hat{X}^{r}(\cdot)\rightarrow\hat{X}(\cdot)$ in distribution in $\sD^I$ where $\hat{X}(\cdot)$ is as introduced above \eqref{eq:eqrbm}.
\end{proposition}
Next two propositions are needed to establish the non-idleness feature of the HGI performance.
% Proof of the next proposition is in Section \ref{sec:thmskorok}.
 \begin{proposition}[Proof in Section \ref{sec:thmskorok}]
 \label{thm:skorokApprox}
For all $r\in \NN$, $y^r \in \cly^r$, $t\geq 0$ and $T\in (0,\infty)$ we have $P_{y^r}$ a.s.,
 \begin{align*}
&\sup_{s\in [0,T]}\left\{| \Gamma\left(\hat{W}^{r}_{i}(t)+\hat{X}^{r}_{i}(t+\cdot)-\hat{X}^{r}_{i}(t)\right)(s)- \hat{W}^{r}_{i}(t+s)|\right\}\\
&\leq 3Jc_{2}r^{\kappa-1}+2r^{-1}C_{i}\sum_{j=1}^{J}\int_{r^{2}t}^{r^{2}(t+T)}\mathcal{I}_{\{ \mathcal{E}_{j}^{r}(s)=1\}}ds.
\end{align*}
 \end{proposition}

% Proof of the following proposition is given in Section \ref{sec:thmidletimebnd}.
\begin{proposition}[Proof in Section \ref{sec:thmidletimebnd}]
\label{thm:idleTimeBndMark}
%Let $y^r\in\mathbb{R}^{3J}_{+}\times\chij$ be arbitrary.  
For each $\epsilon>0$ there exists $B,R\in (0,\infty)$ such that for all $r\geq R$,  $y^r = (\hat{q}^r,\hat{\Upsilon}^r,\tilde{\mathcal{E}^r})\in \cly^r$, $T\geq 1$, and $j\in\AAA_J$ we have
\begin{equation*}
 P_{y^r}\left(\int_{0}^{r^{2}T}\mathcal{I}_{\{ \mathcal{E}_{j}^{r}(s)=1\}}ds\geq r\hat{\Upsilon}^{A,r}_{j}+\epsilon T r^{\frac{7}{4}\kappa}\right)\leq  e^{-B Tr^{\frac{1}{8}\kappa}}.\\
\end{equation*}
\end{proposition}
The proofs of the next two propositions are in Section \ref{sec:workexpandcostdiff}.
The first  gives a key uniform in time exponential moment estimate. The second is key in showing that our policy achieves the instantaneous cost minimization feature of the HGI performance.

\begin{proposition}[Proof in Section \ref{sec:workexpandcostdiff}]
\label{thm:workloadExpBnd}
%Let $\mathcal{Y}^{r}_{0}$ be the set of standard initial conditions as defined in Definition \ref{def:initCond}.  
There exist constants $\tilde{R},\tilde{B}_{1},\tilde{B}_{2},\tilde{B}_{3},\tilde{\delta},\delta>0$ such that for all $r\geq \tilde{R}$, $y^r= (\hat{q}^r,\hat{\Upsilon}^r,\tilde{\mathcal{E}^r})\in\mathcal{Y}^{r}$, $c\in(0,\delta]$, and $t\geq 0$ we have
\begin{equation*}
E_{y^r}\left[e^{c|\hat{W}^{r}(t)|} \right] \leq \tilde{B}_{1}e^{-\tilde{\delta}t+\tilde{B}_{2}\left(|\hat{q}^r|+|\hat{\Upsilon}^r|\right)}+\tilde{B}_{3}.
\end{equation*}
\end{proposition}
%Proof of the next proposition is in Section \ref{sec:costdiffres}.
\begin{proposition}[Proof in Section \ref{sec:workexpandcostdiff}]
\label{thm:costDiffResult}
%Let $\mathcal{Y}^{r}_{0}$ be the set of standard initial conditions as defined in Definition \ref{def:initCond}.  
For any $\epsilon\in (0,1)$ and $M<\infty$ there exist constants $T^{*},R\in(0,\infty)$ such that for all $r\geq R$,   $T\geq T^{*}$, $t\geq 0$, and  $y^r \in \cly^r$ satisfying $\hat{q}^r\leq M$ and $r\hat{\Upsilon}^{r}\leq M$
%$y_{0}\in \mathcal{Y}^{r}_{0}$ 
 we have
\begin{equation*}
E_{y^r}\left[\frac{1}{T}\int_{0}^{T}\left|h\cdot \hat{Q}^{r}(t+s)-\hat{h}\left(\hat{W}^{r}(t+s) \right)\right|ds\right] \leq \epsilon
\end{equation*}
and
\begin{equation*}
E_{y^r}\left[\int_{0}^{\infty}e^{-\vsg s}\left|h\cdot \hat{Q}^{r}(t+s)-\hat{h}\left(\hat{W}^{r}(t+s) \right)\right|ds\right] \leq \epsilon
\end{equation*}
\end{proposition}
\subsection{Proof of Theorem \ref{thm:disccost}}
\label{sec:thmdisccost}
Fix $T\in(0,\infty)$.  From Proposition \ref{thm:XhatLimit}  $\hat{X}^{r}(\cdot)\rightarrow \hat{X}(\cdot)$ in distribution on $D([0,T]: \mathbb{R}^I)$ where $\hat{X}(\cdot)$ is as introduced above \eqref{eq:eqrbm}.  Since $KM^{r}\hat q^{r}\rightarrow w_{0}$, by the continuity of the Skorokhod map, we have
$
\Gamma\left(KM^{r}\hat q^{r}+\hat{X}^{r}(\cdot)\right) \rightarrow \ch W^{w_0}
$
in distribution on $D([0,T]: \mathbb{R}^d)$, where $\ch W^{w_0}(t)$ is the reflected Brownian motion given by \eqref{eq:eqrbm}.  In addition, using Propositions \ref{thm:skorokApprox} and \ref{thm:idleTimeBndMark} (and recalling that $\sup_{r}r\hat{\Upsilon}^{r}<\infty$ and $\kappa <1/4$) we see that
\begin{equation*}
\sup_{t\in[0,T]}|\Gamma\left(KM^{r}\hat q^{r}+\hat{X}^{r}(\cdot)\right)(t)-\hat{W}^{r}(t)|\rightarrow 0
\end{equation*}
in probability.  Combining this gives
$
\hat{W}^{r} \rightarrow \ch W^{w_0}$
in distribution on $D([0,T]: \mathbb{R}^d)$.  For $k\in(0,\infty)$ define
$
\hat{h}_{k}(w)=k\wedge\hat{h}(w)$
and note that $\hat{h}_{k}$ is a bounded, continuous function on $\mathbb{R}^{I}_{+}$.  Consequently, it follows that for any $T<\infty$,
\begin{equation*}
\lim_{r\rightarrow\infty}E\left[\int_{0}^{T}e^{-\vsg t}\hat h_{k}\left(\hat{W}^{r}(t)\right) dt\right]=E\left[\int_{0}^{T}e^{-\vsg t} \hat{h}_{k}\left(\ch W^{w_0}(t)\right)dt \right].
\end{equation*}

% since $\hat{W}^{r}(t) \rightarrow \ch W^{w_0}(t)$ in distribution on $D([0,T]: \mathbb{R}^d)$ and $P\left(\ch W^{w_0}(t)\in C([0,T]: \mathbb{R}^d) \right)=1$.  Now in order to get the result our goal is to show that we can approximate $J^{r}_{D}(B^{r},y^{r}_{0})$ and $E\left[\int_{0}^{\infty}e^{-\vsg t}\hat{h}\left(\ch W^{w_0}(t)\right)dt\right]$ arbitrarily well with $E\left[\int_{0}^{T}e^{-\vsg t}h_{k}\left(\hat{W}^{r}(t)\right) dt\right]$ and $E\left[\int_{0}^{T}e^{-\vsg t}h_{k}\left(\ch W^{w_0}(t)\right)dt \right]$ by choosing $k$ and $T$ sufficiently larger.

Let $\epsilon>0$ be arbitrary.  Proposition \ref{thm:costDiffResult} (recall that $\hat{q}^{r}\rightarrow \tilde{q}$ and $\sup_{r}r\hat{\Upsilon}^{r}<\infty$ ) tells us that there exists $R_{1}\in(0,\infty)$ such that for all $r\geq R_{1}$ we have
\begin{align*}
&\left|E\left[\int_{0}^{\infty}e^{-\vsg t}h\cdot \hat{Q}^{r}(t)dt\right]-E\left[\int_{0}^{\infty}e^{-\vsg t}\hat{h}(\hat{W}^{r}(t))dt\right]\right|\\
&\leq E\left[\int_{0}^{\infty}e^{-\vsg t}|\hat{h}(\hat{W}^{r}(t))-h\cdot \hat{Q}^{r}(t)|dt\right]\leq \frac{\epsilon}{6}.  
\end{align*}
Due to Proposition \ref{thm:workloadExpBnd} (and once more using $\hat{q}^{r}\rightarrow \tilde{q}$ and $\sup_{r}r\hat{\Upsilon}^{r}<\infty$ ) there exists $B_{1}\in (0,\infty)$ and $R_{2}\in[R_{1},\infty)$ such that for all $r\geq R_{2}$ and $t\geq 0$ we have
$
E\left[|\hat{W}^{r}(t)| \right]\leq B_{1}.$
In addition, the definition of $\hat{h}$ implies that there exists $B_{2}<\infty$ such that for all $w\in\mathbb{R}^{I}_{+}$ we have 
$
\hat{h}(w)\leq B_{2}|w|$.
Choose $T_{1} \in (0,\infty)$ sufficiently large that
$
e^{-\vsg T_{1}} \frac{B_{2}B_{1}}{\vsg }\leq \frac{\epsilon}{3}$.
Consequently for all $T\geq T_{1}$ and $r\geq R_{2}$ we have
\begin{eqnarray*}
\left|\int^{\infty}_{0}e^{-\vsg t}E\left[ \hat{h}\left(\hat{W}^{r}(t)\right)\right]dt-\int^{T}_{0}e^{-\vsg t}E\left[ \hat{h}\left( \hat{W}^{r}(t)\right)\right]dt\right|  
% & \int^{\infty}_{T}e^{-\vsg t}E\left[ \hat{h}\left( \hat{W}^{r}(t)\right)\right]dt\\
% &\leq&\int^{\infty}_{T}e^{-\vsg t}B_{\hat{h}}B_{1}dt\\
% &\leq&\frac{B_{\hat{h}}B_{1}}{\theta}e^{-\vsgT}\\
\leq \frac{B_{2}B_{1}}{\vsg }e^{-\vsg T_{1}}
\leq \frac{\epsilon}{3}.
\end{eqnarray*}
Proposition \ref{thm:workloadExpBnd} also implies that there exist constants $K_{1}\in(0,\infty)$ and $R_{3}\in[R_{2},\infty)$ such that for all $r\geq R_{3}$ and $t\geq 0$ we have
$
E\left[\mathcal{I}_{\{|\hat{W}^{r}(t)|\geq K_{1}\}}|\hat{W}^{r}(t)|\right]\leq \vsg\epsilon/(3B_{2})$.
Consequently for $r\geq R_{3}$, $t\geq 0$, and $k\geq K_{1}B_{2}$ we have
\begin{align*}
&E\left[ |\hat{h}(\hat{W}^{r}(t))-\hat{h}_{k}(\hat{W}^{r}(t))|\right] \\
&\le E\left[\mathcal{I}_{\{\hat{h}(\hat{W}^{r}(t))\geq k\}}\hat{h}(\hat{W}^{r}(t)) \right] \leq  B_{2} E\left[\mathcal{I}_{\left\{|\hat W^{r}(t)|\geq K_{1}\right\}}|\hat{W}^{r}(t)| \right]
% \\
% &\leq& E\left[\mathcal{I}_{\{|W^{r}(t)|\geq K_{1}\}}\hat{h}(\hat{W}^{r}(t)) \right]\\
% &\leq& B_{\hat{h}}E\left[\mathcal{I}_{\{|W^{r}(t)|\geq K_{1}\}}\hat{W}^{r}(t) \right]\\
\leq \frac{\vsg\epsilon}{3}.
\end{align*}
Then for all $T<\infty$ we have
\begin{eqnarray*}
\left|E\left[ \int_{0}^{T}e^{-\vsg t}\hat{h}(\hat{W}^{r}(t))dt\right] -E\left[\int_{0}^{T}e^{-\vsg t}\hat{h}_{k}(\hat{W}^{r}(t))dt\right]\right| \leq 
% \int_{0}^{T}e^{-\vsg t}E\left[ |\hat{h}(\hat{W}^{r}(t))-\hat{h}_{k}(\hat{W}^{r}(t))|\right]dt\\
% &\leq& \int_{0}^{\infty}e^{-\vsg t}E\left[ |\hat{h}(\hat{W}^{r}(t))-\hat{h}_{k}(\hat{W}^{r}(t))|\right]dt\\
%&\leq& 
\int_{0}^{\infty}e^{-\vsg t}\frac{\vsg\epsilon}{3}dt 
\leq \frac{\epsilon}{3}.
\end{eqnarray*}

In addition, due to the monotone convergence theorem there exists $T_{2}\in[T_{1},\infty)$ and
$K_{2}\in[K_{1},\infty)$ such that for all $T\geq T_{2}$ and $k\geq K_{2}$ we have
\begin{equation*}
|E\left[\int_{0}^{T}e^{-\vsg t} \hat{h}_{k}\left( \ch W^{w_0}(t)\right) dt\right]-E\left[\int_{0}^{\infty}e^{-\vsg t} \hat{h}\left( \ch W^{w_0}(t)\right) dt\right]|\leq \frac{\epsilon}{6}.
\end{equation*}
Finally, for $k\geq K_{2}$ and $T\geq T_{2}$ we have
\begin{align*}
&\limsup_{r\rightarrow\infty}J^{r}_{D}(B^{r},y^{r}_{0})\\
&= \limsup_{r\rightarrow\infty} E\left[\int_{0}^{\infty}e^{-\vsg t} h\cdot \hat{Q}^{r}(t) dt\right] \le\limsup_{r\rightarrow\infty}E\left[\int_{0}^{\infty}e^{-\vsg t} \hat{h}\left(\hat{W}^{r}(t) \right)dt\right]+\frac{\epsilon}{6}\\
&\leq\limsup_{r\rightarrow\infty}E\left[\int_{0}^{T}e^{-\vsg t} \hat{h}\left( \hat{W}^{r}(t) \right)dt\right]+\frac{3\epsilon}{6}
\leq\limsup_{r\rightarrow\infty} E\left[\int_{0}^{T}e^{-\vsg t} \hat{h}_{k}\left( \hat{W}^{r}(t) \right)dt\right]+\frac{5\epsilon}{6}\\
&= E\left[\int_{0}^{T}e^{-\vsg t} \hat{h}_{k}\left(\ch W^{w_0}(t) \right)dt\right]+\frac{5\epsilon}{6}
\leq E\left[\int_{0}^{\infty}e^{-\vsg t} \hat{h}\left(\ch W^{w_0}(t) \right)dt\right]+\epsilon.
\end{align*}
Similarly
\begin{eqnarray*}
\liminf_{r\rightarrow\infty}J^{r}_{D}(B^{r},y^{r}_{0})\geq E\left[\int_{0}^{\infty}e^{-\vsg t} \hat{h}\left( \ch W^{w_0}(t)\right) dt\right]-\epsilon.\\
\end{eqnarray*}
Since $\epsilon>0$ is arbitrary, the result follows.
\hfill \qed

\subsection{Proof of Theorem \ref{thm:ergcost}}
\label{subsec:erg}
We begin with some auxiliary results.
From Proposition \ref{thm:workloadExpBnd}  it follows that, with $y^r$ as in  Theorem \ref{thm:ergcost},  there exist constants $R,B_1\in(0,\infty)$ and $c>0$ such that for all $r\geq R$ and $t\geq 0$ we have
$
E_{y^r}\left[e^{c|\hat{W}^{r}(t)|} \right] \leq B_{1}$.
As a result $\sup_{r\geq R}\{J^{r}_{E}(B^{r}, y^r)\}<\infty$.
\begin{defn}
\label{def:randMeas}
For $r\in\mathbb{N}$ and $y^r\in \cly^r$ such that $J^{r}_{E}(B^{r}, y^r)<\infty$ choose $T_{r}\in [r,\infty)$ such that $|J^{r,T_{r}}_{E}(B^{r})-J^{r}_{E}(B^{r})|<\frac{1}{r}$.
If $J^{r}_{E}(B^{r}, y^r)=\infty$, set $T_r=1$.  Define the random variable $\nu^{r}$ with values in $\mathcal{P}\left(\mathbb{R}^{I}_{+}\times  D([0,1]: \mathbb{R}^I)\right)$ by
\begin{equation*}
\nu^{r}\doteq \frac{1}{T_{r}}\int_{0}^{T_{r}}\delta_{\left(\hat{W}^{r}(t),\hat{X}^{r}(t+\cdot)-\hat{X}^{r}(t) \right)}dt
\end{equation*}
\end{defn}
The  next two propositions give tightness of the above occupation measure and characterize its weak limit points.
Since analogous results for exponential   primitives were studied in \cite{budjoh} we only give proof sketches
 in Section \ref{sec:tightness}.
\begin{proposition}
\label{thm:tightMeas}
Let $\{\nu^{r}\}$ be as in Definition \ref{def:randMeas} associated with a $\{y^r\} \in \cly^r$ and assume $\sup_{r}\hat{q}^{r}<\infty$ and $\sup_{r}r\hat{\Upsilon}^{r}<\infty$.  Then  $\{\nu^{r}\}$ is tight as a sequence of random variables with values in  $\mathcal{P}\left(\mathbb{R}^{I}_{+}\times D([0,1]: \mathbb{R}^I)\right)$.
\end{proposition}
\begin{proposition}
\label{thm:convMeas}
Assume $\{y^r\}\in \cly^r$ satisfies $\sup_{r}\hat{q}^{r}<\infty$ and $\sup_{r}r\hat{\Upsilon}^{r}<\infty$ and consider a subsequence $\{\nu^{r_{m}}\}_{m=1}^{\infty}$ of the tight sequence in { Proposition} \ref{thm:tightMeas} that converges in distribution to a random variable $\nu^{*}$ with values in $\mathcal{P}\left(\mathbb{R}^{I}_{+}\times D([0,1]: \mathbb{R}^I)\right)$.  Then the coordinate maps $(w, x)$
on  $\mathbb{R}^{I}_{+}\times D([0,1]: \mathbb{R}^I)$ satisfy, under $\nu^*(\omega)$, for a.e. $\omega$, 
\begin{enumerate}[(a)]
\item $x$ is a Brownian motion  with drift $\theta$ and covariance $\Sigma$ with respect to the filtration $\sF(t) = \sigma(w, x(s)) : s\le t)$,   
\item  $\Gamma_{I}(w+x(\cdot))(s) \overset{d}{=} w$ for every $s\in [0,1]$. 
\end{enumerate}
\end{proposition}
We now prove Theorem \ref{thm:ergcost}.  It suffices to show that for any subsequence of $\{J^{r}_{E}(B^{r}, y^r)\}_{r}$ there is a further subsequence that converges to $\mbox{HGI}_E$.  Let $\nu^{r}$ be the random variables given by Definition \ref{def:randMeas} associated with a $y^r \in \cly^r$ as in the statement of Theorem \ref{thm:ergcost}.  For an arbitrary subsequence Propositions \ref{thm:tightMeas} and \ref{thm:convMeas} show that there is a further subsequence (which we will index by $r_m$) which satisfies $\nu^{r_{m}}\rightarrow \nu^{*}$ where $\nu^*$ is such that the coordinate variables $(x,w)$ under $\nu^{*}(\omega)$ for a.e. $\omega$ satisfy $w\overset{d}{=}\Gamma_{I}(w+x(\cdot))$ and $x$ is a Brownian motion with drift $\theta$ and covariance $\Sigma$ with respect to the filtration $\sF(t) = \sigma(w, x(s)) : s\le t)$.  Since the invariant distribution $\pi$ of the reflected Brownian motion in equation (\ref{eq:eqrbm}) is unique it follows that $\nu^{*}_{(1)}(\omega)\overset{d}{=} \pi$ for a.e. $\omega$, where $\nu^{*}_{(1)}$ is the first marginal of $\nu^*$.  Consequently
\begin{equation*}
\int_{\mathbb{R}^{I}_{+}}\hat{h}(w)\pi(dw)=E\left[\int_{\mathbb{R}^{I}_{+}}\hat{h}(w)\nu^{*}_{(1)}(dw)\right].
\end{equation*}
For $k<\infty$ let $\hat{h}_{k}(w) \doteq k\wedge \hat{h}(w)$. Then since $\hat{h}_{k}(\cdot)$ is a bounded, continuous function
\begin{eqnarray*}
\lim_{m\rightarrow\infty}\int_{\mathbb{R}^{I}_{+}}E\left[\hat{h}_{k}(w)\nu^{r_{m}}_{(1)}(dw)\right]=E\left[\int_{\mathbb{R}^{I}_{+}}\hat{h}_{k}(w)\nu^{*}_{(1)}(dw)\right] = \int_{\mathbb{R}^{I}_{+}}\hat{h}_k(w)\pi(dw).
\end{eqnarray*}
  Let
 $\epsilon>0$ be arbitrary.  
As in the proof of Theorem \ref{thm:disccost}, using Proposition\ref{thm:workloadExpBnd} (recall that $\sup_{r}\hat{q}^{r}<\infty$ and $\sup_{r}r\hat{\Upsilon}^{r}<\infty$),  we see that there exist constants $K_{1},R_{1}\in(0,\infty)$ such that for all $r\geq R$, $t\geq 0$, and $k\geq K_{1}$ 
%
% $
% E\left[\mathcal{I}_{\{|\hat{W}^{r}(t)|\geq k\}}|\hat{W}^{r}(t)|\right]\leq  \epsilon/(2B_{\hat{h}})$.
% Consequently for $r\geq R_{1}$, $t\geq 0$, and $k\geq B_{\hat{h}}K_{1}$we have
\begin{equation*}
E\left[ \hat{h}(\hat{W}^{r}(t))-\hat{h}_{k}(\hat{W}^{r}(t))\right] \leq % E\left[\mathcal{I}_{\{\hat{h}(\hat{W}^{r}(t))\geq k\}}\hat{h}(\hat{W}^{r}(t)) \right]
% \leq B_{\hat{h}}  E\left[\mathcal{I}_{\left\{|W^{r}(t)|\geq \frac{k}{B_{\hat{h}}}\right\}} |\hat{W}^{r}(t)| \right]\le  
\frac{\epsilon}{4}.
\end{equation*}
Using the monotone convergence theorem we can choose $K_{2}\in[ K_{1},\infty)$ such that for all $k\geq K_{2}$ 
\begin{equation*}
\left|E\left[\int_{\mathbb{R}^{I}_{+}}\hat{h}(w)\nu^{*}_{(1)}(dw)\right]-E\left[\int_{\mathbb{R}^{I}_{+}}\hat{h}_{k}(w)\nu^{*}_{(1)}(dw)\right]\right|\leq \frac{\epsilon}{2}.
\end{equation*}
In addition, from Proposition \ref{thm:costDiffResult} (recall that $\sup_{r}\hat{q}^{r}<\infty$ and $\sup_{r}r\hat{\Upsilon}^{r}<\infty$) there exists a constant  $R_{2}\in [R_{1},\infty)$ such that for all $r\geq R_{2}$ 
\begin{equation*}
E\left[\frac{1}{T_{r}}\int_{0}^{T_{r}}|h\cdot \hat{Q}^{r}(t)-\hat{h}\left(\hat{W}^{r}(t) \right)|dt\right] \leq \frac{\epsilon}{4}.
\end{equation*}
Consequently for all $k\geq K_{2}$ we have
\begin{align*}
&\limsup_{m\rightarrow\infty}J^{r_{m}}_{E}(B^{r}) \\
&\leq \limsup_{m\rightarrow\infty} E\left[\frac{1}{T_{r_{m}}}\int_{0}^{T_{r_{m}}} h\cdot \hat{Q}^{r}(t) dt\right]\leq \limsup_{m\rightarrow\infty} E\left[\frac{1}{T_{r_{m}}}\int_{0}^{T_{r_{m}}} \hat{h}\left(\hat{W}^{r}(t) \right)dt\right]+\frac{\epsilon}{4}\\
&\leq \limsup_{m\rightarrow\infty} E\left[\frac{1}{T_{r_{m}}}\int_{0}^{T_{r_{m}}} \hat{h}_{k}\left(\hat{W}^{r}(t) \right)dt\right]+\frac{\epsilon}{2}= \limsup_{m\rightarrow\infty}E\left[\int_{\mathbb{R}^{I}_{+}}\hat{h}_{k}(w)\nu^{r_{m}}_{(1)}(dw)\right]+\frac{\epsilon}{2}\\
&= E\left[\int_{\mathbb{R}^{I}_{+}}\hat{h}_{k}(w)\nu^{*}_{(1)}(dw)\right]+\frac{\epsilon}{2} \le E\left[\int_{\mathbb{R}^{I}_{+}}\hat{h}(w)\nu^{*}_{(1)}(dw)\right]+\epsilon= \int_{\mathbb{R}^{I}_{+}}\hat{h}(w)\pi(dw)+\epsilon.
\end{align*}
Similarly,
$
\liminf_{m\rightarrow\infty}J^{r_{m}}_{E}(B^{r})\geq\int_{\mathbb{R}^{I}_{+}}\hat{h}(w)\pi(dw)-\epsilon$.
Since $\epsilon>0$ was arbitrary  the result follows.
\hfill \qed

\section{Proof Sketch of Proposition \ref{thm:XhatLimit}}
\label{sec:pfxhatlim}
Recall the definition of $\hat{X}^{r}(\cdot)$ in \eqref{eq:Xhat}.  The central limit theorem for renewal processes (see e.g. \cite[Theorem 14.6]{BillingsleyConv})
and the fact that $M^{r}\rightarrow M$) implies that $\theta \id(\cdot)+KM^{r}\hat{A}^{r}(\cdot)-KM^{r}\hat{S}^{r}(\rho \id(\cdot))\rightarrow \hat{X}(\cdot)$ in distribution in $\sD^I$, where $\id$ is the identity map.
Clearly $\frac{1}{r}KM^{r}\mathcal{I}_{\{\id(\cdot)\geq\bar{\Upsilon}^{A,r}>0\}}\rightarrow 0$, $\frac{1}{r}KM^{r}\mathcal{I}_{\{\bar{B}^{r}(\cdot)\geq\bar{\Upsilon}^{S,r}>0\}}\rightarrow 0$, and (due to Condition \ref{cond:htc} and the paragraph that follows) $rK(\rho^{r}-\rho)\id(\cdot)\rightarrow \theta\id(\cdot)$ in $\sD^I$.  Since $\lim_{r\rightarrow\infty}\hat{\Upsilon}^{r} =0$ it follows that $\left(\id(\cdot)-\bar{\Upsilon}^{A,r}\right)^{+}\rightarrow \id(\cdot)$, $rK\rho^{r}\left(\id(\cdot)\wedge \bar{\Upsilon}^{A,r}\right)\rightarrow 0$, and $rK\left(\bar{B}^{r}(\cdot)\wedge \bar{\Upsilon}^{S,r}\right)\rightarrow 0$ in $\sD^I$.  In addition, it can be shown using Proposition \ref{thm:workloadExpBnd}, Proposition \ref{thm:expTailBnd},  and the assumptions that $\sup_{r}\hat{q}^{r}<\infty$ and $\sup_{r}r\hat{\Upsilon}^{r}<\infty$ that $\left(\bar{B}^{r}(\cdot)-\bar{\Upsilon}^{S,r}\right)^{+}\rightarrow \rho\id(\cdot)$ in $\sD^I$.  The proof of this is very similar to the proof of \cite[Theorem 15 part 2]{budjoh} and is therefore omitted.  Putting all of this together implies that $\hat{X}^{r}(\cdot)\rightarrow\hat{X}(\cdot)$ in distribution in $\sD^I$ and completes the proof.

%Combining this with  where $\hat{X}(t)$ is an $I$-dimensional Brownian motion as introduced above \eqref{eq:eqrbm}.  Consequently, it suffices to show that $\left(\bar{B}^{r}(t)-\bar{\Upsilon}^{S,r}\right)^{+}\rightarrow \rho t$ in distribution in $C([0,T],\mathbb{R}^{I})$.  The\footnote{discuss} proof of this is  very similar to the proof of Theorem 11 in our paper on the other bound.  The result required here is stronger than what we prove there because it doesn't involve time averaging, but it is easier to achieve because our bound on the magnitude of $\hat{W}^{r}(t)$ here from Proposition \ref{thm:workloadExpBnd} is stronger and doesn't involve time averaging. As a result the details have been omitted.
\hfill \qed

\section{Proof of Proposition
 \ref{thm:skorokApprox}}
\label{sec:thmskorok}
 We begin with an auxiliary result.
 \begin{proposition}
 \label{thm:skorokIneq}
 Let $T\in (0,\infty)$ and $f\in\sD\left([0,T]:\mathbb{R}\right)$ be arbitrary and let $\phi_{1}=\Gamma_1(f)$.  Assume $\phi_{2}=f+h_{2}$ where $h_{2}\in\sD\left([0,T]:\mathbb{R}\right)$ is a nondecreasing function satisfying $0\leq h_{2}(0)\leq  (-f(0))^{+}$ and $\int_{0}^{T}\mathcal{I}_{\{\phi_{2}(t)>0\}}dh_{2}(t)=0$.    Then $\phi_{2} \leq \phi_{1}$.  Let $\phi_{3}=f+h_{3}$ where $h_{3}\in\sD\left([0,T]:\mathbb{R}\right)$ is a nondecreasing function satisfying $h_3(0)\ge 0$, and $\phi_{3}\geq 0$.  Then $ \phi_{1}\leq \phi_{3}$.
 \end{proposition}
 \begin{proof}
 Define $h_{1}(t)=\sup_{s\in[0,t]}\{(-f(s))^{+}\}$ and note that $\phi_{1}=f+h_{1}$.  We will first prove $\phi_{2}\leq \phi_{1}$ and to do so it is sufficient to prove $h_{2}\leq h_{1}$.  Note that $0\leq h_{2}(0)\leq  (-f(0))^{+}$ implies $h_{2}(0)\leq h_{1}(0)$.  Arguing via contradiction, assume there exists $t_{2}^{*}\in(0,T]$ such that $h_{2}(t^{*}_{2})>h_{1}(t^{*}_{2})=\sup_{s\in[0,t^{*}_{2}]}\{(-f(s))^{+}\}$.  For notational convenience let $a\doteq\sup_{s\in[0,t^{*}_{2}]}\{(-f(s))^{+}\}$ and define
$
 t^{*}_{1}=\sup\{s\in[0,T]:h_{2}(s)\leq a\}$.
 Note that since $h_{2}(t^{*}_{2})>a$ we have $t^{*}_{1}\leq t^{*}_{2}$.  If $h_{2}(t^{*}_{1})>a$ then, since $h_{2}(s)\leq a$ for all $s<t^{*}_{1}$, we must have 
 $
 \int_{\{t^{*}_{1}\}}dh_{2}(u)>0$.
However, 
 \begin{align*}
 \phi_{2}(t^{*}_{1})&=f(t^{*}_{1})+h_{2}(t^{*}_{1})>f(t^{*}_{1})+a= f(t^{*}_{1})+\sup_{s\in[0,t^{*}_{2}]}\{(-f(s))^{+}\}\\
 &\geq f(t^{*}_{1})+\sup_{s\in[0,t^{*}_{1}]}\{(-f(s))^{+}\}\geq 0
 \end{align*}
 which contradicts the fact that $\int_{0}^{T}\mathcal{I}_{\{\phi_{2}(t)>0\}}dh_{2}(t)=0$.  Therefore we must have $h_{2}(t^{*}_{1})\leq a$ and so $t^{*}_{1}<t^{*}_{2}$.  By the definition of $t^{*}_{1}$ we have $h_{2}(t)>a$ for all $t\in(t^{*}_{1},t^{*}_{2}]$.  Therefore for any $t\in(t^{*}_{1},t^{*}_{2}]$ we have 
 \begin{equation*}
 \phi_{2}(t)=f(t)+h_{2}(t)>f(t)+a=f(t)+\sup_{s\in[0,t^{*}_{2}]}\{(-f(s))^{+}\}\geq f(t)+\sup_{s\in[0,t]}\{(-f(s))^{+}\}\geq 0.
 \end{equation*}
  However, since $h_{2}(t^{*}_{2})> a \ge h_2(t_1^*)$ we have $\int_{(t^{*}_{1},t^{*}_{2}]}dh_{2}(t)>0$, which, since $\phi_{2}(t)>0$ for all $t\in (t^{*}_{1},t^{*}_{2}]$  contradicts the fact that $\int_{0}^{T}\mathcal{I}_{\{\phi_{2}(t)>0\}}dh_{2}(t)=0$.  Thus we have a contradiction and so  we must have $h_{2}(t)\leq h_{1}(t)$ for all $t\in [0,T]$ which implies $\phi_{2}(t)\leq \phi_{1}(t)$ for all $t\in[0,T]$.  

 Now we will prove that $\phi_{1} \leq \phi_{3}$.  It is sufficient to show that $h_{3}\geq h_{1}$.  Once again arguing via contradiction, assume there exists $t^{*}_{2}\in[0,T]$ such that $h_{3}(t^{*}_{2})<h_{1}(t^{*}_{2})=\sup_{s\in[0,t^{*}_{2}]}\{(-f(s))^{+}\}$.  Then there exists $t^{*}_{1}\in[0,t^{*}_{2}]$ such that $(-f(t^{*}_{1}))^{+}>h_{3}(t^{*}_{2})\geq h_{3}(t^{*}_{1})$ which implies  $(-f(t^{*}_{1}))^{+} =  f(t^{*}_{1}$ and $\phi_{3}(t^{*}_{1})=f(t^{*}_{1})+h_{3}(t^{*}_{1})<f(t^{*}_{1})-f(t^{*}_{1})$ meaning $\phi_{3}(t^{*}_{1})<0$.  However this contradicts the fact that $\phi_{3} \geq 0$ which proves that $h_{3} \geq h_{1}$.
 \end{proof}

We now proceed to the proof of Proposition \ref{thm:skorokApprox}.
Fix $t\ge 0$. Define 
 \begin{equation*}
 \hat{I}^{r}_{i}(s)\doteq \int_{(t,t+s]}\mathcal{I}_{\{\hat{W}^{r}_{i}(u)- Jc_{2}r^{\kappa-1}>0\}}d\hat{U}^{r}_{i}(u), \; s \ge 0.
 \end{equation*}
 Then, from \eqref{eq:eq424}, for $s,t\ge 0$,
 \begin{align*}
\hat{W}^{r}_{i}(t+s)-Jc_{2}r^{\kappa-1}&=\hat{W}^{r}_{i}(t)-Jc_{2}r^{\kappa-1}+\hat{X}^{r}_{i}(t+s)-\hat{X}^{r}_{i}(t)\\
&\quad+\hat{I}^{r}_{i}(s)+\int_{(t,t+s]}\mathcal{I}_{\{\hat{W}^{r}_{i}(u)- Jc_{2}r^{\kappa-1}\leq 0\}}d\hat{U}^{r}_{i}(u)
 \end{align*}
 and consequently from the first part of  Proposition \ref{thm:skorokIneq} we have
  \begin{equation*}
\hat{W}^{r}_{i}(t+s)-Jc_{2}r^{\kappa-1}\leq\Gamma_1\left(\hat{W}^{r}_{i}(t)-Jc_{2}r^{\kappa-1}+\hat{X}^{r}_{i}(t+\cdot)-\hat{X}^{r}_{i}(t)+\hat{I}^{r}_{i}(\cdot)\right)(s)
 \end{equation*}
 for all $s\in [0,T]$.  In addition, 
 \begin{equation*}
\hat{W}^{r}_{i}(t+s)=\hat{W}^{r}_{i}(t)+\hat{X}^{r}_{i}(t+s)-\hat{X}^{r}_{i}(t)+\hat{U}^{r}_{i}(t+s)-\hat{U}^{r}_{i}(t)
 \end{equation*}
 where $\hat{U}^{r}_{i}(t+\cdot)-\hat{U}^{r}_{i}(t)$ is nondecreasing and nonnegative, and $ \hat{W}^{r}_{i}(t+\cdot)\geq 0$ so using the second part of Proposition \ref{thm:skorokIneq} once more, we have
   \begin{equation}\label{eq:916}
\Gamma_1\left(\hat{W}^{r}_{i}(t)+\hat{X}^{r}_{i}(t+\cdot)-\hat{X}^{r}_{i}(t)\right)(s)\leq \hat{W}^{r}_{i}(t+s)
\end{equation}
for all $s\in [0,T]$.  Since 
\begin{align*}
&\sup_{s\in[0,T]}\Big|\Gamma_1\left(\hat{W}^{r}_{i}(t)+\hat{X}^{r}_{i}(t+\cdot)-\hat{X}^{r}_{i}(t)\right)(s)\\
&\quad-\Gamma_1\left(\hat{W}^{r}_{i}(t)-Jc_{2}r^{\kappa-1}+\hat{X}^{r}_{i}(t+\cdot)-\hat{X}^{r}_{i}(t)+\hat{I}^{r}_{i}(\cdot)\right)(s)\Big|\\
&\leq 2\sup_{s\in[0,T]}\Big|\left(\hat{W}^{r}_{i}(t)+\hat{X}^{r}_{i}(t+s)-\hat{X}^{r}_{i}(t)\right)\\
&\quad -\left(\hat{W}^{r}_{i}(t)-Jc_{2}r^{\kappa-1}+\hat{X}^{r}_{i}(t+s)-\hat{X}^{r}_{i}(t)+\hat{I}^{r}_{i}(s)\right)\Big|\\
&\leq 2Jc_{2}r^{\kappa-1}+2\int_{(t,t+T]}\mathcal{I}_{\{\hat{W}^{r}_{i}(u)- Jc_{2}r^{\kappa-1}> 0\}}d\hat{U}^{r}_{i}(u)
\end{align*}
 we have, for all $s\in [0,T]$,
\begin{align}
&\Gamma_1\left(\hat{W}^{r}_{i}(t)+\hat{X}^{r}_{i}(t+\cdot)-\hat{X}^{r}_{i}(t)\right)(s)\nonumber\\
&\geq \Gamma_1\left(\hat{W}^{r}_{i}(t)-Jc_{2}r^{\kappa-1}+\hat{X}^{r}_{i}(t+\cdot)-\hat{X}^{r}_{i}(t)+\hat{I}^{r}_{i}(\cdot)\right)(s)-2Jc_{2}r^{\kappa-1}\nonumber\\
&\quad -2\int_{(t,t+T]}\mathcal{I}_{\{\hat{W}^{r}_{i}(u)- Jc_{2}r^{\kappa-1}> 0\}}d\hat{U}^{r}_{i}(u)\nonumber\\
&\geq  \hat{W}^{r}_{i}(t+s)-3Jc_{2}r^{\kappa-1}-2\int_{(t,t+T]}\mathcal{I}_{\{\hat{W}^{r}_{i}(u)- Jc_{2}r^{\kappa-1}> 0\}}d\hat{U}^{r}_{i}(u).\label{eq:915}
\end{align}
Combining \eqref{eq:916} and \eqref{eq:915}, we have,
\begin{align*}
&\sup_{s\in [0,T]}\left\{| \Gamma_1\left(\hat{W}^{r}_{i}(t)+\hat{X}^{r}_{i}(t+\cdot)-\hat{X}^{r}_{i}(t)\right)(s)- \hat{W}^{r}_{i}(t+s)|\right\}\\
&\leq 3Jc_{2}r^{\kappa-1}+2\int_{(t,t+T]}\mathcal{I}_{\{\hat{W}^{r}_{i}(u)- Jc_{2}r^{\kappa-1}> 0\}}d\hat{U}^{r}_{i}(u).
\end{align*}
In addition due to Proposition \ref{thm:schemeSum} part (d) we have
\begin{eqnarray*}
\int_{(t,t+T]}\mathcal{I}_{\{\hat{W}^{r}_{i}(u)- Jc_{2}r^{\kappa-1}> 0\}}d\hat{U}^{r}_{i}(u)\leq r^{-1}C_{i}\sum_{j=1}^{J}\int_{r^{2}t}^{r^{2}(t+T)}\mathcal{I}_{\{ \mathcal{E}_{j}^{r}(s)=1\}}ds.
\end{eqnarray*}
The result follows.
% \begin{eqnarray*}
% \sup_{s\in [0,T]}\left\{| \Gamma\left(\hat{W}^{r}_{i}(t)+\hat{X}^{r}_{i}(t+\cdot)-\hat{X}^{r}_{i}(t)\right)(s)- \hat{W}^{r}_{i}(t+s)|\right\}\leq 3Jc_{2}r^{\kappa-1}+2r^{-1}C_{i}\sum_{j=1}^{J}\int_{r^{2}t}^{r^{2}(t+T)}\mathcal{I}_{\{ \mathcal{E}_{j}^{r}(s)=1\}}ds.
% \end{eqnarray*}
\hfill \qed

\section{Proof of Proposition \ref{thm:idleTimeBndMark}}
\label{sec:thmidletimebnd}

We will use the following two propositions in the proof of Proposition \ref{thm:idleTimeBndMark}.  The proof of the first proposition is a simpler version of that of Proposition \ref{thm:nextSerTimeBnd} which gives a similar estimate for service times. Since the proof of the latter result is given in full detail in 
Section \ref{sec:propnextser}, we omit the proof of Proposition \ref{thm:nextArrTimeBnd}. We make the convention that $u_j^r(0)=0$ for $r\in \NN$ and $j \in \AAA_J$.
\begin{proposition}
\label{thm:nextArrTimeBnd}
Let $\delta$ be as in Condition \ref{eqn:mgfBnd}.  There exists $R<\infty$ and for any $c<\delta$ a corresponding $K(c)<\infty$ such that for any $j\in\AAA_J$, $y^r \in \cly^r$, and $t\geq 0$
% $\mathcal{F}^{r}(0,0)$-measurable random variable $X\geq 0$ satisfying $P(X<\infty)=1$
we have 
\begin{equation*}
\sup_{r\geq R} E\left[e^{c u^{r}_{j}(\tau^{r,A}_{j}(t))}\right] <K(c).
\end{equation*}
\end{proposition}

The proof of the next proposition is in Section \ref{sec:thmexpTailBnd}.  
\begin{proposition}[Proof in Section \ref{sec:thmexpTailBnd}]
\label{thm:expTailBnd}
Let $j\in \AAA_J$, $c_{1},c_{2} \ge 0$, and $\epsilon>0$ be arbitrary.  Then there exists $B_{1},B_{2},R\in (0,\infty)$ such that for all $T\in [0,\infty)$ and $r\geq R$ we have
\begin{equation}\label{eq:31.d}
P\left(\sup_{0\leq t \leq r^{2c_{1}+c_{2}}T}|A_{j}^{r}(t)-t\alpha^{r}_{j}|\geq \epsilon r^{c_{1}+c_{2}}T\right)\leq B_{1}e^{-r^{c_{2}}T B_{2}}
\end{equation}
and
\begin{equation}\label{eq:31.c}
P\left(\sup_{0\leq t \leq r^{2c_{1}+c_{2}}\max_{i}\{C_{i}\}T}|S_{j}^{r}(t)-t\beta^{r}_{j}|\geq \epsilon r^{c_{1}+c_{2}}T\right)\leq B_{1}e^{-r^{c_{2}}T B_{2}}.
\end{equation}
In particular, for $\kappa \ge 0$,
\begin{equation}\label{eq:31.b}
P\left(\sup_{0\leq t \leq r^{\frac{3}{2}\kappa}T}|A_{j}^{r}(t)-t\alpha^{r}_{j}|\geq \epsilon r^{\kappa}T\right)\leq B_{1}e^{-Tr^{\frac{1}{2}\kappa} B_{2}}.
\end{equation}
and
\begin{equation}\label{eq:31.a}
P\left(\sup_{0\leq t \leq r^{\frac{3}{2}\kappa}\max_{i}\{C_{i}\}T}|S_{j}^{r}(t)-t\beta^{r}_{j}|\geq \epsilon r^{\kappa}T\right)\leq B_{1}e^{-Tr^{\frac{1}{2}\kappa} B_{2}}.
\end{equation}
\end{proposition}
We note that equations (\ref{eq:31.b}) and (\ref{eq:31.a}) are immediate from \eqref{eq:31.d} and
\eqref{eq:31.c} on taking $c_1=c_2 = \kappa/2$.
The former set of equations  are essential to understanding the behavior of $Q^r_j(t)$ in the region $[0, \tilde{c}_2 r^\kappa]$ in the proof of Proposition \ref{thm:idleTimeBndMark}; see for example the proof of \eqref{eq:1204n} completed below \eqref{eq:rev1}, and the estimates in \eqref{eq:revfirstuse} and \eqref{eq:revseconduse}.

We now proceed to the proof of Proposition \ref{thm:idleTimeBndMark}.  \\

\noindent {\bf Overall Proof idea.} Recall, as described in Definition \ref{def:scheme}, that $ \mathcal{E}_{j}^{r}(t)$ changes from $0$ to $1$ when $Q^r_j(t)$ drops below $\tilde{c}_1 r^\kappa$.  This says that we stop processing type $j$ jobs until $Q^r_j(t)$ reaches $\tilde{c}_2 r^\kappa$ at which point $ \mathcal{E}_{j}^{r}(t)$ changes back to $0$ and the processing of type $j$ jobs resumes. The key to this proof is that if $Q^r_j(t)< \tilde{c}_2 r^\kappa $ then type $j$ jobs are processed at a rate lower than their arrival rate (see Proposition \ref{thm:schemeSum} part (e)) in an attempt to boost $Q^r_j(t)$ back up to $\tilde{c}_2 r^\kappa$.  To prove this result we first get an upper bound on  $P(\clc^r )$ where $\clc^r$, defined in (\ref{eq:leaveBndry}), is the event that $Q^r_j(r^2 t)$ fails to exceed $\tilde{c}_1 r^\kappa-1$ by time $t=\bar{\Upsilon}^{A,r}_{j}+Tr^{\frac{3}{2}\kappa-2}$.  Note that due to  Proposition \ref{thm:schemeSum} part (a) once $Q^r_j(r^2 t)$ enters the region $[\tilde{c}_1 r^\kappa -1,\infty)$ it never leaves, and it is convenient for the remainder of this proof to focus on the event $(\clc^r)^c $ where this entrance has occurred by time  $t=\bar{\Upsilon}^{A,r}_{j}+Tr^{\frac{3}{2}\kappa-2}$.  The bulk of the proof is devoted to providing an upper bound on
\begin{equation*}
 P\left( \int_{r\hat{\Upsilon}^{A,r}_{j}+Tr^{\frac{3}{2}\kappa}}^{r^2 T}\mathcal{I}_{\{ \mathcal{E}_{j}^{r}(t)=1\}}dt\geq \epsilon T r^{\frac{7}{4}\kappa}\right).
\end{equation*}
This is accomplished by dividing the time interval $[\bar{\Upsilon}^{A,r}_{j}+Tr^{\frac{3}{2}\kappa-2}, T]$ into subintervals of length $r^{\frac{3}{2}\kappa-2}$ and showing that bad behavior on the $n$th subinterval, as defined by the set $\clu_n^r$ in (\ref{eq:cluSet}), is unlikely.  On the set $(\clc^r)^c $ we can use the frequency of the events $\{\clu_n^r\}_{n\geq 1}$ to bound the amount of time the process is in the state $\mathcal{E}_{j}^{r}(t)=1$ using (\ref{eq:idleTimeBnd}).  Demonstrating that with high probability $\clu_n^r$ only occurs for a small percentage of subintervals completes the proof.

%we don't worry about how often $ \mathcal{E}_{j}^{r}=1$ in the first $\bar{\Upsilon}^{A,r}_{j}+Tr^{\frac{3}{2}\kappa-2}$ part of the $[0,T]$ interval (and make the conservative assumptiont that during this time it is always in this state).  Instead we bound the probability that $Q^r_j$ has not exceeded $\tilde{c}_1 r^\kappa$ prior to 

Let $j\in\AAA_J$ be arbitrary.   Let $R_{1}<\infty$ and $\Delta>0$ be such that for all $r\geq R_{1}$ we have \begin{equation}\label{eq:1141n}
2<r^{\kappa}\frac{\tilde{c}_{2}-\tilde{c}_{1}}{16}, \; r^{\kappa}\frac{5(\tilde{c}_{2}-\tilde{c}_{1})}{\Delta 4}<r^{\frac{3}{2}\kappa}, \mbox{ and if } Q^{r}_{j}(t)<\tilde{c}_{2}r^{\kappa}, \mbox{ then }  \beta^{r}_{j}\yr_{j}(t)\leq  \alpha^{r}_{j}-\Delta.
\end{equation}
Existence of such a $R_1$ and $\Delta$ follows from
Proposition \ref{thm:schemeSum} part (e).

For $n\in \mathbb{N}_0$ define {
$
\hat{t}_{n}=\bar{\Upsilon}^{A,r}_{j}+Tr^{\frac{3}{2}\kappa-2}+(n-1)r^{\frac{3}{2}\kappa-2}$.}
%and note that $\hat{t}_{n}$ is $\mathcal{F}^{r}(0,0)$-measurable.
In addition, for $n\in \mathbb{N}_0$ define the sets
\begin{equation*}
\mathcal{A}^{r,1}_{n}=\left\{ r^{2}\bar{\Upsilon}^{A,r}_{j}(\hat{t}_{n})\geq r^{\kappa}\frac{\tilde{c}_{2}-\tilde{c}_{1}}{\alpha^{r}_{j}16} \right\},\;\;\;
\mathcal{B}^{r,1}_{n}=\left\{ \tau^{r,S}_{j} \left( \hat{t}_{n} \right)<\tau^{r,S}_{j} \left( \hat{t}_{n+1} \right)\right\},
\end{equation*}
and
\begin{eqnarray*}
\mathcal{A}^{r,2}_{n}&=&\Big\{ \sup_{\bar{\xi}^{A,r}_{j}(\hat{t}_{n})\leq s\leq \bar{\xi}^{A,r}_{j}(\hat{t}_{n})+r^{\frac{3}{2}\kappa-2}-\bar{\Upsilon}^{A,r}_{j}(\hat{t}_{n})}\Big|A^{r}_{j}(r^{2}s)-A^{r}_{j}(r^{2}\bar{\xi}^{A,r}_{j}(\hat{t}_{n}))\\
&&\enspace -r^{2}(s-\bar{\xi}^{A,r}_{j}(\hat{t}_{n}))\alpha^{r}_{j}\Big| \quad >r^{\kappa}\frac{\tilde{c}_{2}-\tilde{c}_{1}}{16} \Big\}\\
&=&\Big\{ \sup_{0\leq s\leq r^{\frac{3}{2}\kappa-2} - \bar{\Upsilon}^{A,r}_{j}(\hat{t}_{n})   }\Big|A^{r,\hat{t}_{n}}_{j}(r^{2}s)-r^{2}s \alpha^{r}_{j}\Big| >r^{\kappa}\frac{\tilde{c}_{2}-\tilde{c}_{1}}{16} \Big\},
\end{eqnarray*}

\begin{eqnarray*}
\mathcal{B}^{r,2}_{n}&=&
\Big\{ \sup_{\bar{\xi}^{S,r}_{j}(\hat{t}_{n})\leq s\leq \bar{B}^{r}_{j}(\hat{t}_{n+1})-\bar{B}^{r}_{j}(\hat{t}_{n})+\bar{\xi}^{S,r}_{j}(\hat{t}_{n})-\bar{\Upsilon}^{S,r}_{j}(\hat{t}_{n})}\Big|S^{r}_{j}(r^{2}s)-S^{r}_{j}(r^{2}\bar{\xi}^{S,r}_{j}(\hat{t}_{n}))\\
&&\enspace -r^{2}(s-\bar{\xi}^{S,r}_{j}(\hat{t}_{n}))\beta^{r}_{j}\Big|>r^{\kappa}\frac{\tilde{c}_{2}-\tilde{c}_{1}}{16} \Big\}\\
&=& \Big\{ \sup_{0\leq s\leq \bar{B}^{r}_{j}(\hat{t}_{n+1})-\bar{B}^{r}_{j}(\hat{t}_{n})-\bar{\Upsilon}^{S,r}_{j}(\hat{t}_{n})}\Big|S^{r,\hat{t}_{n}}_{j}(r^{2}s)-r^{2}s\beta^{r}_{j}\Big| >r^{\kappa}\frac{\tilde{c}_{2}-\tilde{c}_{1}}{16} \Big\}.
\end{eqnarray*}
Also define {
\begin{equation}
\label{eq:leaveBndry}
\clc^r = \{Q^r_j(r^2\hat t_1) < r^{\kappa} \tilde c_1 -1\}.
\end{equation}
}
From \eqref{eq:qrt},  we have, for all $s\geq 0$ and $j \in \AAA_J$
\begin{align*}
Q^{r}(r^{2}s)=&q^r_j+A^{r}_{j}(r^{2}(s-\bar{\Upsilon}^{A,r}_{j})^{+})+\mathcal{I}_{\{s\geq \bar{\Upsilon}^{A,r}_{j}>0\}}\\
&-S^{r}_{j}\left(r^{2}\left(\bar{B}^{r}_{j}(s)-\bar{\Upsilon}^{S,r}_{j}\right)^{+}\right)-\mathcal{I}_{\{\bar{B}^{r}_{j}(s)\geq \bar{\Upsilon}^{S,r}_{j}>0\}}
\end{align*}
and for any $n\geq1 $ and $\hat{t}_{n}\leq s \leq \hat{t}_{n+1}$ we have
\begin{align*}
A^{r}_{j}(r^{2}(s-\bar{\Upsilon}^{A,r}_{j})^{+})+\mathcal{I}_{\{s\geq \bar{\Upsilon}^{A,r}_{j}>0\}}=&A^{r}_{j}(r^{2}(\hat{t}_{n}-\bar{\Upsilon}^{A,r}_{j})^{+})+\mathcal{I}_{\{\hat{t}_{n}\geq \bar{\Upsilon}^{A,r}_{j}>0\}}\\
& +\mathcal{I}_{\{s-\hat{t}_{n}\geq\bar{\Upsilon}^{A,r}_{j}(\hat{t}_{n})>0\}}+A^{r,\hat{t}_{n}}_{j}(r^{2}(s-\hat{t}_{n}-\bar{\Upsilon}^{A,r}_{j}(\hat{t}_{n}))^{+})
\end{align*}
and
\begin{align*}
&S^{r}_{j}\left(r^{2}\left(\bar{B}^{r}_{j}(s)-\bar{\Upsilon}^{S,r}_{j}\right)^{+}\right)+\mathcal{I}_{\{\bar{B}^{r}_{j}(s)\geq \bar{\Upsilon}^{S,r}_{j}>0\}}\\
&= S^{r}_{j}\left(r^{2}\left(\bar{B}^{r}_{j}(\hat{t}_{n})-\bar{\Upsilon}^{S,r}_{j}\right)^{+}\right)+\mathcal{I}_{\{\bar{B}^{r}_{j}(\hat{t}_{n})\geq \bar{\Upsilon}^{S,r}_{j}>0\}} +\mathcal{I}_{\{\bar{B}^{r}_{j}(s)-\bar{B}^{r}_{j}(\hat{t}_{n})\geq \bar{\Upsilon}^{S,r}_{j}(\hat{t}_{n})>0\}}\\
&\quad +S^{r,\hat{t}_{n}}_{j}\left(r^{2}\left(\bar{B}^{r}_{j}(s)-\bar{B}^{r}_{j}(\hat{t}_{n})- \bar{\Upsilon}^{S,r}_{j}(\hat{t}_{n})\right)^{+}\right).
\end{align*}
Let 
$
\zeta^{r,n}_{0}\doteq \inf\{s\geq \hat{t}_{n}:Q^{r}_{j}(r^{2}s)\geq r^{\kappa}\tilde{c}_{2}\}
$
and for $l\geq 0$ define
\begin{equation*}
\zeta^{r,n}_{2l+1}\doteq \inf\{s\geq \zeta^{r,n}_{2l}:Q^{r}_{j}(r^{2}s)< r^{\kappa}\tilde{c}_{2}\},\;\;
\zeta^{r,n}_{2l+2}\doteq \inf\{s\geq \zeta^{r,n}_{2l+1}:Q^{r}_{j}(r^{2}s)\geq r^{\kappa}\tilde{c}_{2}\}.
\end{equation*}
If $\zeta^{r,n}_{0}>\hat{t}_{n}$ then for all $s\in [\hat{t}_{n},\zeta^{r,n}_{0})$ we have $Q^{r}_{j}(r^{2}s)<r^{\kappa}\tilde{c}_{2}$.  Consequently for all $s\in [\hat{t}_{n},\zeta^{r,n}_{0})$ we have $\beta^{r}_{j}\left(\bar{B}^{r}_{j}(s)-\bar{B}^{r}_{j}(\hat{t}_{n})\right)\leq \left(s-\hat{t}_{n}\right)\left( \alpha^{r}_{j}-\Delta\right)$.  In addition, on $(\clc^r)^c$, $Q^{r}_{j}(r^{2}s)\geq r^{\kappa}\tilde{c}_{1}-1$ for all $s\geq \hat t_1$ because if $Q^{r}_{j}(r^{2}s)<r^{\kappa}\tilde{c}_{1}$ then $\tilde{\mathcal{E}}^{r}_{j}(s)=1$ so $\frac{d}{ds}\bar{B}^{r}_{j}(s)=0$ due to Proposition \ref{thm:schemeSum} part (a).  Consequently, with 
{
\begin{equation}
\label{eq:cluSet}
\clu_n^r \doteq \mathcal{A}^{r,1}_{n}\cup \mathcal{A}^{r,2}_{n} \cup (\mathcal{B}^{r,1}_{n}\mathcal{B}^{r,2}_{n}),
\end{equation}
}
%$\mathcal{D}^{r}_{n} \doteq \clc^r \cup \mathcal{A}^{r,1}_{n}\cup \mathcal{A}^{r,2}_{n} \cup (\mathcal{B}^{r,1}_{n}\mathcal{B}^{r,2}_{n})$,
 on the set $(\clu_n^r)^{c}\cap (\clc^r)^{c}$ for all $s\in [\hat{t}_{n},\zeta^{r,n}_{0})\cap [\hat{t}_{n}, \hat{t}_{n+1}]$ we have
\begin{align}
Q^{r}(r^{2}s) &= Q^{r}(r^{2}\hat{t}_{n})+\mathcal{I}_{\{s-\hat{t}_{n}\geq \bar{\Upsilon}^{A,r}_{j}(\hat{t}_{n})>0\}} -\mathcal{I}_{\{\bar{B}^{r}_{j}(s)-\bar{B}^{r}_{j}(\hat{t}_{n})\geq \bar{\Upsilon}^{S,r}_{j}(\hat{t}_{n})>0\}}\nonumber\\
&\quad+A^{r,\hat{t}_{n}}_{j}(r^{2}(s-\hat{t}_{n}-\bar{\Upsilon}^{A,r}_{j}(\hat{t}_{n}) )^{+})-S^{r,\hat{t}_{n}}_{j}(r^{2}(\bar{B}^{r}_{j}(s)-\bar{B}^{r}_{j}(\hat{t}_{n})-\bar{\Upsilon}^{S,r}_{j}(\hat{t}_{n}))^{+})\nonumber\\
&\geq r^{\kappa}\tilde{c}_{1}-2+r^{2}\alpha^{r}_{j}(s-\hat{t}_{n}-\bar{\Upsilon}^{A,r}_{j}(\hat{t}_{n}) )^{+}-r^{2}\beta^{r}_{j}(\bar{B}^{r}_{j}(s)-\bar{B}^{r}_{j}(\hat{t}_{n})-\bar{\Upsilon}^{S,r}_{j}(\hat{t}_{n}))^{+} \nonumber\\
&\quad-r^{\kappa}\frac{\tilde{c}_{2}-\tilde{c}_{1}}{8}\nonumber\\
&\geq r^{\kappa}\tilde{c}_{1}-2+r^{2}\alpha^{r}_{j}\left(s-\hat{t}_{n}\right)-r^{2}\alpha^{r}_{j} \bar{\Upsilon}^{A,r}_{j}(\hat{t}_{n})-r^{2}\beta^{r}_{j}\left(\bar{B}^{r}_{j}(s)-\bar{B}^{r}_{j}(\hat{t}_{n})\right)\nonumber\\
&\quad-r^{\kappa}\frac{\tilde{c}_{2}-\tilde{c}_{1}}{8}\nonumber\\
&\geq r^{\kappa}\tilde{c}_{1}-r^{\kappa}\frac{\tilde{c}_{2}-\tilde{c}_{1}}{16}+r^{2}\alpha^{r}_{j}\left(s-\hat{t}_{n}\right)-\alpha^{r}_{j}\left(r^{\kappa}\frac{\tilde{c}_{2}-\tilde{c}_{1}}{16\alpha^{r}_{j}} \right)\nonumber\\
&\quad-r^{2}\beta^{r}_{j}\left(\bar{B}^{r}_{j}(s)-\bar{B}^{r}_{j}(\hat{t}_{n})\right) -r^{\kappa}\frac{\tilde{c}_{2}-\tilde{c}_{1}}{8}\nonumber\\
&\geq r^{\kappa}\tilde{c}_{1}-r^{\kappa}\frac{\tilde{c}_{2}-\tilde{c}_{1}}{4}+\Delta r^{2}\left(s-\hat{t}_{n}\right)\label{eq:227}
\end{align}
where the last two lines use \eqref{eq:1141n}.
% since $2<r^{\kappa}\frac{\tilde{c}_{2}-\tilde{c}_{1}}{16}$ and $ r^{2}\bar{\Upsilon}^{A,r}_{j}(\hat{t}_{n})\leq  r^{\kappa}\frac{\tilde{c}_{2}-\tilde{c}_{1}}{16\alpha^{r}_{j}}$.
Recall that $\Delta>0$ and from \eqref{eq:1141n} $r^2(\hat{t}_{n+1}-\hat{t}_{n})>r^{\kappa}\frac{5(\tilde{c}_{2}-\tilde{c}_{1})}{4\Delta }$ so if $\zeta^{r,n}_{0}-\hat{t}_{n}>r^{\kappa-2}\frac{5(\tilde{c}_{2}-\tilde{c}_{1})}{4\Delta }$ then, on $(\clu_n^r)^{c}\cap (\clc^r)^{c}$,
\begin{equation*}
Q^{r}\left(r^{2}\left(\hat{t}_{n}+r^{\kappa-2}\frac{5(\tilde{c}_{2}-\tilde{c}_{1})}{4\Delta }\right)\right)\geq  r^{\kappa}\tilde{c}_{1}-r^{\kappa}\frac{\tilde{c}_{2}-\tilde{c}_{1}}{4}+\Delta r^{2}\left(r^{\kappa-2}\frac{5(\tilde{c}_{2}-\tilde{c}_{1})}{4\Delta }\right)= r^{\kappa}\tilde{c}_{2}
\end{equation*}
which contradicts the definition of $\zeta^{r,n}_{0}$ so on the set $(\clu_n^r)^{c}\cap (\clc^r)^{c}$ we have
 $\zeta^{r,n}_{0}-\hat{t}_{n}\leq r^{\kappa-2}\frac{5(\tilde{c}_{2}-\tilde{c}_{1})}{4\Delta}$ and consequently $\zeta^{r,n}_{0}<\hat{t}_{n+1}$.  Therefore on the set $(\clu_n^r)^{c}\cap (\clc^r)^{c}$ we have
 \begin{equation}\label{eq:1143n}
 %\int_{r^{2}\hat{t}_{n}}^{r^{2}\zeta^{r,n}_{0}}ds= 
 r^{2}\left(\zeta^{r,n}_{0}-\hat{t}_{n}\right)\leq  r^{\kappa}\frac{5(\tilde{c}_{2}-\tilde{c}_{1})}{4}.
 \end{equation}
 By definition, for all $l\geq 0$,  $Q^{r}_{j}(r^{2}s)\geq r^{\kappa}\tilde{c}_{2}$ for all $s\in [\zeta^{r,n}_{2l},\zeta^{r,n}_{2l+1}]$ so $\mathcal{E}_{j}^{r}(r^{2}s)=0$ for all $s\in [\zeta^{r,n}_{2l},\zeta^{r,n}_{2l+1})$.  On the set $(\clu_n^r)^{c}\cap (\clc^r)^{c}$ for any $l\geq 0$ satisfying $\zeta^{r,n}_{2l+1}<\hat{t}_{n+1}$ and $s\in[\zeta^{r,n}_{2l+1},\zeta^{r,n}_{2l+2})\cap [\zeta^{r,n}_{2l+1}, \hat{t}_{n+1}]$ we have
\begin{align*}
&Q^{r}(r^{2}s)\\
&=Q^{r}(r^{2}\zeta^{r,n}_{2l+1})+\mathcal{I}_{\{s-\hat{t}_{n}\geq \bar{\Upsilon}^{A,r}_{j}(\hat{t}_{n})>\zeta^{r,n}_{2l+1}-\hat{t}_{n}\}} -\mathcal{I}_{\{\bar{B}^{r}_{j}(s)-\bar{B}^{r}_{j}(\hat{t}_{n})\geq  \bar{\Upsilon}^{S,r}_{j}(\hat{t}_{n})>\bar{B}^{r}_{j}(\zeta^{r,n}_{2l+1})-\bar{B}^{r}_{j}(\hat{t}_{n})\}}\\
&\quad+\left(A^{r,\hat{t}_{n}}_{j}(r^{2}(s-\hat{t}_{n}-\bar{\Upsilon}^{A,r}_{j}(\hat{t}_{n}))^{+})-A^{r,\hat{t}_{n}}_{j}(r^{2}(\zeta^{r,n}_{2l+1}-\hat{t}_{n}-\bar{\Upsilon}^{A,r}_{j}(\hat{t}_{n}))^{+})\right)\\
&\quad-\left(S^{r,\hat{t}_{n}}_{j}(r^{2}(\bar{B}^{r}_{j}(s)-\bar{B}^{r}_{j}(\hat{t}_{n})-\bar{\Upsilon}^{S,r}_{j}(\hat{t}_{n}))^{+})\right. \\
&\quad  \left.-S^{r,\hat{t}_{n}}_{j}(r^{2}(\bar{B}^{r}_{j}(\zeta^{r,n}_{2l+1})-\bar{B}^{r}_{j}(\hat{t}_{n})-\bar{\Upsilon}^{S,r}_{j}(\hat{t}_{n}))^{+})\right)\\
&\geq r^{\kappa}\tilde{c}_{2}-2-r^{\kappa}\frac{\tilde{c}_{2}-\tilde{c}_{1}}{8}+r^{2}\alpha^{r}_{j}\left((s-\hat{t}_{n}-\bar{\Upsilon}^{A,r}_{j}(\hat{t}_{n}))^{+}-(\zeta^{r,n}_{2l+1}-\hat{t}_{n}-\bar{\Upsilon}^{A,r}_{j}(\hat{t}_{n}))^{+}\right)\\
&\quad -r^{\kappa}\frac{\tilde{c}_{2}-\tilde{c}_{1}}{8}-r^{2}\beta^{r}_{j}\left((\bar{B}^{r}_{j}(s)-\bar{B}^{r}_{j}(\hat{t}_{n})-\bar{\Upsilon}^{S,r}_{j}(\hat{t}_{n}))^{+}\right.\\
&\quad \left.-(\bar{B}^{r}_{j}(\zeta^{r,n}_{2l+1})-\bar{B}^{r}_{j}(\hat{t}_{n})-\bar{\Upsilon}^{S,r}_{j}(\hat{t}_{n}))^{+}\right)\\
\end{align*}
since $Q^{r}_{j}\left( r^{2}\zeta^{r,n}_{2l+1}\right)\geq r^{\kappa}\tilde{c}_{2}-1$.
Consequently
\begin{align*}
&Q^{r}(r^{2}s)\\
&\geq r^{\kappa}\tilde{c}_{2}-r^{\kappa}\frac{5(\tilde{c}_{2}-\tilde{c}_{1})}{16}+r^{2}\alpha^{r}_{j}\left(s-\zeta^{r,n}_{2l+1}\right)-r^{2}\alpha^{r}_{j} \bar{\Upsilon}^{A,r}_{j}(\hat{t}_{n})-r^{2}\beta^{r}_{j}\left(\bar{B}^{r}_{j}(s)-\bar{B}^{r}_{j}(\zeta^{r,n}_{2l+1})\right)\\
&\geq r^{\kappa}\tilde{c}_{2}-r^{\kappa}\frac{5(\tilde{c}_{2}-\tilde{c}_{1})}{16}+r^{2}\Delta\left(s-\zeta^{r,n}_{2l+1}\right)-\alpha^{r}_{j}\left(r^{\kappa}\frac{\tilde{c}_{2}-\tilde{c}_{1}}{\alpha^{r}_{j}16} \right)\\
&\geq  r^{\kappa}\left(\frac{5}{8}\tilde{c}_{2}+\frac{3}{8}\tilde{c}_{1}\right)+ r^{2}\Delta\left(s-\zeta^{r,n}_{2l+1}\right),\\
\end{align*}
where once more we have used \eqref{eq:1141n}.
% since, $2<r^{\kappa}\frac{\tilde{c}_{2}-\tilde{c}_{1}}{16}$, $ r^{2} \bar{\Upsilon}^{A,r}_{j}(\hat{t}_{n})\leq r^{\kappa}\frac{\tilde{c}_{2}-\tilde{c}_{1}}{\alpha^{r}_{j}16}$, and for all $s\in [\zeta^{r,n}_{2l+1},\zeta^{r,n}_{2l+2})$ we have
%  $\beta^{r}_{j}\left(\bar{B}^{r}_{j}(s)-\bar{B}^{r}_{j}(\zeta^{r,n}_{2l+1})\right)\leq \left(s-\zeta^{r,n}_{2l+1}\right)\left( \alpha^{r}_{j}-\Delta\right)$. 
Because $\Delta>0$ we have $Q^{r}_{j}(r^{2}s)\geq r^{\kappa}\left(\frac{5}{8}\tilde{c}_{2}+\frac{3}{8}\tilde{c}_{1}\right)$ for all $s\in [\zeta^{r,n}_{2l+1},\zeta^{r,n}_{2l+2}\wedge \hat{t}_{n+1}]$ and because by definition $\mathcal{E}_{j}^{r}(r^{2}\zeta^{r,n}_{2l+1})=0$ we have $\mathcal{E}_{j}^{r}(r^{2}s)=0$ for all $s\in [\zeta^{r,n}_{2l+1},\zeta^{r,n}_{2l+2})\cap  [\zeta^{r,n}_{2l+1},\hat{t}_{n+1}]$.  Since $l\geq 0$ such that $\zeta^{r,n}_{2l+1}<\hat{t}_{n+1}$ was arbitrary we have $\mathcal{E}_{j}^{r}(r^{2}s)=0$ for all $s\in [\zeta^{r,n}_{0}, \hat{t}_{n+1}]$ on the set $(\clu_n^r)^{c}\cap (\clc^r)^{c}$.  In addition, we have shown that on the set $(\clu_n^r)^{c}\cap (\clc^r)^{c}$ we have $Q^{r}_{j}(r^{2}s)\geq  r^{\kappa}\left(\frac{5}{8}\tilde{c}_{2}+\frac{3}{8}\tilde{c}_{1}\right)$ for all  $s\in [\zeta^{r,n}_{0}, \hat{t}_{n+1}]$ so $Q^{r}_{j}(r^{2}\hat{t}_{n+1})\geq  r^{\kappa}\left(\frac{5}{8}\tilde{c}_{2}+\frac{3}{8}\tilde{c}_{1}\right)$.  Consequently, from \eqref{eq:1143n}, on the set $(\clu_n^r)^{c}\cap (\clc^r)^{c}$ we have 
 \begin{equation}\label{eq:240}
 \int_{r^{2}\hat{t}_{n}}^{r^{2}\hat{t}_{n+1}}\mathcal{I}_{\{ \mathcal{E}_{j}^{r}(s)=1\}}ds= \int_{r^{2}\hat{t}_{n}}^{r^{2}\zeta^{r,n}_{0}}\mathcal{I}_{\{ \mathcal{E}_{j}^{r}(s)=1\}}ds+ \int_{r^{2}\zeta^{r,n}_{0}}^{r^{2}\hat{t}_{n+1}}\mathcal{I}_{\{ \mathcal{E}_{j}^{r}(s)=1\}}ds\leq r^{\kappa}\frac{5(\tilde{c}_{2}-\tilde{c}_{1})}{4},
 \end{equation}
 \begin{equation}\label{eq:241}
 \mathcal{E}_{j}^{r}(r^{2}\hat{t}_{n+1})=0, \mbox{ and, }
 Q^{r}_{j}(r^{2}\hat{t}_{n+1})\geq  r^{\kappa}\left(\frac{5}{8}\tilde{c}_{2}+\frac{3}{8}\tilde{c}_{1}\right).
 \end{equation}
 In particular, with $H^{r}_{n}=\{ \mathcal{E}_{j}^{r}(r^{2}\hat{t}_{n})=0 \}\cap \left\{Q^{r}_{j}(r^{2}\hat{t}_{n})\geq r^{\kappa}\left(\frac{5}{8}\tilde{c}_{2}+\frac{3}{8}\tilde{c}_{1}\right) \right\}$, for $n> 2$ we have
 \begin{equation}\label{eq:256}
 (H^{r}_{n})^{c}\cap (\clu_{n-1}^r)^{c}\cap (\clc^r)^{c}=\emptyset
 \end{equation}
 Next, on the set $H^{r}_{n}\cap (\clu_n^r)^{c}\cap (\clc^r)^{c}$ for
  $s\in [\hat{t}_{n}, \zeta^{r,n}_{0})$ (recall that $\zeta^{r,n}_{0}<\hat{t}_{n+1}$ on $(\clu_n^r)^{c}\cap (\clc^r)^{c}$) we have, from similar calculations as in \eqref{eq:227}
\begin{eqnarray*}
Q^{r}(r^{2}s)
% &=&Q^{r}(r^{2}\hat{t}_{n})+\mathcal{I}_{\{s-\hat{t}_{n}\geq \bar{\Upsilon}^{A,r}_{j}(\hat{t}_{n})>0\}} -\mathcal{I}_{\{\bar{B}^{r}_{j}(s)\geq r^{-1}\hat{\Upsilon}^{S,r}_{j}>\bar{B}^{r}_{j}(\hat{t}_{n})\}}-\mathcal{I}_{\{\bar{B}^{r}_{j}(s)-\bar{B}^{r}_{j}(\hat{t}_{n})\geq \bar{\Upsilon}^{S,r}_{j}(\hat{t}_{n})>0\}}\\
% &&+A^{r,\hat{t}_{n}}_{j}(r^{2}(s-\hat{t}_{n}- \bar{\Upsilon}^{A,r}_{j}(\hat{t}_{n}))^{+})-S^{r,\hat{t}_{n}}_{j}(r^{2}(\bar{B}^{r}_{j}(s)-\bar{B}^{r}_{j}(\hat{t}_{n})-\bar{\Upsilon}^{S,r}_{j}(\hat{t}_{n}))^{+})\\
% &\geq&  r^{\kappa}\left(\frac{5}{8}\tilde{c}_{2}+\frac{3}{8}\tilde{c}_{1}\right)-2+r^{2}\alpha^{r}_{j}(s-\hat{t}_{n}- \bar{\Upsilon}^{A,r}_{j}(\hat{t}_{n}))^{+}-r^{2}\beta^{r}_{j}(\bar{B}^{r}_{j}(s)-\bar{B}^{r}_{j}(\hat{t}_{n})-\bar{\Upsilon}^{S,r}_{j}(\hat{t}_{n}))^{+}-r^{\kappa}\frac{\tilde{c}_{2}-\tilde{c}_{1}}{8}\\
% &\geq&r^{\kappa}\left(\frac{5}{8}\tilde{c}_{2}+\frac{3}{8}\tilde{c}_{1}\right)+r^{2}\alpha^{r}_{j}\left(s-\hat{t}_{n}\right)-r^{2}\alpha^{r}_{j} \bar{\Upsilon}^{A,r}_{j}(\hat{t}_{n})-r^{2}\beta^{r}_{j}\left(\bar{B}^{r}_{j}(s)-\bar{B}^{r}_{j}(\hat{t}_{n})\right)-r^{\kappa}\frac{3(\tilde{c}_{2}-\tilde{c}_{1})}{16}\\
&\geq&r^{\kappa}\left(\frac{5}{8}\tilde{c}_{2}+\frac{3}{8}\tilde{c}_{1}\right)+r^{2}\Delta \left(s-\hat{t}_{n}\right)-\alpha^{r}_{j}\left(r^{\kappa}\frac{\tilde{c}_{2}-\tilde{c}_{1}}{\alpha^{r}_{j}16} \right)-r^{\kappa}\frac{3(\tilde{c}_{2}-\tilde{c}_{1})}{16}\\
&\geq&r^{\kappa}\left(\frac{3}{8}\tilde{c}_{2}+\frac{5}{8}\tilde{c}_{1}\right)+\Delta r^{2}\left(s-\hat{t}_{n}\right).
\end{eqnarray*}
  % since $2<r^{\kappa}\frac{\tilde{c}_{2}-\tilde{c}_{1}}{16}$, $  r^{2} \bar{\Upsilon}^{A,r}_{j}(\hat{t}_{n})\leq r^{\kappa}\frac{\tilde{c}_{2}-\tilde{c}_{1}}{\alpha^{r}_{j}16}$, and for all $s\in [\hat{t}_{n},\zeta^{r,n}_{0})$ we have
 % $\beta^{r}_{j}\left(\bar{B}^{r}_{j}(s)-\bar{B}^{r}_{j}(\hat{t}_{n})\right)\leq \left(s-\hat{t}_{n}\right)\left( \alpha^{r}_{j}-\Delta\right)$.
 Because $\Delta>0$ this implies $Q^{r}_{j}(r^{2}s)>r^{\kappa}\tilde{c}_{1}$ for all
  $s\in [\hat{t}_{n},\zeta^{r,n}_{0})$ and since $\mathcal{E}_{j}^{r}(r^{2}\hat{t}_{n})=0$ this implies $\mathcal{E}_{j}^{r}(r^{2}s)=0$ for all $s\in  [\hat{t}_{n},\zeta^{r,n}_{0})$.  Consequently on $H^{r}_{n}\cap (\clu_n^r)^{c}\cap (\clc^r)^{c}$  we have that \eqref{eq:241} holds and 
  \begin{equation}\label{eq:237}
 \int_{r^{2}\hat{t}_{n}}^{r^{2}\hat{t}_{n+1}}\mathcal{I}_{\{ \mathcal{E}_{j}^{r}(s)=1\}}ds= \int_{r^{2}\hat{t}_{n}}^{r^{2}\zeta^{r,n}_{0}}\mathcal{I}_{\{ \mathcal{E}_{j}^{r}(s)=1\}}ds+ \int_{r^{2}\zeta^{r,n}_{0}}^{r^{2}\hat{t}_{n+1}}\mathcal{I}_{\{ \mathcal{E}_{j}^{r}(s)=1\}}ds=0,
 \end{equation}
 % \begin{equation}\label{eq:237}
 % \mathcal{E}_{j}^{r}(r^{2}\hat{t}_{n+1})=0,
 % \mbox{ and }
 % Q^{r}_{j}(r^{2}\hat{t}_{n+1})\geq  r^{\kappa}\left(\frac{5}{8}\tilde{c}_{2}+\frac{3}{8}\tilde{c}_{1}\right).
 % \end{equation}

 Therefore for any $N\geq 1$ we have, on $(\clc^r)^c$,
 {
 \begin{eqnarray*}
  \int_{r^{2}\hat{t}_{1}}^{r^{2}\hat{t}_{N}}\mathcal{I}_{\{ \mathcal{E}_{j}^{r}(s)=1\}}ds&=&\sum_{n=1}^{N-1}\int_{r^{2}\hat{t}_{n}}^{r^{2}\hat{t}_{n+1}}\mathcal{I}_{\{ \mathcal{E}_{j}^{r}(s)=1\}}ds\\
  &=&\sum_{n=1}^{N-1}\mathcal{I}_{\clu^r_n}\int_{r^{2}\hat{t}_{n}}^{r^{2}\hat{t}_{n+1}}\mathcal{I}_{\{ \mathcal{E}_{j}^{r}(s)=1\}}ds
  +\sum_{n=1}^{N-1}\mathcal{I}_{H^{r}_{n}\cap (\clu_n^r)^{c}}\int_{r^{2}\hat{t}_{n}}^{r^{2}\hat{t}_{n+1}}\mathcal{I}_{\{ \mathcal{E}_{j}^{r}(s)=1\}}ds\\
   &&+\sum_{n=1}^{N-1}\mathcal{I}_{(H^{r}_{n})^{c}\cap \left(\clu^{r}_n\right)^{c}}\int_{r^{2}\hat{t}_{n}}^{r^{2}\hat{t}_{n+1}}\mathcal{I}_{\{ \mathcal{E}_{j}^{r}(s)=1\}}ds\\
  &\leq& r^{\frac{3}{2}\kappa}\sum_{n=1}^{N-1}\mathcal{I}_{ \clu^{r}_n}
 +r^{\kappa}\frac{5(\tilde{c}_{2}-\tilde{c}_{1})}{4}\sum_{n=1}^{N-1}\mathcal{I}_{(H^{r}_{n})^{c}\cap \left(\clu^{r}_n\right)^{c}} ,
 \end{eqnarray*}}
 where we have used \eqref{eq:241} and \eqref{eq:237} in obtaining the last inequality.

 From \eqref{eq:256} we have, on $(\clc^r)^c$,
 \begin{equation*}
 \sum_{n=1}^{N-1}\mathcal{I}_{(H^{r}_{n})^{c}\cap\left(\mathcal{U}^{r}_{n}\right)^{c}}\leq 1+\sum_{n=1}^{N-1}\mathcal{I}_{\mathcal{U}^{r}_{n}}.
 \end{equation*}
 Therefore, on $(\clc^r)^c$, 
 {
 \begin{equation}
 \label{eq:idleTimeBnd}
 \int_{0}^{r^{2}\hat{t}_{N}}\mathcal{I}_{\{ \mathcal{E}_{j}^{r}(s)=1\}}ds\leq
 r\hat \Upsilon^{A,r}_j +Tr^{\frac{3}{2}\kappa} + r^{\kappa}\frac{5(\tilde{c}_{2}-\tilde{c}_{1})}{4}+ \left( r^{\frac{3}{2}\kappa}+ r^{\kappa}\frac{5(\tilde{c}_{2}-\tilde{c}_{1})}{4}\right)\sum_{n=1}^{N-1}\mathcal{I}_{\mathcal{U}^{r}_{n}}.
 \end{equation}}
From the above estimate, in order to prove the result,  it now suffices to show that there exists $R,B<\infty$ such that for all $r\geq R$ and $T\ge 1$, we have
\begin{equation}\label{eq:1204n}
	P(\clc^r) \le e^{-BTr^{\frac{1}{8}\kappa}},
\end{equation}
and {
  \begin{equation}\label{eq:536}
 P\left(\sum_{n=1}^{\lceil Tr^{2-\frac{3}{2}\kappa}\rceil-1}\mathcal{I}_{ \mathcal{U}^{r}_{n}}\geq 3r^{\frac{1}{8}\kappa}T\right)\leq e^{-BTr^{\frac{1}{8}\kappa}}.
 \end{equation}}
 For \eqref{eq:1204n}, note that since from Proposition \ref{thm:schemeSum} part (a), if $Q^{r}_{j}(r^{2}s)<r^{\kappa}\tilde{c}_{1}$ then  $\frac{d}{ds}\bar{B}^{r}_{j}(s)=0$, we have that
{

 \begin{align}\label{eq:rev1}
 	P(\clc^r) \le P(A^r_j(Tr^{3\kappa/2}) \le r^{\kappa} \tilde c_1).
 \end{align}
}

 The estimate in \eqref{eq:1204n} now follows readily from \eqref{eq:31.b}.

 The estimate in \eqref{eq:536} follows if we can show that there exists $R,B<\infty$ such that for all $r\geq R$ and { $T\geq 1$ } we have
 \begin{equation}\label{eq:512}
 P\left(\sum_{n=1}^{\lceil Tr^{2-\frac{3}{2}\kappa}\rceil-1} 
 \mathcal{I}_{ \left(\mathcal{B}^{r,1}_{n}\mathcal{B}^{r,2}_{n} \right)}\geq r^{\frac{1}{8}\kappa}T\right)\leq  e^{-BTr^{\frac{1}{8}\kappa}}, \\
 \end{equation} 
 \begin{equation}\label{eq:517}
 P\left(\sum_{n=1}^{\lceil Tr^{2-\frac{3}{2}\kappa}\rceil-1} \mathcal{I}_{ \mathcal{A}^{r,1}_{n}}\geq r^{\frac{1}{8}\kappa}T\right)\leq  e^{- BTr^{\frac{1}{8}\kappa}}, \\
 \end{equation}
 and
   \begin{equation}\label{eq:530}
 P\left(\sum_{n=1}^{\lceil Tr^{2-\frac{3}{2}\kappa}\rceil-1} \mathcal{I}_{\left(\mathcal{A}^{r,1}_{n}\right)^{c}\cap\mathcal{A}^{r,2}_{n}}\geq r^{\frac{1}{8}\kappa}T\right)\leq  e^{- BT r^{\frac{1}{8}\kappa}}. \\
 \end{equation}
First we show  \eqref{eq:512}, with $B= \frac{1}{4}$, for $r$ sufficiently large.
Define 
\begin{equation*}
\mathcal{F}^{r,S}_{j}(k)=\sigma\left\{u^{r}_{l}(m^{u}_{l}),v^{r}_{l'}(m^{v}_{l'}), v^{r}_{j}(m^{v}_{j}): 
m^v_{l'}\ge 0, m^u_l\ge 0, l \in \AAA_J, l' \in \AAA_J\setminus\{j\}, m^v_j \le k
\right\}
\end{equation*}
 which is the filtration that contains the information on all inter-arrival times, all service times from queues other than the $j$-th queue, and   the first $k$ service times from queue $j$. 
 % Since $\hat{t}_{n}$ is $\mathcal{F}^{r}(0,0)$-measurable
 Note that $\tau^{r,S}_{j}(\hat{t}_{n})$ is a $\mathcal{F}^{r,S}_{j}(k)$ stopping time and thus with $\clh^j_n \doteq \mathcal{F}^{r,S}_{j}(\tau^{r,S}_{j}(\hat{t}_{n}))$ for $n\geq 0$, $\{\clh^j_n\}_{n\ge 0}$ is a filtration.  Note that for $k<n$ $\mathcal{B}^{r,1}_{k}\cap\mathcal{B}^{r,2}_{k}$ is $\clh^j_n $-measurable.  In addition, $\bar{\xi}^{S,r}_{j}(\hat{t}_{n})$ is $\clh^j_n $-measurable
and $S^{r,\hat{t}_{n}}_{j}$ is independent of $\clh^j_n $. Write
$$\clu_n = \exp\{\frac{1}{2} \mathcal{I}_{\mathcal{B}^{r,1}_{n}\cap\mathcal{B}^{r,2}_{n}}\}  = e^{1/2} 
\mathcal{I}_{\mathcal{B}^{r,1}_{n}\cap\mathcal{B}^{r,2}_{n}}
+ \mathcal{I}_{(\mathcal{B}^{r,1}_{n}\cap\mathcal{B}^{r,2}_{n})^c} .$$
Then we have {
\begin{multline*}
E\left[e^{\frac{1}{2}\sum_{n=1}^{N}\mathcal{I}_{\mathcal{B}^{r,1}_{n}\cap\mathcal{B}^{r,2}_{n}}} \right]=
E\left[\prod_{n=1}^N \clu_n\right] = E\left[ E(\clu_N \mid \clh^j_{N-1})\prod_{n=1}^{N-1} \clu_n\right] \\
 % E\left[e^{\frac{1}{2}\sum_{n=1}^{N-1}\mathcal{I}_{\mathcal{B}^{r,1}_{n}\cap\mathcal{B}^{r,2}_{n}}}E\left[e^{\frac{1}{2}\mathcal{I}_{\mathcal{B}^{r,1}_{N}\cap\mathcal{B}^{r,2}_{N}}}|\mathcal{F}^{r,S}_{j}(\tau^{r,S}_{j}(\hat{t}_{N-1})) \right] \right]\\
% \leq E\left[\Prod_{n=0}^{N-1} \clu_n E\left[e^{\frac{1}{2}\mathcal{I}_{\mathcal{B}^{r,2}_{N}}}|\mathcal{F}^{r,S}_{j}(\tau^{r,S}_{j}(\hat{t}_{N-1})) \right] \right]\\
\leq E\left[\left[\left(1-P\left(\mathcal{B}^{r,2}_{N}| \clh^j_{N-1}\right)\right)+e^{\frac{1}{2}}P\left(\mathcal{B}^{r,2}_{N}| \clh^j_{N-1}\right) \right]\prod_{n=1}^{N-1} \clu_n\right].\\
\end{multline*}}
For $n\geq 0$, since $\bar{\Upsilon}^{S,r}_{j}(\hat{t}_{n})\geq 0$ and $\bar{B}^{r}_{j}(\hat{t}_{n+1})-\bar{B}^{r}_{j}(\hat{t}_{n})\leq \max_{i}\{C_{i}\}r^{\frac{3}{2}\kappa-2}$, we have
\begin{equation*}
P\left(\mathcal{B}^{r,2}_{n}|\clh^j_{n}\right)\leq P\left( \sup_{0\leq s\leq \max_{i}\{C_{i}\}r^{\frac{3}{2}\kappa-2}}\left|S^{r,\hat{t}_{n}}_{j}(r^{2}s)-r^{2}s\beta^{r}_{j}\right| >r^{\kappa}\frac{\tilde{c}_{2}-\tilde{c}_{1}}{16} \right)
\end{equation*}
and due to Proposition \ref{thm:expTailBnd} \eqref{eq:31.a} there exists $R_{2}\in [R_{1},\infty)$ and $B_{1}<\infty$ such that for all $r\geq R_{2}$ and $n\geq 0$ we have
\begin{equation}\label{eq:revfirstuse}
P\left( \sup_{0\leq s\leq \max_{i}\{C_{i}\}r^{\frac{3}{2}\kappa-2}}\right|S^{r,\hat{t}_{n}}_{j}(r^{2}s)-r^{2}s\alpha^{r}_{j}\left| >r^{\kappa}\frac{\tilde{c}_{2}-\tilde{c}_{1}}{16} \right)\leq e^{-r^{\frac{\kappa}{2}}B_{1}}.
\end{equation}
Since $\left(1-p+pe^{\frac{1}{2}}\right)$ is an increasing function of $p$ we have
\begin{equation*}
E\left[e^{\frac{1}{2}\sum_{n=1}^{N}\mathcal{I}_{\mathcal{B}^{r,1}_{n}\cap\mathcal{B}^{r,2}_{n}}} \right]\leq E\left[e^{\frac{1}{2}\sum_{n=1}^{N-1}\mathcal{I}_{\mathcal{B}^{r,1}_{n}\cap\mathcal{B}^{r,2}_{n}}} \right]\left(1+e^{-r^{\frac{\kappa}{2}}B_{1}}\left(e^{\frac{1}{2}}-1\right)\right).
\end{equation*}
Now by a standard recursive argument we see
\begin{equation*}
E\left[e^{\frac{1}{2}\sum_{n=1}^{N}\mathcal{I}_{\mathcal{B}^{r,1}_{n}\cap\mathcal{B}^{r,2}_{n}}} \right]\leq \left(1+e^{-r^{\frac{\kappa}{2}}B_{1}}\left(e^{1/2}-1\right)\right)^{N+1}.
\end{equation*}
Since $e^{\frac{1}{2}}-1\leq 1$ we have
$
1+e^{-r^{\frac{\kappa}{2}}B_{1}}\left(e^{\frac{1}{2}}-1\right)\leq e^{e^{-r^{\frac{\kappa}{2}}B_{1}}}
$
which gives
\begin{equation}\label{eq:941}
E\left[e^{\frac{1}{2}\sum_{n=1}^{N}\mathcal{I}_{\mathcal{B}^{r,1}_{n}\cap\mathcal{B}^{r,2}_{n}}} \right]\leq e^{(N+1)e^{-r^{\frac{\kappa}{2}}B_{1}}}.
\end{equation}
Choose $R_{3}\in [R_{2},\infty)$ such that for all $r\geq R_{3}$ and $T\geq 1$ we have 
% $- \frac{1}{2}r^{\frac{1}{8}\kappa}T+\left(Tr^{\frac{1}{2}-\frac{3}{2}\kappa}+1\right) e^{-r^{\frac{\kappa}{2}}B_{1}}\leq -\frac{1}{4}r^{\frac{1}{8}\kappa}T$
 $$- \frac{1}{2}r^{\frac{1}{8}\kappa}T+\left(Tr^{2-\frac{3}{2}\kappa}+1\right) e^{-r^{\frac{\kappa}{2}}B_{1}}\leq - \frac{1}{4}r^{\frac{1}{8}\kappa}T.$$  Then for all $r\geq R_{3}$ and { $T\geq 1$} we have from \eqref{eq:941} that
  \begin{eqnarray*}
 P\left(\sum_{n=1}^{\lceil Tr^{2-\frac{3}{2}\kappa}\rceil-1} \mathcal{I}_{ \left(\mathcal{B}^{r,1}_{n}\cap \mathcal{B}^{r,2}_{n} \right)}\geq r^{\frac{1}{8}\kappa}T\right)&\leq& e^{- \frac{1}{2}r^{\frac{1}{8}\kappa}T+\lceil Tr^{2-\frac{3}{2}\kappa}\rceil e^{-r^{\frac{\kappa}{2}}B_{1}}}
 % &\leq& e^{- \frac{1}{2}r^{\frac{1}{8}\kappa}T+\left(Tr^{2-\frac{3}{2}\kappa}+1\right) e^{-r^{\frac{\kappa}{2}}B_{1}}}\\
 \leq  e^{-\frac{1}{4}Tr^{\frac{1}{8}\kappa}}. \\
 \end{eqnarray*}
 This completes the proof of \eqref{eq:512}.
Let $\delta>0$ be as in Condition \ref{eqn:mgfBnd}
%Theorem \ref{thm:nextArrTimeBnd} 
and let $\upsilon\in (0,\delta)$ be arbitrary.  Now 
we will show that \eqref{eq:517} holds with $B= \frac{\upsilon}{2}$ for $r$ sufficiently large.
 For  $x\in \mathbb{R}_+$ define
\begin{equation*}
\tilde{\tau}^{r,A}_{j}(x)=\min\left\{k\geq 0:\sum_{l=1}^{k}u^{r}_{j}(l)\geq x\right\},
\end{equation*}
Let $\check u^r_j(x) \doteq \mathcal{I}_{\{x>0\}}u^{r}_{j}\left( \tilde{\tau}^{r,A}_{j}(x)\right)$ and define
\begin{equation*}
\Phi^{r}_{\upsilon}(x)\doteq E\exp\left\{\upsilon\mathcal{I}_{\left\{\check u^r_j(x)>r^{\kappa}\frac{\tilde{c}_{2}-\tilde{c}_{1}}{\alpha^{r}_{j}16}\right\}}+\upsilon\sum_{l=1}^{\infty}\mathcal{I}_{\left\{\check u^r_j(x)>lr^{\frac{3}{2}\kappa}\right\}}\right\}.
\end{equation*}
Let
\begin{equation*}
\mathcal{F}^{r,A}_{j}(k)=\sigma\left\{u^{r}_{l'}(m^{u}_{l'}),v^{r}_{l}(m^{v}_{l}), u^{r}_{j}(m^{u}_{j}): 
m^u_{l'}\ge 0, m^v_l\ge 0, l \in \AAA_J, l' \in \AAA_J\setminus\{j\},  m^u_j \le k
\right\}
\end{equation*}
 which is the filtration that contains the information on all  all service times, all inter-arrival times for queues other than the $j$-th queue, and   the first $k$ arrival times from queue $j$. Note that $\tau^{r,A}_{j}(\hat{t}_{n})$ is a $\mathcal{F}^{r,A}_{j}(k)$-stopping time and thus with $\tilde \clh^j_n \doteq \mathcal{F}^{r,A}_j( \tau^{r,A}_{j}(\hat{t}_{n}))$ for $n\geq 0$, $\{\tilde \clh^j_n \}_{n\geq 0}$ is a filtration.
For  $n\geq 0$ and a $\tilde \clh^j_n$-measurable positive  random variable $X$ define
\begin{equation*}
\tilde{\tau}^{r,A,n}_{j}(X)=\min\left\{k\geq 0:\sum_{l= \tau^{r,A}_{j}(\hat{t}_{n})+1}^{l= \tau^{r,A}_{j}(\hat{t}_{n})+k}u^{r}_{j}(l)\geq X \right\}.
\end{equation*}
and
\begin{eqnarray*}
\psi^{r,n}_{\upsilon}(X)&\doteq& \exp \left[\upsilon\mathcal{I}_{\left\{ X>0, u^{r}_{j}\left( \tilde{\tau}^{r,A}_{j}(X)+\tau^{r,A}_{j}(\hat{t}_{n})\right)>r^{\kappa}\frac{\tilde{c}_{2}-\tilde{c}_{1}}{\alpha^{r}_{j}16}\right\}}\right.\\
&&\quad \quad \left.+\upsilon\sum_{l=1}^{\infty}\mathcal{I}_{\left\{X>0, u^{r}_{j}\left( \tilde{\tau}^{r,A}_{j}(X)+\tau^{r,A}_{j}(\hat{t}_{n})\right)>lr^{\frac{3}{2}\kappa}\right\}}\right].
\end{eqnarray*}
Note that
\begin{equation}
\label{eq:psiCondExp}
E\left[ \psi^{r,n}_{\upsilon}(X)\vert \tilde \clh^j_n  \right]=\Phi^{r}_{\upsilon}(X).
\end{equation}
Due to Proposition \ref{thm:nextArrTimeBnd} there exists $R_{4}\in [R_{3},
\infty)$ and $B_{2}<\infty$ such that 
\begin{equation*}
\sup_{r\geq R_{4},x\in\mathbb{R}_+} E\left[e^{\upsilon \check u^r_j(x) }\right]\leq B_{2}
\end{equation*}
and for all $r\geq R_{4}$ we have $ e^{- \upsilon\left( r^{\frac{3}{2}\kappa}-1\right)}<\frac{1}{2}$, $ r^{\kappa}\frac{\tilde{c}_{2}-\tilde{c}_{1}}{\alpha^{r}_{j}16}\leq r^{\frac{3}{2}\kappa}$, $2\alpha^{r}_{j}\geq \alpha_{j}$, and $r^{\kappa}\frac{\tilde{c}_{2}-\tilde{c}_{1}}{\alpha^{r}_{j}32}>1$.
Consequently for all $x\in\mathbb{R}$ and $l\geq 1$ we have 
\begin{equation*}
P\left( \check u^r_j(x)\geq r^{\kappa}\frac{\tilde{c}_{2}-\tilde{c}_{1}}{\alpha^{r}_{j}16}\right)\leq B_{2}e^{-\upsilon r^{\kappa}\frac{\tilde{c}_{2}-\tilde{c}_{1}}{\alpha^{r}_{j}16}}, \;\; 
P\left( \check u^r_j(x)\geq lr^{\frac{3}{2}\kappa} \right)\leq B_{2}e^{-\upsilon lr^{\frac{3}{2}\kappa}}.
\end{equation*}
% and for all $x\in\mathbb{R}$ and $l\geq 1$ we have
% \begin{equation*}
% P\left( \mathcal{I}_{\{x>0\}}u^{r}_{j}\left(\tilde{\tau}^{r,A}_{j}(x)\right)\geq lr^{\frac{3}{2}\kappa} \right)\leq B_{2}e^{-\upsilon lr^{\frac{3}{2}\kappa}}.
% \end{equation*}
Therefore for $r\geq R_{4}$ we have
\begin{eqnarray*}
\Phi^{r}_{\upsilon}(x)&\leq&1+ e^{\upsilon }B_{2}e^{-\upsilon r^{\kappa}\frac{\tilde{c}_{2}-\tilde{c}_{1}}{\alpha^{r}_{j}16}}+\sum_{l=1}^{\infty}e^{\upsilon (l+1)}B_{2}e^{-\upsilon l r^{\frac{3}{2}\kappa}}\\
&\leq&1+ e^{\upsilon }B_{2}e^{-\upsilon r^{\kappa}\frac{\tilde{c}_{2}-\tilde{c}_{1}}{\alpha^{r}_{j}16}}+e^{2\upsilon}B_{2}e^{-\upsilon r^{\frac{3}{2}\kappa}}\sum_{l=0}^{\infty}e^{-\upsilon l \left( r^{\frac{3}{2}\kappa}-1\right)}\\
% &\leq&1+ e^{\upsilon }B_{2}e^{-\upsilon r^{\kappa}\frac{\tilde{c}_{2}-\tilde{c}_{1}}{\alpha^{r}_{j}16}}+e^{2\upsilon}B_{2}e^{-\upsilon r^{\frac{3}{2}\kappa}}\sum_{l=0}^{\infty}2^{-l}\\
&\leq&1+ e^{\upsilon }B_{2}e^{-\upsilon r^{\kappa}\frac{\tilde{c}_{2}-\tilde{c}_{1}}{\alpha^{r}_{j}16}}+2e^{2\upsilon}B_{2}e^{-\upsilon r^{\frac{3}{2}\kappa}}\\
&\leq&1+ 3B_{2}e^{2\upsilon}e^{-\upsilon r^{\kappa}\frac{\tilde{c}_{2}-\tilde{c}_{1}}{\alpha^{r}_{j}16}}
\leq e^{3B_{2}e^{2\upsilon}e^{-\upsilon r^{\kappa}\frac{\tilde{c}_{2}-\tilde{c}_{1}}{\alpha^{r}_{j}16}}}.\\
\end{eqnarray*}
Letting $B_{3}=3B_{2}e^{2\upsilon}$ and $B_{4}=\upsilon \frac{2(\tilde{c}_{2}-\tilde{c}_{1})}{\alpha_{j}16}$ we have that,   for $r\geq R_{4}$, 
$
\sup_{r\geq R_{4},x\in\mathbb{R}}\left\{\Phi^{r}_{\upsilon}(x)\right\}\leq e^{B_{3}e^{-B_{4}r^{\kappa}}}
$
which combined with equation \eqref{eq:psiCondExp} implies that for any $n\geq 0$ and $\tilde \clh^j_n$-measurable real valued random variable $X$ we have a.e.
\begin{equation}
\label{eq:psiCondExpBnd}
E\left[ \psi^{r,n}_{\upsilon}(X)\vert \tilde \clh^j_n  \right]=\Phi^{r}_{\upsilon}(X)\leq e^{B_{3}e^{-B_{4}r^{\kappa}}}.
\end{equation}
Next,
$$\mathcal{I}_{\mathcal{A}^{r,1}_{N}} \le \mathcal{I}_{\left\{r^{2}\bar{\Upsilon}^{A,r}_{j}(\hat{t}_{N-1})\geq r^{\frac{3}{2}\kappa} \right\}}
 + \mathcal{I}_{\left\{r^{2}\bar{\Upsilon}^{A,r}_{j}(\hat{t}_{N-1})< r^{\frac{3}{2}\kappa} \right\}} \mathcal{I}_{\mathcal{A}^{r,1}_{N}}.$$
 Also,
 \begin{align*}
&\mathcal{I}_{\left\{r^{2}\bar{\Upsilon}^{A,r}_{j}(\hat{t}_{N-1})< r^{\frac{3}{2}\kappa} \right\}} \mathcal{I}_{\mathcal{A}^{r,1}_{N}}\\
& \le 
\mathcal{I}_{\left\{r^{\frac{3}{2}\kappa} -r^{2}\bar{\Upsilon}^{A,r}_{j}(\hat{t}_{N-1})>0, u^{r}_{j}\left(\tau^{r,A}_{j}(\hat{t}_{N-1})+\tilde{\tau}^{r,A,N-1}_{j}\left(r^{\frac{3}{2}\kappa} -r^{2}\bar{\Upsilon}^{A,r}_{j}(\hat{t}_{N-1})\right)\right)>r^{\kappa}\frac{\tilde{c}_{2}-\tilde{c}_{1}}{\alpha^{r}_{j}16}\right\}}.
\end{align*}
Thus
\begin{align*}
e^{\upsilon\mathcal{I}_{\mathcal{A}^{r,1}_{N}} }&\leq e^{\upsilon\mathcal{I}_{\left\{r^{2}\bar{\Upsilon}^{A,r}_{j}(\hat{t}_{N-1})\geq r^{\frac{3}{2}\kappa} \right\}}+\upsilon\mathcal{I}_{\left\{r^{2}\bar{\Upsilon}^{A,r}_{j}(\hat{t}_{N-1})< r^{\frac{3}{2}\kappa} \right\}} \mathcal{I}_{\mathcal{A}^{r,1}_{N}}}\\
& \le  e^{\upsilon\mathcal{I}_{\left\{r^{2}\bar{\Upsilon}^{A,r}_{j}(\hat{t}_{N-1})\geq r^{\frac{3}{2}\kappa} \right\}}}\psi^{r,N-1}_{\upsilon}( r^{\frac{3}{2}\kappa} -r^{2}\bar{\Upsilon}^{A,r}_{j}(\hat{t}_{N-1})).
\end{align*}
By conditioning, using \eqref{eq:psiCondExpBnd},  and the fact that $ r^{\frac{3}{2}\kappa} -r^{2}\bar{\Upsilon}^{A,r}_{j}(\hat{t}_{N-1})$ is  $ \tilde \clh^j_{N-1}$-measurable it then follows
\begin{eqnarray*}
	E\left[ e^{\upsilon\sum_{n=1}^{N}\mathcal{I}_{\mathcal{A}^{r,1}_{n}}}\right]
\leq e^{B_{3}e^{-B_{4}r^{\kappa}}}E  e^{\upsilon\sum_{n=1}^{N-1}\mathcal{I}_{\mathcal{A}^{r,1}_{n}}+\upsilon\mathcal{I}_{\left\{r^{2}\bar{\Upsilon}^{A,r}_{j}(\hat{t}_{N-1})\geq r^{\frac{3}{2}\kappa} \right\}}}.
\end{eqnarray*}
By a successive conditioning argument we now have that
\begin{equation*}
E  e^{\upsilon\sum_{n=1}^{N}\mathcal{I}_{\mathcal{A}^{r,1}_{n}}} \leq e^{NB_{3}e^{-B_{4}r^{\kappa}}}E e^{\upsilon\mathcal{I}_{\mathcal{A}^{r,1}_{0}}+\upsilon\sum_{l=1}^{N}\mathcal{I}_{ \left\{r^{2}\bar{\Upsilon}^{A,r}_{j}(\hat{t}_{0})\geq l r^{\frac{3}{2}\kappa} \right\}}}.
\end{equation*}
By definition $\bar{\Upsilon}^{A,r}_{j}(\hat{t}_{0})=0$ so the expectation on the right side equals
$1$  
which gives
\begin{equation*}
E\left[ e^{\upsilon\sum_{n=1}^{N}\mathcal{I}_{\mathcal{A}^{r,1}_{n}}}\right]\leq e^{NB_{3}e^{-B_{4}r^{\kappa}}}.
\end{equation*}
Choose $R_{5}\in[R_{4},\infty)$ such that for all $r\geq R_{5}$ and $T\geq 1$ we have 
\begin{equation*}
- T\upsilon r^{\frac{1}{8}\kappa}+ Tr^{2-\frac{3}{2}\kappa}B_{3}e^{-B_{4}r^{\kappa}}\leq - T\frac{\upsilon}{2} r^{\frac{1}{8}\kappa}.
\end{equation*}
  Then for all $r\geq R_{5}$ and $T\geq 1$ we have
  \begin{align*}
& P\left(\sum_{n=1}^{\lceil Tr^{2-\frac{3}{2}\kappa}\rceil-1} \mathcal{I}_{ \mathcal{A}^{r,1}_{n}}\geq r^{\frac{1}{8}\kappa}T\right) \leq e^{-\upsilon r^{\frac{1}{8}\kappa}T} E\left[e^{\upsilon \sum_{n=1}^{\lceil Tr^{2-\frac{3}{2}\kappa}\rceil-1} \mathcal{I}_{ \mathcal{A}^{r,1}_{n}} }\right]\\
& \leq  e^{- \upsilon r^{\frac{1}{8}\kappa}T+\left(\lceil Tr^{2-\frac{3}{2}\kappa}\rceil-1 \right)B_{3}e^{-B_{4}r^{\kappa}}}
 \leq e^{- \upsilon r^{\frac{1}{8}\kappa}T+ Tr^{2-\frac{3}{2}\kappa} B_{3}e^{-B_{4}r^{\kappa}}}
 \leq e^{- \frac{\upsilon}{2} r^{\frac{1}{8}\kappa}T}. 
 \end{align*}
 This proves \eqref{eq:517}.
 
Finally, we will show \eqref{eq:530}, with $B= 1/4$ and for $r$ sufficiently large.
Note that for $n\geq 0$ since $A^{r,\hat{t}_{n}}_{j}(\cdot)$ is independent of $\mathcal{F}^{r}(\tau^{r}(\hat{t}_{n}))$, $\bar{\Upsilon}^{A,r}_{j}(\hat{t}_{n})$ is $\mathcal{F}^{r}(\tau^{r}(\hat{t}_{n}))$-measurable, and $\bar{\Upsilon}^{A,r}_{j}(\hat{t}_{n})\geq 0$ we have
\begin{eqnarray*}
P\left(\mathcal{A}^{r,2}_{n}| \mathcal{F}^{r}(\tau^{r}(\hat{t}_{n}))\right)
% &=&P\left( \sup_{0\leq s\leq r^{\frac{3}{2}\kappa-2}-\bar{\Upsilon}^{A,r}_{j}(\hat{t}_{n})}\right|A^{r,\hat{t}_{n}}_{j}(r^{2}s)-r^{2}s \alpha^{r}_{j}\left| >r^{\kappa}\frac{\tilde{c}_{2}-\tilde{c}_{1}}{16}| \mathcal{F}^{r}(\tau^{r}(\hat{t}_{n})) \right)\\
&\leq& P\left( \sup_{0\leq s\leq r^{\frac{3}{2}\kappa-2}}\left|A^{r,\hat{t}_{n}}_{j}(r^{2}s)-r^{2}s \alpha^{r}_{j}\right| >r^{\kappa}\frac{\tilde{c}_{2}-\tilde{c}_{1}}{16}| \mathcal{F}^{r}(\tau^{r}(\hat{t}_{n})) \right)\\
&\leq& P\left( \sup_{0\leq s\leq r^{\frac{3}{2}\kappa-2}}\left|A^{r,\hat{t}_{n}}_{j}(r^{2}s)-r^{2}s \alpha^{r}_{j}\right| >r^{\kappa}\frac{\tilde{c}_{2}-\tilde{c}_{1}}{16} \right).
\end{eqnarray*}
From Proposition \ref{thm:expTailBnd} \eqref{eq:31.b} there exists $R_{6}\in[R_{5},\infty)$ and $B_{5}<\infty$ such that for all $r\geq R_{6}$ we have
\begin{equation}\label{eq:revseconduse}
P\left( \sup_{0\leq s\leq r^{\frac{3}{2}\kappa-2}}\left|A^{r,\hat{t}_{n}}_{j}(r^{2}s)-r^{2}s \alpha^{r}_{j}\right| >r^{\kappa}\frac{\tilde{c}_{2}-\tilde{c}_{1}}{16} \right)\leq e^{-2B_{5}r^{\frac{\kappa}{2}}}
\end{equation}
and $e^{\frac{1}{2}}\leq B_{5}r^{\frac{\kappa}{2}}$.
Therefore for all $r\geq R_{6}$ and $n\geq 0$ we have
\begin{equation*}
P\left(\mathcal{A}^{r,2}_{n}| \mathcal{F}^{r}(\tau^{r}(\hat{t}_{n}))\right)\leq e^{-2B_{5}r^{\frac{\kappa}{2}}}
\end{equation*}
and so
\begin{align*}
	E\left[e^{\frac{1}{2}\mathcal{I}_{\left(\mathcal{A}^{r,1}_{n}\right)^{c}\cap\mathcal{A}^{r,2}_{n}}}|\mathcal{F}^{r}(\tau^{r}(\hat{t}_{n})) \right]
&\le \left(1+e^{\frac{1}{2}}P\left(\mathcal{A}^{r,2}_{n} |\mathcal{F}^{r}(\tau^{r}(\hat{t}_{n}))\right) \right)\\
&\le \left(1+e^{\frac{1}{2}}e^{-2B_{5}r^{\frac{\kappa}{2}}} \right)
\le \left(1+e^{-B_{5}r^{\frac{\kappa}{2}}} \right) .
\end{align*}
Since for $0\leq m<n$ the set $\left(\mathcal{A}^{r,1}_{m}\right)^{c}\cap\mathcal{A}^{r,2}_{m}$ is $\mathcal{F}^{r}(\tau^{r}(\hat{t}_{n}))$-measurable we have by a successive conditioning argument
%
% for $k\geq 1$
% \begin{eqnarray*}
% E\left[e^{\frac{1}{2}\sum_{n=0}^{k}\mathcal{I}_{\left(\mathcal{A}^{r,1}_{n}\right)^{c}\cap\mathcal{A}^{r,2}_{n}}} \right]&=&E\left[E\left[e^{\frac{1}{2}\sum_{n=0}^{k}\mathcal{I}_{\left(\mathcal{A}^{r,1}_{n}\right)^{c}\cap\mathcal{A}^{r,2}_{n}}}|\mathcal{F}^{r}(\tau^{r}(\hat{t}_{k})) \right]\right]\\
% &=&E\left[e^{\frac{1}{2}\sum_{n=0}^{k-1}\mathcal{I}_{\left(\mathcal{A}^{r,1}_{n}\right)^{c}\cap\mathcal{A}^{r,2}_{n}}}E\left[e^{\frac{1}{2}\mathcal{I}_{\left(\mathcal{A}^{r,1}_{k}\right)^{c}\cap\mathcal{A}^{r,2}_{k}}}|\mathcal{F}^{r}(\tau^{r}(\hat{t}_{k})) \right]\right]\\
% &\leq&E\left[e^{\frac{1}{2}\sum_{n=0}^{k-1}\mathcal{I}_{\left(\mathcal{A}^{r,1}_{n}\right)^{c}\cap\mathcal{A}^{r,2}_{n}}}E\left[e^{\frac{1}{2}\mathcal{I}_{\mathcal{A}^{r,2}_{k}}}|\mathcal{F}^{r}(\tau^{r}(\hat{t}_{k})) \right]\right]\\
% &\leq&E\left[e^{\frac{1}{2}\sum_{n=0}^{k-1}\mathcal{I}_{\left(\mathcal{A}^{r,1}_{n}\right)^{c}\cap\mathcal{A}^{r,2}_{n}}}\left(1+e^{\frac{1}{2}}P\left(\mathcal{A}^{r,2}_{k} |\mathcal{F}^{r}(\tau^{r}(\hat{t}_{k}))\right) \right)\right]\\
% &\leq&E\left[e^{\frac{1}{2}\sum_{n=0}^{k-1}\mathcal{I}_{\left(\mathcal{A}^{r,1}_{n}\right)^{c}\cap\mathcal{A}^{r,2}_{n}}}\left(1+e^{\frac{1}{2}-2B_{5}r^{\frac{\kappa}{2}}} \right)\right]\\
% &\leq&\left(1+e^{-B_{5}r^{\frac{\kappa}{2}}} \right)E\left[e^{\frac{1}{2}\sum_{n=0}^{k-1}\mathcal{I}_{\left(\mathcal{A}^{r,1}_{n}\right)^{c}\cap\mathcal{A}^{r,2}_{n}}}\right]\\
% \end{eqnarray*}
% so by induction we have
\begin{align*}
E\left[e^{\frac{1}{2}\sum_{n=1}^{N}\mathcal{I}_{\left(\mathcal{A}^{r,1}_{n}\right)^{c}\cap\mathcal{A}^{r,2}_{n}}} \right] &\leq \left(1+e^{-B_{5}r^{\frac{\kappa}{2}}} \right)^{N}E\left[e^{\frac{1}{2}\mathcal{I}_{\left(\mathcal{A}^{r,1}_{0}\right)^{c}\cap\mathcal{A}^{r,2}_{0}}}\right]\\
&\le \left(1+e^{-B_{5}r^{\frac{\kappa}{2}}} \right)^{N}e^{1/2} \le e^{Ne^{-B_{5}r^{\frac{\kappa}{2}}}} e^
{1/2}.
\end{align*}
% In addition,
% $$
% E\left[e^{\frac{1}{2}\mathcal{I}_{\left(\mathcal{A}^{r,1}_{0}\right)^{c}\cap\mathcal{A}^{r,2}_{0}}}\right]\leq e^{1/2}
% $$
% % E\left[e^{\frac{1}{2}\mathcal{I}_{\mathcal{A}^{r,2}_{0}}}\right]
% % \leq E\left[E\left[e^{\frac{1}{2}\mathcal{I}_{\mathcal{A}^{r,2}_{0}}}| \mathcal{F}^{r}(\tau^{r}(\hat{t}_{0})) \right]\right]\\
% % \leq 1 + e^{\frac{1}{2}}P\left(\mathcal{A}^{r,2}_{0}| \mathcal{F}^{r}(\tau^{r}(\hat{t}_{0})) \right)
% % \leq 1 + e^{\frac{1}{2}-2B_{5}r^{\frac{\kappa}{2}}}
% % \le 1 + e^{-B_{5}r^{\frac{\kappa}{2}}}
% which gives
% \begin{equation*}
% E\left[e^{\frac{1}{2}\sum_{n=1}^{N}\mathcal{I}_{\left(\mathcal{A}^{r,1}_{n}\right)^{c}\cap\mathcal{A}^{r,2}_{n}}} \right] \leq \left(1+e^{-B_{5}r^{\frac{\kappa}{2}}} \right)^{N}e^{1/2}\leq e^{Ne^{-B_{5}r^{\frac{\kappa}{2}}}} e^
% {1/2}.
% \end{equation*}
  Choose $R_{7}\in [R_{6},\infty)$ such that for all $r\geq R_{7}$ and $T\geq 1$ we have $- \frac{1}{2} r^{\frac{1}{8}\kappa}T+Tr^{2-\frac{3}{2}\kappa}e^{-B_{5}r^{\frac{\kappa}{2}}} + \frac{1}{2}\leq - \frac{1}{4} r^{\frac{1}{8}\kappa}T$.  Then for all $r\geq R_{7}$ we have
  \begin{align*}
 P\Big(\sum_{n=1}^{\lceil Tr^{2-\frac{3}{2}\kappa}\rceil-1} \mathcal{I}_{\left(\mathcal{A}^{r,1}_{n}\right)^{c}\cap\mathcal{A}^{r,2}_{n}}\geq r^{\frac{1}{8}\kappa}T\Big)&\leq e^{- \frac{1}{2} r^{\frac{1}{8}\kappa}T+\left(\lceil Tr^{2-\frac{3}{2}\kappa}\rceil-1\right) e^{-B_{5}r^{\frac{\kappa}{2}}}+ \frac{1}{2}}\\
 &\leq e^{- \frac{1}{2} r^{\frac{1}{8}\kappa}T+ Tr^{2-\frac{3}{2}\kappa} e^{-B_{5}r^{\frac{\kappa}{2}}} +\frac{1}{2}}
 \leq e^{- \frac{1}{4} r^{\frac{1}{8}\kappa}T}. 
 \end{align*}
 Consequently for all $r\geq R_{7}$ we have that \eqref{eq:512}-\eqref{eq:530} hold which proves \eqref{eq:536}.
% Now we choose $R_{9}\in [R_{8},\infty)$ such that for all $r\geq R_{9}$ we have $ r^{\kappa}\frac{5(\tilde{c}_{2}-\tilde{c}_{1})}{4}+3 r^{\frac{1}{8}\kappa+\frac{3}{2}\kappa}+3\frac{5(\tilde{c}_{2}-\tilde{c}_{1})}{4} r^{\frac{1}{8}\kappa+\kappa}\leq \epsilon r^{\frac{7}{4}\kappa}$.
% This, combined with (\ref{eq:idleTimeBnd}), implies that for all $r\geq R_{9}$ we have
% \begin{multline*}
%  P\Big(\int_{0}^{r^{2}T}\mathcal{I}_{\{ \mathcal{E}_{j}^{r}(s)=1\}}ds\geq r^{2}\bar{\Upsilon}^{A,r}_{j}+ \epsilon r^{\frac{7}{4}\kappa}T \Big)\leq  P\Big(\int_{r^{2}\hat{t}_{0}}^{r^{2}\hat{t}_{\lceil Tr^{2-\frac{3}{2}\kappa }\rceil}}\mathcal{I}_{\{ \mathcal{E}_{j}^{r}(s)=1\}}ds\geq \epsilon r^{\frac{7}{4}\kappa}T \Big)\\
%  \leq   P\Big(\sum_{n=0}^{\lceil Tr^{2-\frac{3}{2}\kappa}\rceil} \mathcal{I}_{\mathcal{A}^{r,1}_{n}\cup \mathcal{A}^{r,2}_{n} \cup \Big(\mathcal{B}^{r,1}_{n}\cap \mathcal{B}^{r,2}_{n} \Big)}\geq r^{\frac{1}{8}\kappa}T\Big)
% \leq  e^{-B_{6}Tr^{\frac{1}{8}\kappa}}
% \end{multline*}
As noted previously, this completes the proof.
\hfill \qed

\section{Proof of Propositions \ref{thm:workloadExpBnd} 
and \ref{thm:costDiffResult}}
\label{sec:workexpandcostdiff}
In this section we provide proofs of Propositions \ref{thm:workloadExpBnd} 
and \ref{thm:costDiffResult}.
\subsection{Proof of Proposition \ref{thm:workloadExpBnd}}
\label{sec:secworkloadexp}
We begin with some preliminary stability results.
\begin{defn}
\label{def:TildeGamma}
Let $\tilde{\upsilon}>0$ and $\xi \ge 0$ be arbitrary and define
\begin{equation*}
\tilde{\gamma}^{r,\tilde{\upsilon}}_{i,\xi}\doteq \inf\left\{t\geq\xi :\hat{W}^{r}_{i}(t)+\sum_{j=1}^{J}\hat{\Upsilon}^{S,r}_{j}(t)+\sum_{j=1}^{J}\hat{\Upsilon}^{A,r}_{j}(t)<\tilde{\upsilon}\right\}.
\end{equation*}
\end{defn}
Proof of the following Proposition is in Section \ref{sec:thmwaittimeex}.
\begin{proposition}[Proof in Section \ref{sec:thmwaittimeex}]
\label{thm:waitTimeExpBnd}
There exist constants $\tilde{\delta}, {B}_{1}, {B}_{2}, {B}_{3}, {R}\in (0,\infty)$ such that for all $c\in(0,\tilde{\delta}]$, $r\geq  {R}$, $\tilde{\upsilon}>0$, $\xi \ge 0$,
$y^r=(\hat{q}^r,\hat{\Upsilon}^r,\tilde{\mathcal{E}}^r)\in \cly^r$, and $i\in\AAA_I$ we have
\begin{equation*}
E_{y^r}\left[e^{c\tilde{\gamma}^{r,\tilde{\upsilon}}_{i,\xi}} \right]\leq {B}_{1}e^{{B}_{2}\left(\xi+\tilde{\upsilon}+\hat{w}^r_i+\sum_{j=1}^{J}\hat{\Upsilon}^{S,r}_{j}+\sum_{j=1}^{J}\hat{\Upsilon}^{A,r}_{j}\right)}+{B}_{3},
\end{equation*}
where $\hat{w}^r = KM^r\hat{q}^r$.
\end{proposition}
\begin{defn}
\label{def:V}
For $\tilde{\upsilon}>0$ let $\tilde{\delta}$ be as in Proposition \ref{thm:waitTimeExpBnd} and for all $y=(\hat{q},\hat{\Upsilon},\tilde{\mathcal{E}})\in \cly^r$, and $i\in\AAA_I$ define 
\begin{equation*}
V^{r,\tilde{\upsilon}}_{i}\left( y\right)=E_{y}[ e^{\tilde{\delta} \tilde{\gamma}^{r,\tilde{\upsilon}}_{i,0}}].
\end{equation*}
\end{defn}
%Proof of the next lemma is in Section \ref{sec:vmarkovlong}. Recall the Markov process 
$\{\hat Y^r(t)\}$ from Section \ref{sec:mainres}.
\begin{lemma}[Proof in Section \ref{sec:vmarkovlong}]
\label{thm:VMarkovLong}
For $\tilde{\upsilon}>0$ let $\tilde{\delta}$ be as in Proposition \ref{thm:waitTimeExpBnd}.  There exist constants $B,R\in(0,\infty)$ such that for any $y=(\hat{q},\hat{\Upsilon},\tilde{\mathcal{E}})\in \cly^r$, $t\geq 0$, $i\in\AAA_I$, and $r\geq R$ we have
\begin{equation*}
E_{y}\left[ V^{r,\tilde{\upsilon}}_{i}\left(\hat{Y}^{r}(t)\right)\right]\leq e^{-\tilde{\delta}t}V^{r,\tilde{\upsilon}}_{i}\left(y\right)+B.
\end{equation*}
\end{lemma}
%Proof of the next result is in Section \ref{sec:tildefammalow}.
\begin{proposition}[Proof in Section \ref{sec:tildefammalow}]
\label{thm:TildeGammaLower}
For any $\tilde{\upsilon}>0$ there exist constants $B_{1},B_{2},R\in(0,\infty)$ such that for all $y=(\hat{q},\hat{\Upsilon},\tilde{\mathcal{E}})\in \cly^r$, $i\in\AAA_I$, and $r\geq R$ we have, with $\hat{w}^r =KM^r\hat{q}$,
\begin{equation*}
V^{r,\tilde{\upsilon}}_{i}(y) \geq B_{1}e^{B_{2}\left(\hat{w}^r_{i}-C_{i}\max_{j}\{\hat{\Upsilon}^{A}_{j}\}\right)^{+}}.
\end{equation*}
\end{proposition}

We now proceed to the proof of Theorem \ref{thm:workloadExpBnd}.

%Let $y_{0}\in\mathcal{Y}^{r}_{0}$ be arbitrary.  
Fix $\tilde{\upsilon}>0$ and note that from Proposition \ref{thm:TildeGammaLower} there exist constants $B_{1},B_{2},R_{1}\in(0,\infty)$ such that for all $y=(\hat{q},\hat{\Upsilon},\tilde{\mathcal{E}})\in \cly^r$, $i\in\AAA_I$, and $r\geq R_{1}$ we have
\begin{equation*}
B_{1}e^{B_{2}\left(\hat{w}^r_{i}-C_{i}\max_{j}\{\hat{\Upsilon}^{A}_{j}\}\right)^{+}}\leq V^{r,\tilde{\upsilon}}_{i}(y),
\end{equation*}
where $\hat{w}^r=KM^r\hat{q}$.
From Lemma \ref{thm:VMarkovLong} there exist constants $B_{3},\tilde{\delta}\in(0,\infty)$, and $R_{2}\in [R_{1},\infty)$ such that for any $y=(\hat{q},\hat{\Upsilon},\tilde{\mathcal{E}})\in \cly^r$, $t\geq 0$, $i\in\AAA_I$, and $r\geq R_{2}$ we have
\begin{equation*}
E_{y}\left[ V^{r,\tilde{\upsilon}}_{i}\left(\hat{Y}^{r}(t)\right)\right]\leq e^{-\tilde{\delta}t}V^{r,\tilde{\upsilon}}_{i}\left(y\right)+B_{3}.
\end{equation*}
In addition, Proposition \ref{thm:waitTimeExpBnd} implies that there exist constants $B_{4},B_{5},B_{6}\in(0,\infty)$ and $R_{3}\in [R_{2},\infty)$ such that for all $i\in\AAA_I$, $r\geq R_{3}$, and $y=(\hat{q},\hat{\Upsilon},\tilde{\mathcal{E}})\in \cly^r$ we have
\begin{equation*}
V^{r,\tilde{\upsilon}}_{i}(y)\leq B_{4}e^{B_{5}(|\hat{q}|+|\hat{\Upsilon}|)}+B_{6}.
\end{equation*}
Combining these three inequalities, we have, for all $t\geq 0$
\begin{equation}\label{eq:VBndW}
\begin{aligned}
E_{y}\left[e^{B_{2}\left(\hat{W}^{r}_{i}(t)-C_{i}\max_{j}\{\hat{\Upsilon}^{A,r}_{j}(t)\}\right)^{+}}\right]&\leq\frac{1}{B_{1}}
E_{y}\left[ V^{r,\tilde{\upsilon}}_{i}\left( \hat{Y}^{r}(t)\right)\right]\\
&\leq \frac{e^{-\tilde{\delta}t}}{B_{1}}V^{r,\tilde{\upsilon}}_{i}\left(y\right)+\frac{B_{3}}{B_{1}}
\leq \frac{B_{4}}{B_{1}}e^{-\tilde{\delta}t+B_{5}(|\hat{q}|+|\hat{\Upsilon}|)}+\frac{B_{6}+B_{3}}{B_{1}}.\\
\end{aligned}
\end{equation}
Next, from Proposition \ref{thm:nextArrTimeBnd}, there exist constants $B_{7}\in(0,\infty)$ and $R_{4}\in [R_{3}\vee1,\infty)$, such that, with $\delta>0$ as in Condition \ref{eqn:mgfBnd},  for all $j\in\AAA_J$,   and $c\in (0,\delta)$ we have 
\begin{equation*}
\sup_{r\geq R_{4}, y \in \cly^r, t\in [0,\infty )} E_{y}\left[e^{c u^{r}_{j}(\tau^{A, r}_{j}(t))}\right]\leq B_{7}.
\end{equation*}
Then for alll $j\in\AAA_J$, $t\geq 0$, $c\in(0,\delta)$, and $r\geq R_{4}$ we have 
\begin{equation}\label{eq:UpAShortBnd}
\begin{aligned}
E_{y}\left[e^{c\hat{\Upsilon}^{A,r}_{j}(t)} \right] &= E_{y}\left[e^{c\mathcal{I}_{\left\{\tau^{r,A}_{j}(t)=0\right\}}\hat{\Upsilon}^{A}_{j}+c\mathcal{I}_{\left\{\tau^{r,A}_{j}(t)>0\right\}}\hat{\Upsilon}^{A,r}_{j}(t)}\right]\\
&\le e^{c\mathcal{I}_{\left\{\tau^{r,A}_{j}(t)=0\right\}}\hat{\Upsilon}^{A}_{j}}E_{y}\left[e^{\frac{c}{r}\mathcal{I}_{\left\{\tau^{r,A}_{j}(t)>0\right\}}u^{r}_{j}(\tau^{r,A}_{j}(t))} \right]\\
&\leq e^{c\hat{\Upsilon}^{A}_{j}} E_{y}\left[e^{c\mathcal{I}_{\left\{\tau^{r,A}_{j}(t)>0\right\}}u^{r}_{j}(\tau^{r,A}_{j}(t))} \right]
\leq e^{c\hat{\Upsilon}^{A}_{j}}B_{7}.
\end{aligned}
\end{equation}
Since $\left\{\hat{\Upsilon}^{A,r}_{j}(t)\right\}_{j=1}^{J}$ are independent, for all $c\in(0,\delta)$, $t\geq 0$, and $r\geq R_{4}$  we then have
\begin{align}\label{eq:514n}
E_{y}\left[e^{c\max_{j}\{\hat{\Upsilon}^{A,r}_{j}(t)\}} \right] \leq \sum_{j=1}^{J}E_{y}\left[e^{c\hat{\Upsilon}^{A,r}_{j}(t)} \right]
\leq e^{c\max_{j}\{\hat{\Upsilon}^{A}_{j}\}} J B_{7}.
\end{align}
Note that if $t>\max_{j}\{\hat{\Upsilon}^{A}_{j}\}$ and $r\geq R_{4}\geq 1$ then $\tau^{r,A}_{j}(t)>0$ for all $j\in\AAA_J$.  Consequently, a similar estimate as for \eqref{eq:514n} shows that, for all $t>\max_{j}\{\hat{\Upsilon}^{A}_{j}\}$, $c\in(0,\delta)$, and $r\geq R_{4}$ we have 
% \begin{align*}
% E_{y}\left[e^{c\hat{\Upsilon}^{A,r}_{j}(t)} \right] = E_{y}\left[e^{c\mathcal{I}_{\left\{\tau^{r,A}_{j}(t)=0\right\}}\hat{\Upsilon}^{A}_{j}+c\mathcal{I}_{\left\{\tau^{r,A}_{j}(t)>0\right\}}\hat{\Upsilon}^{A}_{j}(t)}\right]=E_{y}\left[e^{\frac{c}{r}\mathcal{I}_{\left\{\tau^{r,A}_{j}(t)>0\right\}}u^{r}_{j}(\tau^{r,A}_{j}(t))} \right]
% \leq  E_{y}\left[e^{c\mathcal{I}_{\left\{\tau^{r,A}_{j}(t)>0\right\}}u^{r}_{j}(\tau^{r,A}_{j}(t))} \right]
% \leq B_{7}.
% \end{align*}
% The independence of $\left\{\hat{\Upsilon}^{A,r}_{j}(t)\right\}_{j=1}^{J}$ implies that for all $c\in(0,\delta)$, $t>\max_{j}\{\hat{\Upsilon}^{A}_{j}\}$, and $r\geq R_{4}$  we  have
\begin{equation}\label{eq:UpALongBnd}
\begin{aligned}
E_{y}\left[e^{c\max_{j}\{\hat{\Upsilon}^{A,r}_{j}(t)\}} \right] \leq \sum_{j=1}^{J}E_{y}\left[e^{c\hat{\Upsilon}^{A,r}_{j}(t)} \right]
\leq J B_{7}.
\end{aligned}
\end{equation}
Let
$
\delta_{1}\doteq\min\left\{B_{2}/2,\delta/(2C_{i})\right\}
$
and note that for all $c\in(0,\delta_{1})$, $t\geq 0$, $r\geq R_{4}$, and $i\in\AAA_I$ we have
\begin{align*}
E_{y}\Big[ e^{c\hat{W}^{r}_{i}(t)}\Big] &\leq E_{y}\Big[\mathcal{I}_{\{\hat{W}^{r}_{i}(t)> 2C_{i}\max_{j}\{\hat{\Upsilon}^{A,r}_{j}(t)\}\}}e^{2c\left(\hat{W}^{r}_{i}(t)-C_{i}\max_{j}\{\hat{\Upsilon}^{A,r}_{j}(t)\}\right)}\\
&\quad +\mathcal{I}_{\{\hat{W}^{r}_{i}(t)\leq 2C_{i}\max_{j}\{\hat{\Upsilon}^{A,r}_{j}(t)\}\}}e^{2cC_{i}\max_{j}\{\hat{\Upsilon}^{A,r}_{j}(t)\}}\Big]\\
&\leq E_{y}\left[e^{2c\left(\hat{W}^{r}_{i}(t)-C_{i}\max_{j}\{\hat{\Upsilon}^{A,r}_{j}(t)\}\right)^{+}}\right]+E_{y}\left[e^{2cC_{i}\max_{j}\{\hat{\Upsilon}^{A,r}_{j}(t)\}}\right]\\
&\leq E_{y}\left[e^{B_{2}\left(\hat{W}^{r}_{i}(t)-C_{i}\max_{j}\{\hat{\Upsilon}^{A,r}_{j}(t)\}\right)^{+}}\right]+E_{y}\left[e^{2cC_{i}\max_{j}\{\hat{\Upsilon}^{A,r}_{j}(t)\}}\right]\\
&\leq \frac{B_{4}}{B_{1}}e^{-\tilde{\delta}t+B_{5}(|\hat{q}|+|\hat{\Upsilon}|)}+\frac{B_{6}+B_{3}}{B_{1}}+E_{y}\left[e^{2cC_{i}\max_{j}\{\hat{\Upsilon}^{A,r}_{j}(t)\}}\right]\\
\end{align*}
where the last line used equation (\refeq{eq:VBndW}).
Combining this with equations \eqref{eq:514n} and \eqref{eq:UpALongBnd} implies that for any $c\in(0,\delta_{1})$, $r\geq R_{4}$, and $i\in\AAA_I$ if $t\in [0,\max_{j}\{\hat{\Upsilon}^{A}_{j}\}]$ we have
\begin{align*}
E_{y}\Big[ e^{c\hat{W}^{r}_{i}(t)}\Big]\leq \frac{B_{4}}{B_{1}}e^{-\tilde{\delta}t+B_{5}(|\hat{q}|+|\hat{\Upsilon}|)}+\frac{B_{6}+B_{3}}{B_{1}}+e^{\delta_{1}\max_{j}\{\hat{\Upsilon}^{A}_{j}\}} J B_{7}
\end{align*}
and if $t>\max_{j}\{\hat{\Upsilon}^{A}_{j}\}$ we have 
\begin{align*}
E_{y}\Big[ e^{c\hat{W}^{r}_{i}(t)}\Big]\leq \frac{B_{4}}{B_{1}}e^{-\tilde{\delta}t+B_{5}(|\hat{q}|+|\hat{\Upsilon}|)}+\frac{B_{6}+B_{3}}{B_{1}}+ J B_{7}.
\end{align*}
Note that since $|\hat{\Upsilon}|\geq \max_{j}\{\hat{\Upsilon}^{A}_{j}\}$ this implies that for any $c\in(0,\delta_{1})$, $r\geq R_{4}$, $i\in\AAA_I$, and $t\geq 0$ we have
\begin{align*}
E_{y}\Big[ e^{c\hat{W}^{r}_{i}(t)}\Big]\leq \left(\frac{B_{4}}{B_{1}}+J B_{7}\right)e^{-\tilde{\delta}t+(B_{5}+\delta_{1}+\tilde{\delta})(|\hat{q}|+|\hat{\Upsilon}|)}+\frac{B_{6}+B_{3}}{B_{1} }+ JB_7.
\end{align*}

Define
$
\delta_{2}\doteq I^{-\frac{1}{2}}\delta_{1}
$
and note that for all $c\in (0,\delta_{2})$, $t\geq 0$, and $r\geq R_{4}$ we have
\begin{align*}
E_{y}\left[ e^{c|\hat{W}^{r}(t)|_{2}}\right] &\leq E_{y}\left[ e^{cI^{\frac{1}{2}}\max_{i}\{\hat{W}^{r}_{i}(t)\}}\right]
\leq \sum_{i=1}^{I} E_{y}\left[e^{cI^{\frac{1}{2}}\hat{W}^{r}_{i}(t)}\right]\\
&\leq \sum_{i=1}^{I}\left(\left(\frac{B_{4}}{B_{1}}+J B_{7}\right)e^{-\tilde{\delta}t+(B_{5}+\delta_{1}+\tilde{\delta})(|\hat{q}|+|\hat{\Upsilon}|)}+
\frac{B_{6}+B_{3}}{B_{1}}+ JB_7\right)\\
&\leq I\left(\frac{B_{4}}{B_{1}}+J B_{7}\right)e^{-\tilde{\delta}t+(B_{5}+\delta_{1}+\tilde{\delta})(|\hat{q}|+|\hat{\Upsilon}|)}+I\left(\frac{B_{6}+B_{3}}{B_{1}}+ JB_7\right).\\
\end{align*}
Since $y\in\mathcal{Y}^{r}$ was arbitrary this completes the proof.
\hfill \qed

\subsection{Proof of Proposition \ref{thm:costDiffResult}}
\label{sec:costdiffres}
We begin with some auxiliary results.

\begin{lemma}
\label{thm:nearBndZ}
For all $s\geq 0$ if $\mathcal{Z}^{r}(s)\not\in \mathcal{M}$ then $\tilde{d}\left(\left(\hat{Q}^{r}(s)-\tilde{c}_{2}r^{\kappa-1}\right)\vee 0\right)=0$.  
\end{lemma}
\begin{proof}
Let $s\geq 0$ be arbitrary and assume $\mathcal{Z}^{r}(s)\not\in \mathcal{M}$.  Note that for all $j\in\AAA_J$ we have 
$
\left(\hat{Q}_{j}^{r}(s)-\tilde{c}_{2}r^{\kappa-1}\right)^+ =0$
if $\mathcal{Z}^{r}_{j}(s)=1$.  Consequently the result follows on applying Proposition\ref{thm:notMprovesOptimal}
 with $q= \left(\hat{Q}^{r}(s)-\tilde{c}_{2}r^{\kappa-1}\right)^+$ and noting that $z^q=\clz^r(s)$.
\end{proof}
The next result, which is analogous to Proposition \ref{thm:nextArrTimeBnd},  is proved in Section \ref{sec:propnextser}. Recall that, by convention, we take $v_j^r(0) =0$ for $r>0$ and $j \in \AAA_J$.
\begin{proposition}[Proof in Section \ref{sec:propnextser}]
\label{thm:nextSerTimeBnd}
Let $\delta$ be as in Condition \ref{eqn:mgfBnd}.  There exists $R<\infty$ and for any $c<\delta$ a corresponding $K(c)<\infty$ such that for any $j\in\AAA_J$, $y^r\in\mathcal{Y}^{r}$ and $t\geq 0$,
% $\mathcal{F}^{r}(0,0)$-measurable random variable $X\geq 0$ satisfying $P(X<\infty)=1$
we have 
\begin{equation*}
\sup_{r\geq R} E_{y^r}\left[e^{c v^{r}_{j}(\tau^{r,S}_{j}(t))}\right] <K(c).
\end{equation*}
\end{proposition}

The next two results are proved in Section \ref{sec:proofsDTildeCostDiffBnds}. Recall the constant $\lambda$ from Section \ref{sec:propcon}.
\begin{proposition}[Proof in Section \ref{sec:proofsDTildeCostDiffBnds}]
\label{thm:dTildeDiffBnd}
There exists a constant $B_{\hat{h}}\in(0,\infty)$ and $R \in (0,\infty)$ such that for all $q^{1},q^{2}\in\mathbb{R}^{J}_{+}$ and $r \ge R$, we have 
\begin{equation*}
|\lambda|(\tilde{d}(q^{2})-\tilde{d}(q^{1}))\leq h\cdot \left(q^{2}-q^{1}\right)+B_{\hat{h}}|KM^rq^{2}-KM^rq^{1}|_{2}.
\end{equation*}
\end{proposition}

\begin{proposition}[Proof in Section \ref{sec:proofsDTildeCostDiffBnds}]
\label{thm:costDiffBndQ}
There exists a constant $B_{\tilde{d}}<\infty$ such that for all $q^{1},q^{2}\in\mathbb{R}^{J}_{+}$ we have
\begin{equation*}
|\tilde{d}(q^{1})-\tilde{d}(q^{2})|\leq B_{\tilde{d}}|q^{1}-q^{2}|_{2}
\end{equation*}
\end{proposition}

We now proceed to the proof of Proposition \ref{thm:costDiffResult}.\\

{\bf Proof of Proposition \ref{thm:costDiffResult}.}
By adding and subtracting $\hat h(KM \hat Q^r)$, and using Proposition \ref{thm:lambdaEqu}, Proposition \ref{thm:workloadExpBnd}, and the Lipschitz property of the map $\hat h$,  it is sufficient to show that for any $\epsilon_0 \in (0,1)$ and $M<\infty$ there exists $T^{*},R\in(0,\infty)$ such that for all $r\geq R$, $T\geq T^{*}$, $y^{r}\in \mathcal{Y}^{r}$ satisfying $\hat{q}^{r}\leq M$ and $\hat{\Upsilon}^{r}\leq r^{-1}M$, and $t\geq 0$ we have
\begin{equation*}
E_{y^{r}}\left[\frac{1}{T}\int_{0}^{T}\tilde{d}\left( \hat{Q}^{r}(t+s)\right)ds\right] \leq \epsilon_0, \;\; 
E_{y^{r}}\left[\int_{0}^{\infty}e^{-\vsg  s}\tilde{d}\left( \hat{Q}^{r}(t+s)\right)ds\right] \leq \epsilon_0.
\end{equation*}
We now fix $\epsilon_0 \in (0,1)$ and let $\epsilon>0$ be arbitrary which will be chosen suitably later depending on $\epsilon_0$.

The main idea behind the proof is that, on average, the vectors $v^{c}(\mathcal{Z}^{r}(t))$ reduce $\tilde{d}(\hat{Q}^{r}(t))$ faster than other factors like randomness and the vectors $v^{b}(\mathcal{Z}^{r}(t))$ (which are used to keep the queues  nonempty) increase $\tilde{d}(\hat{Q}^{r}(t))$.  This is used to show that for  a given $\epsilon>0$ there exists $k<\infty$ such that for all sufficiently large $r$ and all $t\geq 0$ we have {
\begin{equation}
\label{eq:shortIntTildeDBnd}
E_{y^r}\left[\int_{t}^{t+kr^{-1}}\tilde{d}\left( \hat{Q}^{r}(s)\right)ds \right]\leq \epsilon 3 k r^{-1},
\end{equation}
}
and from this the proposition follows readily.
Fix $\epsilon\in(0,1)$ and $M<\infty$.
%  and $t\geq 0$ be arbitrary.\\
% We will use the following inequalities throughout the proof.  To make it a little bit easier to keep track of notation we will use the convention that a constant with a tilde, such as $\tilde{B}$, will not be referenced again while a constant without a tilde, such as $B$, will be referenced throughout this proof.
Proposition \ref{thm:workloadExpBnd} implies that there exist constants $B_{1}<\infty$ and $c>0$ such that for sufficiently large $r$, all $t\geq 0$, and any $y^{r}\in \mathcal{Y}^{r}$ satisfying $\hat{q}^{r}\leq M$ and $\hat{\Upsilon}^{r}\leq r^{-1}M$ we have
$E_{y^{r}}\left[e^{c|\hat{W}^{r}(t)|} \right]\leq B_{1}$.
From the definition of $\tilde{d}$ there exists a constant $B_{2}>0$ such that for all $q\in\mathbb{R}^{J}_{+}$  and $r$ we have
$
\tilde{d}(q)\leq B_{2}e^{\frac{c}{2}|KM^{r}q|}$
and
$
\tilde{d}(q)\leq B_{2}|KM^{r}q|$.
For notational convenience let $\xi=\frac{\epsilon }{\left(1+4|\lambda| |\tilde{\lambda}|^{-1}\right)B_{1}B_{2}}\wedge 1$, where $\tilde \lambda$ was defined in \eqref{def:tillam}, and define for $n \in \NN$,
\begin{eqnarray*}
\mathcal{A}^{r,n,t,j}&=&\left\{\sup_{0\leq s \leq rn}|A_{j}^{r,t}(s)-s\alpha^{r}_{j}|\geq \frac{1}{5} n r^{\frac{1}{2}} \right\}\cup\left\{\sup_{0\leq s \leq r\max_{i}\{C_{i}\}n}|S_{j}^{r,t}(s)-s\beta^{r}_{j}|\geq \frac{1}{5}nr^{\frac{1}{2}} \right\}\\
&&\cup\left\{r\bar{\Upsilon}^{A,r}_{j}(t) \geq  \xi n \right\}\cup\left\{\beta^{r}_{j}r^{\frac{3}{2}}\bar{\Upsilon}^{S,r}_{j}(t) \geq \frac{1}{5}\xi n \right\},\\
\end{eqnarray*}
\begin{equation*}
\mathcal{A}^{r,n,t}=\cup_{j=1}^{J}\mathcal{A}^{r,n,t,j},
\end{equation*}
and
\begin{equation*}
\mathcal{B}^{r,n,t}\doteq \mathcal{A}^{r,n,t}\cup \left\{ \tilde{d}\left(\hat{Q}^{r}(t+ \xi n r^{-1})\right)> \xi n\right\}\cup\Big\{\int_{r^{2}t}^{r^{2}t+r n}\sum_{j=1}^{J}\mathcal{I}_{\{ \mathcal{E}_{j}^{r}(s)= 1\}}ds\geq \xi n r^{\frac{7}{4}\kappa}\Big\}.
\end{equation*}

{\bf Main proof idea.} We will choose $k$ large enough that $P(\mathcal{B}^{r,k,t})$ is small and we will show that on $(\mathcal{B}^{r,k,t})^c$, because $\hat{\Upsilon}^r(t)$ has a small impact after $t+\xi k r^{-1}$, $\hat{Q}^r(s)$ has favorable behavior for $s\in[t+\xi k r^{-1},t+k r^{-1}]$.  When proving the key estimate (\ref{eq:shortIntTildeDBnd}) we break up the interval $[t,t+kr^{-1}]$ into the subintervals $[t,t+\xi k (1+4|\lambda||\tilde{\lambda}|^{-1})]$ and $[t+\xi k (1+4|\lambda||\tilde{\lambda}|^{-1}),t+kr^{-1}]$.  We can think of $\xi$ as determining the size of the first subinterval over which we rely on a crude bound on  $E[\tilde{d}( \hat{Q}^{r}(s))]$ based on Proposition \ref{thm:workloadExpBnd}, but because $\xi$ is small (recall its definition in terms of $\epsilon$) we can use the length of the subinterval to show that
\begin{equation*}
E_{y^r}\left[\int_{t}^{t+\xi k (1+4|\lambda||\tilde{\lambda}|^{-1})}\tilde{d}\left( \hat{Q}^{r}(s)\right)ds \right]\leq \epsilon k r^{-1}.
\end{equation*}
For the second subinterval we demonstrate that on the set $(\mathcal{B}^{r,k,t})^c$ $\tilde{d}\left( \hat{Q}^{r}(s)\right)$ is small for all $s\in [t+\xi k (1+4|\lambda||\tilde{\lambda}|^{-1}),t+kr^{-1}]$, which combined with the fact that  $(\mathcal{B}^{r,k,t})^c$ occurs with high probability, allows us to show
\begin{equation*}
E_{y^r}\left[\int_{t+\xi k (1+4|\lambda||\tilde{\lambda}|^{-1})}^{t+kr^{-1}}\tilde{d}\left( \hat{Q}^{r}(s)\right)ds \right]\leq 2\epsilon k r^{-1}
\end{equation*}
and complete the proof of (\ref{eq:shortIntTildeDBnd}) from which the Proposition follows easily by appealing to Proposition \ref{thm:lambdaEqu}.

Note that for any $j\in\AAA_J$ we have 
\begin{align*}
&r\bar{\Upsilon}^{A,r}_{j}(t)\leq \mathcal{I}_{\{\tau^{r,A}_{j}(t)=0\}}\hat{\Upsilon}^{A,r}_{j}+r^{-1}\mathcal{I}_{\{\tau^{r,A}_{j}(t)>0\}}u^{r}_{j}(\tau^{r,A}_{j}(t))\\
&\leq \mathcal{I}_{\{\tau^{r,A}_{j}(t)=0\}}r^{-1}M+r^{-1}\mathcal{I}_{\{\tau^{r,A}_{j}(t)>0\}}u^{r}_{j}(\tau^{r,A}_{j}(t))
\end{align*}
and, similarly
\begin{equation*}
\beta^{r}_{j}r^{\frac{3}{2}}\bar{\Upsilon}^{S,r}_{j}(t)\leq 
% \mathcal{I}_{\{\tau^{r,S}_{j}(t)=0\}}\beta^{r}_{j}r^{\frac{1}{2}}\hat{\Upsilon}^{S,r}_{j}+\beta^{r}_{j}r^{-\frac{1}{2}}\mathcal{I}_{\{\tau^{r,S}_{j}(t)>0\}} v^{r}_{j}(\tau^{r,S}_{j}(t))\leq
\mathcal{I}_{\{\tau^{r,S}_{j}(t)=0\}}\beta^{r}_{j}r^{-\frac{1}{2}}M+\beta^{r}_{j}r^{-\frac{1}{2}}\mathcal{I}_{\{\tau^{r,S}_{j}(t)>0\}} v^{r}_{j}(\tau^{r,S}_{j}(t)).
\end{equation*}
This (recall that $\sup_{j}\beta^{r}_{j}<\infty$), combined with Propositions \ref{thm:expTailBnd} (\eqref{eq:31.d}, \eqref{eq:31.c}, with $c_1=1/2$, $c_2=0$), \ref{thm:nextSerTimeBnd},  and \ref{thm:nextArrTimeBnd} implies that there exist constants $\tilde{B}_{1},\tilde{B}_{2}<\infty$
 such that for all $t\geq 0$ and all sufficiently large $r$ we have
$
P_{y^r}\left( \mathcal{A}^{r,n,t}\right) \leq \tilde{B}_{1}e^{- n \tilde{B}_{2}}$.
In addition, for $r\geq \frac{1}{\xi}$ we have

\begin{align*}
 &P_{y^r}\left(\int_{r^{2}t}^{r^{2}t+rn}\sum_{j=1}^{J}\mathcal{I}_{\{ \mathcal{E}_{j}^{r}(s)= 1\}}ds\geq \xi nr^{\frac{7}{4}\kappa} \right)
\leq P_{y^r}\left(\int_{r^{2}t}^{r^{2}(t+\xi n)}\sum_{j=1}^{J}\mathcal{I}_{\{ \mathcal{E}_{j}^{r}(s)= 1\}}ds\geq \xi n r^{\frac{7}{4}\kappa} \right)\\
&\leq\sum_{j=1}^{J}P_{y^r}\left(\int_{r^{2}t}^{r^{2}(t+\xi n)}\mathcal{I}_{\{ \mathcal{E}_{j}^{r}(s)= 1\}}ds\geq r\hat{\Upsilon}^{A,r}_{j}(t)+\frac{\xi n}{2J} r^{\frac{7}{4}\kappa} \right)
+\sum_{j=1}^{J}P_{y^r}\left(r\hat{\Upsilon}^{A,r}_{j}(t)\geq\frac{n \xi}{2J}  r^{\frac{7}{4}\kappa} \right).\\
\end{align*}
This, combined with Proposition \ref{thm:idleTimeBndMark}, Proposition \ref{thm:nextArrTimeBnd}, and the fact that for any $j\in\AAA_J$
\begin{eqnarray*}
r\hat{\Upsilon}^{A,r}_{j}(t)&\leq& \mathcal{I}_{\{\tau^{r,A}_{j}(t)=0\}}r\hat{\Upsilon}^{A,r}_{j}+\mathcal{I}_{\{\tau^{r,A}_{j}(t)>0\}}u^{r}_{j}(\tau^{r,A}_{j}(t))\\
&\leq& \mathcal{I}_{\{\tau^{r,A}_{j}(t)=0\}}M+\mathcal{I}_{\{\tau^{r,A}_{j}(t)>0\}}u^{r}_{j}(\tau^{r,A}_{j}(t))
\end{eqnarray*}
implies that there exist constant $\tilde{B}_{3},\tilde{B}_{4}<\infty$ such that for sufficiently large $r$ and for all $t \geq 0$ and $n\geq 0$ we have 
\begin{equation*}
 P_{y^r}\left(\int_{r^{2}t}^{r^{2}t+rn}\sum_{j=1}^{J}\mathcal{I}_{\{ \mathcal{E}_{j}^{r}(s)= 1\}}ds\geq \xi nr^{\frac{7}{4}\kappa} \right)\leq \tilde{B}_{3}e^{-n\tilde{B}_{4}}.
\end{equation*}

Furthermore, since $\tilde{d}(\hat{Q}^{r}(t+ \xi n r^{-1}))\leq B_{2}|\hat{W}^{r}(t+ \xi n r^{-1})|$ and $E_{y^{r}}\left[e^{c|\hat{W}^{r}(t+ \xi n r^{-1})|} \right]\leq B_{1}$ for sufficiently large $r$ we have
\begin{equation*}
P_{y^r}\left(\tilde{d}\left(\hat{Q}^{r}(t+ \xi n r^{-1})\right)> \xi n \right)\leq P_{y^r}\left(|\hat{W}^{r}(t+ \xi n r^{-1})|> \frac{\xi n}{B_{2}}\right)\leq B_{1}e^{-\frac{c\xi}{B_{2}}n}.
\end{equation*}

Consequently we know there exist constants $B_{3},B_{4}<\infty$ such that for sufficiently large $r$ we have for all $t\geq 0$ and $n\geq 0$
$
P_{y^r}\left( \mathcal{B}^{r,n,t}\right)\leq B_{3}e^{- n B_{4}}$.

 Our goal is to show that for any $\epsilon>0$ we can choose corresponding $k,R<\infty$ sufficiently large such that for all $r\geq R$ (\ref{eq:shortIntTildeDBnd}) holds.  Our choices of $k,R<\infty$ will have to satisfy a large number of inequalities used throughout the remainder of the proof, so for the ease of the reader we list them all in one place below and reference them later when needed.

\begin{summary}
Recall the definitions of $\mathcal{B}^{r,n,t}$,  $B_{1}$, $B_{2}$, $B_{3}$, $B_{4}$, $\xi$, and $c$.  We choose the constants $k\geq \frac{1}{\xi}$ and $R<\infty$ such that, $B^{\frac{1}{2}}_{3}e^{-k \frac{B_{4}}{2}}B_{2}B^{\frac{1}{2}}_{1}\leq \epsilon$ and, for all $r\geq R$,   the following hold:
\label{sum:costDiffResultSum}
\begin{enumerate}[(a)]
\item  For all $t\geq 0$ we have $E_{y^{r}}\left[ e^{c|\hat{W}^{r}(t)|}\right]\leq B_{1}$
\item For all $q\in\mathbb{R}^{J}_{+}$, we have $\tilde{d}(q)\leq B_{2}e^{\frac{c}{2}|KM^{r}q|}$
  and  $\tilde{d}(q)\leq B_{2}|KM^{r}q|$
\item For all $t\geq 0$ we have $P_{y^r}\left( \mathcal{B}^{r,k,t}\right)\leq B_{3}e^{- k B_{4}}$.
\item  $\epsilon r^{\frac{1}{8}}\geq k$, $r\geq \frac{1}{\xi}$,  $r\geq 25$, and  $\epsilon r^{\frac{1}{16}}\geq \xi k$.
\item  $2|\theta|+|\frac{1}{\beta^{r}}|_{1}\leq r^{\frac{1}{8}}$
\item For all $i\in\AAA_I$ we have $r\left(\sum_{j=1}^{J}K_{i,j}\rho^{r}_{j}-C_{i}\right)\geq 2\theta_{i}$
\item $J (c_{2}+1)r^{\kappa-1}+C_{i} \epsilon r^{\frac{7}{4}\kappa-\frac{15}{16}}+|\frac{1}{\beta^{r}}|_{1}\epsilon r^{-\frac{3}{8}} \leq \epsilon r^{-\frac{1}{4}}$
\item For all $i\in\AAA_I$ we have $\left(\sum_{j=1}^{J}K_{i,j}\rho^{r}_{j}-C_{i}\right)<0$
\item  $|h||\beta^{r}||\rho -\rho^{r} |+|h||\beta^{r}-\beta|\left(\max_{z\in\mathcal{M}}|v^{c}(z)|+\max_{z\in\left\{0,1\right\}^{J}}|v^{b}(z)|\right)\leq \frac{|\tilde{\lambda}|}{4}$
\item $|\alpha^{r}|\leq 2|\alpha|$ 
\item  $\left(\frac{|\tilde{\lambda}|}{2}+2|h||\alpha|\right)\epsilon r^{\frac{7}{4}\kappa-\frac{15}{16}}+|h|\epsilon r^{-\frac{3}{8}}\leq \epsilon r^{-\frac{1}{4}}$
\item  $\frac{\epsilon}{|\lambda|} r^{-\frac{1}{4}}+\frac{B_{\hat{h}}}{|\lambda|}I^{\frac{1}{2}}\epsilon r^{-\frac{1}{4}}\leq \xi k $
\item  $B_{\tilde{d}} J^{\frac{1}{2}}\tilde{c}_{2} r^{\kappa-1}+J^{\frac{1}{2}}B_{\tilde{d}}r^{-1} +\frac{\epsilon}{|\lambda|} r^{-\frac{1}{4}}+2B_{\hat{h}}I^{\frac{1}{2}}\frac{\epsilon}{|\lambda|} r^{-\frac{1}{4}}\leq \epsilon r^{-\frac{1}{8}}$.
\end{enumerate}
\end{summary}

We now  bound $\tilde{d}\left(\hat{Q}^{r}(\cdot)\right)$ from above on the set $\left(\mathcal{B}^{r,k,t}\right)^{c}$.
Using \eqref{eq:328},
% For any $j\in\AAA_J$ and $s\geq t$ we have
% \begin{eqnarray*}
% Q^{r}_{j}(sr^{2})&=&Q^{r}_{j}(r^{2}t)+\mathcal{I}_{\{s-t\geq \bar{\Upsilon}^{A,r}_{j}(t)\}}+A_{j}^{r,t}\left(r^2\left(s-t-\bar{\Upsilon}^{A,r}_{j}(t)\right)^{+}\right)\\
% &&-\mathcal{I}_{\{\bar{B}^{r}_{j}(s)-\bar{B}^{r}_{j}(t)\geq \bar{\Upsilon}^{S,r}_{j}(t)\}}-S_{j}^{r,t}\left(r^2\left(\bar{B}^{r}_{j}(s)-\bar{B}^{r}_{j}(t)-\bar{\Upsilon}^{S,r}_{j}(t)\right)^{+}\right).\\
% \end{eqnarray*}
% Consequently 
on the set $\left( \mathcal{B}^{r,k,t}\right)^{c}$ for all $r\geq R$, any $t+ \xi k r^{-1}\leq s_{1}<s_{2}\leq t+k r^{-1}$, and any $j\in \AAA_J$, we have 
\begin{eqnarray*}
&&\hat{Q}^{r}_{j}(s_{2})-\hat{Q}^{r}_{j}(s_{1})\\
&=&\frac{1}{r}\mathcal{I}_{\{s_{2}-t\geq \bar{\Upsilon}^{A,r}_{j}(t)>s_{1}-t\}}- \frac{1}{r}\mathcal{I}_{\{\bar{B}^{r}_{j}(s_{2})-\bar{B}^{r}_{j}(t)\geq \bar{\Upsilon}^{S,r}_{j}(t)>\bar{B}^{r}_{j}(s_{1})-\bar{B}^{r}_{j}(t)\}}\\
&&+ \frac{1}{r}A_{j}^{r,t}\left(r^2\left(s_{2}-t-\bar{\Upsilon}^{A,r}_{j}(t)\right)^{+}\right)-\frac{1}{r}A_{j}^{r,t}\left(r^2\left(s_{1}-t-\bar{\Upsilon}^{A,r}_{j}(t)\right)^{+}\right)\\
&&-\frac{1}{r}S_{j}^{r,t}\left(r^2\left(\bar{B}^{r}_{j}(s_{2})-\bar{B}^{r}_{j}(t)-\bar{\Upsilon}^{S,r}_{j}(t)\right)^{+}\right)+\frac{1}{r}S_{j}^{r,t}\left(r^2\left(\bar{B}^{r}_{j}(s_{1})-\bar{B}^{r}_{j}(t)-\bar{\Upsilon}^{S,r}_{j}(t)\right)^{+}\right)\\
&\leq& \frac{1}{r}A_{j}^{r,t}\left(r^2\left(s_{2}-t-\bar{\Upsilon}^{A,r}_{j}(t)\right)\right)-\frac{1}{r}A_{j}^{r,t}\left(r^2\left(s_{1}-t-\bar{\Upsilon}^{A,r}_{j}(t)\right)\right)\\
&&-\frac{1}{r}S_{j}^{r,t}\left(r^2\left(\bar{B}^{r}_{j}(s_{2})-\bar{B}^{r}_{j}(t)-\bar{\Upsilon}^{S,r}_{j}(t)\right)^{+}\right)+\frac{1}{r}S_{j}^{r,t}\left(r^2\left(\bar{B}^{r}_{j}(s_{1})-\bar{B}^{r}_{j}(t)-\bar{\Upsilon}^{S,r}_{j}(t)\right)^{+}\right),
\end{eqnarray*}
since on the set $\left( \mathcal{B}^{r,k,t}\right)^{c}$ we have $ \bar{\Upsilon}^{A,r}_{j}(t)\leq \xi k r^{-1}$ meaning $\mathcal{I}_{\{s_{2}-t\geq \bar{\Upsilon}^{A,r}_{j}(t)>s_{1}-t\}}=0$,
and so
\begin{eqnarray*}
\hat{Q}^{r}_{j}(s_{2})-\hat{Q}^{r}_{j}(s_{1})
&\leq& \alpha^{r}_{j}r\left(s_{2}-s_{1}\right)+\frac{4}{5}k r^{-\frac{1}{2}}-\beta^{r}_{j}r\left(\left(\bar{B}^{r}_{j}(s_{2})-\bar{B}^{r}_{j}(t)-\bar{\Upsilon}^{S,r}_{j}(t)\right)^{+}\right.\\
&&\left.-\left(\bar{B}^{r}_{j}(s_{1})-\bar{B}^{r}_{j}(t)-\bar{\Upsilon}^{S,r}_{j}(t)\right)^{+} \right)\\
&\leq& \alpha^{r}_{j}r\left(s_{2}-s_{1}\right)-\beta^{r}_{j}r\left(\bar{B}^{r}_{j}(s_{2})-\bar{B}^{r}_{j}(s_{1})\right)+\beta^{r}_{j}r\bar{\Upsilon}^{S,r}_{j}(t)+\frac{4}{5}k r^{-\frac{1}{2}}\\
&\leq& \alpha^{r}_{j}r\left(s_{2}-s_{1}\right)-\beta^{r}_{j}r\left(\bar{B}^{r}_{j}(s_{2})-\bar{B}^{r}_{j}(s_{1})\right)+kr^{-\frac{1}{2}}\\
&\leq& \alpha^{r}_{j}r\left(s_{2}-s_{1}\right)-\beta^{r}_{j}r\left(\bar{B}^{r}_{j}(s_{2})-\bar{B}^{r}_{j}(s_{1})\right)+\epsilon r^{-\frac{3}{8}}\\
\end{eqnarray*}
 due to Summary \ref{sum:costDiffResultSum} part (d) and the fact that % $\left(\bar{B}^{r}_{j}(s_{2})-\bar{B}^{r}_{j}(t)-\bar{\Upsilon}^{S,r}_{j}(t)\right)^{+}-\left(\bar{B}^{r}_{j}(s_{1})-\bar{B}^{r}_{j}(t)-\bar{\Upsilon}^{S,r}_{j}(t)\right)^{+}\geq \bar{B}^{r}_{j}(s_{2})-\bar{B}^{r}_{j}(s_{1})-\bar{\Upsilon}^{S,r}_{j}(t)$,
 on the set $\left( \mathcal{B}^{r,k,t}\right)^{c}$ we have $\beta^{r}_{j}r^{\frac{3}{2}}\bar{\Upsilon}^{S,r}_{j}(t)\leq \frac{1}{5}\xi k \leq \frac{1}{5}k $.
 
 A similar calculation, using the inequalities,  $$\left(\bar{B}^{r}_{j}(s_{2})-\bar{B}^{r}_{j}(t)-\bar{\Upsilon}^{S,r}_{j}(t)\right)^{+}-\left(\bar{B}^{r}_{j}(s_{1})-\bar{B}^{r}_{j}(t)-\bar{\Upsilon}^{S,r}_{j}(t)\right)^{+}\leq \bar{B}^{r}_{j}(s_{2})-\bar{B}^{r}_{j}(s_{1}),$$
  $- \frac{1}{r}\mathcal{I}_{\{\bar{B}^{r}_{j}(s_{2})-\bar{B}^{r}_{j}(t)\geq \bar{\Upsilon}^{S,r}_{j}(t)>\bar{B}^{r}_{j}(s_{1})-\bar{B}^{r}_{j}(t)\}}\geq -\frac{1}{r}$,   properties in Summary \ref{sum:costDiffResultSum} {part} (d), and the fact that on the set $\left( \mathcal{B}^{r,k,t}\right)^{c}$ we have $ \bar{\Upsilon}^{A,r}_{j}(t)\leq \xi k r^{-1}$ shows that
 on the set $\left( \mathcal{B}^{r,k,t}\right)^{c}$ for any $t+ \xi k r^{-1}\leq s_{1}<s_{2}\leq t+k r^{-1}$, $r\geq R$, and $j\in \AAA_J$ we have
\begin{eqnarray*}
\hat{Q}^{r}_{j}(s_{2})-\hat{Q}^{r}_{j}(s_{1})
&\geq& \alpha^{r}_{j}r\left(s_{2}-s_{1}\right)-\beta^{r}_{j}r\left(\bar{B}^{r}_{j}(s_{2})-\bar{B}^{r}_{j}(s_{1})\right)-\epsilon r^{-\frac{3}{8}}.
\end{eqnarray*}
Consequently on the set $\left( \mathcal{B}^{r,k,t}\right)^{c}$ for any $t+\xi k r^{-1}\leq s_{1}<s_{2}\leq t+k r^{-1}$, $r\geq R$, and $j\in \AAA_J$ we have
\begin{equation}
\label{eq:queueBnd}
|\hat{Q}^{r}_{j}(s_{2})-\hat{Q}^{r}_{j}(s_{1})-r\left( \alpha^{r}_{j}(s_{2}-s_{1})-\beta^{r}_{j}\left(\bar{B}^{r}_{j}(s_{2})-\bar{B}^{r}_{j}(s_{1}) \right)\right)|\leq \epsilon r^{-\frac{3}{8}}.
\end{equation}
We will next use Proposition \ref{thm:dTildeDiffBnd} to bound $\tilde{d}(\hat{Q}^{r}(\cdot))$ from above on the interval $[t+\xi k r^{-1},t+k r^{-1}]$. For this  we will need to bound
$
\sup_{s\in[t+\xi k r^{-1},t+k r^{-1}]}|\hat{W}^{r}_{i}(s)-\hat{W}^{r}_{i}(t+\xi k r^{-1})|
$
for all $i\in\AAA_I$.  Let $i\in\AAA_I$ be arbitrary.  We will first bound\\
 $\inf_{s\in[t+\xi k r^{-1},t+k r^{-1}]}\left(\hat{W}^{r}_{i}(s)-\hat{W}^{r}_{i}(t+\xi k r^{-1})\right)$ from below.  
On the set $\left(\mathcal{B}^{r,k,t}\right)^{c}$ for all $r\geq R$ and $s\in[t+\xi k r^{-1},t+k r^{-1}]$ we have from \eqref{eq:queueBnd}
\begin{eqnarray*}
\hat{W}^{r}_{i}(s)&\geq&\hat{W}^{r}_{i}(t+\xi k r^{-1})- |\frac{1}{\beta^{r}}|_{1}\epsilon r^{-\frac{3}{8}}\\
&&+r\sum_{j=1}^{J}K_{i,j}\frac{1}{\beta^{r}_{j}}\left(\alpha^{r}_{j}(s-t-\xi k r^{-1})-\beta^{r}_{j}\left(\bar{B}^{r}_j(s)-\bar{B}^{r}_j(t+\xi k r^{-1}) \right) \right)\\
&\geq&\hat{W}^{r}_{i}(t+\xi k r^{-1})+r(s-t-\xi k r^{-1})\left(\sum_{j=1}^{J}K_{i,j}\rho^{r}_{j}-C_{i}\right)-|\frac{1}{\beta^{r}}|_{1}\epsilon r^{-\frac{3}{8}}\\
&\geq&\hat{W}^{r}_{i}(t+\xi k r^{-1})+2\theta_{i}(s-t-\xi k r^{-1})-|\frac{1}{\beta^{r}}|_{1}\epsilon r^{-\frac{3}{8}}\\
&\geq&\hat{W}^{r}_{i}(t+\xi k r^{-1})+2\theta_{i} k r^{-1}-|\frac{1}{\beta^{r}}|_{1}\epsilon r^{-\frac{3}{8}}\\
&\geq&\hat{W}^{r}_{i}(t+\xi k r^{-1})+2\theta_{i}\epsilon r^{-\frac{7}{8}}-|\frac{1}{\beta^{r}}|_{1}\epsilon r^{-\frac{3}{8}}\\
&\geq&\hat{W}^{r}_{i}(t+\xi k r^{-1})-\left(2|\theta_{i}|+|\frac{1}{\beta^{r}}|_{1}\right)\epsilon r^{-\frac{3}{8}}\geq\hat{W}^{r}_{i}(t+\xi k r^{-1})-\epsilon r^{-\frac{1}{4}}\\
\end{eqnarray*}
due to equation (\ref{eq:queueBnd}) and Summary \ref{sum:costDiffResultSum} parts (d), (e), and (f).  Therefore on the set $\left(\mathcal{B}^{r,k,t}\right)^{c}$ for all $r\geq R$ we have
\begin{equation*}
\inf_{s\in[t+\xi k r^{-1},t+k r^{-1}]}\left(\hat{W}^{r}_{i}(s)-\hat{W}^{r}_{i}(t+\xi k r^{-1})\right)\geq -\epsilon r^{-\frac{1}{4}}.
\end{equation*}
Now we  bound $\sup_{s\in[t+\xi k r^{-1},t+ k r^{-1}]}\left(\hat{W}^{r}_{i}(s)-\hat{W}^{r}_{i}(t+\xi k r^{-1})\right)$ from above. 
For all $i\in\AAA_I$ define
\begin{equation*}
\gamma^{r,i}_{0}\doteq\inf\left\{s\geq t+\xi k r^{-1}:\hat{W}^{r}_{i}(s)\geq J c_{2}r^{\kappa-1} \right\}
\end{equation*}
and for $l\geq 0$ define {
\begin{equation*}
\gamma^{r,i}_{2l+1}\doteq\inf\left\{s\geq \gamma^{r,i}_{2l}: \hat{W}^{r}_{i}(s)< J c_{2}r^{\kappa-1} \right\},\;\;
\gamma^{r,i}_{2l+2}\doteq\inf\left\{s\geq\gamma^{r,i}_{2l+1}:  \hat{W}^{r}_{i}(s)\geq J c_{2}r^{\kappa-1} \right\}.
\end{equation*}
}
Using  Proposition \ref{thm:schemeSum} part (d) and \eqref{eq:queueBnd},  for $r\geq R$, on the set $\left(\mathcal{B}^{r,k,t}\right)^{c}$, for any $l\geq 0$ and all
 $s\in [\gamma^{r,i}_{2l},\gamma^{r,i}_{2l+1})\cap [\gamma^{r,i}_{2l},t+ k r^{-1}]$, 
\begin{align*}
&\hat{W}^{r}_{i}(s)\\
&\leq \hat{W}^{r}_{i}(\gamma^{r,i}_{2l})+r\sum_{j=1}^{J}K_{i,j}\frac{1}{\beta^{r}_{j}}\left(\alpha^{r}_{j}(s-\gamma^{r,i}_{2l})-\beta^{r}_{j}\left(\bar{B}^{r}(s)-\bar{B}^{r}(\gamma^{r,i}_{2l}) \right) \right)+|\frac{1}{\beta^{r}}|_{1}\epsilon r^{-\frac{3}{8}}\\
&\leq \hat{W}^{r}_{i}(t+\xi k r^{-1})+ J( c_{2}+1)r^{\kappa-1}+r(s-\gamma^{r,i}_{2l})\left(\sum_{j=1}^{J}K_{i,j}\rho^{r}_{j}-C_{i}\right)\\
&\quad+r^{-1}C_{i}\int_{r^{2}t}^{r^{2}t+rk}\sum_{j=1}^{J}\mathcal{I}_{\{ \mathcal{E}_{j}^{r}(u)= 1\}}du+|\frac{1}{\beta^{r}}|_{1}\epsilon r^{-\frac{3}{8}}\\
&\leq \hat{W}^{r}_{i}(t+\xi k r^{-1})+ J (c_{2}+1)r^{\kappa-1}+r(s-\gamma^{r,i}_{2l})\left(\sum_{j=1}^{J}K_{i,j}\rho^{r}_{j}-C_{i}\right)\\
& \quad+ C_{i}\xi k r^{\frac{7}{4}\kappa-1}+|\frac{1}{\beta^{r}}|_{1}\epsilon r^{-\frac{3}{8}}\\
&\leq \hat{W}^{r}_{i}(t+\xi k r^{-1})+J (c_{2}+1)r^{\kappa-1}+r(s-\gamma^{r,i}_{2l})\left(\sum_{j=1}^{J}K_{i,j}\rho^{r}_{j}-C_{i}\right)\\
& \quad+ C_{i}\epsilon r^{\frac{7}{4}\kappa-\frac{15}{16}}+|\frac{1}{\beta^{r}}|_{1}\epsilon r^{-\frac{3}{8}}\\
&\leq \hat{W}^{r}_{i}(t+\xi k r^{-1})+r(s-\gamma^{r,i}_{2l})\left(\sum_{j=1}^{J}K_{i,j}\rho^{r}_{j}-C_{i}\right) +\epsilon r^{-\frac{1}{4}}
\leq \hat{W}^{r}_{i}(t+\xi k r^{-1}) +\epsilon r^{-\frac{1}{4}}
\end{align*}
due to  Summary \ref{sum:costDiffResultSum} parts (d), (g) and (h), and the fact that on the set $\left(\mathcal{B}^{r,k,t}\right)^{c}$ we have $r^{-1}\int_{r^{2}t}^{r^{2}t+rk}\sum_{j=1}^{J}\mathcal{I}_{\{ \mathcal{E}_{j}^{r}(s)= 1\}}ds\leq \xi k r^{\frac{7}{4}\kappa-1}$.  
In addition, for $r\geq R$ on the set $\left(\mathcal{B}^{r,k,t}\right)^{c}$, for all $s\in [t+\xi k r^{-1},\gamma^{r,i}_{0})\cap [t+\xi k r^{-1}, t+kr^{-1}]$ and $s\in [\gamma^{r,i}_{2l+1},\gamma^{r,i}_{2l+2})\cap  [\gamma^{r,i}_{2l+1}, t+kr^{-1}]$ for any $l\geq 0$ we have
\begin{eqnarray*}
\hat{W}^{r}_{i}(s)&\leq& J(c_{2}+1)r^{\kappa-1}\leq J(c_{2}+1)r^{\kappa-1}+ \hat{W}^{r}_{i}(t+\xi k r^{-1})\leq  \hat{W}^{r}_{i}(t+\xi k r^{-1})+\epsilon r^{-\frac{1}{4}}
\end{eqnarray*}
due to Summary \ref{sum:costDiffResultSum} part (g). Since $i\in\AAA_I$ was arbitrary we have shown that on the set $\left(\mathcal{B}^{r,k,t}\right)^{c}$ for all $r\geq R$ and $i\in\AAA_I$  we have
\begin{equation}
\label{eq:workloadBndComp}
\sup_{s\in[t+\xi k r^{-1},t+k r^{-1}]}|\hat{W}^{r}_{i}(s)-\hat{W}^{r}_{i}(t+\xi k r^{-1})|\leq \epsilon r^{-\frac{1}{4}}.
\end{equation}
Now we will apply Proposition \ref{thm:dTildeDiffBnd} to bound $\tilde{d}(\hat{Q}^{r}(\cdot))$ from above on the interval $[t+\xi k r^{-1},t+k r^{-1}]$.
From Proposition \ref{thm:costDiffBndQ} ,
% there exists a constant $B_{\tilde{d}}>0$ such that for all $q^{1},q^{2}\in\mathbb{R}^{J}_{+}$
% \begin{equation*}
% |\tilde{d}(q^{1})-\tilde{d}(q^{2})|\leq B_{\tilde{d}}|q^{1}-q^{2}|_{2}.
% \end{equation*}
% Therefore
for any $q\in\mathbb{R}^{J}_{+}$,
 % if we define $\left(q- \tilde{c}_{2}r^{\kappa}\right) \vee 0$ component-wise so
% \begin{equation*}
% \left(q- \tilde{c}_{2}r^{\kappa}\right) \vee 0=\left(\left(q_{1}-\tilde{c}_{2}r^{\kappa}\right) \vee 0,\left(q_{2}- \tilde{c}_{2}r^{\kappa}\right) \vee 0,...,\left(q_{J}- \tilde{c}_{2}r^{\kappa}\right) \vee 0\right),
% \end{equation*}
%  we have
\begin{equation}\label{eq:tilddiff}
|\tilde{d}(q)-\tilde{d}(\left(q- \tilde{c}_{2}r^{\kappa}\right) \vee 0)|\leq B_{\tilde{d}}|q-\left(q- \tilde{c}_{2}r^{\kappa}\right) \vee 0|_{2}\leq B_{\tilde{d}}J^{\frac{1}{2}}\tilde{c}_{2}r^{\kappa}.
\end{equation}
Define
\begin{equation*}
\zeta^{r}_{0}\doteq \inf\left\{ s\geq t+\xi k r^{-1}: \tilde{d}\left(Q^{r}(r^{2}s) \right)\leq B_{\tilde{d}} J^{\frac{1}{2}}\tilde{c}_{2} r^{\kappa} \right\}
\end{equation*}
and for $l\geq 0$ define 
\begin{align*}
\zeta^{r}_{2l+1}&\doteq \inf\left\{ s\geq \zeta^{r}_{2l}: \tilde{d}\left(Q^{r}(r^{2}s) \right)>B_{\tilde{d}}J^{\frac{1}{2}}\tilde{c}_{2} r^{\kappa} \right\},\\
\zeta^{r}_{2l+2}&\doteq \inf\left\{ s\geq \zeta^{r}_{2l+1}: \tilde{d}\left(Q^{r}(r^{2}s) \right)\leq  B_{\tilde{d}} J^{\frac{1}{2}}\tilde{c}_{2} r^{\kappa} \right\}.
\end{align*}
Note that if $\mathcal{Z}^{r}(s)\not\in\mathcal{M}$ then due to Lemma \ref{thm:nearBndZ} we have $\tilde{d}\left(\left(Q^{r}(r^{2}s)-\tilde{c}_{2}r^{\kappa}\right)\vee 0\right)=0$ so, from \eqref{eq:tilddiff},
\begin{eqnarray*}
\tilde{d}\left(Q^{r}(r^{2}s)\right)&\leq& \tilde{d}\left(\left(Q^{r}(r^{2}s)-\tilde{c}_{2}r^{\kappa}\right)\vee 0\right)+|\tilde{d}(Q^{r}(r^{2}s))-\tilde{d}(\left(Q^{r}(r^{2}s)- \tilde{c}_{2}r^{\kappa}\right) \vee 0)|\\
% \\
% &\leq& B_{\tilde{d}}|Q^{r}(r^{2}s)-\left(Q^{r}(r^{2}s)- \tilde{c}_{2}r^{\kappa}\right) \vee 0|_{2}
&\leq& B_{\tilde{d}}J^{\frac{1}{2}}\tilde{c}_{2}r^{\kappa}.
\end{eqnarray*}
Consequently if $\tilde{d}\left(\hat{Q}^{r}(s)\right)>B_{\tilde{d}}J^{\frac{1}{2}}\tilde{c}_{2}r^{\kappa-1}$ then $\mathcal{Z}^{r}(s)\in\mathcal{M}$ (here we have used the fact that for any $r> 0$ and $q\in\mathbb{R}^{J}_{+}$ we have $\tilde{d}(\frac{1}{r}q)=\frac{1}{r}\tilde{d}(q)$).  In particular, $\mathcal{Z}^{r}(s)\in\mathcal{M}$ for all $s\in[t+\xi k r^{-1},\zeta^{r}_{0})\cap [t+\xi k r^{-1}, t+k r^{-1}]$ and all $s\in[\zeta^{r}_{2l+1},\zeta^{r}_{2l+2})\cap [\zeta^{r}_{2l+1}, t+k r^{-1}]$ for any $l\geq 0$.

 Next note that from Definition \ref{def:scheme} for all $r\geq R$ if $\mathcal{Z}^{r}(s)\in\mathcal{M}$ and $\sum_{j=1}^{J} \mathcal{E}_{j}^{r}(s)=0$ we have
$
\yr(s)=\rho-v^{c}(\mathcal{Z}^{r}(s))-v^{b}(\mathcal{Z}^{r}(s))
$
and so
\begin{align}
&h\cdot \left(\alpha^{r}-\beta^{r}\yr(s) \right)=h\cdot \left(\alpha^{r}-\beta^{r}\left( \rho-v^{c}(\mathcal{Z}^{r}(s))-v^{b}(\mathcal{Z}^{r}(s))\right) \right)\nonumber\\
&= h\cdot \left(- \beta^{r}\left(\rho -\rho^{r} \right)+\beta v^{c}(\mathcal{Z}^{r}(s))+\beta v^{b}(\mathcal{Z}^{r}(s)) \right. \nonumber \\
&\quad \left.+(\beta^{r}-\beta) v^{c}(\mathcal{Z}^{r}(s))+(\beta^{r}-\beta)v^{b}(\mathcal{Z}^{r}(s))\right)\nonumber\\
&\leq |h||\beta^{r}||\rho -\rho^{r} |+|h||\beta^{r}-\beta|\left(\max_{z\in\mathcal{M}}|v^{c}(z)|+\max_{z\in\left\{0,1\right\}^{J}}|v^{b}(z)|\right)\nonumber\\
 &\quad +h\beta\cdot v^{c}(\mathcal{Z}^{r}(s))+h\beta\cdot v^{b}(\mathcal{Z}^{r}(s))\nonumber \\
&\leq\frac{|\tilde{\lambda}|}{4}-|\tilde{\lambda}|+|h||\beta|\frac{|\tilde{\lambda}|}{4|h|\beta|} \leq -\frac{|\tilde{\lambda}|}{2} \label{eq:453}
\end{align}
where we have used Summary \ref{sum:costDiffResultSum} part (i) and the fact that due to Definition \ref{def:scheme} we have $\max_{z\in \mathcal{Z}}\left\{h\beta\cdot v^{c}(z)\right\}\leq -|\tilde{\lambda}|$ and $\max_{z\in\left\{0,1\right\}}|v^{b}(z)|\leq \frac{|\tilde{\lambda}|}{4|\beta||h|}$.
Next, on the set $\left(\mathcal{B}^{r,k,t}\right)^{c}$ for all $r\geq R$ and $s\in[t+\xi k r^{-1},\zeta^{r}_{0})\cap [t+\xi k r^{-1},t+k r^{-1}]$ we have from \eqref{eq:queueBnd}
\begin{align*}
&h\cdot \left(\hat{Q}^{r}(s)-\hat{Q}^{r}(t+\xi k r^{-1}) \right)- |h|\epsilon r^{-\frac{3}{8}}\\
&\leq rh\cdot \left(\alpha^{r}\left(s-t-\xi k r^{-1} \right)-\beta^{r}\left(\bar{B}^{r}(s)-\bar{B}^{r}(t+\xi k r^{-1}) \right) \right)\\
&\leq \int_{t+\xi k r^{-1}}^{s}\mathcal{I}_{\{\sum_{j=1}^{J} \mathcal{E}_{j}^{r}(r^{2}z)=0\}} rh\cdot \left(\alpha^{r}-\beta^{r}\yr(r^{2}z)\right)dz\\
&\quad +\int_{t+\xi k r^{-1}}^{s}\mathcal{I}_{\{\sum_{j=1}^{J} \mathcal{E}_{j}^{r}(r^{2}z)>0\}} rh\cdot \left(\alpha^{r}-\beta^{r}\yr(r^{2}z)\right)dz\\
&\leq  -r\left(s-t-\xi k r^{-1} -\int_{t+\xi k r^{-1}}^{s}\mathcal{I}_{\{\sum_{j=1}^{J} \mathcal{E}_{j}^{r}(r^{2}z)>0\}}dz \right)\frac{|\tilde{\lambda}|}{2}\\
&\quad +r(h\cdot \alpha^{r})\int_{t+\xi k r^{-1}}^{s}\mathcal{I}_{\{\sum_{j=1}^{J} \mathcal{E}_{j}^{r}(r^{2}z)>0\}}dz
\end{align*}
where the last line is from \eqref{eq:453}. Thus
\begin{align}
&h\cdot \left(\hat{Q}^{r}(s)-\hat{Q}^{r}(t+\xi k r^{-1}) \right)\nonumber\\
&\leq -r\left(s-t-\xi k r^{-1} \right)\frac{|\tilde{\lambda}|}{2}+\left(\frac{|\tilde{\lambda}|}{2}+2|h||\alpha|\right)\xi k r^{\frac{7}{4}\kappa-1}+|h|\epsilon r^{-\frac{3}{8}}\nonumber\\
&\leq -r\left(s-t-\xi k r^{-1} \right)\frac{|\tilde{\lambda}|}{2}+\left(\frac{|\tilde{\lambda}|}{2}+2|h||\alpha|\right)\epsilon r^{\frac{7}{4}\kappa-\frac{15}{16}}+|h|\epsilon r^{-\frac{3}{8}}\nonumber\\
&\leq -r\left(s-t-\xi k r^{-1} \right)\frac{|\tilde{\lambda}|}{2}+\epsilon r^{-\frac{1}{4}} \label{eq:500}
\end{align}
due to Summary \ref{sum:costDiffResultSum} parts (d), (j), and (k), and the fact that on the set $\left(\mathcal{B}^{r,k,t}\right)^{c}$ we have\\ $\int_{t}^{t+k r^{-1}}\mathcal{I}_{\{\sum_{j=1}^{J} \mathcal{E}_{j}^{r}(r^{2}z)>0\}}dz\leq \xi k r^{\frac{7}{4}\kappa-2}.$  An almost identical argument shows that on the set $\left(\mathcal{B}^{r,k,t}\right)^{c}$ for all $r\geq R$, any $l\geq 0$, and all $s\in [\zeta^{r}_{2l+1},\zeta^{r}_{2l+2})\cap[\zeta^{r}_{2l+1},t+k r^{-1}]$ we have
\begin{equation}
h\cdot \left(\hat{Q}^{r}(s)-\hat{Q}^{r}(\zeta^{r}_{2l+1}) \right) \leq -r\left(s-\zeta^{r}_{2l+1} \right)\frac{|\tilde{\lambda}|}{2}+\epsilon r^{-\frac{1}{4}}. \label{eq:503}
\end{equation}
From Proposition \ref{thm:dTildeDiffBnd}, 
% there exists a constant $B_{\hat{h}}>0$ such that for all $q^{1},q^{2}\in\mathbb{R}^{J}_{+}$ we have
% \begin{equation*}
% \tilde{d}(q^{2})-\tilde{d}(q^{1})\leq h\cdot \left(q^{2}-q^{1}\right)+B_{\hat{h}}|KMq^{2}-KMq^{1}|_{2}.
% \end{equation*}
for $r\geq R$, on the set $\left(\mathcal{B}^{r,k,t}\right)^{c}$ for all $s\in [t+\xi k r^{-1},\zeta^{r}_{0})\cap [t+\xi k r^{-1}, t+k r^{-1}]$  we have
\begin{eqnarray*}
\tilde{d}\left( \hat{Q}^{r}(s)\right)&\leq& \tilde{d}\left( \hat{Q}^{r}(t+\xi k r^{-1})\right) +\frac{1}{|\lambda|} h\cdot \left(\hat{Q}^{r}(s)-\hat{Q}^{r}(t+\xi k r^{-1})\right)\\
&&+\frac{B_{\hat{h}}}{|\lambda|}| \hat{W}^{r}(s)-\hat{W}^{r}(t+\xi k r^{-1}) |_{2}\\
&\leq& \xi k -r\left(s-t-\xi k r^{-1} \right)\frac{|\tilde{\lambda}|}{2|\lambda|}+\frac{\epsilon}{|\lambda|} r^{-\frac{1}{4}}+\frac{B_{\hat{h}}}{|\lambda|}I^{\frac{1}{2}}\epsilon r^{-\frac{1}{4}}\\
&\leq& 2\xi k -r\left(s-t-\xi k r^{-1} \right)\frac{|\tilde{\lambda}|}{2|\lambda|}
\end{eqnarray*}
where we have used \eqref{eq:500}, \eqref{eq:workloadBndComp},
Summary \ref{sum:costDiffResultSum} part (l), and the fact that on the set $\left(\mathcal{B}^{r,k,t}\right)^{c}$ we have  $\tilde{d}\left( \hat{Q}^{r}(t+\xi k r^{-1})\right) \leq \xi k$.   This implies that on the set $\left(\mathcal{B}^{r,k,t}\right)^{c}$ for $r\geq R$  if $\zeta^{r}_{0}-t-\xi k r^{-1}>\frac{4\xi k|\lambda| }{|\tilde{\lambda}|}r^{-1}$ we have
\begin{equation*}
\tilde{d}\left( \hat{Q}^{r}\left( t+\xi k r^{-1}+\frac{4\xi k|\lambda|}{|\tilde{\lambda}|}r^{-1}\right)\right)\leq2\xi k-r\left(\frac{4\xi k|\lambda|}{|\tilde{\lambda}|}r^{-1} \right)\frac{|\tilde{\lambda}|}{2|\lambda|}\leq 0
\end{equation*}
which contradicts the definition of $\zeta^{r}_{0}$.  Consequently, on the set $\left(\mathcal{B}^{r,k,t}\right)^{c}$ for $r\geq R$ we have 
\begin{equation}
	\zeta^{r}_{0}\leq t+\xi k r^{-1}+\frac{4\xi k|\lambda|}{|\tilde{\lambda}|}r^{-1}. \label{eq:516}
\end{equation}
Similarly, for $r\geq R$ on the set $\left(\mathcal{B}^{r,k,t}\right)^{c}$  for all $s\in [\zeta^{r}_{2l+1},\zeta^{r}_{2l+2})\cap [\zeta^{r}_{2l+1}, t+kr^{-1}]$ for any $l\geq 0$ we have 
\begin{eqnarray*}
\tilde{d}\left( \hat{Q}^{r}(s)\right)&\leq& \tilde{d}\left( \hat{Q}^{r}(\zeta^{r}_{2l+1})\right) + \frac{1}{|\lambda|}h\cdot \left(\hat{Q}^{r}(s)-\hat{Q}^{r}(\zeta^{r}_{2l+1})\right)+\frac{B_{\hat{h}}}{|\lambda|}| \hat{W}^{r}(s)-\hat{W}^{r}(\zeta^{r}_{2l+1}) |_{2}\\
&\leq& \tilde{d}\left( \hat{Q}^{r}(\zeta^{r}_{2l+1})\right) +\frac{1}{|\lambda|} h\cdot \left(\hat{Q}^{r}(s)-\hat{Q}^{r}(\zeta^{r}_{2l+1})\right)\\
&&+\frac{B_{\hat{h}}}{|\lambda|}| \hat{W}^{r}(s)-\hat{W}^{r}(t+\xi k r^{-1}) |_{2}+\frac{B_{\hat{h}}}{|\lambda|}| \hat{W}^{r}(\zeta^{r}_{2l+1}) -\hat{W}^{r}(t+\xi k r^{-1}) |_{2}\\
&\leq&  B_{\tilde{d}} J^{\frac{1}{2}}\tilde{c}_{2} r^{\kappa-1}+J^{\frac{1}{2}}B_{\tilde{d}}r^{-1} -r\left(s-\zeta^{r}_{2l+1} \right)\frac{|\tilde{\lambda}|}{2|\lambda|}+\frac{\epsilon}{|\lambda|} r^{-\frac{1}{4}}+2B_{\hat{h}}I^{\frac{1}{2}}\frac{\epsilon}{|\lambda|} r^{-\frac{1}{4}}\\
&\leq& \epsilon r^{-\frac{1}{8}} -r\left(s-\zeta^{r,\epsilon}_{2l+1} \right)\frac{|\tilde{\lambda}|}{2|\lambda|}\\
\end{eqnarray*}
where we have used \eqref{eq:workloadBndComp}, \eqref{eq:503}, 
Summary \ref{sum:costDiffResultSum} part (m), and the fact that by the definition of $\zeta^{r}_{2l+1}$ and Proposition \ref{thm:costDiffBndQ} we have $\tilde{d}\left( \hat{Q}^{r}(\zeta^{r}_{2l+1})\right) \leq B_{\tilde{d}} J^{\frac{1}{2}}\tilde{c}_{2} r^{\kappa-1}+J^{\frac{1}{2}}B_{\tilde{d}}r^{-1}$.
This, combined with the fact that for $r\geq R$, any $l\geq 0$, and all $s\in [\zeta^{r}_{2l},\zeta^{r}_{2l+1})$ we have $\tilde{d}\left( \hat{Q}^{r}(s)\right)\leq B_{\tilde{d}} J^{\frac{1}{2}}\tilde{c}_{2} r^{\kappa-1}\leq  \epsilon r^{-\frac{1}{8}}$ due to the definition of $\zeta^{r}_{2l+1}$ and Summary \ref{sum:costDiffResultSum} part (m), implies that on the set $\left(\mathcal{B}^{r,k,t}\right)^{c}$ for $r\geq R$ and all $s\in [\zeta^{r}_{0},t+kr^{-1}]$ we have
$
\tilde{d}\left( \hat{Q}^{r}(s)\right) \leq \epsilon r^{-\frac{1}{8}}$.
Therefore, from \eqref{eq:516}, on the set $\left(\mathcal{B}^{r,k,t}\right)^{c}$ for $r\geq R$ we have
\begin{equation}\label{eq:526}
\tilde{d}\left( \hat{Q}^{r}(s)\right) \leq  \epsilon r^{-\frac{1}{8}}
\mbox{ for all } s\in[ t+(\xi k+4\xi k|\lambda||\tilde{\lambda}|^{-1})r^{-1},t+k r^{-1}].
\end{equation}
Note that for $r\geq R$ we have
\begin{align}
E_{y^r} \int_{t}^{t+k r^{-1}}\tilde{d}\left( \hat{Q}^{r}(s)\right)ds &\leq E_{y^r} \int_{t}^{ t+\left(\xi k+\frac{4\xi k|\lambda|}{|\tilde{\lambda}|}\right)r^{-1}}\tilde{d}\left( \hat{Q}^{r}(s)\right)ds \nonumber \\
& +E_{y^r} \int_{ t+\left(\xi k+\frac{4\xi k|\lambda|}{|\tilde{\lambda}|}\right)r^{-1}}^{t+k r^{-1}}\tilde{d}\left( \hat{Q}^{r}(s)\right)ds. \label{eq:527}
\end{align}
 and
\begin{align}
E_{y^r}\int_{t}^{ t+\left(\xi k+\frac{4\xi k|\lambda|}{|\tilde{\lambda}|}\right)r^{-1}}\tilde{d}\left( \hat{Q}^{r}(s)\right)ds &\leq \int_{t}^{ t+\left(\xi k+\frac{4\xi k|\lambda|}{|\tilde{\lambda}|}\right)r^{-1}}B_{2}E_{y^r}\left[e^{c |\hat{W}^{r}(s)|}\right]ds\\
&\leq r^{-1}\left(\xi k+\frac{4\xi k|\lambda|}{|\tilde{\lambda}|}\right)B_{2}B_{1}
\leq  r^{-1}k\epsilon \label{eq:529}
\end{align}
due to our choice of $\xi$ and Summary \ref{sum:costDiffResultSum} parts (a) and  (b). 

In addition, due to Summary \ref{sum:costDiffResultSum} parts (a), (b), (c), and our choice of $k$, for all $r\geq R$, $t\geq 0$, and $s\geq 0$ we have 
\begin{eqnarray*}
E_{y^r}\left[\mathcal{I}_{\mathcal{B}^{r,k,t}}\tilde{d}\left(\hat{Q}^{r}(s)\right) \right]&\leq& E_{y^r}\left[\mathcal{I}_{\mathcal{B}^{r,k,t}}B_{2}e^{\frac{c}{2}|\hat{W}^{r}(s)|} \right]\\
&\leq& P_{y^r}(\mathcal{B}^{r,k,t})^{\frac{1}{2}}B_{2}E_{y^r}\left[e^{c|\hat{W}^{r}(s)|} \right]^{\frac{1}{2}}\leq  B^{\frac{1}{2}}_{3}e^{-k\frac{B_{4}}{2}}B_{2}B^{\frac{1}{2}}_{1}
\leq \epsilon.
\end{eqnarray*}
So for $r\geq R$ we have
\begin{align}
&E_{y^r}\left[\int_{ t+\left(\xi k+\frac{4\xi k|\lambda|}{|\tilde{\lambda}|}\right)r^{-1}}^{t+k r^{-1}}\tilde{d}\left( \hat{Q}^{r}(s)\right)ds\right] \nonumber \\
&\leq  E_{y^r}\left[\mathcal{I}_{\left(\mathcal{B}^{r,k,t}\right)^{c}}\int_{ t+\left(\xi k+\frac{4\xi k|\lambda|}{|\tilde{\lambda}|}\right)r^{-1}}^{t+k r^{-1}}\tilde{d}\left( \hat{Q}^{r}(s)\right)ds\right]\nonumber\\
&\quad + E_{y^r}\left[\mathcal{I}_{\mathcal{B}^{r,k,t}}\int_{ t+\left(\xi k+\frac{4\xi k|\lambda|}{|\tilde{\lambda}|}\right)r^{-1}}^{t+k r^{-1}}\tilde{d}\left( \hat{Q}^{r}(s)\right)ds\right]\nonumber\\
% &\leq  r^{-1}\left(( 1-\xi)k-\frac{4\xi k|\lambda|}{|\tilde{\lambda}|}\right)\epsilon r^{-\frac{1}{8}}+\int_{ t+\left(\xi k+\frac{4\xi k|\lambda|}{|\tilde{\lambda}|}\right)r^{-1}}^{t+k r^{-1}} E_{y^r}\left[\mathcal{I}_{\mathcal{B}^{r,k,t}}\tilde{d}\left( \hat{Q}^{r}(s)\right)\right]ds\nonumber\\
&\leq r^{-1}\left(( 1-\xi)k-\frac{4\xi k|\lambda|}{|\tilde{\lambda}|}\right)\epsilon r^{-\frac{1}{8}}+ r^{-1}\left(( 1-\xi)k-\frac{4\xi k|\lambda|}{|\tilde{\lambda}|}\right)\epsilon
\leq r^{-1}k2\epsilon. \label{eq:533}
\end{align}
Combining  \eqref{eq:527}, \eqref{eq:529}, and \eqref{eq:533}, we have
 that for $r\geq R$ and any $t\geq 0$
\begin{equation*}
E_{y^r}\left[\int_{t}^{t+k r^{-1}}\tilde{d}\left( \hat{Q}^{r}(s)\right)ds\right]\leq r^{-1}3 k \epsilon.
\end{equation*}

We can now complete the proof. Consider first the discounted case.
For $r\geq R$ we have
\begin{multline*}
E_{y^r}\left[\int_{0}^{\infty}e^{-s\vsg}\tilde{d}\left( \hat{Q}^{r}(s)\right) ds\right]
 \leq\sum_{l=0}^{\infty}e^{-l\vsg\left( k r^{-1}\right)}E_{y^r}\left[\int_{lk r^{-1}}^{(l+1)k r^{-1}}\tilde{d}\left( \hat{Q}^{r}(s)\right)ds\right]\\
 \leq \sum_{l=0}^{\infty}e^{-l\vsg\left(k r^{-1}\right)}r^{-1}3k\epsilon 
  \leq \frac{1}{1-e^{-\vsg\left(k r^{-1}\right)}}r^{-1}3k \epsilon  \\
  \leq \frac{1}{e^{-\vsg\left( k r^{-1}\right)}\vsg\left( k r^{-1}\right)}3r^{-1}k\epsilon\leq \frac{3\epsilon e^{\vsg\left(k r^{-1}\right)}}{\vsg} 
  \leq\frac{3\epsilon e^{\vsg\epsilon r^{-\frac{7}{8}}}}{\vsg} 
  \leq \frac{3\epsilon e^{\vsg\epsilon}}{\vsg} 
\end{multline*}
due to Summary \ref{sum:costDiffResultSum} part (d). Taking $\epsilon = \frac{\vsg e^{-\vsg}}{3}\epsilon_0 \wedge 1$  proves the first statement in the proposition. 

Consider now the ergodic cost. Let $T^{*}=kR^{-1}$.  For any $t\geq 0$, $T\geq T^{*}$, and $r\geq R$ we have
\begin{multline*}
E_{y^r}\left[\frac{1}{T}\int_{t}^{t+T}\tilde{d}\left( \hat{Q}^{r}(s)\right) ds \right]  \leq  \frac{1}{T}\sum_{l=1}^{\lceil \frac{T}{kr^{-1}} \rceil }E_{y^r}\left[\int_{t+(l-1)kr^{-1}}^{t+l kr^{-1}}\tilde{d}\left( \hat{Q}^{r}(s)\right)ds\right]\\
\leq \frac{1}{T}\sum_{l=1}^{\lceil \frac{T}{kr^{-1}} \rceil }r^{-1}3k \epsilon 
\leq \frac{1}{T}\left( \frac{T}{kr^{-1}}+1 \right) r^{-1}3k \epsilon 
\leq 3\epsilon+\frac{3kR^{-1}\epsilon}{T^{*} }
\leq 6\epsilon.
\end{multline*}
Taking $\epsilon= \epsilon_0/6$ completes the proof of the second statement in the proposition. \hfill \qed

\section{Proofs of Some Stability Results}
\label{sec:stabproofs}
In this section we prove the key stability results, namely Proposition \ref{thm:waitTimeExpBnd}, Lemma \ref{thm:VMarkovLong}, and Proposition \ref{thm:TildeGammaLower}, that were used in the proof of Proposition \ref{thm:workloadExpBnd} in Section \ref{sec:secworkloadexp}.

\subsection{Proof of Proposition \ref{thm:waitTimeExpBnd}}
\label{sec:thmwaittimeex}

We begin with the following  auxiliary result which is proved in Section \ref{sec:proofLargeWaitBnd}. 
% \begin{proposition}
% \label{thm:expTailBndBasic}
% Let $j\in \AAA_J$, $c>0$, and $\epsilon>0$ be arbitrary.  Then there exists $B_{1},B_{2},R\in (0,\infty)$ such that for all $S\in [0,\infty)$ and $r\geq R$ we have
% \begin{equation*}
% P\left(\sup_{0\leq t \leq r^{2c}S}|A_{j}^{r}(t)-t\alpha^{r}_{j}|\geq \epsilon r^{c}S\right)\leq B_{1}e^{-S B_{2}},
% \end{equation*}
% \begin{equation*}
% P\left(\sup_{0\leq t \leq r^{2c}\max_{i}\{C_{i}\}S}|S_{j}^{r}(t)-t\beta^{r}_{j}|\geq \epsilon r^{c}S\right)\leq B_{1}e^{-S B_{2}},
% \end{equation*}
% \begin{equation*}
% P\left(\sup_{0\leq t \leq S}|A_{j}^{r}(t)-t\alpha^{r}_{j}|\geq \epsilon S\right)\leq B_{1}e^{-S B_{2}},
% \end{equation*}
% and
% \begin{equation*}
% P\left(\sup_{0\leq t \leq \max_{i}\{C_{i}\}S}|S_{j}^{r}(t)-t\beta^{r}_{j}|\geq \epsilon S\right)\leq B_{1}e^{-S B_{2}}
% \end{equation*}
% \end{proposition}
\begin{proposition}
\label{thm:largeWaitBnd}
For any $c>0$ and $\epsilon>0$ there exist constants $B,R\in(0,\infty)$ such that for all $T\geq 1$, $j\in\AAA_J$, and $r\geq R$ we have
\begin{equation}
P\left(\sum_{l=1}^{\lceil r^{2}T \rceil}\mathcal{I}_{\{v^{r}_{j}(l)\geq r c\}}v^{r}_{j}(l) \geq \epsilon T \right) \leq e^{-BT} \label{eq:eq706a}
\end{equation}
and
\begin{equation}
P\left(\sum_{l=1}^{\lceil r^{2}T \rceil}\mathcal{I}_{\{u^{r}_{j}(l)\geq r c\}}u^{r}_{j}(l) \geq \epsilon T \right) \leq e^{-BT}. \label{eq:eq706b}
\end{equation}
\end{proposition}
We now complete the proof of Proposition \ref{thm:waitTimeExpBnd}.\\

\noindent 
{\bf Proof of Proposition \ref{thm:waitTimeExpBnd}.} 
Recall $\rho^*$ introduced below Proposition \ref{thm:fullCapVect} and $\theta$ from Condition \ref{cond:htc}.
Fix $i\in\AAA_I$ and $y^r = (\hat{q}^r,\hat{\Upsilon}^r,\tilde{\mathcal{E}}^r)\in \cly^r$.  Let $\tilde{\upsilon}>0$, $\xi \ge 0$ be given. 
Let $\epsilon_{1}=\frac{\frac{|\theta_{i}|}{12}}{JC_{i}\left(1+\frac{4}{\rho^*}\right)}$, $\epsilon_{2}=\frac{|\theta_{i}|}{24JC_{i} }$, and $\epsilon_{3}=\frac{|\theta_{i}|\min_{j}\{\beta_{j}\}}{96J}$.  Define the sets
\begin{eqnarray*}
\mathcal{A}^{r,T}_{j}&=&\left\{\sup_{0\leq t \leq r^{2}T}|A_{j}^{r}(t)-t\alpha^{r}_{j}|\leq \epsilon_{3} rT \right\}\cap\left\{\sup_{0\leq t \leq r^{2}\max_{i}\{C_{i}\}T}|S_{j}^{r}(t)-t\beta^{r}_{j}|\leq \epsilon_{3} rT \right\}\\
&&\cap\left\{ \int_{0}^{r^{2}T}\mathcal{I}_{\{ \mathcal{E}_{j}^{r}(s)=1\}}ds\leq r\hat{\Upsilon}^{A,r}_{j}+\epsilon_{2} T r \right\}\\
&&\cap\left\{ \sum_{l=1}^{\lceil 2r^{2}\max_{i}\{C_{i}\}\beta^{r}_{j}T \rceil}\mathcal{I}_{\{v^{r}_{j}(l)> r \frac{\tilde{\upsilon}}{2J+1}\}}v^{r}_{j}(l) \leq r \epsilon_{1} T\right\} \\
&&\cap\left\{ \sum_{l=1}^{\lceil 2r^{2}\alpha^{r}_{j}T \rceil}\mathcal{I}_{\{u^{r}_{j}(l)> r \frac{\tilde{\upsilon}}{2J+1}\}}u^{r}_{j}(l) \leq r \epsilon_{1} T\right\}\\
\end{eqnarray*}
and
$
\mathcal{A}^{r,T}=\cap_{j=1}^{J}\mathcal{A}^{r,T}_{j}$.
Due to Propositions \ref{thm:expTailBnd} (equations \eqref{eq:31.d} and \eqref{eq:31.c} with $c_1=1$ and $c_2=0$), \ref{thm:largeWaitBnd} , and \ref{thm:idleTimeBndMark} (and since $\kappa <1/4$) we know there exist constants $B_{1},B_{2},R\in (0,\infty)$ such that for all $r\geq R$, $y^r \in \cly^r$, and $T\geq 1$ we have
\begin{equation}
P_{y^r}\left( \left( \mathcal{A}^{r,T}\right)^{c}\right)\leq B_{1}e^{-B_{2}T},\;\;
\epsilon_{3} T\leq r\max_{i}\{C_{i}\}\beta^{r}_{j}T,\;\; \mbox{ and }
\epsilon_{3} T\leq r\alpha^{r}_{j}T. \label{eq:606}
\end{equation}
In addition, we assume $R$ is sufficiently large that for all $r\geq R$ we have $\frac{3\theta}{4}\geq r(K\rho^{r}-C)$, $\tilde{\upsilon}>Jr^{-1}$, $\frac{\tilde{\upsilon}}{2J+1}\geq Jr^{\kappa-1} c_{2}$, $2\min_{j}\{\beta^{r}_{j}\}\geq \min_{j}\{\beta_{j}\}$,and for all $j\in\AAA_J$ and $t\geq 0$ we have $x^{r}_{j}(t)\geq \frac{\rho^{*}}{4}$ (see Proposition \ref{thm:schemeSum} part (c)).  For the rest of the proof we will restrict ourselves to values of $r$ satisfying $r\geq R$.  Note that if $\hat{W}^{r}_{i}(s)\geq \frac{\tilde{\upsilon}}{2J+1}\geq Jc_{2}r^{\kappa-1}$ and $\sum_{j=1}^{J}\mathcal{I}_{\{\tilde{\mathcal{E}}^{r}_{j}(s)=1\}}=0$ then due to Proposition \ref{thm:schemeSum} part (d) we have
$
(K\yr)_{i}(r^{2}s)=C_{i}$.
This implies that if 
$
\sum_{j=1}^{J}\mathcal{I}_{\{\tilde{\mathcal{E}}^{r}_{j}(s)=1\}}=0$,
\begin{equation*}
\sum_{j=1}^{J}\mathcal{I}_{\{\hat{\Upsilon}^{S,r}_{j}(s)> \frac{\tilde{\upsilon}}{2J+1}\}}+\sum_{j=1}^{J}\mathcal{I}_{\{\hat{\Upsilon}^{A,r}_{j}(s)> \frac{\tilde{\upsilon}}{2J+1}\}}=0, \mbox{ and }
\hat{W}^{r}_{i}(s)+\sum_{j=1}^{J}\hat{\Upsilon}^{S,r}_{j}(s)+\sum_{j=1}^{J}\hat{\Upsilon}^{A,r}_{j}(s)\geq \tilde{\upsilon}
 \end{equation*}
  then
$(K\yr)_{i}(r^{2}s)=C_{i}$.
Define 
\begin{equation*}
\hat{t}\doteq\sup\{t\in [0,\xi]: \hat{W}^{r}_{i}(t)+\sum_{j=1}^{J}\hat{\Upsilon}^{S,r}_{j}(t)+\sum_{j=1}^{J}\hat{\Upsilon}^{A,r}_{j}(t)<\tilde{\upsilon}\}
\end{equation*}
with the convention that if 
\begin{equation*}
\hat{W}^{r}_{i}(t)+\sum_{j=1}^{J}\hat{\Upsilon}^{S,r}_{j}(t)+\sum_{j=1}^{J}\hat{\Upsilon}^{A,r}_{j}(s)\geq \tilde{\upsilon}
\mbox{ for all } t\in [0,\xi] \mbox{ then }
\hat{t}\doteq 0.
\end{equation*}
Define
\begin{equation*}
T^{*}\doteq 1+2\xi+\frac{4}{|\theta_{i}|}\left(\hat{w}^{r}_{i}+3\tilde{\upsilon}+ \left( C_{i}+1\right)\sum_{j=1}^{J}\hat{\Upsilon}^{S,r}_{j}+3 C_{i}\sum_{j=1}^{J}\hat{\Upsilon}^{A,r}_{j}\right) 
\end{equation*}
and let $T>T^{*}$.  Then on the set $\left\{\tilde{\gamma}^{r,\tilde{\upsilon}}_{i,\xi}> T\right\}\cap\mathcal{A}^{r,T}$,   for all $s\in [\hat{t},T]$, we have 
\begin{equation*}
\hat{W}^{r}_{i}(s)+\sum_{j=1}^{J}\hat{\Upsilon}^{S,r}_{j}(s)+\sum_{j=1}^{J}\hat{\Upsilon}^{A,r}_{j}(s)\geq \tilde{\upsilon}
\end{equation*}
so for any $s\in [\hat{t},T]$ satisfying
$
\sum_{j=1}^{J}\mathcal{I}_{\{\mathcal{E}^{r}_{j}(s)=1\}}=0,
$
and
\begin{equation*}
\sum_{j=1}^{J}\mathcal{I}_{\{\hat{\Upsilon}^{S,r}_{j}(s)> \frac{\tilde{\upsilon}}{2J+1}\}}+\sum_{j=1}^{J}\mathcal{I}_{\{\hat{\Upsilon}^{A,r}_{j}(s)> \frac{\tilde{\upsilon}}{2J+1}\}}=0,
\end{equation*}
we have
$
(K\yr)_{i}(s)=C_{i}$.
Consequently, on the set $\left\{\tilde{\gamma}^{r,\tilde{\upsilon}}_{i,\xi}> T\right\}\cap\mathcal{A}^{r,T}$,
\begin{align}
&\sum_{j=1}^{J}K_{i,j}\left(\bar{B}^{r}_{j}(T)-\bar{B}^{r}_{j}(\hat{t}) \right) = \int_{ \hat{t}}^{T}\sum_{j=1}^{J}K_{i,j}\yr_{j}(r^{2}s)ds \nonumber\\
&\geq C_{i}(T- \hat{t})- C_{i}\sum_{j=1}^{J}\int_{\hat{t}}^{T}\mathcal{I}_{\{\tilde{\mathcal{E}}^{r}_{j}(s)=1\}}ds\nonumber \\
&\quad - C_{i}\sum_{j=1}^{J}\int_{\hat{t}}^{T}\mathcal{I}_{\{\hat{\Upsilon}^{S,r}_{j}(s)> \frac{\tilde{\upsilon}}{2J+1}\}}ds- C_{i}\sum_{j=1}^{J}\int_{\hat{t}}^{T}\mathcal{I}_{\{\hat{\Upsilon}^{A,r}_{j}(s)> \frac{\tilde{\upsilon}}{2J+1}\}}ds\nonumber \\
&\geq C_{i}(T-\hat{t})- C_{i}\sum_{j=1}^{J}\int_{0}^{T}\mathcal{I}_{\{\tilde{\mathcal{E}}^{r}_{j}(s)=1\}}ds\nonumber \\
&\quad - C_{i}\sum_{j=1}^{J}\int_{0}^{T}\mathcal{I}_{\{\hat{\Upsilon}^{S,r}_{j}(s)> \frac{\tilde{\upsilon}}{2J+1}\}}ds- C_{i}\sum_{j=1}^{J}\int_{0}^{T}\mathcal{I}_{\{\hat{\Upsilon}^{A,r}_{j}(s)> \frac{\tilde{\upsilon}}{2J+1}\}}ds. \label{eq:609}
\end{align}
For all $j\in\AAA_J$ we have 
\begin{align*}
&\int_{0}^{T}\mathcal{I}_{\left\{\hat{\Upsilon}^{S,r}_{j}(s)> \frac{\tilde\upsilon}{2J+1}r\right\}}ds\\
&\leq r^{-1}\hat{\Upsilon}^{S,r}_{j}+\sum_{l=1}^{\tau^{r,S}_{j}(T)}\mathcal{I}_{\{ v^{r}_{j}(l)> \frac{\tilde{\upsilon}}{2J+1}r \}}\frac{1}{r^{2}}\int_{0}^{r^{2}T}\mathcal{I}_{\left\{ \sum_{k=1}^{l-1}v^{r}_{j}(k)\leq B^{r}_{j}(s) \leq \sum_{k=1}^{l}v^{r}_{j}(k) \right\}}ds.
\end{align*}
Recall that due to Proposition \ref{thm:schemeSum} part (c) and our assumption on the size of $r$ we have $\yr_{j}(s)\geq \frac{\rho^*}{4}$ unless $\mathcal{E}^{r}_{j}(s)=1$ so for all $l\in\{1,...,\tau^{r,S}_{j}(T)\}$ we have
\begin{eqnarray*}
&&\int_{0}^{r^{2}T}\mathcal{I}_{\left\{ \sum_{k=1}^{l-1}v^{r}_{j}(k)\leq B^{r}_{j}(s) \leq \sum_{k=1}^{l}v^{r}_{j}(k) \right\}}ds\\
 &\leq &\int_{0}^{r^{2}T}\mathcal{I}_{\{\mathcal{E}^{r}_{j}(s)=0\}}\mathcal{I}_{\left\{ \sum_{k=1}^{l-1}v^{r}_{j}(k)\leq B^{r}_{j}(s) \leq \sum_{k=1}^{l}v^{r}_{j}(k) \right\}}ds\\
&&+\int_{0}^{r^{2}T}\mathcal{I}_{\{\mathcal{E}^{r}_{j}(s)=1\}}\mathcal{I}_{\left\{ \sum_{k=1}^{l-1}v^{r}_{j}(k)\leq B^{r}_{j}(s) \leq \sum_{k=1}^{l}v^{r}_{j}(k) \right\}}ds\\
&\leq& \frac{4}{\rho^*}v^{r}_{j}(l)+\int_{0}^{r^{2}T}\mathcal{I}_{\{\mathcal{E}^{r}_{j}(s)=1\}}\mathcal{I}_{\left\{ \sum_{k=1}^{l-1}v^{r}_{j}(k)\leq B^{r}_{j}(s) \leq \sum_{k=1}^{l}v^{r}_{j}(k) \right\}}ds.\\
\end{eqnarray*}
Combining these last two inequalities gives
\begin{align*}
&\int_{0}^{T}\mathcal{I}_{\left\{\hat{\Upsilon}^{S,r}_{j}(s)> \frac{\tilde{\upsilon}}{2J+1}r\right\}}ds
% \leq r^{-1}\hat{\Upsilon}^{S}_{j}+\sum_{l=1}^{\tau^{r,S}_{j}(T)}\mathcal{I}_{\{ v^{r}_{j}(l)> \frac{\tilde{\upsilon}}{2J+1}r \}}\frac{1}{r^{2}}\int_{0}^{r^{2}T}\mathcal{I}_{\left\{ \sum_{k=1}^{l-1}v^{r}_{j}(k)\leq B^{r}_{j}(s) \leq \sum_{k=1}^{l}v^{r}_{j}(k) \right\}}ds\\
\leq r^{-1}\hat{\Upsilon}^{S,r}_{j}+\sum_{l=1}^{\tau^{r,S}_{j}(T)}\mathcal{I}_{\{ v^{r}_{j}(l)> \frac{\tilde{\upsilon}}{2J+1}r \}}\frac{1}{r^{2}}\frac{4}{\rho^*}v^{r}_{j}(l)\\
&\quad+\sum_{l=1}^{\tau^{r,S}_{j}(T)}\mathcal{I}_{\{ v^{r}_{j}(l)> \frac{\tilde{\upsilon}}{2J+1}r \}}\frac{1}{r^{2}}\int_{0}^{r^{2}T}\mathcal{I}_{\{\mathcal{E}^{r}_{j}(s)=1\}}\mathcal{I}_{\left\{ \sum_{k=1}^{l-1}v^{r}_{j}(k)\leq B^{r}_{j}(s) \leq \sum_{k=1}^{l}v^{r}_{j}(k) \right\}}ds\\
&\leq r^{-1}\hat{\Upsilon}^{S,r}_{j}+\frac{4}{\rho^*}\frac{1}{r^{2}}\sum_{l=1}^{\tau^{r,S}_{j}(T)}\mathcal{I}_{\{ v^{r}_{j}(l)> \frac{\tilde{\upsilon}}{2J+1}r \}}v^{r}_{j}(l)+\int_{0}^{T}\mathcal{I}_{\{\tilde{\mathcal{E}}^{r}_{j}(s)=1\}}ds.
\end{align*}
Note that on the set $\mathcal{A}^{r,T}$ we have $\tau^{r,S}_{j}(T)\leq \lceil 2 r^{2}\max_{i}\{C_{i}\}\beta^{r}_{j}T \rceil$ and
\begin{equation*}
\frac{1}{r^{2}}\sum_{l=1}^{\lceil 2 r^{2}\max_{i}\{C_{i}\}\beta^{r}_{j}T \rceil}\mathcal{I}_{\{ v^{r}_{j}(l)> \frac{\tilde{\upsilon}}{2J+1}r \}}v^{r}_{j}(l)\leq r^{-1}\epsilon_{1} T
\end{equation*}
which gives
\begin{eqnarray*}
\int_{0}^{T}\mathcal{I}_{\left\{\hat{\Upsilon}^{S,r}_{j}(s)\geq \frac{\tilde{\upsilon}}{2J+1}r\right\}}ds=r^{-1}\hat{\Upsilon}^{S,r}_{j}+r^{-1}\frac{4\epsilon_{1}}{\rho^*}T+\int_{0}^{T}\mathcal{I}_{\{\tilde{\mathcal{E}}^{r}_{j}(s)=1\}}ds.
\end{eqnarray*}
Also, on the set $\mathcal{A}^{r,T}$, we have $\tau^{r,A}_{j}(T)\leq \lceil 2 r^{2}\alpha^{r}_{j}T \rceil$ and
\begin{equation*}
\frac{1}{r^{2}}\sum_{l=1}^{\lceil 2 r^{2}\alpha^{r}_{j}T \rceil}\mathcal{I}_{\{ u^{r}_{j}(l)> \frac{\tilde{\upsilon}}{2J+1}r \}}u^{r}_{j}(l)\leq r^{-1}\epsilon_{1} T
\end{equation*}
so
\begin{align*}
&\int_{0}^{T}\mathcal{I}_{\left\{\hat{\Upsilon}^{A,r}_{j}(s)> \frac{\tilde{\upsilon}}{2J+1}r\right\}}ds\\
&\leq r^{-1}\hat{\Upsilon}^{A,r}_{j}+\sum_{l=1}^{\tau^{r,A}_{j}(T)}\mathcal{I}_{\{ u^{r}_{j}(l)> \frac{\tilde{\upsilon}}{2J+1}r \}}\frac{1}{r^{2}}\int_{0}^{r^{2}T}\mathcal{I}_{\left\{ \sum_{k=1}^{l-1}u^{r}_{j}(k)\leq s \leq \sum_{k=1}^{l}u^{r}_{j}(k) \right\}}ds\\
&\leq r^{-1}\hat{\Upsilon}^{A,r}_{j}+\frac{1}{r^{2}}\sum_{l=1}^{\tau^{r,A}_{j}(T)}\mathcal{I}_{\{ u^{r}_{j}(l)> \frac{\tilde{\upsilon}}{2J+1}r \}}u^{r}_{j}(l)
\leq r^{-1}\hat{\Upsilon}^{A,r}_{j}+r^{-1}\epsilon_{1} T.
\end{align*}
Thus from \eqref{eq:609}, on $\left\{\tilde{\gamma}^{r,\tilde{\upsilon}}_{i,\xi}> T\right\}\cap\mathcal{A}^{r,T}$,
\begin{align}
&\sum_{j=1}^{J}K_{i,j}\left(\bar{B}^{r}_{j}(T)-\bar{B}^{r}_{j}( \hat{t}) \right)\nonumber\\
% &\geq&C_{i}(T- \hat{t})- C_{i}\sum_{j=1}^{J}\int_{0}^{T}\mathcal{I}_{\{\mathcal{E}^{r}_{j}(r^{2}s)=1\}}ds- C_{i}\sum_{j=1}^{J}\int_{0}^{T}\mathcal{I}_{\{\hat{\Upsilon}^{S,r}_{j}(s)> \frac{\tilde{\upsilon}}{2J+1}\}}ds- C_{i}\sum_{j=1}^{J}\int_{0}^{T}\mathcal{I}_{\{\hat{\Upsilon}^{A,r}_{j}(s)> \frac{\tilde{\upsilon}}{2J+1}\}}ds\\
&\geq C_{i}(T- \hat{t})- 2C_{i}\sum_{j=1}^{J}\int_{0}^{T}\mathcal{I}_{\{\mathcal{E}^{r}_{j}(r^{2}s)=1\}}ds-r^{-1}JC_{i}\epsilon_{1} T\left(1+\frac{4}{\rho^*}\right) \nonumber\\
&\quad - r^{-1}C_{i}\sum_{j=1}^{J}\hat{\Upsilon}^{S,r}_{j}- r^{-1}C_{i}\sum_{j=1}^{J}\hat{\Upsilon}^{A,r}_{j}\nonumber \\
&\geq C_{i}(T- \hat{t})- r^{-1}C_{i}J2\epsilon_{2} T-r^{-1}JC_{i}\epsilon_{1} T\left(1+\frac{4}{\rho^*}\right)\nonumber \\
&\quad - r^{-1}C_{i}\sum_{j=1}^{J}\hat{\Upsilon}^{S,r}_{j}-3 r^{-1}C_{i}\sum_{j=1}^{J}\hat{\Upsilon}^{A,r}_{j}. \label{eq:615}
\end{align}
Consequently on$\left\{\tilde{\gamma}^{r,\tilde{\upsilon}}_{i,\xi}> T\right\}\cap\mathcal{A}^{r,T}$ we have
\begin{align*}
&\hat{W}^{r}_{i}(T)\\
&\leq\hat{W}^{r}_{i}(\hat{t})+Jr^{-1}+\sum_{j=1}^{J}K_{i,j}\frac{1}{\beta^{r}_{j}}\left(\hat{A}^{r}_{j}\left(\left(T-\bar{\Upsilon}^{A,r}_{j}\right)^{+}\right)-\hat{A}^{r}_{j}\left(\left( \hat{t}-\bar{\Upsilon}^{A,r}_{j}\right)^{+}\right) \right)\\
& \quad -\sum_{j=1}^{J}K_{i,j}\frac{1}{\beta^{r}_{j}}\left(\hat{S}^{r}_{j}\left(\left(\bar{B}^{r}_{j}(T)-\bar{\Upsilon}^{S,r}_{j}\right)^{+}\right)-\hat{S}^{r}_{j}\left(\left(\bar{B}^{r}_{j}( \hat{t})-\bar{\Upsilon}^{S,r}_{j}\right)^{+}\right)\right)+\sum_{j=1}^{J}K_{i,j}\hat{\Upsilon}^{S,r}\\
&\quad+(T-\hat{t})r((K\rho^{r})_{i}-C_{i}) +r\left(C_{i}(T-\hat{t})-\sum_{j=1}^{J}K_{i,j}(\bar{B}^{r}(T)-\bar{B}^{r}(\hat{t}))\right)\\
&\leq \hat{W}^{r}_{i}(\hat{t})+\tilde{\upsilon}+\frac{4J}{\min_{j}\{\beta^{r}_{j}\}}\epsilon_{3}T+\frac{3\theta_i}{4}(T-\hat{t})+ C_{i}J2\epsilon_{2} T\\
&\quad+JC_{i}\epsilon_{1} T\left(1+\frac{4}{\rho^*}\right)+(C_{i}+1)\sum_{j=1}^{J}\hat{\Upsilon}^{S,r}_{j}+3 C_{i}\sum_{j=1}^{J}\hat{\Upsilon}^{A,r}_{j}.\\
\end{align*}
Note that $\hat{W}^{r}_{i}(\hat{t})\leq \max\{\hat{w}^r_i,\tilde{\upsilon}+r^{-1}J\}\leq \hat{w}^r_i +2\tilde{\upsilon}$ which, combined with our choices of $\epsilon_i$, gives
\begin{eqnarray*}
\hat{W}^{r}_{i}(T)
% &\leq& \max\{\hat{w}^r_i,2\tilde{\upsilon}\}+\tilde{\upsilon}+\frac{8J}{\min_{j}\{\beta_{j}\}}\epsilon_{3}T+\frac{3\theta}{4}(T-\hat{t})+ C_{i}J2\epsilon_{2} T\\
% &&+JC_{i}\epsilon_{1} T\left(1+\frac{4}{\rho^*}\right)+(C_{i}+1)\sum_{j=1}^{J}\hat{\Upsilon}^{S,r}_{j}+3 C_{i}\sum_{j=1}^{J}\hat{\Upsilon}^{A,r}_{j}\\
&\leq& \hat{w}^r_i+3\tilde{\upsilon}-\frac{\theta_{i}}{12}T+\frac{3\theta_i}{4}(T-\hat{t})-\frac{\theta_{i}}{12}T-\frac{\theta_{i}}{12}T+(C_{i}+1)\sum_{j=1}^{J}\hat{\Upsilon}^{S,r}_{j}+3 C_{i}\sum_{j=1}^{J}\hat{\Upsilon}^{A,r}_{j}\\
&\leq& \hat{w}^r_i+3\tilde{\upsilon}+\frac{3\theta_i}{4}(T-\xi)-\frac{\theta_{i}}{4}T+(C_{i}+1)\sum_{j=1}^{J}\hat{\Upsilon}^{S,r}_{j}+3 C_{i}\sum_{j=1}^{J}\hat{\Upsilon}^{A,r}_{j}\\
% &\leq& \hat{w}^r_i+3\tilde{\upsilon}+\frac{3\theta}{4}(T-\xi)-\frac{2\theta_{i}}{4}(T-\xi)+(C_{i}+1)\sum_{j=1}^{J}\hat{\Upsilon}^{S,r}_{j}+3 C_{i}\sum_{j=1}^{J}\hat{\Upsilon}^{A,r}_{j}\\
&\leq& \hat{w}^r_i+3\tilde{\upsilon}+(C_{i}+1)\sum_{j=1}^{J}\hat{\Upsilon}^{S,r}_{j}+3 C_{i}\sum_{j=1}^{J}\hat{\Upsilon}^{A,r}_{j}+\frac{\theta_i}{4}(T-\xi),\\
\end{eqnarray*}
where in the last inequality we have used $T< 2(T-\xi)$ which follows from our choice of $T$.
Since $T>\xi+\frac{4}{|\theta_{i}|}\left(\hat{w}^r_i+3\tilde{\upsilon}+ \left( C_{i}+1\right)\sum_{j=1}^{J}\hat{\Upsilon}^{S}_{j}+3 C_{i}\sum_{j=1}^{J}\hat{\Upsilon}^{A}_{j}\right)$
and $\theta_{i}<0$ this implies
\begin{eqnarray*}
\hat{W}^{r}_{i}(T)&<&\hat{w}^r_i+3\tilde{\upsilon}+ \left( C_{i}+1\right)\sum_{j=1}^{J}\hat{\Upsilon}^{S}_{j}+3 C_{i}\sum_{j=1}^{J}\hat{\Upsilon}^{A}_{j}\\
&&+\frac{1}{4}\theta_{i}\left(-\frac{4}{\theta_{i}}\left(w^r_i+3\tilde{\upsilon}+ \left( C_{i}+1\right)\sum_{j=1}^{J}\hat{\Upsilon}^{S}_{j}+3 C_{i}\sum_{j=1}^{J}\hat{\Upsilon}^{A}_{j}\right)\right)
=0
\end{eqnarray*}
which contradicts the fact that $\hat{W}^{r}_{i}(t)\geq 0$ for all $t\geq 0$.  Therefore for all $T>T^{*}$ we have $\left\{\tilde{\gamma}^{r,\tilde{\upsilon}}_{i,\xi}> T\right\}\cap\mathcal{A}^{r,T}=\emptyset$ which says that $\left\{\tilde{\gamma}^{r,\tilde{\upsilon}}_{i,\xi}> T\right\}\subset \left(\mathcal{A}^{r,T}\right)^{c}$ and consequently from \eqref{eq:606}
$
 P_{y^r}\left(\tilde{\gamma}^{r,\tilde{\upsilon}}_{i,\xi}> T \right)\leq P_{y^r}\left( \left(\mathcal{A}^{r,T}\right)^{c} \right)\leq B_{1}e^{-B_{2}T}$.
% Note that for an arbitrary nonnegative random variable $Y$ with pdf $\mu(dy)$ we have for any $c>0$
% \begin{eqnarray*}
% E[e^{c Y}]\le
% % \int_{0}^{\infty}e^{c y}\mu(dy)&=&\int_{0}^{\infty}\int_{0}^{e^{c y}}dx\mu(dy)\\
% % &=&\int_{0}^{1}\int_{0}^{\infty}\mu(dy)dx+\int_{1}^{\infty}\int_{\left(\frac{\ln(x)}{c},\infty\right)}\mu(dy)dx\\
% % &=&
% 1+\int_{1}^{\infty}P\left(Y> \frac{\ln(x)}{c} \right)dx.
% \end{eqnarray*}
Finally, for any $c\in (0,\frac{1}{2}B_{2}]$ and $r\geq R$ we have
\begin{eqnarray*}
E_{y^r}[e^{c\tilde{\gamma}^{r,\tilde{\upsilon}}_{i,\xi}}]
%&\leq&1+\int_{1}^{\infty}P\left(\tilde{\gamma}^{r,\tilde{\upsilon}}_{i,\xi}> \frac{\ln(x)}{c} \right)dx
&\leq & e^{c T^{*}}+\int_{\left(e^{c T^{*}},\infty\right)}P_{y^r}\left(\tilde{\gamma}^{r,\tilde{\upsilon}}_{i,\xi}> \frac{\ln(x)}{c} \right)dx\\
&\leq& e^{c T^{*}}+\int_{\left(e^{c T^{*}},\infty\right)}B_{1}e^{-\frac{B_{2}\ln(x)}{c}}dx
\\
&\leq&e^{\frac{1}{2}B_{2} T^{*}}+B_{1}.
% &\leq&e^{\frac{1}{2}B_{2} \left(1+2\xi-\frac{4}{\theta_{i}}\left(\hat{W}^{r}_{i}(0)+3\tilde{\upsilon}+ \left( C_{i}+1\right)\sum_{j=1}^{J}\hat{\Upsilon}^{S}_{j}+3 C_{i}\sum_{j=1}^{J}\hat{\Upsilon}^{A}_{j}\right) \right)}\\
% &&+B_{1}e^{ -\frac{1}{2}B_{2}\left(1+2\xi-\frac{4}{\theta_{i}}\left(\hat{W}^{r}_{i}(0)+3\tilde{\upsilon}+ \left( C_{i}+1\right)\sum_{j=1}^{J}\hat{\Upsilon}^{S}_{j}+3 C_{i}\sum_{j=1}^{J}\hat{\Upsilon}^{A}_{j}\right) \right)}\\
% &\leq&B_{3}e^{B_{4}\left(\xi+\tilde{\upsilon}+\hat{W}^{r}_{i}(0)+\sum_{j=1}^{J}\hat{\Upsilon}^{S}_{j}+\sum_{j=1}^{J}\hat{\Upsilon}^{A}_{j}\right)}+B_{1}
\end{eqnarray*}
The result follows on recalling the definition of $T^*$.
% where
% $
% B_{3}\doteq e^{\frac{1}{2}B_{2}}$
% and
% $
% B_{4} \doteq \max\left\{B_{2},-\frac{6B_{2}}{\theta_{i}},-\frac{2B_{2}}{\theta_{i}},-\frac{2B_{2}}{\theta_{i}}\left(C_{i}+1\right), -\frac{6B_{2}}{\theta_{i}}C_{i}\right\}$.
% Because $i\in\AAA_I$ was arbitrary this completes the proof.
\hfill \qed

\subsection{Proof of Lemma 
\ref{thm:VMarkovLong}}
\label{sec:vmarkovlong}
Since $\hat{Y}^{r}(\cdot)$ is a $\clg^{r}(t)$ Markov process, we have,  for $0\leq s<t$, 
\begin{eqnarray*}
E_{y}\left[ V^{r,\tilde{\upsilon}}_{i}\left(\hat{Y}^{r}(t)\right)\right]&=&E_{y}\left[e^{\tilde{\delta}( \tilde{\gamma}^{r,\tilde{\upsilon}}_{i,t}-t)}\right]
=E_{y}\left[E_{y}\left[ e^{\tilde{\delta}(\tilde{\gamma}^{r,\tilde{\upsilon}}_{i,t}-t)}|\clg^{r}(s) \right]\right]\\
&=&E_{y}\left[E_{\hat{Y}^{r}(s)}\left[ e^{\tilde{\delta}(\tilde{\gamma}^{r,\tilde{\upsilon}}_{i,t-s}-(t-s))} \right]\right]
= E_{y}\left[E_{\hat{Y}^{r}(s)}\left[ V^{r,\tilde{\upsilon}}_{i}\left(\hat{Y}^{r}(t-s)\right)\right]\right].
\end{eqnarray*}
\begin{proposition}
\label{thm:VMarkovShort}
For $\tilde{\upsilon}>0$ let $\tilde{\delta}$ be as in Proposition \ref{thm:waitTimeExpBnd}. There exist constants $\tilde{R},\tilde{B}\in (0,\infty)$ such that for any $y=(\hat{q},\hat{\Upsilon},\tilde{\mathcal{E}})\in \cly^r$, $i\in\AAA_I$, $s\in [0,1]$, and $r\geq\tilde{R}$ we have
\begin{equation*}
E_{y}\left[ V^{r,\tilde{\upsilon}}_{i}\left(\hat{Y}^{r}(s)\right)\right]\leq e^{-\tilde{\delta}s}V^{r,\tilde{\upsilon}}_{i}\left(y\right)+\tilde{B}.
\end{equation*}
\end{proposition}
\begin{proof}

Let $\tilde{\upsilon}>0$ be fixed, let $B_{1},B_{2},B_{3},R$ be as in Theorem \ref{thm:waitTimeExpBnd}, and let $y\in \cly^r$, $i\in\AAA_I$, and $s\in [0,1]$ be arbitrary.  Due to the fact that $\hat{Y}^{r}(t)$ is a $\mathcal{G}^{r}(t)$ strong Markov process  and $\tilde{\gamma}^{r,\tilde{\upsilon}}_{i,0}\wedge s$ is a bounded $\mathcal{G}^{r}(t)$ stopping time, for all $r\geq R$ we have,
\begin{eqnarray*}
E_{y}\left[ V^{r,\tilde{\upsilon}}_{i}\left(\hat{Y}^{r}(s)\right)\right]&\leq&E_{y}\left[e^{\tilde{\delta}( \tilde{\gamma}^{r,\tilde{\upsilon}}_{i,s}-s)}\right]\\
&=&E_{y}\left[\mathcal{I}_{\{\tilde{\gamma}^{r,\tilde{\upsilon}}_{i,0}\geq s\}}e^{\tilde{\delta}( \tilde{\gamma}^{r,\tilde{\upsilon}}_{i,s}-s)}\right]+E_{y}\left[\mathcal{I}_{\{\tilde{\gamma}^{r,\tilde{\upsilon}}_{i,0}< s\}}e^{\tilde{\delta}( \tilde{\gamma}^{r,\tilde{\upsilon}}_{i,s}-s)}\right]\\
% &\leq&E_{y}\left[\mathcal{I}_{\{\tilde{\gamma}^{r,\tilde{\upsilon}}_{i,0}\geq s\}}e^{\tilde{\delta}( \tilde{\gamma}^{r,\tilde{\upsilon}}_{i,0}-s)}\right]+E_{y}\left[\mathcal{I}_{\{\tilde{\gamma}^{r,\tilde{\upsilon}}_{i,0}< s\}}e^{\tilde{\delta}( \tilde{\gamma}^{r,\tilde{\upsilon}}_{i,s}-s)}\right]\\
&=&E_{y}\left[e^{\tilde{\delta}( \tilde{\gamma}^{r,\tilde{\upsilon}}_{i,0}-s)}\right]+E_{y}\left[E_{y}\left[ \mathcal{I}_{\{\tilde{\gamma}^{r,\tilde{\upsilon}}_{i,0}< s\}}e^{\tilde{\delta}( \tilde{\gamma}^{r,\tilde{\upsilon}}_{i,s}-s)}|\mathcal{G}^{r}(\tilde{\gamma}^{r,\tilde{\upsilon}}_{i,0}\wedge s) \right]\right]\\
% &\leq&E_{y}\left[e^{\tilde{\delta}( \tilde{\gamma}^{r,\tilde{\upsilon}}_{i,0}-s)}\right]+E_{y}\left[\mathcal{I}_{\{\tilde{\gamma}^{r,\tilde{\upsilon}}_{i,0}< s\}}E_{\hat{Y}^{r}(\tilde{\gamma}^{r,\tilde{\upsilon}}_{i,0}\wedge t)}\left[e^{\tilde{\delta}\left( \tilde{\gamma}^{r,\tilde{\upsilon}}_{i,(s-\tilde{\gamma}^{r,\tilde{\upsilon}}_{i,0})^{+}}-s\right)}\right]\right]\\
&=&e^{-\tilde{\delta}s}E_{y}\left[e^{\tilde{\gamma}^{r,\tilde{\upsilon}}_{i,0}}\right]+\sup_{(y,z):w_{i}+|\hat{\Upsilon}|\leq \tilde{\upsilon},z\in[0, s]}\left\{E_{y}\left[e^{\tilde{\delta} \tilde{\gamma}^{r,\tilde{\upsilon}}_{i,z}}\right]\right\}\\
&\leq&e^{-\tilde{\delta}s}V^{r,\tilde{\upsilon}}_{i}(y)+B_{1}e^{B_{2}\left(1+2\tilde{\upsilon}\right)}+B_{3}\\
\end{eqnarray*}
where the last line uses Proposition \ref{thm:waitTimeExpBnd}.  This completes the proof.
\end{proof}

We now complete the proof of Lemma \ref{thm:VMarkovLong}.\\

\noindent {\bf Proof of Lemma \ref{thm:VMarkovLong}.} Let $\tilde{\upsilon}>0$ be fixed and note that, with $\tilde \delta$ as in Proposition \ref{thm:waitTimeExpBnd},  due to Proposition \ref{thm:VMarkovShort}, there exist constants $B_{1},R\in (0,\infty)$ such that for all $y\in \cly^r$, $r\geq R$, $i\in\AAA_I$, and $t\in [0,1]$ we have
\begin{equation*}
E_{y}\left[ V^{r,\tilde{\upsilon}}_{i}\left(\hat{Y}^{r}(t)\right)\right]\leq e^{-\tilde{\delta}t}V^{r,\tilde{\upsilon}}_{i}\left(y\right)+B_{1}.
\end{equation*}
Consequently for $T\in [0,1]$ we already have the result (with $B=B_1$).  Let $T>1$ be arbitrary and note that there exists $t\in \left[\frac{1}{2},1\right]$ and $n\in\mathbb{N}$ such that $T=nt$.  Using the Markov property we have 
\begin{align*}
&E_{y}\left[V^{r,\tilde{\upsilon}}_{i}\left(\hat{Y}^{r}(T)\right)  \right]\\
&=E_{y}\left[V^{r,\tilde{\upsilon}}_{i}\left(\hat{Y}^{r}(nt)\right) \right]
= E_{y}\left[E\left[ V^{r,\tilde{\upsilon}}_{i}\left(\hat{Y}^{r}(nt)\right) | \mathcal{G}^{r}\left((n-1)t\right)\right]\right]\\
&=E_{y}\left[E_{\hat{Y}^{r}((n-1)t)}\left[V^{r,\tilde{\upsilon}}_{i}\left(\hat{Y}^{r}(t)\right) \right]\right]\leq e^{-\tilde{\delta}t}E_{y}\left[V^{r,\tilde{\upsilon}}_{i}\left( \hat{Y}^{r}((n-1)t)\right) \right]+B_{1}.
\end{align*}
% In addition, if for $m\in\{1,...,n-1\}$ we have
% \begin{equation*}
% E_{y}\left[V^{r,\tilde{\upsilon}}_{i}\left(\hat{Y}^{r}(T)\right)  \right]\leq e^{-\tilde{\delta}mt} E_{y}\left[V^{r,\tilde{\upsilon}}_{i}\left( \hat{Y}^{r}((n-m)t)\right) \right]+B_{1}\sum_{l=0}^{m-1}e^{-\tilde{\delta}l t}
% \end{equation*}
% Then due to the Markov property and Theorem \ref{thm:VMarkovShort} we have
% \begin{eqnarray*}
% E_{y}\left[V^{r,\tilde{\upsilon}}_{i}\left(\hat{Y}^{r}(T)\right)  \right]&\leq &e^{-\tilde{\delta}mt} E_{y}\left[V^{r,\tilde{\upsilon}}_{i}\left( \hat{Y}^{r}((n-m)t)\right) \right]+B_{1}\sum_{l=0}^{m-1}e^{-\tilde{\delta}l t}\\
% &\leq &e^{-\tilde{\delta}mt} E_{y}\left[E \left[V^{r,\tilde{\upsilon}}_{i}\left( \hat{Y}^{r}((n-m)t)\right)| \mathcal{F}^{r}\left(\tau^{r}((n-m-1)t)\right) \right] \right]+B_{1}\sum_{l=0}^{m-1}e^{-\tilde{\delta}l t}\\
% &\leq &e^{-\tilde{\delta}mt} E_{y}\left[E_{\hat{Y}^{r}((n-m-1)t)}\left[V^{r,\tilde{\upsilon}}_{i}\left( \hat{Y}^{r}(t)\right) \right] \right]+B_{1}\sum_{l=0}^{m-1}e^{-\tilde{\delta}l t}\\
% &\leq &e^{-\tilde{\delta}mt} \left( e^{-\tilde{\delta}t}E_{y}\left[V^{r,\tilde{\upsilon}}_{i}\left( \hat{Y}^{r}((n-m-1)t)\right) \right] +B_{1}\right)+B_{1}\sum_{l=0}^{m-1}e^{-\tilde{\delta}l t}\\
% &\leq& e^{-\tilde{\delta}(m+1)t} E_{y}\left[V^{r,\tilde{\upsilon}}_{i}\left( \hat{Y}^{r}((n-m-1)t)\right) \right]+B_{1}\sum_{l=0}^{m}e^{-\tilde{\delta}l t}.\\
% \end{eqnarray*}
Now a standard recursive argument shows that
\begin{eqnarray*}
E_{y}\left[V^{r,\tilde{\upsilon}}_{i}\left(\hat{Y}^{r}(T)\right)  \right]&\leq &e^{-\tilde{\delta}nt} E_{y}\left[V^{r,\tilde{\upsilon}}_{i}\left( \hat{Y}^{r}(0)\right) \right]+B_{1}\sum_{l=0}^{n-1}e^{-\tilde{\delta}l t}\\
&\leq &e^{-\tilde{\delta}T}V^{r,\tilde{\upsilon}}_{i}\left( y\right) +B_{1}\sum_{l=0}^{\infty}e^{-\tilde{\delta}l \frac{1}{2}}
\leq  e^{-\tilde{\delta}T}V^{r,\tilde{\upsilon}}_{i}\left( y\right) +\frac{B_{1}}{1-e^{-\frac{1}{2}\tilde{\delta}}}.
\end{eqnarray*}
This completes the proof on taking $B= B_1/(1-e^{-\frac{1}{2}\tilde{\delta}})$.
\hfill \qed

\subsection{Proof of Proposition \ref{thm:TildeGammaLower}}
\label{sec:tildefammalow}

Let $\tilde{\upsilon}>0$ be fixed and let $y=(\hat{q},\hat{\Upsilon},\tilde{\mathcal{E}})\in \cly^r$ and $i\in\AAA_I$ be arbitrary.  For $T\geq 1$ define
\begin{eqnarray*}
\mathcal{A}^{r,T}_{j}&=&\left\{\sup_{0\leq t \leq r^{2}T}|A_{j}^{r}(t)-t\alpha^{r}_{j}|\leq \frac{|\theta_{i}|}{4J} rT \right\}\cap\left\{\sup_{0\leq t \leq r^{2}\max_{i}\{C_{i}\}T}|S_{j}^{r}(t)-t\beta^{r}_{j}|\leq \frac{|\theta_{i}|}{4J} rT \right\}\\
\end{eqnarray*}
and
$
\mathcal{A}^{r,T}=\cap_{j=1}^{J}\mathcal{A}^{r,T}_{j}$.
From Proposition  \ref{thm:expTailBnd} (equations \eqref{eq:31.d} and \eqref{eq:31.c} with $c_1=1$, $c_2=0$) we know there exist constants $B_{1},B_{2},R\in (0,\infty)$ such that for all $r\geq R$ and $T\geq 1$ we have
\begin{equation*}
P\left( \mathcal{A}^{r,T}\right)\geq 1- B_{1}e^{-B_{2}T},\;\;
r^{-1}\sum_{j=1}^{J}\frac{1}{\beta^{r}_{j}}\leq\tilde{\upsilon},
\;\;\mbox{ and, }
r \left( \sum_{j=1}^{J} K_{i,j}\rho^{r}_{j} - C_{i} \right) \geq \frac{3 \theta_{i} }{2}.
\end{equation*}
Choose $L\in[1,\infty)$ such that $1- B_{1}e^{-B_{2}L}\geq\frac{1}{2}$.
If 
\begin{equation*}
\hat{w}_{i}-C_{i}\max_{j}\{\hat{\Upsilon}^{A}_{j}\}\geq 2|\theta_{i}|L+2\tilde{\upsilon},
\mbox{ let }
\tilde{T} \doteq \frac{1}{2|\theta_{i}|}\left(\hat{w}_{i}-C_{i}\max_{j}\{\hat{\Upsilon}^{A}_{j}\}-2\tilde{\upsilon}\right)
\end{equation*}
and note that $\tilde{T}\geq L$.
For all $r\geq R$ and $s\in[0,\tilde{T})$ on the set $\mathcal{A}^{r,\tilde{T}}$ we have
\begin{eqnarray*}
\hat{W}^{r}_{i}(s)&=&\hat{w}_{i}+ \sum_{j=1}^{J}K_{i,j}\frac{1}{r\beta^{r}_{j}}A^{r}_{j}\left(r^{2}\left(s-\bar{\Upsilon}_{j}^{A}\right)^{+} \right) +\sum_{j=1}^{J}K_{i,j}\frac{1}{r\beta^{r}_{j}}\mathcal{I}_{\{s\geq \bar{\Upsilon}_{j}^{A}>0 \}}\\
&&-\sum_{j=1}^{J}K_{i,j}\frac{1}{r\beta^{r}_{j}}S^{r}_{j}\left(r^{2}\left(\bar{B}^{r}(s)-\bar{\Upsilon}_{j}^{S}\right)^{+} \right) -\sum_{j=1}^{J}K_{i,j}\frac{1}{r\beta^{r}_{j}}\mathcal{I}_{\{\bar{B}^{r}(s)\geq \bar{\Upsilon}_{j}^{S} >0\}}\\
&\geq& \hat{w}_{i}+\sum_{j=1}^{J}K_{i,j}\frac{1}{r\beta^{r}_{j}}A^{r}_{j}\left(r^{2}\left(s-\bar{\Upsilon}^{A}_{j}\right)^{+} \right)-\sum_{j=1}^{J}K_{i,j}\frac{1}{r\beta^{r}_{j}}S^{r}_{j}\left(r^{2}\bar{B}^{r}_{j}(s) \right)-r^{-1}\sum_{j=1}^{J}\frac{1}{\beta^{r}_{j}}\\
&\geq& \hat{w}_{i}+r\sum_{j=1}^{J}K_{i,j}\rho^{r}_{j}\left(s-\bar{\Upsilon}^{A}_{j}\right)^{+}-r\sum_{j=1}^{J}K_{i,j}\bar{B}^{r}_{j}(s)+\frac{\theta_{i}}{2} \tilde{T}-\tilde{\upsilon}\\
&\geq& \hat{w}_{i}+r\sum_{j=1}^{J}K_{i,j}\rho^{r}_{j}\left(s-\max_{j}\{\bar{\Upsilon}^{A}_{j}\}\right)^{+}-rsC_{i}+\frac{\theta_{i}}{2} \tilde{T}-\tilde{\upsilon}\\
&\geq& \hat{w}_{i}-C_{i}\max_{j}\{\hat{\Upsilon}^{A}_{j}\}+sr \left(\sum_{j=1}^{J}K_{i,j}\rho^{r}_{j}-C_{i}\right)+\frac{\theta_{i}}{2} \tilde{T}-\tilde{\upsilon}\\
&\geq& \hat{w}_{i}-C_{i}\max_{j}\{\hat{\Upsilon}^{A}_{j}\}+\frac{3\theta_{i}}{2}s+\frac{\theta_{i}}{2} \tilde{T}-\tilde{\upsilon}
> \hat{w}_{i}-C_{i}\max_{j}\{\hat{\Upsilon}^{A}_{j}\}-\tilde{\upsilon}+2\theta_{i}\tilde{T}.
\end{eqnarray*}
Since 
$
\hat{w}_{i}-C_{i}\max_{j}\{\hat{\Upsilon}^{A}_{j}\}-\tilde{\upsilon}+2\theta_{i}\tilde{T}=\tilde{\upsilon}$
this means that for $r\geq R$ and $s\in[0,\tilde{T})$ on the set $\mathcal{A}^{r,\tilde{T}}$ we have $\hat{W}^{r}_{i}(s)>\tilde{\upsilon}$ and consequently
$
\{\tilde{\gamma}^{r,\tilde{\upsilon}}_{i,0}\geq \tilde{T}\}\supset \mathcal{A}^{r,\tilde{T}}$.
Therefore for $r\geq R$ if $\hat{w}_{i}-C_{i}\max_{j}\{\hat{\Upsilon}^{A}_{j}\}\geq 2|\theta_{i}|L+2\tilde{\upsilon}$ we have
\begin{align}
V^{r,\tilde{\upsilon}}_{i}(y)&=E_{y}\left[ e^{\tilde{\delta}\tilde{\gamma}^{r,\tilde{\upsilon}}_{i,0}}\right]
\geq P\left(\mathcal{A}^{r,\tilde{T}}\right)e^{\tilde{\delta}\tilde{T}}\nonumber\\
&\geq\left(1- B_{1}e^{-B_{2}L}\right)e^{\tilde{\delta}\tilde{T}}
\geq \frac{e^{-\frac{\tilde{\delta}\tilde{\upsilon}}{|\theta_{i}|}}}{2}e^{\frac{\tilde{\delta}}{2|\theta_{i}|}\left(\hat{w}_{i}-C_{i}\max_{j}\{\hat{\Upsilon}^{A}_{j}\}\right)}. \label{eq:919b}
\end{align}
If $\hat{w}_{i}-C_{i}\max_{j}\{\hat{\Upsilon}^{A}_{j}\}<2|\theta_{i}|L+2\tilde{\upsilon}$ by definition we have
$
V^{r,\tilde{\upsilon}}_{i}(y)=E_{y}\left[ e^{\tilde{\delta}\tilde{\gamma}^{r,\tilde{\upsilon}}_{i,0}}\right]\geq 1$.
Let 
\begin{equation*}
\tilde{B}_{1}\doteq\min\left\{e^{-\frac{\tilde{\delta}}{2|\theta_{i}|}\left( 2|\theta_{i}|L+2\tilde{\upsilon}\right)} ,\frac{e^{-\frac{\tilde{\delta}\tilde{\upsilon}}{|\theta_{i}|}}}{2}\right\},
\mbox{ and, }
\tilde{B}_{2}\doteq \frac{\tilde{\delta}}{2|\theta_{i}|}.
\end{equation*}
Then if $\hat{w}_{i}-C_{i}\max_{j}\{\hat{\Upsilon}^{A}_{j}\}\leq 2|\theta_{i}|L+2\tilde{\upsilon}$ we have
$
\tilde{B}_{1}e^{\tilde{B}_{2}\left(\hat{w}_{i}-C_i\max_{j}\{\hat{\Upsilon}^{A}_{j}\}\right)^{+}}\leq 1\leq V^{r,\tilde{\upsilon}}_{i}(y)$
and if $\hat{w}_{i}-C_{i}\max_{j}\{\hat{\Upsilon}^{A}_{j}\}\geq  2|\theta_{i}|L+2\tilde{\upsilon}$ we have from \eqref{eq:919b}
\begin{equation*}
\tilde{B}_{1}e^{\tilde{B}_{2}\left(\hat{w}_{i}-C_{i}\max_{j}\{\hat{\Upsilon}^{A}_{j}\}\right)^{+}}\leq \frac{e^{-\frac{\tilde{\delta}\tilde{\upsilon}}{|\theta_{i}|}}}{2}e^{\frac{\tilde{\delta}}{2|\theta_{i}|}\left(\hat{w}_{i}-C_{i}\max_{j}\{\hat{\Upsilon}^{A}_{j}\}\right)}\leq V^{r,\tilde{\upsilon}}_{i}(y).
\end{equation*}
The result follows.
% so for all $r\geq R$
% $
% \tilde{B}_{1}e^{-\tilde{B}_{2}\left(w_{i}-C_{i}\max_{j}\{\hat{\Upsilon}^{A}_{j}\}\right)^{+}}\leq V^{r,\tilde{\upsilon}}_{i}(y)$.
% Since $y=(q,\hat{\Upsilon},\tilde{\mathcal{E}})\in \cly^r$ and $i\in\AAA_I$ were arbitrary and $\tilde{B}_{1},\tilde{B}_{2}$, and $R$ did not depend on $y$ this completes the proof.
\hfill \qed

\section{Proofs of Some Exponential Estimates}
\label{sec:expoestim}
The following lemma gives a key independence property.
% \begin{proposition}
% \label{thm:stopTime}
% Let $t \in \R_+$ and $r \in \N$.
% %$X\geq 0$ be a $\mathcal{F}^{r}(0,0)$-measurable random variable.
% Then $\tau^{r}(t)$ is a $\mathcal{F}^{r}(n,m)$-stopping time.
% \end{proposition}
% \begin{proof}
% 	 Note that $t\rightarrow \tau^{r}(t)$ is a $\mathbb{N}_{0}^{2J}$-valued left continuous non-decreasing jump process.
% Let $\{t_{l}\}_{l=1}^{\infty}$ be a countable set of nonnegative real numbers which are dense in $[0,\infty)$ and let $(\tilde{n},\tilde{m})\geq (0,0)$ be arbitrary.  Consequently
% \begin{equation*}
% \left\{ \tau^{r}(t)\leq (\tilde{n},\tilde{m})\right\}=\cup_{l=1}^{\infty}\left( \left\{ t_{l}\leq t\right\}\cap\left\{ \tau^{r}(t_{l})\leq (\tilde{n},\tilde{m}) \right\}\right)
% \end{equation*}
% and since $t$ is $\mathcal{F}^{r}(0,0)$-measurable and $\tau^{r}(t_{l})$ is a $\mathcal{F}^{r}(n,m)$ stopping time for all $l\in\mathbb{N}$
% we have
% \begin{equation*}
% \left\{t\leq t_{l}\right\}\cap\left\{ \tau^{r}(t_{l})\leq (\tilde{n},\tilde{m}) \right\}\in \mathcal{F}^{r}(\tilde{n},\tilde{m})
% \end{equation*}
%  for all $l\in\mathbb{N}$.  Therefore
% \begin{equation*}
% \left\{ \tau^{r}(t)\leq (\tilde{n},\tilde{m})\right\}=\cup_{l=1}^{\infty}\left( \left\{t\leq t_{l}\right\}\cap\left\{ \tau^{r}(t_{l})\leq (\tilde{n},\tilde{m}) \right\}\right)\in \mathcal{F}^{r}(\tilde{n},\tilde{m})
% \end{equation*}
% and since $(\tilde{n},\tilde{m})\geq (0,0)$ was arbitrary this shows that $\tau^{r}(t)$ is a $\mathcal{F}^{r}(n,m)$-stopping time.
% \end{proof}

\begin{lemma}
\label{thm:sameDist}
Let $t\in \RR_+$  and $r \in \NN$.
% be an arbitrary $\mathcal{F}^{r}(0,0)$-measurable random variable satisfying $P(X<\infty)=1$ and let $r>0$ be arbitrary.
Then $(\hat{A}^{r,t}(\cdot),\hat{S}^{r,t}(\cdot))$ is independent of $\mathcal{G}^{r}(t)$ and $(\hat{A}^{r,t}(\cdot),\hat{S}^{r,t}(\cdot))$ has the same distribution as $(\hat{A}^{r}(\cdot),\hat{S}^{r}(\cdot))$.
\end{lemma}

\begin{proof}
For  $a,b\in\mathbb{N}^{J}$ define
\begin{equation*}
\tilde{u}^{r}_{a}(b)=\left(u^{r}_{1}(a_{1}+1)...,u^{r}_{1}(a_{1}+b_{1}),u^{r}_{2}(a_{2}+1)...,u^{r}_{2}(a_2+b_{2})...,u^{r}_{J}(a_{J}+1)...,u^{r}_{J}(a_{J}+b_{J})\right)
\end{equation*}
and
\begin{equation*}
\tilde{v}^{r}_{a}(b)=\left(v^{r}_{1}(a_{1}+1)...,v^{r}_{1}(a_{1}+b_{1}),v^{r}_{2}(a_{2}+1)...,v^{r}_{2}(a_{2}+b_{2}),...,v^{r}_{J}(a_{J}+1)...,v^{r}_{J}(a_{J}+b_{J})\right).
\end{equation*}
It suffices to show that for all $(\tilde{n},\tilde{m})\geq (0,0)$ $(\tilde{u}^{r}_{\tau^{r,A}(t)}(\tilde{n}),\tilde{v}^{r}_{\tau^{r,S}(t)}(\tilde{m}))$ is independent of  $\clg^r(t)=\mathcal{F}^{r}(\tau^{r}(t))$ and $(\tilde{u}^{r}_{\tau^{r,A}(t)}(\tilde{n}),\tilde{v}^{r}_{\tau^{r,S}(t)}(\tilde{m}))$ has the same distribution as $(\tilde{u}^{r}_{0}(\tilde{n}),\tilde{v}^{r}_{0}(\tilde{m}))$.  Let $(n,m)\geq (0,0)$, $f:\mathbb{R}_{+}^{|\tilde{n}|_{1}+|\tilde{m}|_{1}}\rightarrow \mathbb{R}$, and $G\in \mathcal{F}^{r}(\tau^{r}(t))$ be arbitrary.  Then
% \begin{equation*}
% 1=\sum_{(n,m)\geq (0,0)}P\left( \tau^{r}(t)=(n,m) \right)
% \end{equation*}
% so
\begin{align*}
&E\left[ \mathcal{I}_{G}f\left(\tilde{u}^{r}_{\tau^{r,A}(t)}(\tilde{n}),\tilde{v}^{r}_{\tau^{r,S}(t)}(\tilde{m}) \right)\right]\\
&=\sum_{(n,m)\geq (0,0)}E\left[ \mathcal{I}_{\{\tau^{r}(t)=(n,m)\}}\mathcal{I}_{G}f\left(\tilde{u}^{r}_{\tau^{r,A}(t)}(\tilde{n}),\tilde{v}^{r}_{\tau^{r,S}(t)}(\tilde{m}) \right)\right]\\
&=\sum_{(n,m)\geq (0,0)}E\left[\mathcal{I}_{G\cap \{\tau^{r}(t)=(n,m)\}}f\left(\tilde{u}^{r}_{n}(\tilde{n}),\tilde{v}^{r}_{m}(\tilde{m}) \right)\right]\\
&=\sum_{(n,m)\geq (0,0)}E\left[\mathcal{I}_{G\cap \{\tau^{r}(t)=(n,m)\}}\right]E\left[f\left(\tilde{u}^{r}_{n}(\tilde{n}),\tilde{v}^{r}_{m}(\tilde{m}) \right)\right]\\
&=E\left[f\left(\tilde{u}^{r}_{0}(\tilde{n}),\tilde{v}^{r}_{0}(\tilde{m}) \right)\right]\sum_{(n,m)\geq (0,0)}E\left[\mathcal{I}_{G\cap \{\tau^{r}(t)=(n,m)\}}\right]
=P\left(G\right)\left[f\left(\tilde{u}^{r}_{0}(\tilde{n}),\tilde{v}^{r}_{0}(\tilde{m}) \right)\right]
\end{align*} 
where the third equality is from  the fact that $G\cap \{\tau^{r}(t)=(n,m)\}$ is $\mathcal{F}^{r}(n,m)$-measurable and $ (\tilde{u}^{r}_{n}(\tilde{n}),\tilde{v}^{r}_{m}(\tilde{m}))$ is independent of $\mathcal{F}^{r}(n,m)$ and the fourth equality uses the fact that $ (\tilde{u}^{r}_{n}(\tilde{n}),\tilde{v}^{r}_{m}(\tilde{m}))$ has the same distribution as $ (\tilde{u}^{r}_{0}(\tilde{n}),\tilde{v}^{r}_{0}(\tilde{m}))$.  
The result follows.
\end{proof}

\subsection{Proof of Proposition \ref{thm:nextSerTimeBnd}}
\label{sec:propnextser}
  %Recall that  $\tau^{r}(t)$ is a $\mathcal{F}^{r}(n,m)$ stopping time.  
  Fix $j\in\AAA_J$. Define
\begin{equation*}
\mathcal{F}^{r,S}_{j}(k)\doteq\sigma\left\{u^{r}_{l}(m^{u}_{l}),v^{r}_{l'}(m^{v}_{l'}): m^{u}_{l}\in\mathbb{N},  l\in\AAA_J, m^{v}_{l'}\in\mathbb{N},  l'\in\AAA_J\setminus \left\{ j \right\}, \text{ and }
m^{v}_{j}\leq k \right\}
\end{equation*}
 which is the filtration that contains the information from all inter-arrival times, all service times from queues other than the $j$-th queue, and
 the first $k$ service times from queue $j$.  Note that 
 % for any $(n,m)$ satisfying $m_{j}=k$ we have $\mathcal{F}^{r}(n,m)\subset\mathcal{F}^{r,S}_{j}(k)$ and consequently
% $
% \left\{\tau^{r,S}_{j}(t)\leq k \right\}=\cup_{(n,m):m_{j}=k}\left\{\tau^{r}(t)\leq (n,m) \right\}\subset\mathcal{F}^{r,S}_{j}(k)
% $
% so
$\tau^{r,S}_{j}(t)$ is a $\mathcal{F}^{r,S}_{j}(k)$ stopping time.  For all $n\geq 1$ define 
$
\tilde{L}^r_{j}(n)=\sup\left\{s\geq 0 : B^{r}_{j}(s)<\sum_{l=1}^{n}v^{r}_{j}(l) \right\}
$
and  since $B^{r}_{j}(\cdot)$ is continuous we have $B^{r}_{j}\left(\tilde{L}^r_{j}(n)\right)=\sum_{l=1}^{n}v^{r}_{j}(l)$.

  Define 
\begin{equation*}
o^{r}_{j}(n)=\inf\left\{s\geq \tilde{L}^{r}_{j}(n):\mathcal{E}^{r}_{j}(s)=0\right\}
\end{equation*}
and note that due to property (a) of { Proposition} \ref{thm:schemeSum} for all $s\in [\tilde{L}^{r}_{j}(n),o^{r}_{j}(n)]$ we have $B^{r}_{j}\left(s \right)=\sum_{l=1}^{n}v^{r}_{j}(l)$.  Define $\tilde{a}=\frac{\rho^*}{4}$ and note that due to properties (b) and (c) of { Proposition} \ref{thm:schemeSum} for all $r$ sufficiently large we have $\yr_{j}(s)\geq \tilde{a}>0$ for all $s\in  [o^{r}_{j}(n),\tilde{L}^{r}_{j}(n+1)]$.\\

  Let $n\geq 2$ be arbitrary and assume that $\tau^{r,S}_{j}\left(\left(t-\frac{v^{r}_{j}(n)}{\tilde{a}r^{2}}\right)^{+}\right)>n-1$ and $\tau^{r,S}_{j}(t)=n$. Note that
   %
   % if $t\leq\frac{v^{r}_{j}(n)}{\tilde{a}r^{2}}$ then $\tau^{r,S}_{j}\left(\left(t-\frac{v^{r}_{j}(n)}{\tilde{a}r^{2}}\right)^{+}\right)=0$ which contradictions our assumption so
   in this case we must have $t>\frac{v^{r}_{j}(n)}{\tilde{a}r^{2}}$.  Since $\tau^{r,S}_{j}\left(t-\frac{v^{r}_{j}(n)}{\tilde{a}r^{2}}\right)>n-1$ implies $B^{r}_{j}\left(r^{2}\left(t-\frac{v^{r}_{j}(n)}{\tilde{a}r^{2}} \right)\right)>\sum_{l=1}^{n-1}v^{r}_{j}(l)$ we  have
 $o^{r}_{j}(n-1)<r^{2}\left(t-\frac{v^{r}_{j}(n)}{\tilde{a}r^{2}} \right)$ because $o^{r}_{j}(n-1)\geq r^{2}\left(t-\frac{v^{r}_{j}(n)}{\tilde{a}r^{2}} \right)$ implies $B^{r}_{j}\left(r^{2}\left(t-\frac{v^{r}_{j}(n)}{\tilde{a}r^{2}} \right)\right)=\sum_{l=1}^{n-1}v^{r}_{j}(l)$.  In addition, $\tau^{r,S}_{j}(t)=n$ implies $\tilde{L}^{r}_{j}(n)\geq t r^{2}$ so $ [ r^{2}\left(t-\frac{v^{r}_{j}(n)}{r^{2}\tilde{a}}\right),r^{2} t ]\subset [o^{r}_{j}(n-1),\tilde{L}^{r}_{j}(n)]$ and consequently  $\yr_{j}(s)\geq \tilde{a}$ for all $s\in [ r^{2}\left(t-\frac{v^{r}_{j}(n)}{r^{2}\tilde{a}}\right),r^{2} t ] $ and thus in particular $B^{r}_{j}(r^{2}t)-B^{r}_{j}\left(r^{2}\left(t-\frac{v^{r}_{j}(n)}{r^{2}\tilde{a}} \right)\right)\geq v^{r}_{j}(n)$.  Therefore
\begin{equation*}
\sum_{l=1}^{n}v^{r}_{j}(l)=\sum_{l=1}^{n-1}v^{r}_{j}(l)+v^{r}_{j}(n)<B^{r}_{j}\left(r^{2}\left(t-\frac{v^{r}_{j}(n)}{r^{2}\tilde{a}} \right)\right)+v^{r}_{j}(n)\le B^{r}_{j}\left(t r^{2}\right)
\end{equation*}
which contradicts the assumption that $\tau^{r,S}_{j}(t)=n$.  Consequently 
\begin{equation*}
\left\{\tau^{r,S}_{j}\left(\left(t-\frac{v^{r}_{j}(n)}{\tilde{a}r^{2}} \right)^{+}\right)>n-1\right\}\cap\left\{\tau^{r,S}_{j}(t)=n\right\}=\emptyset
\end{equation*}
and so
\begin{equation*}
\left\{\tau^{r,S}_{j}(t)=n\right\}=\left\{\tau^{r,S}_{j}\left(\left(t-\frac{v^{r}_{j}(n)}{\tilde{a}r^{2}} \right)^{+}\right)\leq n-1\right\}\cap\left\{\tau^{r,S}_{j}(t)=n\right\}.
\end{equation*}
% and since
% \begin{equation*}
% \left\{\tau^{r,S}_{j}(t)=n\right\}=\left\{\tau^{r,S}_{j}(t)=n\right\}\cap\left\{\tau^{r,S}_{j}(t)>n-1\right\}
% \end{equation*}
Thus we have
\begin{eqnarray*}
\left\{\tau^{r,S}_{j}(t)=n\right\}&=& \left\{\tau^{r,S}_{j}(t)>n-1\right\}\cap\left\{\tau^{r,S}_{j}(t)=n\right\}\\
&=& \left\{\tau^{r,S}_{j}(t)>n-1\right\}\cap\left\{\tau^{r,S}_{j}\left(\left(t-\frac{v^{r}_{j}(n)}{\tilde{a}r^{2}} \right)^{+}\right)\leq n-1\right\}\cap\left\{\tau^{r,S}_{j}(t)=n\right\}\\
&\subset& \left\{\tau^{r,S}_{j}(t)>n-1\right\}\cap\left\{\tau^{r,S}_{j}\left(\left(t-\frac{v^{r}_{j}(n)}{\tilde{a}r^{2}} \right)^{+}\right)\leq n-1\right\}.\\
\end{eqnarray*}
Note that for any $z\geq 0$, 
%$\mathcal{F}^{r}(0,0)$-measurable random variable 
$t\geq 0$, and $n\geq 2$ both $\left\{\tau^{r,S}_{j}(t)>n-1\right\}$ and \\
 $\left\{\tau^{r,S}_{j}\left(\left(t-\frac{z}{\tilde{a}r^{2}} \right)^{+}\right)\leq n-1\right\}$ are $\mathcal{F}^{r,S}_{j}(n-1)$-measurable and $v^{r}_{j}(n)$ is independent of $\mathcal{F}^{r,S}_{j}(n-1)$. \\

Let $\gamma^{r}_{j}(dz)$ denote the pdf of $v^{r}_{j}(1)$ (recall that the $\{v^{r}_{j}(l)\}_{l=1}^{\infty}$ are i.i.d.)and let $c\in (0,\delta )$ be arbitrary where $\delta$ is as in Condition \ref{eqn:mgfBnd}.  Since $P(\tau^{r,S}_{j}(t)<\infty)=1$, we have, recalling the convention $v_j^r(0)=0$,
\begin{align}
E\left[e^{c v^{r}_{j}(\tau^{r,S}_{j}(t))}\right]&\leq 1 + E\left[e^{c v^{r}_{j}(1)}\right]+\sum_{n=2}^{\infty}E\left[e^{c v^{r}_{j}(n)}\mathbb{I}_{\{ \tau^{r,S}_{j}(t)=n \}}\right] \nonumber\\
% &\leq&E\left[e^{c v^{r}_{j}(1)}\right] +\sum_{n=2}^{\infty}E\left[e^{c v^{r}_{j}(n)}\mathbb{I}_{\{ \tau^{r,S}_{j}(t)=n \}}\right] \\
% &\leq&E\left[e^{c v^{r}_{j}(1)}\right] +\sum_{n=2}^{\infty}E\left[ e^{c v^{r}_{j}(n)}E\left[\mathbb{I}_{\{ \tau^{r,S}_{j}(t)=n \}} \vert v^{r}_{j}(n) \right] \right]\\
% &\leq&E\left[e^{c v^{r}_{j}(1)}\right] +\sum_{n=2}^{\infty}\int_{(0,\infty)}e^{cz}E\left[\mathbb{I}_{\{ \tau^{r,S}_{j}(t)=n \}}\vert v^{r}_{j}(n)=z\right]\gamma^{r}_{j}(dz)\\
&\le 1+ E\left[e^{c v^{r}_{j}(1)}\right] +\sum_{n=2}^{\infty} E\left[e^{cv^{r}_{j}(n)}\mathbb{I}_{ \left\{\tau^{r,S}_{j}(t)>n-1\right\}\cap\left\{\tau^{r,S}_{j}\left(\left(t-\frac{v^{r}_{j}(n)}{r^{2}\tilde{a}} \right)^{+}\right)\leq n-1\right\}}\right]\nonumber\\
% &\leq&E\left[e^{c v^{r}_{j}(1)}\right] +\sum_{n=2}^{\infty}\int_{(0,\infty)}e^{cz}P\left( \left\{\tau^{r,S}_{j}(t)>n-1\right\}\cap\left\{\tau^{r,S}_{j}\left(\left(t-\frac{z}{r^{2}\tilde{a}} \right)^{+}\right)\leq n-1\right\}\right)\gamma^{r}_{j}(dz)\\
&=1+ \int_{(0,\infty)}e^{cz}E\left[1+\sum_{n=1}^{\infty}\mathbb{I}_{\left\{\tau^{r,S}_{j}\left(\left(t-\frac{z}{\tilde{a}r^2}\right)^{+}\right)\leq n \right\}\cap \left\{\tau^{r,S}_{j}(t)> n\right\}}\right]\gamma^{r}_{j}(dz)  \nonumber\\
&=1+\int_{(0,\infty)}e^{cz}\left(1+E\left[\tau^{r,S}_{j}(t)-\tau^{r,S}_{j}\left(\left(t-\frac{z}{\tilde{a}r^2}\right)^{+}\right)\right]\right)\gamma^{r}_{j}(dz).  \label{eq:857b}
\end{align}
Let
\begin{equation*}
\hat{\tau}^{r,S,t}_{j}(z)=\min\left\{n\geq 1:\sum_{l=\tau^{r,S}_{j}\left(\left(t-\frac{z}{\tilde{a}r^2}\right)^{+}\right)+1}^{\tau^{r,S}_{j}\left(\left(t-\frac{z}{\tilde{a}r^2}\right)^{+}\right)+n}v^{r}_{j}(l)\geq \frac{\max_{i}\{C_{i}\}z}{\tilde{a}}\right\}
\end{equation*}
and note that since $B^{r}_{j}(r^{2}t)-B^{r}_{j}\left(r^{2}\left( t-\frac{z}{\tilde{a}r^{2}}\right)^{+}\right)\leq \frac{\max_{i}\{C_{i}\}z}{\tilde{a}}$ and
 $B^{r}_{j}\left(r^{2}\left( t-\frac{z}{\tilde{a}r^{2}}\right)^{+}\right)\leq \sum_{l=1}^{\tau^{r,S}_{j}\left(\left(t-\frac{z}{\tilde{a}r^2}\right)^{+}\right)}v^{r}_{j}(l)$ we have
$
\hat{\tau}^{r,S,t}_{j}(z)\geq \tau^{r,S}_{j}(t)-\tau^{r,S}_{j}\left(\left(t-\frac{z}{\tilde{a}r^2}\right)^{+}\right).
$
In addition, if we define
$
\hat{\tau}^{r,S}_{j}(z)=\min\left\{n\geq 1:\sum_{l=1}^{n}v^{r}_{j}(l)\geq \frac{\max_{i}\{C_{i}\}z}{\tilde{a}}\right\}
$
then we have
$
 \hat{\tau}^{r,S,t}_{j}(z) \overset{d}{=}  \hat{\tau}^{r,S}_{j}(z)
$
due to { Lemma} \ref{thm:sameDist}.
Consequently, from \eqref{eq:857b},
\begin{equation}
E\left[e^{c v^{r}_{j}(\tau^{r,S}_{j}(t))}\right]
% =\int_{(0,\infty)}e^{cz}\left(1+E\left[\tau^{r,S}_{j}(t)-\tau^{r,S}_{j}\left(\left(t-\frac{z}{\tilde{a}r^2}\right)^{+}\right)\right]\right)\gamma^{r}_{j}(dz)\\
%\leq \int_{(0,\infty)}e^{cz}\left(1+E\left[\hat{\tau}^{r,S,t}_{j}(z)\right]\right)\gamma^{r}_{j}(dz)
\leq 1+ \int_{(0,\infty)}e^{cz}\left(1+E\left[\hat{\tau}^{r,S}_{j}(z)\right]\right)\gamma^{r}_{j}(dz).
\label{eq:902}
\end{equation}
We will now bound $E\left[\hat{\tau}^{r,S}_{j}(z)\right]$ from above.
For all $r$ sufficiently large we have $\frac{2}{\beta_{j}}\geq \frac{1}{\beta^{r}_{j}} \geq \frac{3}{4\beta_{j}}$ and $2\sigma^{v}_{j}\geq \sigma^{v,r}_{j}\geq \frac{1}{2}\sigma^{v}_{j}$.  Due to Condition \ref{eqn:mgfBnd} there exists $1\le K<\infty$ such that
$
\sup_{r}\left\{ E\left[ v^{r}_{j}(1) \mathbb{I}_{\{v^{r}_{j}(1) > K\}} \right] \right\}\leq \frac{1}{4\beta_{j}}$.
Therefore for all $r$ sufficiently large we have
\begin{eqnarray*}
\frac{3}{4\beta_{j}} \leq \frac{1}{\beta^{r}_{j}} 
\leq E\left[ v^{r}_{j}(1) \mathbb{I}_{\{v^{r}_{j}(1)\leq K\}} \right]+E\left[ v^{r}_{j}(1) \mathbb{I}_{\{v^{r}_{j}(1) > K\}} \right]
\end{eqnarray*}
and so 
\begin{eqnarray*} 
E\left[ v^{r}_{j}(1) \mathbb{I}_{\{v^{r}_{j}(1)\leq K\}} \right] \geq \frac{3}{4\beta_{j}}-E\left[ v^{r}_{j}(1) \mathbb{I}_{\{v^{r}_{j}(1) > K\}} \right] 
\geq \frac{1}{2\beta_{j}}.
\end{eqnarray*}
  In addition,
\begin{multline*}
E\left[ v^{r}_{j}(1) \mathbb{I}_{\{v^{r}_{j}(1)\leq K\}} \right] \leq E\left[ v^{r}_{j}(1) \mathbb{I}_{\{v^{r}_{j}(1) < \frac{1}{4\beta_{j}}\}} \right]+E\left[ v^{r}_{j}(1) \mathbb{I}_{\{\frac{1}{4\beta_{j}}\leq v^{r}_{j}(1)\leq K\}} \right] \\
\leq \frac{1}{4\beta_{j}}P\left(v^{r}_{j}(1) < \frac{1}{4\beta_{j}}\right)+KP\left(\frac{1}{4\beta_{i}}\leq v^{r}_{j}(1) \leq K\right) 
\leq \frac{1}{4\beta_{j}}+KP\left(\frac{1}{4\beta_{j}}\leq v^{r}_{j}(1) \right)
\end{multline*}
so for all $r$ sufficiently large we have
\begin{equation}
P\left(\frac{1}{4\beta_{j}}\leq v^{r}_{j}(1) \right) \geq \frac{1}{K}E\left[ v^{r}_{j}(1) \mathbb{I}_{\{v^{r}_{j}(1)\leq K\}} \right]-\frac{1}{4K \beta_{j}} \geq\frac{1}{2K \beta_{j}}- \frac{1}{4K \beta_{j}}\geq \frac{1}{4K \beta_{j}}. \label{eq:904}
\end{equation}
Define 
$
C^{r}_{j}(n)=\sum_{l=1}^{n}\mathbb{I}_{\{ v^{r}_{j}(l)\geq \frac{1}{4K \beta_{j}} \}}$
and 
$
\zeta^{r}_{j}(z)=\min\{n\geq0:C^{r}_{j}(n)= \lceil 4\frac{z}{\tilde{a}}K\max_i\{C_i\} \beta_{j}\rceil \}
$
and note that 
$
E\left[\hat{\tau}^{r,S}_{j}(z)\right] \leq E\left[ \zeta^{r}_{j}(z)\right].
$
However, because the $\{v^{r}_{j}(l)\}_{l=1}^{\infty}$ are i.i.d. it follows that $\zeta^{r}_{j}(z)$ is just the sum of $ \lceil 4\frac{z}{\tilde{a}}K \max_i\{C_i\}\beta_{j}\rceil$ independent geometric distributions with probability of success $p\geq \frac{1}{4K \beta_{j}}$ for all $r$ sufficiently large which gives
\begin{eqnarray*}
E\left[\hat{\tau}^{r,S}_{j}(z)\right] \le E\left[ \zeta^{r}_{j}(z) \right] 
%\leq \lceil 4\frac{z}{\tilde{a}}K \beta_{j}\rceil 4K \alpha_{j}
\leq (1+ 4\frac{z}{\tilde{a}}K\max_i\{C_i\} \beta_{j}) 4K \beta_{j} 
\leq 4K \beta_{j} + 16\frac{z}{\tilde{a}}K^2\max_i\{C_i\} \beta^{2}_{j}.
\end{eqnarray*}
% Consequently for all $r$ sufficiently large we have
% $
% E\left[\hat{\tau}^{r,S}_{j}(z)\right] \leq 4K \beta_{j} + 16\frac{z}{\tilde{a}}K \beta^{2}_{j}$
% and we can now use this upper bound on $E\left[\hat{\tau}^{r,S}_{j}(z)\right] $ to get an upper bound on $E\left[e^{cv^{r}_{j}(\tau^{r,S}_{j}(t))}\right]$.  
Thus, from \eqref{eq:902} we have, for all $r$ sufficiently large,
\begin{align*}
&E\left[e^{c v^{r}_{j}(\tau^{r,S}_{j}(t))}\right]\leq 1+\int_{(0,\infty)}e^{cz}\left(1+E\left[\hat{\tau}^{r,S}_{j}(z)\right]\right)\gamma^{r}_{j}(dz)\\
&\leq  1+\int_{(0,\infty)}e^{cz}(1+ 4K \beta_{j} + 16\frac{z}{\tilde{a}}K^2\max_i\{C_i\} \beta^{2}_{j})\gamma^{r}_{j}(dz) \\
&\leq 1+\left(1+ 4K \beta_{j}\right)\int_{(0,\infty)}e^{cz} \gamma^{r}_{j}(dz)+ \frac{16K^2 \beta^{2}_{j}}{\tilde{a}}\max_i\{C_i\}\int_{(0,\infty)}ze^{cz}\gamma^{r}_{j}(dz). 
\end{align*}
The result follows on using Condition \ref{eqn:mgfBnd}.
% Due to Condition \ref{eqn:mgfBnd}, the fact that $c<\delta$, and the fact that the $\mathcal{F}^{r}(0,0)$-measurable random variable $t\geq 0$ satisfying $P(X<\infty)=1$ and $j\in\AAA_J$ were arbitrary completes the proof.
\hfill \qed

\subsection{Proof of Proposition \ref{thm:expTailBnd}.}
\label{sec:thmexpTailBnd}
In this section we prove the key large deviation estimates given in { Proposition} \ref{thm:expTailBnd} which have been used on several occasions.

Fix $j\in\AAA_J$.  Since the proof is identical for $A^{r}_{j}$ and $S^{r}_{j}$ we will only present it for $A^{r}_{j}$. Throughout in the proof we suppress the subscript $j$.
Define 
$
C^r(n)\doteq \sum_{l=1}^{n} u^{r}(1)$
and
$
\Lambda^{r}(y)=\log(E[ e^{y(u^{r}(1)-\frac{1}{\alpha^{r}})}]).
$
Note that $y\mapsto \Lambda^{r}(y)$ is infinitely differentiable  for $|y|<\delta$ and due to Jensen's inequality $\Lambda^{r}(y)\geq 0$ for all $|y|<\delta$. Due to Condition \ref{eqn:mgfBnd} there exists $K_{\Lambda}<\infty$ such that 
\begin{eqnarray*}
\sup_{r,|y|\leq\frac{\delta}{2}}\Big|\frac{d^3\Lambda^{r}}{dy^3}(y)\Big|
% \\
% &=&\sup_{r,|y|\leq\frac{\delta}{2}}\left|\frac{E\left[(u^{r}_{j}-\frac{1}{\alpha^{r}_{j}})^3 e^{y(u^{r}_{j}-\frac{1}{\alpha^{r}_{j}})}\right]}{E\left[ e^{y(u^{r}_{j}-\frac{1}{\alpha^{r}_{j}})}\right]}-\frac{3E\left[(u^{r}_{j}-\frac{1}{\alpha^{r}_{j}}) e^{y(u^{r}_{j}-\frac{1}{\alpha^{r}_{j}})}\right]E\left[(u^{r}_{j}-\frac{1}{\alpha^{r}_{j}})^2 e^{y(u^{r}_{j}-\frac{1}{\alpha^{r}_{j}})}\right]}{E\left[ e^{y(u^{r}_{j}-\frac{1}{\alpha^{r}_{j}})}\right]^2}\right. \\
% &&\left.+\frac{2E\left[(u^{r}_{j}-\frac{1}{\alpha^{r}_{j}}) e^{y(u^{r}_{j}-\frac{1}{\alpha^{r}_{j}})}\right]}{E\left[ e^{y(u^{r}_{j}-\frac{1}{\alpha^{r}_{j}})}\right]}\right|\\
\leq K_{\Lambda}.
\end{eqnarray*}
 Since, $\Lambda^{r}(0)=0$, $\frac{d\Lambda^{r}}{dy}(0)=0$,
and
$
\frac{d^2\Lambda^{r}}{dy^2}(0)=\left(\sigma^{u,r}\right)^2
$
we have for  $|y| \le \delta/2$,
% \begin{eqnarray*}
% \Lambda^{r}(y)=\Lambda^{r}(0)+y\frac{d\Lambda^{r}}{dy}(0)+\frac{y^{2}}{2}\frac{d^2\Lambda^{r}}{dy^2}(0)+\frac{y^{3}}{6}\frac{d^3\Lambda^{r}}{dy^3}(y^{*})
% =\frac{y^{2}}{2}\left(\sigma^{u,r}\right)^2+\frac{y^{3}}{6}\frac{d^3\Lambda^{r}}{dy^3}(y^{*})
% \end{eqnarray*}
% which gives
\begin{equation}
\left|\Lambda^{r}(y)-\frac{y^{2}}{2}\left(\sigma^{u,r}\right)^2 \right|\leq \frac{|y|^{3}}{6}K_{\Lambda} \label{eq:959}.
\end{equation}
We will first prove that, for any $c\ge 0$ there exist  $B_{1},B_{2},R\in(0,\infty)$ such that for all $r\geq R$ and $T\geq 0$ we have
\begin{equation}
P\left(\sup_{0\leq t \leq r^{2c}T}|A^{r}(t)-t\alpha^{r}|\geq \epsilon r^{c}T\right)=P\left(\sup_{0\leq t \leq T}|A^{r}( r^{2c}t)- r^{2c}t\alpha^{r}|\geq \epsilon r^{c}T\right)\leq B_{1}e^{-T B_{2}}. \label{eq:1001}
\end{equation}
Due to the fact that $ A^{r}$ is integer-valued and the definition of $C^r$,  for all $t\in [0,T]$ we have
\begin{eqnarray*}
\left\{ A^{r}(r^{2c}t)-\alpha^{r}r^{2c}t\geq  r^{c}T\epsilon  \right\}&=&\left\{ A^{r}(r^{2c}t)\geq \lceil \alpha^{r}r^{2c}t+r^{c}T\epsilon \rceil \right\}\\
&=&\left\{C^r(\lceil \alpha^{r}r^{2c}t +r^{c}T\epsilon \rceil)\leq r^{2c}t \right\}
\end{eqnarray*}
and 
\begin{eqnarray*}
\left\{ A^{r}(r^{2c}t)-\alpha^{r}r^{2c}t\leq  -r^{c}T\epsilon  \right\}&=&\left\{ A^{r}(r^{2c}t)\leq \lfloor \alpha^{r}r^{2c}t-r^{c}T\epsilon \rfloor \right\}\\
&=&\left\{C^r(\lfloor \alpha^{r}r^{2c}t -r^{c}T\epsilon \rfloor+1)> r^{2c}t \right\}.
\end{eqnarray*}
In addition, since for all $t\in [0,T]$ we have 
\begin{equation*}
\frac{1}{\alpha^{r}}\lceil \alpha^{r}r^{2c}t +r^{c}T\epsilon \rceil-\frac{1}{\alpha^{r}}\left( r^{c}T\epsilon\right)\geq r^{2c}t
\end{equation*}
and
\begin{equation*}
\frac{1}{\alpha^{r}}\left(\lfloor \alpha^{r}r^{2c}t -r^{c}T\epsilon \rfloor+1\right)+\frac{1}{\alpha^{r}} \left(r^{c}T\epsilon-1\right)\leq r^{2c}t 
\end{equation*}
it follows that for all $t\in [0,T]$ we have
\begin{align*}
&\left\{C^r(\lceil \alpha^{r}r^{2c}t +r^{c}T\epsilon \rceil)\leq r^{2c}t \right\}\\
% \left\{-C^r_{j}(\lceil \alpha^{r}_{j}r^{2c}t +r^{c}T\epsilon \rceil)\geq- r^{2c}t \right\} \\
&\subset\left\{\frac{1}{\alpha^{r}}\lceil \alpha^{r}r^{2c}t +r^{c}T\epsilon \rceil-C^r(\lceil \alpha^{r}r^{2c}t +r^{c}T\epsilon \rceil)\geq\frac{1}{\alpha^{r}}\left( r^{c}T\epsilon \right) \right\}\\
&\subset\left\{\frac{1}{\alpha^{r}}\lceil \alpha^{r}r^{2c}t +r^{c}T\epsilon \rceil-C^r(\lceil \alpha^{r}r^{2c}t +r^{c}T\epsilon \rceil)>\frac{1}{\alpha^{r}}\left( r^{c}T\epsilon -1\right) \right\}\\
\end{align*}
and
\begin{eqnarray*}
&&\left\{C^r(\lfloor \alpha^{r}r^{2c}t -r^{c}T\epsilon \rfloor+1)> r^{2c}t \right\}\\
% &\subset&\left\{-C^r_{j}(\lfloor \alpha^{r}_{j}r^{2c}t -r^{c}T\epsilon \rfloor+1)<- r^{2c}t \right\}\\
&&\subset\left\{\frac{1}{\alpha^{r}}\left(\lfloor \alpha^{r}r^{2c}t -r^{c}T\epsilon \rfloor+1\right)-C^r(\lfloor \alpha^{r}r^{2c}t -r^{c}T\epsilon \rfloor+1)<-\frac{1}{\alpha^{r}} \left(r^{c}T\epsilon-1\right)\right\}\\
&&=\left\{C^r(\lfloor \alpha^{r}r^{2c}t -r^{c}T\epsilon \rfloor+1)-\frac{1}{\alpha^{r}}\left(\lfloor \alpha^{r}r^{2c}t -r^{c}T\epsilon \rfloor+1\right)>\frac{1}{\alpha^{r}} \left(r^{c}T\epsilon-1\right)\right\}.\\
\end{eqnarray*}
Using these observations, we have, with 
$p_c^r \doteq \lceil \alpha^{r}r^{2c}T+r^{c}\epsilon T \rceil$
\begin{equation}
\left\{\sup_{0\leq t \leq T}\left(A^{r}(r^{2c}t)- \alpha^{r}r^{2c}t \right)\geq r^{c}T\epsilon\right\}\subset \left\{ \sup_{0\leq n \leq p_c^r}\left(\frac{1}{\alpha^{r}}n-C^r(n)\right) >  \frac{1}{\alpha^{r}}\left(r^{c}\epsilon T-1\right) \right\} \label{eq:1009b}
\end{equation}
and with $q_c^r = \lfloor \alpha^{r} r^{2c}T-r^{c}\epsilon T  \rfloor$
\begin{equation}
\left\{\inf_{0\leq t \leq T}\left(A^{r}(r^{2c}t)- \alpha^{r}r^{2c}t \right)\leq -r^{c}T\epsilon\right\}\subset \left\{ \sup_{0\leq n \leq  q_c^r+1}\left(C^r(n)-\frac{1}{\alpha^{r}}n \right)> \frac{1}{\alpha^{r}}\left(r^{c}\epsilon T-1\right) \right\}. \label{eq:1032b}
\end{equation}
For $0<y<\frac{\delta}{4}$
\begin{equation}
P\Big(  \sup_{0\leq n \leq p_c^r}\Big(\frac{1}{\alpha^{r}}n-C^r(n)\Big) >  \frac{1}{\alpha^{r}}\Big(r^{c}\epsilon T-1\Big) \Big)e^{\frac{y}{\alpha^{r}}\Big(r^{c}\epsilon T-1\Big)}
% &\leq&  E\Big[e^{y \sup_{0\leq n \leq p_c^r}\Big(\frac{1}{\alpha^{r}}n-C^r(n)\Big)} \Big]
\leq E\Big[\sup_{0\leq n \leq p_c^r}e^{y \Big(\frac{1}{\alpha^{r}}n-C^r(n)\Big)} \Big]. \label{eq:1019b}
\end{equation}
Since $\Lambda^{r}(-\frac{y}{2})\geq 0$ we have
\begin{align}
E\left[\sup_{0\leq n \leq p_c^r}\left(e^{y \left(\frac{1}{\alpha^{r}}n-C^r(n)\right)}\right) \right]
&= E\left[\sup_{0\leq n \leq p_c^r}\left(e^{y \left(\frac{1}{\alpha^{r}}n-C^r(n)\right)-n2\Lambda^{r}(-\frac{y}{2})}e^{n2 \Lambda^{r}(-\frac{y}{2})}\right) \right]\nonumber\\
&\leq e^{p_c^r 2\Lambda^{r}(-\frac{y}{2})}E\left[\sup_{0\leq n \leq p_c^r}\left(e^{y \left(\frac{1}{\alpha^{r}}n-C^r(n)\right)-n2\Lambda^{r}\left(-\frac{y}{2}\right)}\right) \right]
\label{eq:1014}
\end{align}
Since $e^{\frac{y}{2} \left(\frac{1}{\alpha^{r}}n-C^r(n)\right)-n\Lambda^{r}(-\frac{y}{2})}$ is a nonnegative martingale and
$
Ee^{\frac{y}{2} \left(\frac{1}{\alpha^{r}}n-C^r(n)\right)-n\Lambda^{r}(-\frac{y}{2})}<\infty
$
for all $n\in\mathbb{N}$ the Doob-Kolmogorov inequality gives
\begin{eqnarray*}
E\left[\sup_{0\leq n \leq p_c^r}\left(e^{y \left(\frac{1}{\alpha^{r}}n-C^r(n)\right)-2n\Lambda^{r}(-\frac{y}{2})}\right) \right]
% &=&E\left[\sup_{0\leq n \leq \lceil \alpha^{r}_{j} r^{2c}T+r^{c}\epsilon T  \rceil}\left(e^{\frac{y}{2} \left(\frac{1}{\alpha^{r}_{j}}n-C^r_{j}(n)\right)-n\Lambda^{r}_{j}(-\frac{y}{2})}\right)^{2} \right] \\
% &\leq&4E\left[\left(e^{\frac{y}{2} \left(\frac{1}{\alpha^{r}_{j}}\lceil \alpha^{r}_{j} r^{2c}T+r^{c}\epsilon T  \rceil-C^r_{j}(\lceil \alpha^{r}_{j} r^{2c}T+r^{c}\epsilon T  \rceil)\right)-\lceil \alpha^{r}_{j} r^{2c}T+r^{c}\epsilon T  \rceil\Lambda^{r}_{j}(-\frac{y}{2})}\right)^{2} \right] \\
&\leq& 4E\left[e^{y \left(\frac{1}{\alpha^{r}}p_c^r-C^r(p_c^r)\right)-2p_c^r\Lambda^{r}(-\frac{y}{2})} \right]\\
% &\leq& 4e^{-2p_c^r\Lambda^{r}(-\frac{y}{2})}E\left[e^{y \left(\frac{1}{\alpha^{r}}p_c^r-C^r(p_c^r)\right)} \right]\\
&\leq& 4e^{-2p_c^r\Lambda^{r}(-\frac{y}{2})}e^{p_c^r\Lambda^{r}(-y)}.
%\leq 4e^{p_c^r\Lambda^{r}(-y)}.
\end{eqnarray*}
Consequently, from  \eqref{eq:1009b}, \eqref{eq:1019b} and \eqref{eq:1014},
\begin{align}
&P\left(  \sup_{0\leq t \leq T}\left(A^{r}(r^{2c}t)- \alpha^{r}r^{2c}t \right)\geq r^{c}T\epsilon \right)\nonumber \\
&\leq  4e^{-\frac{y}{\alpha^{r}}\left(r^{c}\epsilon T-1\right)+p_c^r \Lambda^{r}(-y)}
\leq  4e^{-\frac{y}{\alpha^{r}}\left(r^{c}\epsilon T-1\right)+p_c^r \left(\frac{y^{2}}{2}\left(\sigma^{u,r}\right)^2+\frac{|y|^{3}}{6}K_{\Lambda}\right)}\nonumber\\
&\leq 4e^{-\frac{y}{\alpha^{r}}\left(r^{c}\epsilon T-1\right)+\left(\alpha^{r} r^{2c}T+r^{c}\epsilon T +1 \right) \left(\frac{y^{2}}{2}\left(\sigma^{u,r}\right)^2+\frac{|y|^{3}}{6}K_{\Lambda}\right)},
%+\left(\frac{y^{2}}{2}\left(\sigma^{u,r}\right)^2+\frac{|y|^{3}}{6}K_{\Lambda}\right)},
\label{eq:1024b}
\end{align}
where the second inequality is from \eqref{eq:959}.

Choose $R<\infty$ sufficiently large that for all $r\geq R$ we have $ \alpha/2 \le \alpha^{r}\leq 2\alpha$, $\left(\sigma^{u,r}\right)^2\leq 2\left(\sigma^{u}\right)^2$.

Consider first the case where $c>0$.
Define 
$
k=\epsilon/(16\left(\alpha\sigma^{u}\right)^2)$,
and assume without loss that $R$  is large enough that for all $r\geq R$
\begin{equation*}
r^{-c}\frac{1}{6}\alpha^{r}k^{3}K_{\Lambda}+r^{-c}\frac{1}{2}\epsilon k^{2}\left(\sigma^{u,r}\right)^2+r^{-2c}\frac{1}{6}\epsilon k^{3}K_{\Lambda}\leq \frac{k\epsilon }{8\alpha},
\end{equation*}
\begin{equation*}
r^{-c}\frac{k}{ \alpha^{r}}+r^{-2c}\frac{1}{2}k^{2}\left(\sigma^{u,r}\right)^2+r^{-3c}\frac{1}{6}k^{3}K_{\Lambda}\leq 1,\;\; \mbox{ and }
kr^{-c}<\frac{\delta}{4}.
 \end{equation*}
Then for all $r\geq R$, the exponent in \eqref{eq:1024b} with $y = kr^{-c}$ satisfies 
\begin{eqnarray*}
&&-\frac{ kr^{-c}}{\alpha^{r}}\left(r^{c}\epsilon T-1\right)+\left(\alpha^{r} r^{2c}T+r^{c}\epsilon T  \right) \left(\frac{\left( kr^{-c}\right)^{2}}{2}\left(\sigma^{u,r}\right)^2+\frac{\left( kr^{-c}\right)^{3}}{6}K_{\Lambda}\right)\\
&&+\left(\frac{\left( kr^{-c}\right)^{2}}{2}\left(\sigma^{u,r}\right)^2+\frac{\left( kr^{-c}\right)^{3}}{6}K_{\Lambda}\right)\\
&&=-\frac{k\epsilon T}{\alpha^{r}}+\frac{k}{r^{c} \alpha^{r}}+\frac{1}{2}\alpha^{r}Tk^{2}\left(\sigma^{u,r}\right)^2+\frac{1}{6r^{c}}\alpha^{r}Tk^{3}K_{\Lambda}+\frac{1}{2r^{c}}\epsilon Tk^{2}\left(\sigma^{u,r}\right)^2+\frac{1}{6r^{2c}}\epsilon Tk^{3}K_{\Lambda}\\
&&\quad+\frac{1}{2r^{2c}}k^{2}\left(\sigma^{u,r}\right)^2+\frac{1}{6r^{3c}}k^{3}K_{\Lambda}\\
&&=T\left(\frac{1}{2}\alpha^{r}k^{2}\left(\sigma^{u,r}\right)^2-\frac{k\epsilon }{\alpha^{r}}\right)+T\left(r^{-c}\frac{1}{6}\alpha^{r}k^{3}K_{\Lambda}+r^{-c}\frac{1}{2}\epsilon k^{2}\left(\sigma^{u,r}\right)^2+r^{-2c}\frac{1}{6}\epsilon k^{3}K_{\Lambda}\right)\\
&&\quad+r^{-c}\frac{k}{ \alpha^{r}}+r^{-2c}\frac{1}{2}k^{2}\left(\sigma^{u,r}\right)^2+r^{-3c}\frac{1}{6}k^{3}K_{\Lambda}\\
&&\leq T\left(2\alpha k^{2}\left(\sigma^{u}\right)^2-\frac{k\epsilon}{2\alpha}\right)+T\left(\frac{k\epsilon }{8\alpha}\right)+1
=T\left(\frac{k\epsilon }{8\alpha}-\frac{k\epsilon }{2\alpha}\right)+T\left(\frac{k\epsilon }{8\alpha}\right)+1
=-T\frac{k\epsilon }{4\alpha}+1.
\end{eqnarray*}
Consequently for $r\geq R$,  substituting $y=kr^{-c}$ in \eqref{eq:1024b} gives
\begin{equation*}
P\left(  \sup_{0\leq t \leq T}\left(A^{r}(r^{2c}t)- \alpha^{r}r^{2c}t \right)\geq r^{c}T\epsilon \right)
% &\leq &4e^{-\frac{kr^{-c}}{\alpha^{r}}\left(r^{c}\epsilon T-1\right)+\left(\alpha^{r} r^{2c}T+r^{c}\epsilon T  \right) \left(\frac{(kr^{-c})^{2}}{2}\left(\sigma^{u,r}\right)^2+\frac{(kr^{-c})^{3}}{6}K_{\Lambda}\right)+\left(\frac{(kr^{-c})^{2}}{2}\left(\sigma^{u,r}\right)^2+\frac{(kr^{-c})^{3}}{6}K_{\Lambda}\right)}\\
\leq 4e^{-T\frac{k\epsilon }{4\alpha}+1}\leq 4ee^{-T\frac{k\epsilon }{4\alpha}}.
\end{equation*}
Similarly, using the  inequality \eqref{eq:1032b},
an almost identical argument (which we omit for the sake of brevity) provides a similar bound on $P\left( \inf_{0\leq t \leq T}\left(A^{r}(r^{2c}t)- \alpha^{r}r^{2c}t \right)\leq -r^{c}T\epsilon\right)$ and combining the two bounds we have the estimate in \eqref{eq:1001} for $c>0$.

Now consider the case $c=0$ and define 
\begin{equation*}
q=\min\left\{1,\frac{\delta}{4},\frac{\epsilon}{8\alpha\left(\alpha\left(\sigma^{u}\right)^2+ \frac{\alpha}{6}K_{\Lambda}+\epsilon\left(\sigma^{u}\right)^2+  \frac{\epsilon}{6}K_{\Lambda}\right)}\right\}.
\end{equation*}
Then for all $r\geq R$ we have, for the exponent in \eqref{eq:1024b} with $y=q$,
\begin{eqnarray*}
&&-\frac{q}{\alpha^{r}}\left(\epsilon T-1\right) +\left(\alpha^{r} T+\epsilon T  \right) \left(\frac{q^{2}}{2}\left(\sigma^{u,r}\right)^2+\frac{q^{3}}{6}K_{\Lambda}\right)+\left(\frac{q^{2}}{2}\left(\sigma^{u,r}\right)^2+\frac{q^{3}}{6}K_{\Lambda}\right)\\
&&\leq-\frac{\epsilon Tq}{2\alpha}+ \frac{2q}{\alpha}+\alpha^{r} Tq^{2}\left(\sigma^{u}\right)^2+\alpha^{r} T\frac{q^{3}}{6}K_{\Lambda}+\epsilon Tq^{2}\left(\sigma^{u}\right)^2+ \epsilon T\frac{q^{3}}{6}K_{\Lambda}  +q^{2}\left(\sigma^{u}\right)^2+\frac{q^{3}}{6}K_{\Lambda}\\
&&\leq-\frac{\epsilon Tq}{2\alpha}+2q^{2}T\left(\alpha\left(\sigma^{u}\right)^2+ \frac{\alpha}{6}K_{\Lambda}+\epsilon\left(\sigma^{u}\right)^2+  \frac{\epsilon}{6}K_{\Lambda}\right)  +q^{2}\left(\sigma^{u}\right)^2+\frac{2q}{\alpha}+\frac{q^{3}}{6}K_{\Lambda}\\
&&\leq-\frac{\epsilon Tq}{2\alpha}+\frac{\epsilon Tq}{4\alpha}  +q^{2}\left(\sigma^{u}\right)^2+\frac{2q}{\alpha}+\frac{q^{3}}{6}K_{\Lambda}
=-T\frac{\epsilon q}{4\alpha}  +q^{2}\left(\sigma^{u}\right)^2+\frac{2q}{\alpha}+\frac{q^{3}}{6}K_{\Lambda}.
\end{eqnarray*}
 Consequently for $r\geq R$, substituting  $y=q$ in \eqref{eq:1024b} with $c=0$, we have
\begin{eqnarray*}
P\left(  \sup_{0\leq t \leq T}\left(A^{r}(t)- \alpha^{r}t \right)\geq T\epsilon \right)&
% \leq&4e^{ -\frac{q}{\alpha^{r}}\left(\epsilon T-1\right) +\left(\alpha^{r} T+\epsilon T  \right) \left(\frac{q^{2}}{2}\left(\sigma^{u,r}\right)^2+\frac{q^{3}}{6}K_{\Lambda}\right)+\left(\frac{q^{2}}{2}\left(\sigma^{u,r}\right)^2+\frac{q^{3}}{6}K_{\Lambda}\right)}\\
\leq& \left(4e^{q^{2}\left(\sigma^{u}\right)^2+\frac{2q}{\alpha}+\frac{q^{3}}{6}K_{\Lambda}}\right)e^{ -T\frac{\epsilon q}{4\alpha}}.
\end{eqnarray*}
Finally, using the inclusion in \eqref{eq:1032b} with $c=0$ we have,
% \begin{equation*}
% \left\{\inf_{0\leq t \leq T}\left(A^{r}(t)- \alpha^{r}t \right)\leq -T\epsilon\right\}\subset \left\{ \sup_{0\leq n \leq \lfloor \alpha^{r} T-\epsilon T  \rfloor+1}\left(C^r(n)-\frac{1}{\alpha^{r}}n \right)> \frac{1}{\alpha^{r}}\left(\epsilon T-1\right) \right\}.
% \end{equation*}
by an almost identical argument (which is omitted), a similar upper bound on $P\left(\inf_{0\leq t \leq T}\left(A^{r}(t)- \alpha^{r}t \right)\leq -T\epsilon\right)$.
Combining the two bounds we now have the estimate in \eqref{eq:1001} for $c=0$.
We have thus proved \eqref{eq:31.d} with $c_2=0$.
Now \eqref{eq:31.d} for general $c_1, c_2\ge 0$ (and $T$ replaced by $S\ge 0$) follows on taking $c=c_1$ and $T=Sr^{c_{2}}$ in \eqref{eq:1001}. Finally, the inequality in \eqref{eq:31.b} follows on taking $c_1=c_2=\frac{\kappa}{2}$ in \eqref{eq:31.d}. \hfill \qed

\subsection{Proof of Proposition \ref{thm:largeWaitBnd}.}
\label{sec:proofLargeWaitBnd}
The proof for \eqref{eq:eq706b} is the same as that for \eqref{eq:eq706a} and so 
we will only show the latter.  Let $c>0$, $\epsilon>0$, and $j\in\AAA_J$ be arbitrary.  
Once again we suppress $j$ from the notation.
Due to Condition \ref{eqn:mgfBnd} we have
$
\sup_{r>0}\{E e^{\frac{3\delta}{4}v^{r}(1)}\}=\tilde{K}<\infty
$
where $\delta$ is as in Condition \ref{eqn:mgfBnd}.  Choose $R<\infty$ such that for all $r\geq R$ we have
$
\left(r^{2}+1\right)\tilde{K}e^{-r\frac{\delta}{4}c}\leq \frac{\epsilon\delta}{4}.
$
Consequently
\begin{eqnarray*}
E\left[ e^{\frac{\delta}{2}\mathcal{I}_{\{v^{r}(1)\geq r c\}}v^{r}(1)}\right]&=& 1+E\left[ \mathcal{I}_{\{v^{r}(1)\geq r c\}}e^{\frac{\delta}{2}v^{r}(1)}\right]
\leq 1+ E\left[ \mathcal{I}_{\{v^{r}(1)\geq r c\}}e^{-r\frac{\delta}{4}c}e^{\frac{3\delta}{4}v^{r}(1)}\right]\\
&\leq&1+ e^{-r\frac{\delta}{4}c}E\left[ e^{\frac{3\delta}{4}v^{r}(1)}\right]
\leq 1+ \tilde{K}e^{-r\frac{\delta}{4}c}
\leq e^{ \tilde{K}e^{-r\frac{\delta}{4}c}}.
\end{eqnarray*}
So for all $T\geq 1$ and $r\geq R$ we have
\begin{equation*}
E\left[ e^{\frac{\delta}{2}\sum_{l=1}^{\lceil r^{2}T \rceil}\mathcal{I}_{\{v^{r}(l)\geq r c \}}v^{r}_{l}}\right] \leq
%  E\left[ \prod_{l=1}^{\lceil r^{2}T \rceil}e^{\frac{\delta}{2}\mathcal{I}_{\{v^{r}(l)\geq r c\}}v^{r}(l)}\right]
% \leq  \prod_{l=1}^{\lceil r^{2}T \rceil} E\left[e^{\frac{\delta}{2}\mathcal{I}_{\{v^{r}(l)\geq r c\}}v^{r}(l)}\right]\\
% \leq \prod_{l=1}^{\lceil r^{2}T\rceil} e^{ \tilde{K}e^{-r\frac{\delta}{4}c}}
% \leq e^{\lceil Tr^{2}\rceil \tilde{K}e^{-r\frac{\delta}{4}c}}
% \leq e^{ Tr^{2} \tilde{K}e^{-r\frac{\delta}{4}c}+ \tilde{K}e^{-r\frac{\delta}{4}c}}
 e^{ T\left(r^{2}+1\right) \tilde{K}e^{-r\frac{\delta}{4}c}}
\leq e^{ T\frac{\epsilon\delta}{4}}
\end{equation*}
which implies
\begin{equation*}
P\left(\sum_{l=1}^{\lceil r^{2}T \rceil}\mathcal{I}_{\{v^{r}(l)\geq r c \}}v^{r}(l) \geq \epsilon T \right)
% \leq e^{-\frac{\epsilon\delta}{2} T}E\left[ e^{\frac{\delta}{2}\sum_{l=1}^{\lceil r^{2}T \rceil}\mathcal{I}_{\{v^{r}(l)\geq r c\}}v^{r}(l)}\right]
\leq e^{-\frac{\epsilon\delta}{2} T}e^{ T\frac{\epsilon\delta}{4}}
\leq e^{ -\frac{\epsilon\delta}{4} T}.
\end{equation*}
\hfill \qed

\section{General Network and Cost Properties}
\label{sec:gennetm}
\subsection{Proof of Proposition \ref{thm:lambdaEqu}.}
\label{sec:gennet}
From the discussion above Proposition \ref{thm:lambdaEqu}, the result is clearly true when 
$\mathcal{C}_{K}^{h} = \text{ker}(K)$. Consider now the complementary case, namely,
$\text{dim}(\mathcal{C}_{K}^{h})=\text{dim}(\text{ker}(K))-1=J-I-1$.
Let $q\in\mathbb{R}_{+}^{J}$ be arbitrary.  For  $\tilde{q}\in\Lambda\left(KMq\right)$ define $\tilde{v} = (\tilde q - q)/\beta$.
% \begin{equation*}
% \tilde{v}=\left(\frac{\tilde{q}_{1}-q_{1}}{\beta_{1}},\frac{\tilde{q}_{2}-q_{2}}{\beta_{2}},...,\frac{\tilde{q}_{J}-q_{J}}{\beta_{J}}\right).
% \end{equation*}
% so $\tilde{q}=q+\beta \tilde{v}$ where $\beta \tilde{v}$ is defined component-wise so
% \begin{equation*}
% \beta \tilde{v}=\left(\beta_{1} \tilde{v}_{1},\beta_{2} \tilde{v}_{2},...,\beta_{J} \tilde{v}_{J}\right)
% \end{equation*}
Note that 
$
0\leq \tilde{q}_{j}=q_{j}+(\tilde{q}_{j}-q_{j})=q_{j}+\tilde{v}_{j}\beta_{j}
$
for all $j\in\AAA_J$ and
$
K\tilde{v}=KM(\tilde{q}-q)=KM\tilde{q}-KMq=0
$
and so $\tilde{v}\in\Xi(q)$.  In addition, for any $\hat{v}\in\Xi(q)$ define $\hat{q}=q+\beta\hat{v}$ and note that by the definition of $\Xi(q)$ we have $\hat{q}\in\mathbb{R}^{J}_{+}$ and $KM\hat{q}=KM(q+\beta\hat{v})=KMq+K\hat{v}=KMq$ so $\hat{q}\in\Lambda\left(KMq\right)$.  Consequently $\inf_{\tilde{q}\in\Lambda(KMq)}(h\cdot \tilde{q})
= \inf_{\tilde{v}\in\Xi(q)}(h\cdot\left(  q+\beta\tilde{v}\right))$.  Therefore
\begin{align*}
&h\cdot q-\hat{h}\left(KMq\right) = h\cdot q-\inf_{\tilde{q}\in\Lambda(KMq)}(h\cdot \tilde{q})
= h\cdot q-\inf_{\tilde{v}\in\Xi(q)}(h\cdot\left(  q+\beta\tilde{v}\right))
\\
&= \sup_{\tilde{v}\in\Xi(q)}(-h\cdot\left( \beta\tilde{v}\right))
=\sup_{\tilde{v}\in\Xi(q)}\left(-\beta h\cdot \sum_{j=1}^{J-I}u_{j}(u_{j} \cdot \tilde{v} )\right)
= \sup_{\tilde{v}\in\Xi(q)}\left(-\beta h\cdot u_{J-I}(u_{J-I}\cdot\tilde{v} )\right)\\
&= |\lambda|\sup_{\tilde{v}\in\Xi(q)}\left(u_{J-I}\cdot\tilde{v} \right)
=|\lambda|\tilde d(q),
\end{align*}
where the last equality on the second line uses $\beta h \cdot u_j=0$ for $j=1, \ldots J-I-1$, and the last line follows on recalling the definition of $\lambda$ and $\tilde d$. The result follows.
% For any $\hat{v}\in\Xi(q)$ define $\hat{q}=q+\beta\hat{v}$.  Then
% \begin{equation*}
% 0\leq q_{j}+\beta_{j}\hat{v}_{j}=\hat{q}_{j}
% \end{equation*}
% for all $j\in\mathbb{R}^{J}_{+}$ and
% \begin{equation*}
% KM(\hat{q})=KM(q+\beta\hat{v})=KMq+K\hat{v}=KMq
% \end{equation*}
% so $\hat{q}\in\Lambda\left(KMq\right)$.  Therefore
% \begin{eqnarray*}
% |\lambda|\tilde{d}(q)&\geq&|\lambda |\sup_{\hat{v}\in\Xi(q)}\left(u_{J-I}\cdot\hat{v} \right) \\
% &\geq&\left(-\beta h\cdot u_{J-I}\right)\sup_{\hat{v}\in\Xi(q)}\left(\beta h\cdot u_{J-I}(u_{J-I}\cdot\hat{v} )\right)\\
% &\geq&\sup_{\hat{v}\in\Xi(q)}\left(-\beta h\cdot u_{J-I}(u_{J-I}\cdot\hat{v} )\right)\\
% &\geq&\sup_{\hat{v}\in\Xi(q)}\left(-\beta h\cdot \sum_{j=1}^{J-I}u_{j}(u_{j} \cdot \hat{v} )\right)\\
% &\geq&\sup_{\hat{v}\in\Xi(q)}(-h\cdot\left( \beta\hat{v}\right))\\
% &\geq&\sup_{\hat{q}\in\Lambda(KMq)}(h\cdot\left(q-\hat{q}\right))\\
% &\geq&-\inf_{\hat{q}\in\Lambda(KMq)}(h\cdot\left(\hat{q}-q\right))\\
% &\geq&h\cdot q-\inf_{\hat{q}\in\Lambda(KMq)}(h\cdot \hat{q})\\
% &\geq&h\cdot q-\hat{h}(KMq).\\
% \end{eqnarray*}
% Because $q\in\mathbb{R}^{J}_{+}$ was arbitrary this completes the proof.
\hfill \qed

\subsection{Proof of Proposition \ref{thm:fullCapVect}}
\label{sec:fullcapvec}
Fix $z \in \chi_J$.
Due to the local traffic assumption (Condition \ref{cond:locTraffic}) for each $l\in\mathcal{A}_{z}$ we can choose $s_{l}\in\AAA_J$ such that $K_{l,s_{l}}=1$ and $\sum_{i\in\AAA_I}K_{i,s_{l}}=1$.  Define $S=\{s_{l}:l\in\mathcal{A}_{z}\}$, $P=\{j\in\AAA_J:z_{j}=1\}\setminus S$ and $N=\{j\in\AAA_J:z_{j}=0\}\setminus S$.  For $j\in N$ define $v_{j}=-J$, for $j\in P$ define $v_{j}=1$, and for all $l\in\mathcal{A}_{z}$ define $v_{s_{l}}=-\sum_{j\neq s_{l}}K_{l,j}v_{j}$.  For any $l\in\mathcal{A}_{z}$ we have
\begin{equation*}
\sum_{j\in\AAA_J}K_{l,j}v_{j}=\sum_{j\neq s_{l}}K_{l,j}v_{j}+v_{s_{l}}=\sum_{j\neq s_{l}}K_{l,j}v_{j}-\sum_{j\neq s_{l}}K_{l,j}v_{j}=0.
\end{equation*}
% so
% \begin{equation*}
% \sum_{j\in\AAA_J}K_{i,j}v_{j}=0
% \end{equation*}
% for all $i\in\mathcal{A}_{z}$.
This verifies the first statement in the proposition with the above choice of $v$.
Next consider $j \in \NN_j$ such that $z_{j}=1$. If $j \in P$, then by definition $v_j>0$.
To complete the proof consider now $l\in\mathcal{A}_{z}$ with  $z_{s_{l}}=1$.  Due to the definition of $\mathcal{A}_{z}$ we have
$
\sum_{j:z_{j}=0}K_{l,j}\geq 1$
and we also know that $K_{l,s_{i}}=0$ for all $i\in\mathcal{A}_{z}\setminus\{l\}$.  This implies that there exists $j^{*}\in N$ such that $K_{l,j^{*}}=1$ and for all $j\in\mathbb{N_{J}}\setminus\{j^{*},s_{l}\}$ such that $K_{l,j}=1$, either $v_{j}=1$ or $v_{j}=-J$.  Since $v_{j^{*}}=-J$ we have
\begin{equation*}
v_{s_{l}}=-\sum_{j\neq s_{l}}K_{l,j}v_{j}=J-\sum_{j\in\mathbb{N}\setminus\{j^{*},s_{l}\}}K_{l,j}v_{j}\geq J -(J-2)\geq 2.
\end{equation*}
Thus we have shown 
that for all $j\in\AAA_J$ such that $z_{j}=1$ we have $v_{j}>0$. This proves the second statement in the proposition.
\hfill \qed

\subsection{Proofs of Propositions \ref{thm:dTildeDiffBnd}
and \ref{thm:costDiffBndQ}.}
\label{sec:proofsDTildeCostDiffBnds}
We begin with two auxiliary results.
\begin{lemma}
\label{thm:bndCostDifNeg}
There exists a $\hat{B}_{\hat{h}}^{-} \in (0,\infty)$  such that for all $w^{1},w^{2}\in\mathbb{N}^{I}_{+}$ satisfying $w^{1}\geq w^{2}$, we have
\begin{equation*}
\hat{h}(w^{2})\leq \hat{h}(w^{1})+\hat{B}_{\hat{h}}^{-}|w^{1}-w^{2}|_{2}.
\end{equation*}
\end{lemma}

\begin{proof}
For $i\in \AAA_I$ let $\mathcal{C}_{i}=\left\{j\in\AAA_J: K_{i,j}=1 \right\}$.  %Specifically, $\mathcal{C}_{i}$ consists of the jobs which impact server $i$.  
For  $j\in\mathcal{C}_{i}$ define\\
 $\mathcal{O}^{j}_{i}=\left\{l\in\AAA_I: K_{l,j}=1 \text{ and } l\neq i \right\}$.  Thus the set $\mathcal{O}^{j}_{i}$ consists of all the resources that job type $j$ impacts aside from resource $i$.  For each $i\in \AAA_I$ let $s_{i}\in\AAA_J$ be such that $\sum_{l=1}^{I}K_{l,s_{i}}=K_{i,s_{i}}=1$.  For any $j\in\AAA_J$ let $e^{j}\in\mathbb{R}^{J}_{+}$ be the unit vector with one on the $j$th coordinate so $e^{j}_{j}=1$ and $e^{j}_{l}=0$ for $l\neq j$.  Note that for any $c>0$ and $j\in \AAA_J$ if we define $\tilde{w}^{1}=KM c \beta_{j}e^{j}= cK_{\cdot,j}$ (here $K_{\cdot,j}$ is the $j$th column vector of the matrix $K$) and $\tilde{w}^{2}=KM\left(\sum_{l\in\mathcal{O}^{j}_{i}} c \beta_{s_{l}}e^{s_{l}}\right)$ then $\tilde{w}^{2}_{l}=\tilde{w}^{1}_{l}$ for $l\neq i$ but $\tilde{w}^{2}_{i}=0$ and $\tilde{w}^{1}_{i}=c$.  In other words, by replacing $c\beta_{j}e^{j}$ with $\sum_{l\in\mathcal{O}^{i}_{j}}c\beta_{s_{l}}e^{s_{l}}$ in the queue length vector we've reduced the workload for server $i$ by $c$ and we've changed the cost by $c\left( \sum_{l\in\mathcal{O}^{j}_{i}}h_{s_{l}}\beta_{s_{l}}-h_{j}\beta_{j}\right)$.  This is the key idea in the proof.

Define 
\begin{equation*}
\tilde{R_{i}}=\max_{j\in\mathcal{C}_{i}} \left( \sum_{l\in\mathcal{O}^{j}_{i}}h_{s_{l}}\beta_{s_{l}}-h_{j}\beta_{j}\right)^{+} \;\;
\mbox{ and }
\tilde{R}=\max_{i\in \AAA_I}\left\{\tilde{R_{i}} \right\}.
\end{equation*}
Now let $w^{1},w^{2}\in\mathbb{N}^{I}_{+}$ be such that $w^{1}\geq w^{2}$ and let $\tilde{q}^{0}\in\mathbb{R}^{J}_{+}$ satisfy $w^{1}=KM\tilde{q}^{0}$ and $\hat{h}(w^{1})=h\cdot \tilde{q}^{0}$.  Since $ \sum_{j\in\mathcal{C}_{1}}\frac{1}{\beta_{j}}\tilde{q}^{0}_{j}=
\sum_{j\in \AAA_J}K_{1,j}\frac{1}{\beta_{j}}\tilde{q}^{0}_{j} =
w^{1}_{1}\geq w^{1}_{1}-w^{2}_{1}$ we can choose $\tilde{c}^{1}_{j}\in [0, \tilde{q}^{0}_{j}/\beta_{j}]$ for all $j\in\mathcal{C}_{1}$ such that $\sum_{j\in\mathcal{C}_{1}}\tilde{c}^{1}_{j}=w^{1}_{1}-w^{2}_{1}$.  Define
\begin{equation*}
\tilde{q}^{1}=\tilde{q}^{0}+\sum_{j\in\mathcal{C}_{1}}\left(\sum_{l\in\mathcal{O}^{j}_{1}}\tilde{c}^{1}_{j}\beta_{s_{l}}e^{s_{l}}-\tilde{c}^{1}_{j}\beta_{j}e^{j} \right)
\end{equation*}
and note that $\tilde{q}^{1}\in\mathbb{R}^{J}_{+}$,
\begin{equation*}
KM\tilde{q}^{1}=\left(w^{1}_{1}-\sum_{j\in\mathcal{C}_{1}}\tilde{c}^{1}_{j}, w^{1}_{2},w^{1}_{3},...,w^{1}_{I}\right)=\left(w^{2}_{1}, w^{1}_{2},w^{1}_{3},...,w^{1}_{I}\right),
\end{equation*}
and
\begin{eqnarray*}
h\cdot \tilde{q}^{1}&=& h\cdot \tilde{q}^{0} +\sum_{j\in\mathcal{C}_{1}}\left(\sum_{l\in\mathcal{O}^{j}_{1}}\tilde{c}^{1}_{j}h_{s_{l}}\beta_{s_{l}}-\tilde{c}^{1}_{j}h_{j}\beta_{j}\right)\\
&\leq& \hat{h}(w^{1})+\sum_{j\in\mathcal{C}_{1}}\tilde{c}^{1}_{j}\left(\sum_{l\in\mathcal{O}^{j}_{1}}h_{s_{l}}\beta_{s_{l}}-h_{j}\beta_{j}\right)^{+}
\leq  \hat{h}(w^{1}) +\tilde{R}|w^{1}_{1}-w^{2}_{1}|.
\end{eqnarray*}
Now assume that for $k\in\{1,...,I-1\}$ there exists $\tilde{q}^{k}\in \mathbb{R}^{J}_{+}$ such that 
\begin{equation}
KM\tilde{q}^{k}=(w^{2}_{1},w^{2}_{2}, ...,w^{2}_{k},w^{1}_{k+1},w^{1}_{k+2},...,w^{1}_{I}),\;\;
\mbox{ and }
h\cdot \tilde{q}^{k}\leq \hat{h}(w^{1})+\tilde{R}\sum_{i=1}^{k}|w^{1}_{i}-w^{2}_{i}|. \label{eq:409b}
\end{equation}
Since $\sum_{j\in\mathcal{C}_{k+1}}\tilde{q}^{k}_{j}/\beta_{j}=\sum_{j\in\AAA_J}K_{k+1,j}\tilde{q}^{k}_{j}/
\beta_{j}=w^{1}_{k+1}\geq w^{2}_{k+1}$ we can choose $\tilde{c}^{k+1}_{j}\in [0, \tilde{q}^{k}_{j}/\beta_{j}]$ for all $j\in\mathcal{C}_{k+1}$ such that $\sum_{j\in\mathcal{C}_{k+1}}\tilde{c}^{k+1}_{j}=w^{1}_{k+1}-w^{2}_{k+1}$.  Define
\begin{equation*}
\tilde{q}^{k+1}=\tilde{q}^{k}+\sum_{j\in\mathcal{C}_{k+1}}\left(\sum_{l\in\mathcal{O}^{j}_{k+1}}\tilde{c}^{k+1}_{j}\beta_{s_{l}}e^{s_{l}}-\tilde{c}^{k+1}_{j}\beta_{j}e^{j} \right).
\end{equation*}
Then, as before $\tilde{q}^{k+1} \in \RR_+^J$, and from \eqref{eq:409b},
\begin{eqnarray*}
KM\tilde{q}^{k+1}&=&\left(w^{2}_{1},w^{2}_{2}, ...,w^{2}_{k},w^{1}_{k+1}-\sum_{j\in\mathcal{C}_{k+1}}\tilde{c}^{k+1}_{j},w^{1}_{k+2},...,w^{1}_{I}\right)\\
&=&\left(w^{2}_{1},w^{2}_{2}, ...,w^{2}_{k},w^{2}_{k+1},w^{1}_{k+2},...,w^{1}_{I}\right)
\end{eqnarray*}
and
\begin{eqnarray*}
h\cdot \tilde{q}^{k+1}&=& h\cdot \tilde{q}^{k} +\sum_{j\in\mathcal{C}_{k+1}}\left(\sum_{l\in\mathcal{O}^{j}_{k+1}}\tilde{c}^{k+1}_{j}h_{s_{l}}\beta_{s_{l}}-\tilde{c}^{k+1}_{j}h_{j}\beta_{j}\right)\\
&\leq& \hat{h}(w^{1})+\tilde{R}\sum_{i=1}^{k}|w^{1}_{i}-w^{2}_{i}|+\sum_{j\in\mathcal{C}_{k+1}}\tilde{c}^{k+1}_{j}\left(\sum_{l\in\mathcal{O}^{j}_{k+1}}h_{s_{l}}\beta_{s_{l}}-h_{j}\beta_{j}\right)^{+}\\
&\leq& \hat{h}(w^{1})+\tilde{R}\sum_{i=1}^{k}|w^{1}_{i}-w^{2}_{i}| +\tilde{R}|w^{1}_{k+1}-w^{2}_{k+1}|
= \hat{h}(w^{1})+\tilde{R}\sum_{i=1}^{k+1}|w^{1}_{i}-w^{2}_{i}|.
\end{eqnarray*}
By induction this implies that there exists $\tilde{q}^{I}\in \mathbb{R}^{J}_{+}$ such that
\begin{equation*}
KM\tilde{q}^{I}=w^{2}, \mbox{ and, }
h\cdot \tilde{q}^{I} \leq \hat{h}(w^{1})+\tilde{R}\sum_{i=1}^{I}|w^{1}_{i}-w^{2}_{i}|.
\end{equation*}
Since $\sum_{i=1}^{I}|w^{1}_{i}-w^{2}_{i}|\leq \sqrt{J}|w^{1}-w^{2}|_{2}$ and $\hat{h}(w^{2})\leq h\cdot \tilde{q}^{I}$ due to the fact that $KM\tilde{q}^{I}=w^{2}$, we have
\begin{equation*}
\hat{h}(w^{2})\leq \hat{h}(w^{1}) +\tilde{R}\sqrt{J}|w^{1}-w^{2}|_{2}.
\end{equation*}
This completes the proof.
\end{proof}
\begin{lemma}
\label{thm:costDiffBndW}
There exists $\hat{B}_{\hat{h}}\in (0,\infty)$ such that for all $w^{1},w^{2}\in\mathbb{R}^{I}_{+}$ we have
\begin{equation*}
|\hat{h}(w^{1})-\hat{h}(w^{2})|\leq \hat{B}_{\hat{h}}|w^{1}-w^{2}|_{2}.
\end{equation*}
\end{lemma}
\begin{proof}
Let $w^{1},w^{2}\in\mathbb{R}^{I}_{+}$ be arbitrary.    Let $q^{*}\in \mathbb{R}_{+}^{J}$ satisfy $KMq^{*}=w^{1}\wedge w^{2}$ and $h\cdot q^{*}=\hat{h}\left(w^{1}\wedge w^{2}\right)$.  Note that for every $i\in\AAA_I$ we have $|w^{1}_{i}-w^{2}_{i}|=|w^{1}_{i}-w^{1}_{i}\wedge w^{2}_{i}|+|w^{2}_{i}-w^{1}_{i}\wedge w^{2}_{i}|$ so 
\begin{equation*}
|w^{1}-w^{1}\wedge w^{2}|_{2}\leq |w^{1}-w^{2}|_{2}, \mbox{ and, }
|w^{2}-w^{1}\wedge w^{2}|_{2}\leq |w^{1}-w^{2}|_{2},
\end{equation*}
 and 
\begin{equation*}
 |w^{1}-w^{2}|_{2}\leq |w^{1}-w^{1}\wedge w^{2}|_{2}+|w^{2}-w^{1}\wedge w^{2}|_{2}.  
 \end{equation*}
Due to Lemma \ref{thm:bndCostDifNeg} we have
\begin{equation*}
\hat{h}(w^{1}\wedge w^{2}) \leq \hat{h}(w^{1})+\hat{B}^{-}_{\hat{h}}|w^{1}-w^{1}\wedge w^{2}|_{2}\leq \hat{h}(w^{1})+\hat{B}^{-}_{\hat{h}}|w^{1}-w^{2}|_{2}
\end{equation*}
and
\begin{equation*}
\hat{h}(w^{1}\wedge w^{2})\leq  \hat{h}(w^{2}) +\hat{B}^{-}_{\hat{h}}|w^{2}-w^{1}\wedge w^{2}|_{2}\leq  \hat{h}(w^{2})+\hat{B}^{-}_{\hat{h}}|w^{1}-w^{2}|_{2}.
\end{equation*}
Let $\hat{B}_{\hat{h}}^{+}\in (0,\infty)$ be such that for all $w\in\mathbb{R}^{I}_{+}$ we have $\hat{h}(w)\leq \hat{B}_{\hat{h}}^{+}|w|_{2}$.  Then, using the definition of $\hat h$,
\begin{eqnarray*}
\hat{h}(w^{1})&\leq& \hat{h}(w^{1}\wedge w^{2})+\hat{h}(w^{1}-w^{1}\wedge w^{2})\leq \hat{h}(w^{1}\wedge w^{2})+\hat{B}^{+}_{\hat{h}}|w^{1}-w^{1}\wedge w^{2}|_{2}\\
&\leq& \hat{h}(w^{1}\wedge w^{2})+\hat{B}^{+}_{\hat{h}}|w^{1}- w^{2}|_{2}
\end{eqnarray*}
and similarly
\begin{equation*}
\hat{h}(w^{2})\leq \hat{h}(w^{1}\wedge w^{2})+\hat{B}^{+}_{\hat{h}}|w^{1}- w^{2}|_{2}.
\end{equation*}
Consequently
\begin{equation*}
|\hat{h}(w^{1})-\hat{h}(w^{1}\wedge w^{2})|\leq \max\{\hat{B}^{+}_{\hat{h}},\hat{B}^{-}_{\hat{h}}\}|w^{1}- w^{2}|_{2}
\end{equation*}
 and 
\begin{equation*}
|\hat{h}(w^{2})-\hat{h}(w^{1}\wedge w^{2})|\leq \max\{\hat{B}^{+}_{\hat{h}},\hat{B}^{-}_{\hat{h}}\}|w^{1}- w^{2}|_{2}
\end{equation*}
and so
\begin{eqnarray*}
|\hat{h}(w^{1})-\hat{h}(w^{2})|&\leq&|\hat{h}(w^{1})-\hat{h}(w^{1}\wedge w^{2})|+|\hat{h}(w^{2})-\hat{h}(w^{1}\wedge w^{2})|\\
&\leq& 2\max\{\hat{B}^{+}_{\hat{h}},\hat{B}^{-}_{\hat{h}}\}|w^{1}- w^{2}|_{2}.
\end{eqnarray*}
This completes the proof.
\end{proof}

We can now complete the proof of Proposition \ref{thm:dTildeDiffBnd}.\\

\noindent {\bf Proof of Proposition \ref{thm:dTildeDiffBnd}}
For any $q^{1},q^{2}\in\mathbb{R}^{J}_{+}$, from Proposition  \ref{thm:lambdaEqu}, we have
\begin{align*}
&|\lambda|(\tilde{d}(q^{2})-\tilde{d}(q^{1})) = h\cdot q^{2} -\hat{h}(KMq^{2})-\left(h\cdot q^{1}-\hat{h}(KMq^{1}) \right)\\
&= h\cdot q^{2} -h\cdot q^{1}-\left(\hat{h}(KMq^{2})-\hat{h}(KMq^{1}) \right)
\leq h\cdot \left(q^{2} -q^{1}\right) + |\hat{h}(KMq^{2})-\hat{h}(KMq^{1}) |.
\end{align*}
The result now follows from Lemma \ref{thm:costDiffBndW} on observing that we can find  $R \in(0,\infty)$  such that for all $x \in \RR^J$ and $r \ge R$,
$|KMx|_2 \le  2 |KM^rx|_2$.
% Due to Theorem \ref{thm:costDiffBndW} there exists a constant $B_{\hat{h}}<\infty$ such that for all $w^{1},w^{2}\in\mathbb{R}^{I}_{+}$ we have
% $
% |\hat{h}(w^{1})-\hat{h}(w^{2})|\leq B_{\hat{h}}|w^{1}-w^{2}|_{2}$.
% Consequently
% \begin{eqnarray*}
% \tilde{d}(q^{2})-\tilde{d}(q^{1})&\leq& h\cdot \left(q^{2} -q^{1}\right) + |\hat{h}(MKq^{2})-\hat{h}(MKq^{1}) |
% \leq h\cdot \left(q^{2} -q^{1}\right) +B_{\hat{h}}|KMq^{2}-KMq^{1}|_{2}
% \end{eqnarray*}
% and because $q^{1},q^{2}\in\mathbb{R}^{J}_{+}$ were arbitrary this completes the proof.
\hfill \qed

Finally we complete the proof of Proposition \ref{thm:costDiffBndQ}.\\

\noindent {\bf Proof of Proposition  \ref{thm:costDiffBndQ}}
Let  $q^{1},q^{2}\in\mathbb{R}^{J}_{+}$ be arbitrary.  Due to Proposition \ref{thm:dTildeDiffBnd} we have
\begin{eqnarray*}
|\lambda||\tilde{d}(q^{2})-\tilde{d}(q^{1})| &\leq& |h\cdot \left(q^{2}-q^{1}\right)|+B_{\hat{h}}|KMq^{2}-KMq^{1}|_{2}\\
&\leq& |h|_{2}|q^{2}-q^{1}|_{2}+ B_{\hat{h}}|KM||q^{2}-q^{1}|_{2}
\leq (|h|_{2}+ B_{\hat{h}}|KM|)|q^{2}-q^{1}|_{2}
\end{eqnarray*}
which completes the proof.
\hfill \qed

\section{Proofs of Propositions \ref{thm:tightMeas} and \ref{thm:convMeas}}
\label{sec:tightness}
Results analogous to  Propositions \ref{thm:tightMeas} and \ref{thm:convMeas} for exponential primitives were studied in \cite{budjoh} and therefore we only give proof sketches here.
 
\subsection{Proof of Proposition \ref{thm:tightMeas}.}
  The fact that $J^{r}_{E}(B^{r},y^{r})<\infty$ for $r$ sufficiently large follows from Proposition \ref{thm:workloadExpBnd}  and the fact that there exists a constant $B<\infty$ such that $h\cdot q \leq B|w|$ for all $q \in \Lambda(w)$ and $w\in\mathbb{R}^{I}_{+}$.
%  , and the assumptions that $\sup_{r}\hat{q}^{r}<\infty$ and $\sup_{r}r\hat{\Upsilon}^{r}<\infty$.  
  To prove tightness of $\{\nu^r\}$ it is sufficient to show that the two marginals are tight. The tightness of $\nu^{r}_{(1)}$ follows immediately from Proposition \ref{thm:workloadExpBnd}.  To show the tightness of the second marginal, $\nu^{r}_{(2)}(dx)$, since $x(0)=0$ $\nu_{(2)}^{r}$-a.s. it is sufficient to show that for any $\epsilon_{1},\epsilon_{2}>0$ there exists $\delta>0$ and $R<\infty$ such that for all $r\geq R$ we have
\begin{equation}
E\left[\nu^{r}_{(2)}\left( \left\{ \sup_{s,t\in [0,1],|s-t|<\delta}\|x(s)-x(t)\|>\epsilon_{1}\right\} \right)\right]<\epsilon_{2}.\label{eq:XhatTightCond}
\end{equation}
Note that the left side above equals
$$\frac{1}{T_{r}}\int_{0}^{T_{r}} P\left( \sup_{s,t\in [u,u+1],|s-t|<\delta}\|\hat{X}^{r}(s)-\hat{X}^{r}(t)\|>\epsilon_{1}\right)du$$
and for any $t,u\geq 0$ we have
\begin{align}
\hat{X}^{r}(u+t)=&\hat{X}^{r}(u)+ KM^{r}\hat{A}^{r,u}\left(\left(t-\bar{\Upsilon}^{A,r}(u)\right)^{+}\right)+\frac{1}{r}KM^{r}\mathcal{I}_{\{t\geq\bar{\Upsilon}^{A,r}(u)>0\}}\nonumber\\
&-KM^{r}\hat{S}^{r,u}\left(\left(\bar{B}^{r}(t+u)-\bar{B}^{r}(u)-\bar{\Upsilon}^{S,r}(u)\right)^{+}\right)\nonumber \\
&-\frac{1}{r}KM^{r}\mathcal{I}_{\{\bar{B}^{r}(t+u)-\bar{B}^{r}(u)\geq\bar{\Upsilon}^{S,r}(u)>0\}} 
+rtK(\rho^{r}-\rho)\nonumber\\
 &-rK\rho^{r}\left(t\wedge \bar{\Upsilon}^{A,r}(u)\right) 
+rK\left((\bar{B}^{r}(t+u)-\bar{B}^{r}(u))\wedge \bar{\Upsilon}^{S,r}(u)\right).\label{eq:XhatDiff} 
\end{align}
From Proposition \ref{thm:expTailBnd} and Lemma \ref{thm:sameDist} it follows that for any $\epsilon_{1},\epsilon_{2}>0$ there exists $\delta>0$ and $R<\infty$ such that for all $r\geq R$ and $j\in\AAA_J$ we have
\begin{equation*}
\frac{1}{T_{r}}\int_{0}^{T_{r}} P\left( \sup_{s,t\in [u,u+1],|s-t|<\delta}\|\hat{A}_{j}^{r,u}(s)-\hat{A}_{j}^{r,u}(t)\|>\epsilon_{1}\right)du<\epsilon_{2}
\end{equation*}
and
\begin{equation*}
\frac{1}{T_{r}}\int_{0}^{T_{r}} P\left( \sup_{s,t\in [u,u+\max_{i}\{C_{i}\}],|s-t|<\max_{i}\{C_{i}\}\delta}\|\hat{S}_{j}^{r,u}(s)-\hat{S}_{j}^{r,u}(t)\|>\epsilon_{1}\right)du<\epsilon_{2}.
\end{equation*}
In addition, due to Proposition \ref{thm:nextArrTimeBnd}, Proposition \ref{thm:nextSerTimeBnd}, and the assumption that $\sup_{r}r\hat{\Upsilon}^{r}<\infty$ it follows that for any $\epsilon_{1},\epsilon_{2}>0$ there exists $\delta>0$ and $R<\infty$ such that for all $r\geq R$ we have
\begin{equation*}
\frac{1}{T_{r}}\int_{0}^{T_{r}} P\left( \|rK\rho^{r}\bar{\Upsilon}^{A,r}(u)\|>\epsilon_{1}\right)du<\epsilon_{2}
\end{equation*}
and
\begin{equation*}
\frac{1}{T_{r}}\int_{0}^{T_{r}} P\left( \|rK\bar{\Upsilon}^{S,r}(u)\|>\epsilon_{1}\right)du<\epsilon_{2}.
\end{equation*}
Also note that $rK(\rho^{r}-\rho)\rightarrow\theta$  (due to Condition \ref{cond:htc} and the paragraph that follows) and for all $s,t\geq 0$ and $j\in\AAA_J$ we have $|\bar{B}_{j}^{r}(s)-\bar{B}_{j}^{r}(t)|\leq \max_{i}\{C_{i}\}(t-s)$.
These observations, together with the form of $\|\hat{X}^{r}(s)-\hat{X}^{r}(t)\|$ for $s,t\geq u$ given by \eqref{eq:XhatDiff}, give \eqref{eq:XhatTightCond} and complete the proof.
\hfill \qed

\subsection{Proof of Proposition \ref{thm:convMeas}.}
Let $\nu^*$ and $(w,x)$ be as in the statement of the proposition. In what follows we let $P_{\nu^{*}(\omega)}$ and $E_{\nu^{*}(\omega)}$ denote probability and expectation under $\nu^{*}(\omega)$.  The proof that for a.e. $\omega$,  $P_{\nu^{*}(\omega)}(x\in\mathcal{C}([0,1]:\RR^I)=1$ is the same as the proof of \cite[Theorem 15 part 1]{budjoh}.  To complete the proof of part (a) it suffices to show that for any $f\in\mathcal{C}^{2}_{c}(\mathbb{R}^{I})$ (which is  the space of continuous, real-valued functions on $\mathbb{R}^{I}$ with compact support and continuous first and second derivatives) $f(x(t))-\int_{0}^{t}\mathcal{L}f(x(s))ds$ is a $\sigma(w,x(s):s\leq t)$-martingale under $\nu^{*}(\omega)$ for a.e $\omega$ where
\begin{equation*}
 \mathcal{L}f(y)=\sum_{i=1}^{I}\theta_{i}\frac{\partial f}{\partial y_{i}}(y)+\frac{1}{2}\sum_{i=1}^{I}\sum_{j=1}^{I}\Sigma_{i,j}\frac{\partial^2 f}{\partial y_{i} \partial y_{j}}(y), \; y \in \RR^I.
\end{equation*}
  Let $0\leq s< t \leq 1$ be fixed and let $g_{s}\in\mathcal{C}_{b}(\mathbb{R}_{+}^{I}\times\mathcal{D}([0,s]))$ and $f\in\mathcal{C}^{2}_{c}(\mathbb{R}^{I})$ be arbitrary (here $\mathcal{C}_{b}(\mathbb{R}_{+}^{I}\times\mathcal{D}([0,s])$ denotes the space of continuous, bounded real-valued functions on $\mathbb{R}_{+}^{I}\times\mathcal{D}([0,s])$).
Denoting, for $s \in [0,1]$, by $x_s$ the restriction of $x$ on $[0,s]$, we  see that 
\begin{equation*}
 g_{s}(w,x_s)\left (f(x(t))-f(x(s))-\int_{s}^{t}\mathcal{L}f(x(z))dz \right)
\end{equation*}
 is a bounded, continuous function on $\mathbb{R}_{+}^{I}\times \mathcal{D}([0,1]:\mathbb{R}^{I})$ so 
\begin{align}
&E\left[E_{\nu^{*}(\omega)}\left[g_{s}(w,x_s)\left(f(x(t))-f(x(s))-\int_{s}^{t}\mathcal{L}f(x(u))du\right) \right]^2\right]\nonumber\\
&=\lim_{m\rightarrow\infty}E\left[E_{\nu^{r_{m}}(\omega)}\left[g_{s}(w,x_s)\left(f(x(t))-f(x(s))-\int_{s}^{t}\mathcal{L}f(x(u))du\right) \right]^2\right].\label{eq:900mm}
\end{align}  
For  $(w,y)\in\mathbb{R}^{I}_{+}\times\sD^{I}$  and $0\le s \le 1$, define
\begin{equation*}
G^{u}_{s}(w,y)=g_{s}(w,[y(u+\cdot)-y(u)]_s),
\end{equation*}
where $[y(u+\cdot)-y(u)]_s$ is the restriction of $y(u+\cdot)-y(u)$ on $[0,s]$ and for $0\le s \le t \le 1$,
\begin{equation*}F^{u}_{s,t}(y)=f(y(u+t)-y(u))-f(y(u+s)-y(u))-\int_{s}^{t}\mathcal{L}f(y(u+z)-y(u))dz.
\end{equation*}
  Then the expectation in \eqref{eq:900mm} can be written as 

\begin{align*}
&E\left[\left(\frac{1}{T_{r_{m}}}\int_{0}^{T_{r_{m}}}G^{u}_{s}(\hat{W}^{r_{m}}(u),\hat{X}^{r_{m}})F^{u}_{s,t}(\hat{X}^{r_{m}})du \right)^2\right]\\
% &=E\left[\frac{1}{T^{2}_{r_{m}}}\int_{0}^{T_{r_{m}}}\int_{0}^{T_{r_{m}}}G^{u}_{s}(\hat{W}^{r_{m}}(u),\hat{X}^{r_{m}})F^{u}_{s,t}(\hat{X}^{r_{m}}) G^{v}_{s}(\hat{W}^{r_{m}}(v),\hat{X}^{r_{m}})F^{v}_{s,t}(\hat{X}^{r_{m}})dvdu\right]\\
&=E\left[\frac{2}{T^{2}_{r_{m}}}\int_{0}^{T_{r_{m}}}\int_{0}^{u}G^{u}_{s}(\hat{W}^{r_{m}}(u),\hat{X}^{r_{m}})F^{u}_{s,t}(\hat{X}^{r_{m}}) G^{v}_{s}(\hat{W}^{r_{m}}(v),\hat{X}^{r_{m}})F^{v}_{s,t}(\hat{X}^{r_{m}})dvdu\right]
\end{align*}

It can be shown using Proposition \ref{thm:workloadExpBnd}, Proposition \ref{thm:expTailBnd}, Lemma \ref{thm:sameDist}, Propositions \ref{thm:nextArrTimeBnd} and \ref{thm:nextSerTimeBnd}, and the assumptions that $\sup_{r}\hat{q}^{r}<\infty$ and $\sup_{r}r\hat{\Upsilon}^{r}<\infty$, that
\begin{equation}
\lim_{m\rightarrow\infty}\sup_{u\geq 0} P\left(\sup_{t\in [0,1]}\|\bar{B}^{r}(t+u)-\bar{B}^{r}(u)-t\rho\|>\epsilon\right)= 0.
\label{eq:BBarToRhoU}
\end{equation}
The proof of the above assertion is very similar to that of  \cite[Theorem 15 part 2 ]{budjoh} and is therefore omitted.

For any $u\geq 0$ define
\begin{equation*}
\tilde{X}^{r,u}(t)=KM^{r}\hat{A}^{r,u}\left(t\right)-KM^{r}\hat{S}^{r,u}\left(\rho t\right)+rtK(\rho^{r}-\rho).
\end{equation*}
and 
\begin{equation*}
\breve{X}^{r_{m},u}(t)\doteq 
\begin{cases}
\hat{X}^{r_{m}}(t) &\mbox{ if } u\geq t\\
\hat{X}^{r_{m}}(u)+\tilde{X}^{r,u}(t-u)&\mbox{ otherwise.}
\end{cases}
\end{equation*}
Note that due to Lemma \ref{thm:sameDist} the distribution of $\tilde{X}^{r_{m},u}$ does not depend on $u$, meaning $\tilde{X}^{r_{m},u}\overset{d}{=}\tilde{X}^{r_m,0}$ for all $u\geq 0$, and from the central limit theorem for renewal processes (see e.g. \cite[Theorem 14.6]{BillingsleyConv})
 $\tilde{X}^{r,0}\rightarrow\hat{X}$ in distribution on $\sD^{I}$ where $\hat{X}(\cdot)$ is as introduced above \eqref{eq:eqrbm}.  Since $\hat{X}\in\sC^{I}$ a.s., it follows from \eqref{eq:BBarToRhoU}, Propositions \ref{thm:nextArrTimeBnd} and \ref{thm:nextSerTimeBnd}, and the assumption that $\sup_{r}\hat{\Upsilon}^{r}<\infty$ that, for any $\epsilon>0$,
\begin{equation*}
\lim_{m\rightarrow\infty}\sup_{u\geq 0} P\left(\sup_{z\in [0,1]}\|\hat{X}^{r_{m}}(u+z)-\breve{X}^{r_{m},u}(u+z)\|>\epsilon\right)= 0.
\end{equation*}
Since $f$ and it's first and second derivatives are bounded and uniformly continuous it follows that 
\begin{equation}
\lim_{m\rightarrow\infty}\sup_{u\geq 0} E\left[\|F^{u}_{s,t}(\hat{X}^{r_{m}})-F^{u}_{s,t}(\breve{X}^{r_{m},u+s})\|\right]= 0.\label{eq:diffFXHatFXBreve}
\end{equation}
and using the fact that the distribution of $\tilde{X}^{r_{m},u}$ does not depend on $u$, that $f \in \clc^2_c(\RR^{I})$ (in particular it has compact support), and that
\begin{equation*}
E\left[f(\hat{X}(t-s)+x)-f(x)-\int_{s}^{t}\mathcal{L}f(\hat{X}(z-s)+x)dz\right]=0
\end{equation*}
for all $x\in\mathbb{R}^{I}$ we have
\begin{equation}
\lim_{m\rightarrow\infty}\sup_{x\in\mathbb{R}^{I},u\geq 0}\left\{E\left[f(\tilde{X}^{r_{m},u+s}(t-s)+x)-f(x)-\int_{s}^{t}\mathcal{L}f(\tilde{X}^{r_{m},u+s}(z-s)+x)dz\right]\right\}= 0. \label{eq:fXTilde0}
\end{equation}
Recall that 
\begin{align*}
F^{u}_{s,t}(\breve{X}^{r_{m},u+s}) =&f(\tilde{X}^{r_{m},u+s}(t-s)+\hat{X}^{r_{m}}(u+s)-\hat{X}^{r_{m}}(u))-f(\hat{X}^{r_{m}}(u+s)-\hat{X}^{r_{m}}(u))\\
&-\int_{s}^{t}\mathcal{L}f(\tilde{X}^{r_{m},u+s}(z-s)+\hat{X}^{r_{m}}(u+s)-\hat{X}^{r_{m}}(u))dz
\end{align*}
and note that since $\tilde{X}^{r_{m},u+s}$ is independent of $\clg^{r_{m}}(u+s)$ and $\hat{X}^{r_{m}}(u+s)-\hat{X}^{r_{m}}(u)$ is $\clg^{r_{m}}(u+s)$-measurable we have
\begin{align}
&\sup_{u\geq 0} E\left[F^{u}_{s,t}(\breve{X}^{r_{m},u+s}) | \clg^{r_{m}}(u)\right]\nonumber\\
& \leq\sup_{x\in\mathbb{R}^{I},u\geq 0} E\left[f(\tilde{X}^{r_{m},u+s}(t-s)+x)-f(x)-\int_{s}^{t}\mathcal{L}f(\tilde{X}^{r_{m},u+s}(z-s)+x)dz\right]. \label{eq:FXBreveUnif}
\end{align}
The fact that $F^{u}_{s,t}$ and $G^{u}_{s}$ are uniformly bounded in $u$ gives
\begin{align*}
&\lim_{m\rightarrow \infty}E \frac{2}{T^{2}_{r_{m}}}\int_{0}^{T_{r_{m}}}\int_{0}^{u}G^{u}_{s}(\hat{W}^{r_{m}}(u),\hat{X}^{r_{m}})F^{u}_{s,t}(\hat{X}^{r_{m}}) G^{v}_{s}(\hat{W}^{r_{m}}(v),\hat{X}^{r_{m}})F^{v}_{s,t}(\hat{X}^{r_{m}})dvdu \\
&=\lim_{m\rightarrow \infty}E\frac{2}{T^{2}_{r_{m}}}\int_{0}^{T_{r_{m}}}\int_{0}^{u-1}G^{u}_{s}(\hat{W}^{r_{m}}(u),\hat{X}^{r_{m}})F^{u}_{s,t}(\breve{X}^{r_{m},u+s})\\
& \hspace{1.75in} G^{v}_{s}(\hat{W}^{r_{m}}(v),\hat{X}^{r_{m}})F^{v}_{s,t}(\hat{X}^{r_{m}})dvdu \\
&=\lim_{m\rightarrow \infty}E\frac{2}{T^{2}_{r_{m}}}\int_{0}^{T_{r_{m}}}\!\!\!\!\int_{0}^{u-1}\!\!E\left[F^{u}_{s,t}(\breve{X}^{r_{m},u+s})| \clg^{r_{m}}(u+s) \right]\! G^{u}_{s}(\hat{W}^{r_{m}}(u),\hat{X}^{r_{m}}) \\
&\hspace{1.75in}G^{v}_{s}(\hat{W}^{r_{m}}(v),\hat{X}^{r_{m}})F^{v}_{s,t}(\hat{X}^{r_{m}})dvdu=0
\end{align*}
where the first equality comes from \eqref{eq:diffFXHatFXBreve}, and the second comes from the fact that for $u-1\geq v$ $G^{u}_{s}(\hat{W}^{r_{m}}(u),\hat{X}^{r_{m}})G^{v}_{s}(\hat{W}^{r_{m}}(v),\hat{X}^{r_{m}})F^{v}_{s,t}(\hat{X}^{r_{m}})$ and $\hat{X}^{r_{m}}(u+s)-\hat{X}^{r_{m}}(u)$ are $\clg^{r_{m}}(u+s)$-measurable, and the third comes from \eqref{eq:FXBreveUnif} and \eqref{eq:fXTilde0}.  
% In addition, since $F^{u}_{s,t}$ and $G^{u}_{s}$ are uniformly bounded in $u$ and $T_{r_{m}}\rightarrow\infty$ we have
% \begin{align*}
% &\lim_{m\rightarrow \infty}E\left[\frac{2}{T^{2}_{r_{m}}}\int_{0}^{T_{r_{m}}}\int_{(u-1)\vee 0}^{u}G^{u}_{s}(\hat{W}^{r_{m}}(u),\hat{X}^{r_{m}})F^{u}_{s,t}(\hat{X}^{r_{m}}) G^{v}_{s}(\hat{W}^{r_{m}}(v),\hat{X}^{r_{m}})F^{v}_{s,t}(\hat{X}^{r_{m}})dvdu\right]=0
% \end{align*}
Putting this all together gives $$E\left[E_{\nu^{*}(\omega)}\left[g_{s}(w,x(\cdot))(f(x(t))-f(x(s))-\int_{s}^{t}\mathcal{L}f(x(u))du \right]^2\right]=0.$$  
Proof of (a) now follows by a standard separability argument.
% so $E_{\nu^{*}(\omega)}\left[g_{s}(w,x(\cdot))(f(x(t))-f(x(s))-\int_{s}^{t}\mathcal{L}f(x(u))du \right]=0$ for a.e. $\omega$.  Because $g_{s}\in\mathcal{C}_{b}(\mathbb{R}_{+}\times\mathcal{D}([0,s]))$ was arbitrary it follows that for  for a.e. $\omega$ we have $E_{\nu^{*}(\omega)}\left[\mathcal{I}_{A_{i}}(f(x(t))-f(x(s))-\int_{s}^{t}\mathcal{L}f(x(u))du \right]=0$ all sets in a countable $\pi$-system $\{A_{i}\}_{i=1}^{\infty}$ which generates $\sigma(w,x(z):z\leq s)$ meaning
% \begin{equation}
% E_{\nu^{*}(\omega)}\left[ f(x(t)-x(s)-\int_{s}^{t}\mathcal{L}f(x(z))dz\Big|\sigma(w,x(z):z\leq s)\right]=0\label{eq:mart}
% \end{equation}
% for a.e. $\omega$.  Since $f\in\mathcal{C}^{2}_{c}(\mathbb{R}^{I})$ was arbitrary and $\mathcal{C}^{2}_{c}(\mathbb{R}^{I})$ is separable under the norm $||f||=\sum_{|\alpha|\leq 2}||D^{\alpha}f||_{\infty}$\footnote{I don't have a reference but I'm pretty sure it's true - DJ} it follows for a.e. $\omega$  \eqref{eq:mart} holds for all $f\in\mathcal{C}^{2}_{c}(\mathbb{R}^{I})$.  Finally, since $0\leq s <t\leq 1$ were arbitrary and $P_{\nu^{*}(\omega)}(x\in\mathcal{C}([0,1])=1$ for a.e. $\omega$ it follows that for a.e. $\omega$ \eqref{eq:mart} holds for all $0\leq s <t\leq 1$ and $f\in\mathcal{C}^{2}_{c}(\mathbb{R}^{I})$ which completes the proof of part (a).\\

  We will now prove part (b).  Let $f\in\mathcal{C}_{c}(\mathbb{R}^{I})$ (the space of continuous functions on $\mathbb{R}^{I}$ with compact support)) and $s\in [0,1]$ be arbitrary 
 From part (a), $P_{\nu^{*}(\omega)}(x\in\mathcal{C}([0,1]: \RR^I)=1$ for a.e. $\omega$ which implies
 % \begin{equation*}
 % \lim_{m\rightarrow\infty}\int_{\mathbb{R}^{I}_{+}\times D([0,1]: \mathbb{R}^d)}f(w)\nu^{r_{m}}(dw,dx)(\omega)=\int_{\mathbb{R}^{I}_{+}\times D([0,1]: \mathbb{R}^d)}f(w)\nu^{*}(dw,dx)(\omega)
 % \end{equation*}
 % and
 % \begin{equation*}
 % \lim_{m\rightarrow\infty}\int_{\mathbb{R}^{I}_{+}\times D([0,1]: \mathbb{R}^d)}f(\Gamma(w+x(\cdot))(s))\nu^{r_{m}}(dw,dx)(\omega)=\int_{\mathbb{R}^{I}_{+}\times D([0,1]: \mathbb{R}^d)}f(\Gamma(w+x(\cdot))(s))\nu^{*}(dw,dx)(\omega).
 % \end{equation*}
 % for a.a. $\omega$.  
 %Consequently
 \begin{equation*}
 E\left|E_{\nu^{*}(\omega)}\left[ f(w)-f(\Gamma(w+x(\cdot))(s))\right]\right|=\lim_{m\rightarrow\infty}E\left|E_{\nu^{r_{m}}(\omega)}\left[ f(w)-f(\Gamma(w+x(\cdot))(s))\right]\right|.
 \end{equation*}
 The expectation on the right side is bounded above by
 \begin{align*}
 &E\left[\left|\frac{1}{T_{r_{m}}}\int_{0}^{T_{r_{m}}} f(\hat{W}^{r_{m}}(u+s))-f(\Gamma(\hat{W}^{r_{m}}(u)+\hat{X}^{r_{m}}(u+\cdot)-\hat{X}^{r_{m}}(u))(s))du\right|\right]\\
 &+ E\left[\left|\frac{1}{T_{r_{m}}}\int_{0}^{T_{r_{m}}} f(\hat{W}^{r_{m}}(u))-f(\hat{W}^{r_{m}}(u+s))du\right|\right].
 \end{align*}
 Due to Proposition \ref{thm:skorokApprox}, Proposition \ref{thm:idleTimeBndMark}, Propositions \ref{thm:nextArrTimeBnd} and \ref{thm:nextSerTimeBnd}, the assumption that $\sup_{r}\hat{\Upsilon}^{r}<\infty$, and the fact that $f$ is uniformly continuous and bounded we have
 \begin{equation*}
 \lim_{m\rightarrow\infty}E\left|\frac{1}{T_{r_{m}}}\int_{0}^{T_{r_{m}}} f(\hat{W}^{r_{m}}(u+s))-f(\Gamma(\hat{W}^{r_{m}}(u)+\hat{X}^{r_{m}}(u+\cdot)-\hat{X}^{r_{m}}(u))(s))du\right|=0.
 \end{equation*}
 In addition, since $f$ is bounded and $T_{r_{m}}\rightarrow\infty$ we have
 \begin{equation*}
 \lim_{m\rightarrow\infty}E\left[\left|\frac{1}{T_{r_{m}}}\int_{0}^{T_{r_{m}}} [f(\hat{W}^{r_{m}}(u))-f(\hat{W}^{r_{m}}(u+s))]du\right|\right]=0.
 \end{equation*}
Together these observations show $E\left[\left|E_{\nu^{*}(\omega)}\left[ f(w)-f(\Gamma(w+x(\cdot))(s))\right]\right|\right]=0$.
By standard separability arguments it now follows that  $w\overset{d}{=}\Gamma(w+x(\cdot))(s)$ for all $s\in[0,1]$ under $\nu^{*}(\omega)$ for a.e. $\omega$ which proves part (b).
 \hfill \qed

%%%%%%%%%%%%%%%%%%%%%%%%%%%%%%%%%%%%%%%%%%%%%%
%% Single Appendix:                         %%
%%%%%%%%%%%%%%%%%%%%%%%%%%%%%%%%%%%%%%%%%%%%%%
%\begin{appendix}
%\section*{???}%% if no title is needed, leave empty \section*{}.
%\end{appendix}
%%%%%%%%%%%%%%%%%%%%%%%%%%%%%%%%%%%%%%%%%%%%%%
%% Multiple Appendixes:                     %%
%%%%%%%%%%%%%%%%%%%%%%%%%%%%%%%%%%%%%%%%%%%%%%
%\begin{appendix}
%\section{???}
%
%\section{???}
%
%\end{appendix}

%%%%%%%%%%%%%%%%%%%%%%%%%%%%%%%%%%%%%%%%%%%%%%
%% Support information, if any,             %%
%% should be provided in the                %%
%% Acknowledgements section.                %%
%%%%%%%%%%%%%%%%%%%%%%%%%%%%%%%%%%%%%%%%%%%%%%
\begin{acks}[Acknowledgments]
We would like to thank Prof. Mike Harrison whose many questions and suggestions significantly improved the exposition of the work.
\end{acks}
%%%%%%%%%%%%%%%%%%%%%%%%%%%%%%%%%%%%%%%%%%%%%%
%% Funding information, if any,             %%
%% should be provided in the                %%
%% funding section.                         %%
%%%%%%%%%%%%%%%%%%%%%%%%%%%%%%%%%%%%%%%%%%%%%%
\begin{funding}
Research of the first author was supported in part by  awards DMS-2134107, DMS-2152577 from the NSF.  
% The second author was supported in part by ...
\end{funding}

%%%%%%%%%%%%%%%%%%%%%%%%%%%%%%%%%%%%%%%%%%%%%%
%% Supplementary Material, including data   %%
%% sets and code, should be provided in     %%
%% {supplement} environment with title      %%
%% and short description. It cannot be      %%
%% available exclusively as external link.  %%
%% All Supplementary Material must be       %%
%% available to the reader on Project       %%
%% Euclid with the published article.       %%
%%%%%%%%%%%%%%%%%%%%%%%%%%%%%%%%%%%%%%%%%%%%%%
%\begin{supplement}
%\stitle{???}
%\sdescription{???.}
%\end{supplement}

%%  MR numbers will be added by VTeX.                      %%
%%                                                         %%
%%  Use \cite{...} to cite references in text.             %%
%%                                                         %%
%%%%%%%%%%%%%%%%%%%%%%%%%%%%%%%%%%%%%%%%%%%%%%%%%%%%%%%%%%%%%

%% if your bibliography is in bibtex format, uncomment commands:
\bibliographystyle{imsart-number} % Style BST file (imsart-number.bst or imsart-nameyear.bst)
\bibliography{networks}       % Bibliography file (usually '*.bib')

%% or include bibliography directly:
% \begin{thebibliography}{}
% \bibitem{b1}
% \end{thebibliography}

\end{document}